\documentclass{amsart}

\usepackage{geometry}
\usepackage{amsmath}
\usepackage{amssymb}
\usepackage{amsthm}
\usepackage{color}
\usepackage[shortlabels]{enumitem}

\usepackage{hyperref}

\newtheorem{prop}{Proposition} 
\newtheorem{cor}{Corollary}
\newtheorem{lemma}{Lemma}
\newtheorem{thm}{Theorem}
\newtheorem{defn}{Definition}

\numberwithin{equation}{section}

\begin{document}

\title{On the mean square gap between primes}
\date{}
\author{Julia Stadlmann}\thanks{\hspace{-5.4mm} University of Illinois at Urbana-Champaign, Department of Mathematics, 1305 W Green St, Urbana, IL, 61801, USA \\ Email address: stadl@illinois.edu \\ 
MSC-class: 11N05 (primary), 11N36, 11M41 (secondary)}

\begin{abstract}
We prove that the average size of the squares of differences between consecutive primes less than $x$ is $O(x^{0.23+\varepsilon})$ for any fixed $\varepsilon>0$. This improves on a result of Peck, who gave bound $O(x^{0.25+\varepsilon})$ in the place of $O(x^{0.23+\varepsilon})$. Key ingredients are Harman's sieve,  Heath-Brown's mean value theorem for sparse Dirichlet
polynomials and Heath-Brown's  $R^*$ bound.
\end{abstract}

\maketitle

\vspace{-5mm}

\section{Introduction} \label{sec:intro}

Let $p_1, p_2, \dots$ denote the sequence of primes and  write $\pi(x) = \#\{n: p_n \leq x\}$. We study the average size of the squares of gaps between consecutive primes less than $x$, given by $\frac{1}{\pi(x)}\sum_{p_n \leq x} (p_{n+1}-p_n)^2$. Under assumption of the Riemann hypothesis, Selberg~\cite{Selberg:1943:RH} first   showed that
\begin{align*} 
\sum_{p_n \leq x} (p_{n+1}-p_n)^2 \ll  x \log(x)^3.
\end{align*} 
Assuming the Lindelöf hypothesis, Yu~\cite{Yu:1996:DCP} proved that $\sum_{p_n \leq x} (p_{n+1}-p_n)^2 = O_\varepsilon(x^{1+\varepsilon})$ for every $\varepsilon>0$,  implying  that the average size of $(p_{n+1}-p_n)^2$ with $p_n \leq x$ is $O_\varepsilon(x^\varepsilon)$. 
 
 \vspace{3mm}
A first unconditional result was given by Heath-Brown in~\cite{Heath-Brown:1978:CP1}. He showed that 
\begin{equation} \label{equ:mainsumnew}
\sum_{p_n \leq x} (p_{n+1}-p_n)^2 \ll_\varepsilon  x^{1+\nu + \varepsilon}
\end{equation} holds with $\nu=1/3= 0.\dot{3}$. In~\cite{Heath-Brown:1997:CP3} he improved this bound to $\nu=5/18=0.2\dot{7}$. The problem was further studied by Peck~\cite{Peck:1996:Thesis} and later Maynard~\cite{Maynard:2012:DCP}, who both obtained (\ref{equ:mainsumnew}) with $\nu = 1/4=0.25$. 
New phenomena occur when $\nu<0.25$ and handling these issues is the key technical innovation of our
paper.  We will prove the following result: 
\begin{thm} \label{thm:theorem1}
For any $\varepsilon>0$, we have 
\begin{align*}
\sum_{p_n \leq x} (p_{n+1}-p_n)^2 \ll_\varepsilon  x^{1.23 + \varepsilon}.
\end{align*}
\end{thm}
At the core of our proof is a parametric version of Harman's sieve~\cite{Harman:2007:PDS}. To obtain the information required for the application of this sieve, we rely on a complicated combination of bounds on large
values of Dirichlet polynomials, in particular  Heath-Brown's $R^*$
bound~\cite{Heath-Brown:1979:ZDE} and Heath-Brown's sparse mean value theorem~\cite{Heath-Brown:2019:CP5}.

\subsection{Key ideas} \label{ssec:ideas}

First  we recall the argument of Peck. Given some value of $\tau$ and a prime gap $ p_{n+1}- p_n$ with $p_{n+1}-p_n>6x / \tau$ and $p_n \leq 2x$, we note that   $y\in (p_n, (p_n+p_{n+1})/2)$  implies that the interval $[y,y+y/\tau]$ contains no
primes. This differs
significantly from the expected number of primes in   an interval of length $y/\tau$ and, by Perron's formula, 
roughly implies that 
$$ \Bigg| \int_{\tau}^{2\tau}y^{it} \sum_{p <2x}  p^{-it}\, \mbox{d}t \Bigg|  $$
is unusually large. We can show this happens rarely by combining
$L^\infty$, $L^2$ and $L^4$ bounds over $y\in[x,2x]$. To get adequate bounds we
apply a combinatorial decomposition to the primes and reduce the
situation to obtaining suitable mean value estimates for products of Dirichlet
polynomials. Doing this for all relevant values of $\tau$ gives result (\ref{equ:mainsumnew}) with $\nu =0.25$.

\vspace{3mm}
Unfortunately, this argument breaks down when $\nu<1/4$, as there is a
discontinuity in the estimates of Peck~\cite{Peck:1996:Thesis} and Maynard~\cite{Maynard:2012:DCP}. By refining the
underlying estimates, most of the problematic terms can be handled,
but products of  4 Dirichlet polynomials of length roughly $x^{1/4}$ cannot be treated
using existing mean value theorems. (Limitations correspond to the $3/4$--line, just like in many similar problems.)  To handle this
issue, we instead count primes with a minorant chosen via Harman's sieve, allowing us to exclude some small number of bad terms. In principle,  this new argument should  then be sufficient to give   (\ref{equ:mainsumnew}) with $\nu=1/4-\delta$ for some very small $\delta>0$.

\vspace{3mm}
However, as $\nu$ decreases further away from $1/4$,  we quickly encounter many new bad factor lengths,
which cannot be treated using the  approaches of Peck~\cite{Peck:1996:Thesis} and Maynard~\cite{Maynard:2012:DCP}. (Here  issues are also no longer only clustered  around the $3/4$--line, there are many bad Dirichlet polynomial sizes.) To handle many of these new bad terms, we introduce a new
estimate based on a mean value theorem of Heath-Brown for sparse Dirichlet
polynomials, which combines well with the previous techniques, particularly Heath-Brown's $R^*$ bound.  Which terms still cause issues depends on the
parameter $\tau$  and so we require a  parametric version of
Harman's sieve. In this parametric version of the sieve,  the retained terms,  which must have no bad sizes, must  depend on
$\tau$ in a suitably concrete way that we can verify certain integral
computations for all relevant values of $\tau$ and thus can ensure that the  minorant constructed
produces primes. It is handling these difficulties where our key new 
ideas appear.

\subsection{Proof outline} \label{ssec:outline}

To prove Theorem~\ref{thm:theorem1}, we must show that the interval $[y,y+y/\tau]$ with $\tau \leq x^{0.77}$ contains primes for most choices of $y \in [x,3x] \cap \mathbb{N}$ -- no more than  $O(\tau x^{0.23+\varepsilon})$ exceptions are allowed. When $\tau$ is large, $[y,y+y/\tau]$ is a short interval and we lack good estimates for $\pi(y+y/\tau)-\pi(y)$, the number of primes it contains. However, we guess that for most choices of $y$, the number of primes in short interval $[y,y+y/\tau]$ is roughly proportional to that in the longer interval $[y,y+y/x^b]$, where $b$ is a small constant. This larger interval is known to contain $(1+o(1))y/(x^b \log(x))$ primes. Hence if $y$ satisfies 
\begin{align} \label{introcomparison}
(\pi(y+y/\tau)-\pi(y))- \left(\dfrac{x^b}{\tau}\right)(\pi(y+y/x^b)-\pi(y)) > - \dfrac{\alpha y}{\tau \log(x)}
\end{align} for some fixed $\alpha \in (0,1)$, then $[y,y+y/\tau]$ contains primes. To show that such an inequality indeed holds for all but $O(\tau x^{0.23+\varepsilon})$ choices of $y$, we use the Buchstab identity, an inclusion--exclusion argument, which lets us compare the number of primes in $[y,y+y/\tau]$ and $[y,y+y/x^b]$ by comparing the number of multiples of products of primes of certain sizes. These multiples  sometimes provide additional structure in computations, making them easier to count than the primes  themselves.  This  process of comparison of multiples of certain primes in a short and in a long interval is known as Harman's sieve. Our application of Harman's sieve and proof of Theorem~\ref{thm:theorem1} are split into three propositions and one lemma,  each corresponding to one section of this paper. The precise formulation of these results can be found in Section~\ref{sec:key}, but we first give a rough outline of the main arguments. For clarity of exposition, details are suppressed. 

\vspace{3mm}
In Section~\ref{sec:sums} we compare certain sums over  $[y,y+y/\tau]$ and $[y,y+y/x^b]$. We  consider
\begin{align} \label{introdifference}
\sum_{\substack{k_1 \dots k_{r} \in [y,y+y/\tau] \\ k_i \in (N_i,2N_i]  }}  \left(\prod_{1 \leq i \leq r} a^{({i})}_{k_{i}} \right) - \dfrac{x^b}{\tau}\sum_{\substack{k_1 \dots k_{r} \in [y,y+y/x^b] \\ k_i \in (N_i,2N_i]  }}  \left(\prod_{1 \leq i \leq r} a^{({i})}_{k_{i}} \right),
\end{align}
where $a_{n}^{(i)} = 1_{\mathbb{P}}(n)$ (the prime indicator function) or $a_{n}^{(i)} = 1$  whenever $N_i$ is large.  We also consider corresponding Dirichlet polynomials $F(s) = \prod_{i=1}^j S_i(s)$ which look  roughly like $\prod_{i=1}^r \sum_{k_i \in (N_i, 2N_i]} a_{k_i}^{(i)} k_i^{-s}$, except factors of the form $ \sum_{ k_i \in (N_i, 2N_i]} 1_{\mathbb{P}}(k_i) k_i^{-s}$ are  replaced by their combinatorial decomposition into smooth or short factors via Heath--Brown's identity.  By Perron's formula, (\ref{introdifference}) is $O_A(x/(\tau \log(x)^A))$ unless there is some corresponding Dirichlet polynomial $F(s) = \prod_{i=1}^j S_i(s)$ for which    the integral
\begin{align} \label{introintegral} \dfrac{1}{\tau}\left| \int^{\tau}_{x^b} y^{1+it} F(1+it)\,\mbox{{\rm d}}t \right|
\end{align}
fails to be bounded by $x/(\tau \log(x)^A)$. Denoting by $B_i$ the length of Dirichlet polynomial $S_i(s)$, we split the domain of integration of (\ref{introintegral}) into $O(\log(x)^{j+1})$ regions, considering separately the contribution of  $\bigcup_{n \in \mathcal{S}_1} [n,n+1]$ where $\mathcal{S}_1$ is the set of $n \in [T,2T] \cap \mathbb{N} \subseteq [x^b,\tau]$ for which each factor $S_i(1+it)$ is of size between $B_i^{\sigma_i-1}$ and $2B_i^{\sigma_i-1}$ on $[n,n+1]$  and for  which $F(1+it)$ has a total size of about $x^{\sigma-1}$ on $[n,n+1]$.  If $\#\mathcal{S}_1<x^{1-\sigma-\varepsilon}$,  this region contributes only little to (\ref{introintegral}) for every $y$ and can be ignored. If instead $\#\mathcal{S}_1\geq x^{1-\sigma -\varepsilon}$, we wish to show that the contribution of  $\bigcup_{n \in \mathcal{S}_1}[n,n+1]$ to (\ref{introintegral})  exceeds $x/(\tau \log(x)^{A+j+1})$ for only $O(\tau x^{0.23+\varepsilon})$ values $y \in [x,3x] \cap \mathbb{N}$.
 Taking a second moment over $y \in [x,3x]$,  
 \begin{align*}  \int_x^{3x} \left| \sum_{n \in \mathcal{S}_1}\int^{n+1}_{n} y^{1+it} F(1+it)\,\mbox{{\rm d}}t \right|^2 \mbox{{\rm d}}y,
\end{align*}
we get the desired bound on the number of bad $y$ whenever $\#\mathcal{S}_1< \tau x^{1-2\sigma+0.23}$. We denote $\#\mathcal{S}_1$ by $R(F,(S_i), (\sigma_i),T)$.  Alternatively, by taking a fourth moment, 
we find that we  also have a suitable bound if the number of quadruples $n_1,\dots,n_4 \in \mathcal{S}_1$ with $n_1+n_2=n_3+n_4$ is bounded above by $\tau x^{3-4\sigma+0.23}$. We denote this number of quadruples by $R^*(F,(S_i), (\sigma_i),T)$.  For a third approach, we denote by $\mathcal{J}$ the set of $m \in [1,\tau ] \cap \mathbb{N}$ for which the contribution of $\bigcup_{n \in \mathcal{S}_1} [n,n+1]$ to (\ref{introintegral}) with $y =  \lceil 3x/\tau \rceil m$ exceeds $x/(\tau \log(x)^{A+j+1})$. Effectively, we sample bad $y$, leaving large gaps to ensure that $ \# \mathcal{J} \leq \tau$. Multiplying  by some $\xi_m \in \{z \in \mathbb{C}: |z|=1\}$ to remove  absolute value signs and summing over $\mathcal{J}$, we find 
\begin{align*}  \# \mathcal{J}  \leq \log(x)^{A+j+2} \left|\sum_{n \in \mathcal{S}_1}\int^{n}_{n+1}  \left(\sum_{m \in \mathcal{J}} \xi_m   m^{it} \right) F(1+it)\,\mbox{{\rm d}}t \right|.
\end{align*}
We have introduced a second Dirichlet polynomial $G(s) = \sum_{ m\in \mathcal{J}} \overline{\xi_m} m^{-s}$. If the integral on the RHS is small, there are only  few bad $y$.  We split up the sum, restricting our attention to $n \in \mathcal{S}_2$, where  $\mathcal{S}_2$ is the set of $n \in [T,2T] $ for  which  the size of $G(it)$ is contained in a prescribed dyadic interval $(L^\gamma, 2L^\gamma]$. We find that the desired bound $ \#\mathcal{J} = O(\tau^2 x^{0.23-1})$ holds if $\#(\mathcal{S}_1 \cap \mathcal{S}_2) \leq (\#\mathcal{J})^{7/8-\gamma} \tau^{1/4} x^{7/8-\sigma+0.23/8}$ or $\#(\mathcal{S}_1 \cap \mathcal{S}_2) \leq (\#\mathcal{J})^{1-\gamma} x^{1-\sigma}$. We write $\# (\mathcal{S}_1 \cap \mathcal{S}_2) = Q(F,G,(S_i),(\sigma_i),\gamma,T)$ and have reduced the problem of obtaining a  good bound of $O(x/(\tau \log(x)^A))$  on (\ref{introcomparison}) to showing that $R(F,(S_i), (\sigma_i),T)$, $R^*(F,(S_i), (\sigma_i),T)$ and $Q(F,G,(S_i), (\sigma_i), \gamma,T)$ satisfy one of five inequality conditions for every choice of $(\sigma_i)$, $\gamma$ and $T$. The formal statement of this result is Proposition~\ref{proposition1}, which is proved in Section~\ref{sec:sums}.

\vspace{3mm}
In Section~\ref{sec:values} we then give various bounds for $R$, $R^*$ and $Q$. We use Heath-Brown's bound on sparse Dirichlet polynomials to replace condition  $Q\leq \max\{(\#\mathcal{J})^{7/8-\gamma} \tau^{1/4} x^{7/8-\sigma+0.23/8},(\#\mathcal{J})^{1-\gamma} x^{1-\sigma} \}$  by an  alternative condition $R \leq  \min\{\tau^{-1/4}M^{2\beta-1}  x^{7/4-2\sigma+\nu/4},M^{2\beta-2}x^{2-2\sigma}\}$ where $M = \prod_{i \in I} B_i^{m_i}$ and $M^\beta = \prod_{i \in I} B_i^{m_i \sigma_i}$ for some $I \subseteq \{1, \dots, j\}$ and $m_i \in \mathbb{N}$. We also  use Heath-Brown's $R^*$ bound  to replace condition $R^* \leq \tau x^{3-4\sigma+0.23}$ by an alternative condition $ R \leq  \min_{(U,V,W,X,Y,Z) \in \mathcal{Z}} \tau^{U}M^{V\beta-W}  x^{X-Y\sigma+0.23Z
} $, where $\mathcal{Z}$ is a set of  nine tuples $(U,V,W,X,Y,Z)$ specified  in Section~\ref{sec:values}.  This reduces the problem of bounding (\ref{introdifference}) to obtaining good upper bounds on $R(F,(S_i),(\sigma_i),T)$, the number of intervals $[n,n+1]$ on which Dirichlet polynomials $S_i(s)$  are of size about $B_i^{\sigma_i-1}$. We then use  results such as   Montgomery's mean value estimate and
Huxley's large values estimate to obtain such bounds. These  are very sensitive to the lengths $B_i$ of Dirichlet polynomials $S_i(s)$. We show that we get adequate bounds on $R$ provided that some combination of factors $S_i(s)$ is of a convenient length (depending on $\tau$). The formal statement of this result is Proposition~\ref{proposition2}, proved in Section~\ref{sec:values}.

\vspace{3mm} In Section~\ref{sec:comparison} we  use our bounds on the differences (\ref{introdifference}) to prove that if $P_1, \dots, P_r$ are in certain good ranges, then there exists  $\mathcal{I} \subseteq [x,3x] \cap \mathbb{N}$ with $\# \mathcal{I} = O(\tau x^{0.23 + \varepsilon})$ such that $y \in ([x,3x] \cap \mathbb{N}) \setminus \mathcal{I}$ implies 
\begin{align} \label{approxdifference}
\sum_{ \substack{k_1 \dots k_r m \in [y,y+y/\tau]\\ k_i \in (P_i,2P_i] \\ (m, p)=1 \, \mbox{\scriptsize for } p < z(p_k)}} \!\!\!\!\!1_{\mathbb{P}}(k_1) \dots 1_{\mathbb{P}}(k_r) - 
\dfrac{x^b}{\tau} \sum_{ \substack{k_1 \dots k_r m \in [y,y+y/x^b]\\ k_i \in (P_i,2P_i] \\ (m, p)=1 \, \mbox{\scriptsize for } p < z(p_k)}}  \!\!\!\!\! 1_{\mathbb{P}}(k_1) \dots 1_{\mathbb{P}}(k_r) \approx 0.
\end{align}
Here $z$ is a function of $p_k$ with $z(p_k) = p_k$ or $z(p_k)=x^\beta$ for some small $\beta$.  
This allows us to compare how many elements of $[y,y+y/\tau]$  and $[y,y+y/x^b]$ are of the form $p_1 \dots p_r m$ with $p_i \in (P_i,2P_i]$ and $m$ not divisible by any prime less than $z(p_k)$. The formal statement of this result is Lemma~\ref{lemma1}.

\vspace{3mm} Finally, recall that we are looking to bound  $(\pi(y+y/\tau)-\pi(y))- (x^b/\tau)(\pi(y+y/x^b)-\pi(y))$ from below. We  may use the Buchstab identity, an inclusion--exclusion argument,  to rewrite the difference $(\pi(y+y/\tau)-\pi(y))- (x^b/\tau)(\pi(y+y/x^b)-\pi(y))$ as a sum 
 of terms which look like the LHS of (\ref{approxdifference}) weighted by $(-1)^r$. From Lemma~\ref{lemma1} we know that the LHS of (\ref{approxdifference}) is close to zero for most $y$ when $P_1, \dots, P_r$ are in certain good ranges. On the other hand, when  $P_1, \dots, P_r$ are bad and $r$ is even, then the LHS of  (\ref{approxdifference}), weighted by $(-1)^r=1$, can  be bounded below by discarding the contribution of short interval $[y,y+y/\tau]$ and estimating the sum over long interval $[y,y+y/x^b]$ asymptotically. This is possible because counting multiples of certain primes in a long interval is easy. 

\vspace{3mm} 
 If we choose our Buchstab decomposition so that the overall bound on the bad $[y,y+y/x^b]$--sums does not exceed $(\alpha y) / (x^b \log(x))$ 
for some fixed $\alpha \in (0,1)$, then this gives the desired lower bound (\ref{introcomparison}) for all but $O(\tau x^{0.23+\varepsilon})$ choices of $y \in [x,3x] \cap \mathbb{N}$, which in turn proves Theorem~\ref{thm:theorem1}.  In practice, finding  suitable  Buchstab iterations is a long and calculation-heavy process. These computations,  reformulated as seeking a suitable minorant for $1_{\mathbb{P}}(n)$, can be found  Section~\ref{sec:harman}. Since the good ranges of $P_i$ vary with $\tau$, we apply Harman's sieve  six separate times. The formal statement of the results of Section~\ref{sec:harman} is  Proposition~\ref{proposition3}. 
 
\subsection{Notation} \label{ssec:notation} We will use the following notation from sieve theory: For $\mathcal{C} \subseteq \mathbb{N}$ and $z>0$,    
\begin{align*}
S(\mathcal{C},z) = \# \{n \in \mathcal{C}: (n,P(z))=1\} \quad \mbox{ where } \quad P(z)=\prod_{p< z} p.
\end{align*} 
We also let $V(z)=\prod_{p < z} (1-1/p)$ and $\mathcal{C}_d = \{ n \in \mathbb{N}: nd \in \mathcal{C}\}$. 
For any $n \in \mathbb{N} $ and any $z > 0$, we take
\begin{align*}
\psi(n,z) = \begin{cases}
&1 \quad  \mbox{ if } \, p \mid n \Rightarrow p \geq z\\
&0 \quad \mbox{ otherwise. }
\end{cases}
\end{align*}
In particular, $\sum_{md \in \mathcal{C}} \psi(m,z) = S(\mathcal{C}_d,z)$. 

 \vspace{3mm}
$1_\mathcal{C}$ is the indicator function of $\mathcal{C}$ and $\mathbb{P}$ is  the set of primes.
We  write $n \sim N$ to mean $n \in (N, 2N]\cap\mathbb{N}$.

 \vspace{3mm}
Throughout this paper, we will work with a fixed, but arbitrary constant $\varepsilon>0$ and with quantities $\nu = 0.23$, $b = 1/10^{5}$, $\varepsilon_1 = \varepsilon/10^5 $,  $T_0= \tau  x^{\varepsilon_1}$ and $T_1= \gamma x^{b/2-\varepsilon_1}$, where $\gamma \in [1,2]$ is chosen such that $\log(T_0/T_1)/\log(2) \in \mathbb{Z}$, allowing for a nice dyadic decomposition of $[T_1,T_0]$. 
We  compare sets
\begin{align*}
\mathcal{A}(y) = \left[y, y+\dfrac{y}{\tau}\right] \cap \mathbb{N} \quad \mbox{ and } \quad 
\mathcal{B}(y) = \left[y, y+\dfrac{y}{x^b}\right] \cap \mathbb{N}.
\end{align*} 
  Throughout all our proofs, we will  assume without loss of generality that $x$ is very large and $\varepsilon$ is  small.

\section{Key propositions} \label{sec:key}

In this section we now introduce important definitions, formally state our key propositions and lemma and then show that they combine to give a proof of Theorem~\ref{thm:theorem1}.

\subsection{About Proposition 1} \label{ssec:proposition1}
Our first proposition relates bounds on  the differences between sums over $\mathcal{A}(y)$ and $\mathcal{B}(y)$ to values of Dirichlet polynomials. For  fixed $\varepsilon>0$ and $\tau \in [x^{0.475-\varepsilon},x^{0.77-\varepsilon}]$ we consider
\begin{align*}
\sum_{n \in \mathcal{A}(y)} a_n - \dfrac{x^b}{\tau}\sum_{n \in \mathcal{B}(y)} a_n,
\end{align*}
where $(a_n)$ is a sequence with the following properties: There exist  $K\in \mathbb{N}$, $N_i \in [ \frac{1}{2}, \infty)$, $r_1, r_2\in \mathbb{N}$ with   $0 \leq r_1  \leq r_2 \leq 10^5$ and  $a_{n}^{(i)} \in \mathbb{C}$ with $a_n^{(i)}=O( \log(n+3))$ for $i > r_1$ such that 
\begin{align*}
a_n= \sum_{\substack{k_0k_1 \dots k_{r_2} =n \\ k_i \sim N_i  }} \left( 1_{\mathbb{N}}(k_0) \prod_{1 \leq i \leq r_1} 1_\mathbb{P}(k_i) \prod_{r_1< i \leq r_2} a^{({i})}_{k_{i}} \right).
\end{align*}
Here we require that $2^{-r_2}x \leq \prod_{i=0}^{r_2} N_i \leq 6x$ and  $N_i < x^{1/K}$ for $i>r_1$ and $N_i \geq x^{1/K}$ for $1 \leq i \leq r_1$.  
We additionally assume that $r_1$, $N_0, \dots, N_{r_2}$ and $(a_n^{(i)})$ satisfy one of the following three options:  
 \begin{enumerate}[{\rm (i)}] 
 \item $r_1\geq 2$ or $N_0 \geq x^{0.95}$.
 \item   $r_1 =1$  and  $N_0 \geq x^{1/\log\log x}$.
 \item $r_1 \in \{0,1\}$  and there exists $i>r_1$ with $a_n^{(i)} = 1_{\mathbb{P}}(n)$ for all  $n \in \mathbb{N}$  and $N_i \geq x^{1/\log\log x}$.
 \end{enumerate}

\vspace{3mm} We will associate with $(a_n)$ a collection of Dirichlet polynomials with various nice properties. In particular, the factors of these Dirichlet polynomials will be of the form described below.

\begin{defn} \label{def:propertiesz1z2d}
For a given $T \in [1,x]$ and a Dirichlet polynomial $H(s) = \sum_{n \sim N} h_n n^{-s}$, we say that $H(s)$ has property $(\ref{propertyz1})$  with respect to $T$ if there exists  $w \in [-T/2,T/2]$ such that 
\begin{align}
H(s+iw) = \sum_{n \sim N} \dfrac{1}{n^s} \quad \quad &\mbox{ or } \quad \quad H(s+iw) = \sum_{n \sim N} \dfrac{\log(n)}{n^s}. \label{propertyz1} \tag{Z1}
\end{align}
 We say that $H(s)$ has property $(\ref{propertyz2})$ with respect to $T$ if there exists  $w \in [-T/2,T/2]$ such that 
\begin{align}
H(s+iw) =  \sum_{k=2}^\infty \sum_{p^k \sim N} \dfrac{1}{k p^{ks}}. \tag{Z2} \label{propertyz2}
\end{align}
We say that $H(s)$ has property $(\ref{propertyd})$ with respect to $T$  if there  exists $w \in [-T/2,T/2]$ such that $H(s+iw)$ equals one of the following five Dirichlet polynomials: 
\begin{align}
H(s+iw) \in \left\{\sum_{n \sim N} \dfrac{1}{n^s},  \quad \sum_{n \sim N} \dfrac{\log(n)}{n^s},   \quad   \sum_{k=2}^\infty \sum_{p^k \sim N} \dfrac{1}{k p^{ks}},  \quad \sum_{n \sim N} \dfrac{1_{\mathbb{P}}(n)}{n^s}, \quad \sum_{n \sim N} \dfrac{\mu(n)}{n^s} \right\}. \label{propertyd} \tag{D} 
\end{align}
\end{defn}

For  $(a_n)$ with corresponding $(N_0, \dots, N_{r_1})$ and $K$, we consider the following sets of Dirichlet polynomials: 

\begin{defn} \label{def:polynomialcollections}
 $\mathcal{F}^*(T,J,K)$ is the set of Dirichlet polynomials $F(s)$  with  $$F(s) = \prod_{i=1}^j S_i(s)$$ for  some $j\leq J$ and some $S_i(s) = \sum_{n \sim B_i}  b_n^{(i)} n^{-s}$ which satisfy:
 \begin{enumerate}[{\rm (1)}]
  \item   $b_n^{(i)} \in \mathbb{C}$ with $b_n^{(i)} =O( \log(n))$. 
  \item $B_i \geq \log(x)$ and $2^{-J}x \leq \prod_{i=1}^j B_i \leq 2^J x$.
  \item If  $B_\ell \geq x^{2/K}$, then  $S_\ell(s)$ has property $(\ref{propertyz1})$ or $(\ref{propertyz2})$ with respect to $T$. 
 \item At least one of  the following three  conditions is satisfied: 
 \begin{enumerate}[{\rm (A)}]
 \item There exists $\ell \in \{1, \dots, j\}$ with $B_\ell \geq x^{2/K}$ and there exists $i_0 \neq \ell$ such that $B_{i_0} \geq x^{1/\log\log(x)}$ and such that $S_{i_0}(s)$ has property $(\ref{propertyd})$ with respect to $T$.
 \item For every $i \in \{1, \dots, j\}$, $B_i < x^{2/K}$, and there exists $i_0 \in \{1, \dots, j\}$ such that $B_{i_0} \geq x^{1/\log\log(x)}$ and such that $S_{i_0}(s)$ has property $(\ref{propertyd})$ with respect to $T$.
 \item There exists $\ell \in \{1, \dots, j\}$ with $B_\ell \geq x^{0.95}$.
 \end{enumerate}
  \end{enumerate}

\vspace{3mm}
$\mathcal{F}((N_0, \dots, N_{r_1}),T,J,K)$ is the set of $F(s) \in \mathcal{F}^*(T,J,K)$ for which factorisation  $\prod_{i=1}^j S_i(s)$ can be chosen so that   the following three  conditions are all satisfied: 
 \begin{enumerate}[{\rm (i)}]
  \item There  exist disjoint subsets $X_1, \dots, X_{r_1}$ of $\{1, \dots, j\}$ with 
\begin{align*} 
\log(x)^{-2K-1} N_\ell \leq \prod_{i \in X_\ell} B_i \leq \log(x)^{2K+1} N_\ell \quad \mbox{ \rm for } \quad 1 \leq \ell \leq r_1.
\end{align*}
   \item If $N_0 \geq \log(x)$,  there also exists $i_0 \in \{1, \dots, j\} \setminus (\cup_{\ell=1}^{r_1} X_\ell)$ with $S_{i_0}(s) = \sum_{k \sim N_0} k^{-s}$.
    \item  If $i \not\in \bigcup_{\ell=1}^{r_1} X_\ell$ and $i \neq i_0$, then $B_i <  x^{1/K}$. 
 \end{enumerate}
 \vspace{3mm}

 $\mathcal{G}(\tau,\varepsilon)$ is the set of Dirichlet polynomials with the following property:
 If $G(s) \in \mathcal{G}(\tau,\varepsilon)$, then   $G(s) = \sum_{n \in \mathcal{J}} \xi_n n^{-s}$ for some $\mathcal{J} \subseteq [1,\tau  x^{\varepsilon_1}]$ with $\#\mathcal{J} \geq \tau^2 x^{-0.77}$  and some $\xi_n \in \mathbb{C}$ with $|\xi_n| = 1$.
 \end{defn}
 
 Next we give the formal definition of quantities $R$, $R^*$ and $Q$, first introduced in Section~\ref{ssec:outline}. 

 \begin{defn}\label{def:counting}
Let $c = 1 +1/\log(x)$. For  Dirichlet polynomials  $F(s) = \prod_{i=1}^j S_i(s)= \prod_{i=1}^j (\sum_{n \sim B_i} \frac{b_n^{(i)} }{n^{s}})$, $G(s) = \sum_{n \in \mathcal{J}} \frac{\xi_n}{ n^{s}} $ and   $(\sigma_i) \in \mathbb{R}^j$, $\gamma \in \mathbb{R}$ and $T \in [T_1,T_0]$,  we define:
\begin{align}
&\mathcal{S}_1(F,(S_i),(\sigma_i))= \bigcap_{i=1}^j\Big\{n \in \mathbb{N}:  \sup_{t\in[n,n+1]}|S_i(c+it)| \in (B_i^{-c+\sigma_i},2B_i^{-c+\sigma_i}]  \Big\},
\nonumber \\
&R(F,(S_i),(\sigma_i),T)= \#(\mathcal{S}_1(F,(S_i),(\sigma_i)) \cap [ T,2T]), \label{firstR}
\\
&\mathcal{S}_1^*(F,(S_i),(\sigma_i)) = \left\{ (n_1,n_2,n_3,n_4) \in \mathcal{S}_1(F,(S_i),(\sigma_i))^4: n_1+n_2 = n_3+n_4 \right\}, \nonumber\\
&R^*(F,(S_i),(\sigma_i),T)= \#(\mathcal{S}_1^* (F,(S_i),(\sigma_i))\cap [T,2T]^4), \label{firstR*}
\\ 
&\mathcal{S}_2(G,\gamma)= \Big\{n \in \mathbb{N}:  \sup_{t\in[n,n+1]}|G(it)| \in ((\#\mathcal{J})^{\gamma},2(\#\mathcal{J})^{\gamma}] \Big\}, \nonumber\\ 
&Q(F,G,(S_i),(\sigma_i),\gamma,T)= \#(\mathcal{S}_1(F,(S_i),(\sigma_i))\cap\mathcal{S}_2(G,\gamma) \cap [T,2T]). \label{firstQ}
\end{align}
\end{defn} 

We are now ready to state Proposition~\ref{proposition1}.

\begin{prop} \label{proposition1}
Let $\varepsilon>0$, $\tau \in [x^{0.475-\varepsilon}, x^{0.77-\varepsilon}]$ and $A \in \mathbb{N}$. Suppose that $(a_n)$ is as described at the start of Section~{\rm\ref{ssec:proposition1}}, with corresponding constants $r_1, r_2$, $(N_0, \dots, N_{r_2})$ and $K$. Let $J=A+10^7K$.

\vspace{3mm}
 
Suppose each choice of  $T\in [T_1,T_0]$, $F(s) = \prod_{i=1}^j S_i(s)= \prod_{i=1}^j (\sum_{n \sim B_i} \frac{b_n^{(i)} }{n^{s}}) \in \mathcal{F}((N_0, \dots, N_{r_1}),T,J,K)$, $G(s) = \sum_{n \in \mathcal{J}} \xi_n n^{-s} \in \mathcal{G}(\tau,\varepsilon)$,    $\gamma \in (- \infty, 1)$ and $(\sigma_i)$ with  $\sigma_i \in \mathbb{R}$ and $\sigma = \frac{\sum_{i=1}^j \log(B_i) \sigma_i}{\sum_{i=1}^j \log(B_i) }$  satisfies    at least one of the following five conditions:
\begin{align} 
& R(F,(S_i),(\sigma_i),T) \leq \dfrac{x^{1-\sigma}}{ \log(x)^{2J}},  \tag{C1} \label{cc1}\\
&
R(F,(S_i),(\sigma_i),T) \leq \tau x^{1.23-2\sigma+2\varepsilon_1},
 \tag{C2} \label{cc2}
 \\ 
&
R^*(F,(S_i),(\sigma_i),T) \leq \tau x^{3.23-4\sigma+2\varepsilon_1}, \tag{C3} \label{cc3} \\
&  Q(F,G,(S_i),(\sigma_i),\gamma,T) \leq (\#\mathcal{J})^{7/8-\gamma}\tau^{1/4} x^{7/8-\sigma+0.23/8+\varepsilon_1}, \tag{C4.A} \label{cc4.A} \\
&  Q(F,G,(S_i),(\sigma_i),\gamma,T) \leq (\#\mathcal{J})^{1-\gamma} x^{1-\sigma-\varepsilon_1}. \tag{C4.B} \label{cc4.B}
\end{align} 
Then there exists $\mathcal{I} \subseteq [x,3x] \cap \mathbb{N}$ such that  $\#\mathcal{I} = O( \tau x^{0.23+\varepsilon/2})$ and  such that    $y    \in   ( [x,3x] \cap \mathbb{N}) \setminus \mathcal{I}$ implies
\begin{align*}
\left|\sum_{ n \in \mathcal{A}(y)} a_n - \dfrac{x^b}{\tau} \sum_{ n \in \mathcal{B}(y) } a_n \right|  \leq \dfrac{x}{\tau \log(x)^A} .
\end{align*}
Here the constant implied by big O notation only depends on $A$, $K$ and $\varepsilon$.
\end{prop}

\subsection{About Proposition 2} The second proposition then gives conditions on the factor lengths of $F(s) = \prod_{i=1}^j S_i(s)$ which ensure that one of  (\ref{cc1}), (\ref{cc2}), (\ref{cc3}), (\ref{cc4.A}) or (\ref{cc4.B}) holds. For a fixed $\varepsilon > 0$, $\varepsilon_1= \varepsilon/10^5$ and a given $a$, we  define a quantity  $\chi_0(a)$ and regions $\chi_1(a)$, $\chi_2(a)$ and $\chi_3(a)$  as follows: 
\begingroup
\allowdisplaybreaks
\begin{align} 
&\chi_0(a) = \begin{cases}
0.290-\varepsilon_1 \quad \mbox{ if } a \leq 0.53, \\
0.315-\varepsilon_1 \quad \mbox{ if } a \in (0.53,0.545], \\
0.335-\varepsilon_1 \quad \mbox{ if } a \in (0.545,0.57], \\
0.330-\varepsilon_1 \quad \mbox{ if } a \in (0.57,0.59], \\
0.330-\varepsilon_1 \quad \mbox{ if } a \in (0.59,0.61], \\
0.320-\varepsilon_1 \quad \mbox{ if } a > 0.61.
\end{cases} \label{chi0}\end{align}
\begin{align}
&\widetilde{\chi}_1(a) =  \begin{cases}
[0.290-\varepsilon_1, 0.360+\varepsilon_1]  \qquad\qquad\qquad\qquad\qquad\qquad\mbox{ if } a \leq 0.53, \\
[0.315-\varepsilon_1, 0.345+\varepsilon_1] \cup [0.427-\varepsilon_1, 0.474+\varepsilon_1]  \quad \mbox{ if } a \in (0.53,0.545], \\
[0.400-\varepsilon_1, 0.475+\varepsilon_1]\qquad\qquad\qquad\qquad\qquad\qquad \mbox{ if } a \in (0.545,0.57], \\
[0.380-\varepsilon_1, 0.420+\varepsilon_1] \qquad\qquad\qquad\qquad\qquad\qquad \mbox{ if } a \in (0.57,0.59], \\
[0.365-\varepsilon_1, 0.420+\varepsilon_1] \qquad\qquad\qquad\qquad\qquad\qquad \mbox{ if } a \in (0.59,0.61], \\
[0.355-\varepsilon_1, 0.420+\varepsilon_1] \qquad\qquad\qquad\qquad\qquad\qquad \mbox{ if } a  > 0.61.
\end{cases} \nonumber \\
&\chi_1(a) = \widetilde{\chi}_1(a) \cup (1 - \widetilde{\chi}_1(a))= \widetilde{\chi}_1(a) \cup \{1-\ell: \ell \in \widetilde{\chi}_1(a) \}. \label{chi1} \\ 
&\chi_2(a) =  \begin{cases}
\chi_1(a) \qquad\qquad\qquad\qquad\qquad\qquad\qquad \qquad\qquad\quad\,\,  \mbox{ if } a \leq 0.53, \\
[0.405-\varepsilon_1, 0.485+\varepsilon_1] \cup [0.515-\varepsilon_1, 0.595+\varepsilon_1]  \quad \mbox{ if } a \in (0.53,0.545], \\
\chi_1(a) \qquad\qquad\qquad\qquad\qquad\qquad\qquad \qquad\qquad\quad\,\, \mbox{ if } a \in (0.545,0.57], \\
[0.380-\varepsilon_1, 0.455+\varepsilon_1] \cup [0.545-\varepsilon_1, 0.620+\varepsilon_1]  \quad \mbox{ if } a \in (0.57,0.59], \\
[0.365-\varepsilon_1, 0.435+\varepsilon_1] \cup [0.565-\varepsilon_1, 0.635+\varepsilon_1]  \quad \mbox{ if } a \in (0.59,0.61], \\
\chi_1(a) \qquad\qquad\qquad\qquad\qquad\qquad\qquad \qquad\qquad\quad\,\, \mbox{ if } a > 0.61.
\end{cases}
 \label{chi2} \\
&\chi_3(a) =  \begin{cases}
\chi_1(a) \qquad\qquad\qquad\qquad\qquad\qquad\qquad \qquad\qquad\quad\,\,  \mbox{ if } a \leq 0.53, \\
[0.285-\varepsilon_1, 0.375+\varepsilon_1] \cup [0.625-\varepsilon_1, 0.715+\varepsilon_1]  \quad \mbox{ if } a \in (0.53,0.545], \\
\chi_1(a) \qquad\qquad\qquad\qquad\qquad\qquad\qquad \qquad\qquad\quad\,\, \mbox{ if } a \in (0.545,0.57], \\
[0.315-\varepsilon_1, 0.420+\varepsilon_1] \cup [0.580-\varepsilon_1, 0.685+\varepsilon_1]  \quad \mbox{ if } a \in (0.57,0.59], \\
[0.330-\varepsilon_1, 0.420+\varepsilon_1] \cup [0.580-\varepsilon_1, 0.670+\varepsilon_1]  \quad \mbox{ if } a \in (0.59,0.61], \\
\chi_1(a) \qquad\qquad\qquad\qquad\qquad\qquad\qquad \qquad\qquad\quad\,\, \mbox{ if } a > 0.61.
\end{cases}  \label{chi3}
\end{align}
 \endgroup

The following result then holds for  $\chi_0(a)$, $\chi_1(a)$, $\chi_2(a)$ and $\chi_3(a)$  as given in  $(\ref{chi0})$, $(\ref{chi1})$, $(\ref{chi2})$ and $(\ref{chi3})$:
\begin{prop} \label{proposition2} 
Let $\varepsilon>0$,  $a \in [0.475-\varepsilon,0.77-\varepsilon]$, $\tau = x^a$, $K=2000$, $J>10^7K$ and $T\in [T_1,T_0]$.  Let $\mathcal{F}^*(T,J,K)$ and $\mathcal{G}(\tau,\varepsilon)$ be as described in Definition~{\rm \ref{def:polynomialcollections}} and suppose that $F(s)= \prod_{i=1}^j S_i(s) = \prod_{i=1}^j\sum_{n \sim B_i}  b_n^{(i)} n^{-s} \in \mathcal{F}^*(T,J,K)$ and  $G(s) \in \mathcal{G}(\tau,\varepsilon)$.
 
\vspace{3mm}
Write $\ell_k = \frac{\log(B_k)}{\sum_{i=1}^j \log(B_i)}$ for $1 \leq k \leq j$.   Suppose we have one of the following three options:
\begin{enumerate} [{\rm (1)}]
\item There exists $k \in \{1, \dots, j\}$  with $\ell_k \geq \chi_0(a)$. 
\item There exists $I_1 \subseteq \{1,\dots,j\}$ with  $\sum_{i \in I_1} \ell_i  \in \chi_1(a) $.
\item There exist $I_2 \subseteq \{1,\dots,j\}$ with  $\sum_{i \in I_2} \ell_i  \in\chi_2(a) $ and 
$I_3 \subseteq \{1,\dots,j\}$ with  $\sum_{i \in I_3} \ell_i  \in  \chi_3(a)   $. 
\end{enumerate}
Then there exists a large constant $C$, dependent only on $ J$ and $\varepsilon$, such that for $x \geq C$     and for every $\gamma \in (- \infty, 1)$ and $(\sigma_i)$ with  $\sigma_i \in \mathbb{R}$ and $\sigma = \frac{\sum_{i=1}^j \log(B_i) \sigma_i}{\sum_{i=1}^j \log(B_i) }$,   one of the five conditions {\rm(\ref{cc1})}, {\rm(\ref{cc2})}, {\rm(\ref{cc3})}, {\rm(\ref{cc4.A})} or {\rm(\ref{cc4.B})}, described in Proposition~{\rm\ref{proposition1}}, holds.
\end{prop}

\subsection{About Lemma 1}  Proposition~\ref{proposition1} and~\ref{proposition2} are then combined in order to compare sifted sets.  Criteria concerning the size of Dirichlet polynomials are now replaced by  purely combinatorial conditions. 

\begin{defn} \label{def:combinatorialconditions} 

Let $\beta \in [0.01,0.15]$, $r \in \{0,\dots, 5\}$ and $\beta \leq \ell_i^*\leq 0.5+\varepsilon$,  $\ell_i^* \geq \ell^*_{i+1}$ and $\ell_1^*+ \dots + \ell_r^* \leq 0.75$.
We denote by $\Xi^\star(\ell_1^*, \dots, \ell_r^*, \beta)$ the set of finite sequences $\{\ell_i\}_{i=1}^j$ with $\ell_1, \dots, \ell_j \in [0,1]$, $\sum_{i=1}^j \ell_i =1$ and $j \leq 10^{20}$ for which there exist disjoint subsets $X_1, \dots, X_r$ of $\{1, \dots, j\}$ with
 \begin{align*}
 \sum_{i \in X_s} \ell_i \in \left[\ell_s^*-\dfrac{\varepsilon }{10^{100}}, \, \ell_s^*+\dfrac{\varepsilon }{10^{100}}\right]
 \end{align*}
 for $s \leq r$ and $\ell_i \leq \beta+10^{-100}\varepsilon$ for all but at most one $i \in \{1, \dots, j \} \setminus \bigcup_{i=1}^r X_{i} $.

 \vspace{3mm}
Let  $r \in \{2,4,6\}$ and  $0.01-\varepsilon \leq \ell_i^*\leq 0.5+\varepsilon$,  $\ell_i^* \geq \ell^*_{i+1}$ and $\ell_1^*+ \dots + \ell_r^* \leq 0.99$.

We denote by $\Xi^{\star\star}(\ell_1^*, \dots, \ell_r^*)$ the set of finite sequences $\{\ell_i\}_{i=1}^j$ with $\ell_1, \dots, \ell_j \in [0,1]$, $\sum_{i=1}^j \ell_i =1$ and $j \leq 10^{20}$ for which there exist disjoint subsets $X_1, \dots, X_r$ of $\{1, \dots, j\}$ with
 \begin{align*}
 \sum_{i \in X_s} \ell_i \in \left[\ell_s^*-\dfrac{\varepsilon }{10^{100}}, \, \ell_s^*+\dfrac{\varepsilon }{10^{100}}\right] \quad \mbox{ for } s\leq r.
 \end{align*}  
Suppose  $\{\ell_i\}_{i=1}^j \in \Xi^\star(\ell_1^*, \dots, \ell_r^*, \beta)$  or $\{\ell_i\}_{i=1}^j \in \Xi^{\star\star}(\ell_1^*, \dots, \ell_r^*)$. We then  say  that $\{\ell_i\}_{i=1}^j$  satisfies one of the options {\rm (1)}, {\rm (2)} or {\rm (3)}  for a given $a \in [0.475-\varepsilon,0.77-\varepsilon]$ if one of the following   conditions holds: 
\begin{enumerate} [{\rm (1)}]
\item There exists $k \in \{1, \dots, j\}$  with $\ell_k \geq \chi_0(a)$. 
\item There exists $I_1 \subseteq \{1,\dots,j\}$ with  $\sum_{i \in I_1} \ell_i  \in \chi_1(a) $.
\item There exist $I_2 \subseteq \{1,\dots,j\}$ with  $\sum_{i \in I_2} \ell_i  \in\chi_2(a) $ and 
$I_3 \subseteq \{1,\dots,j\}$ with  $\sum_{i \in I_3} \ell_i  \in  \chi_3(a)   $. 
\end{enumerate}

\end{defn}

We have the following relationships between sifted sets:

\begin{lemma} \label{lemma1}
Let $\varepsilon>0$. Let $a \in [0.475-\varepsilon,0.77-\varepsilon]$ and set $\tau = x^a$. 

\vspace{3mm}
Let $\beta \in [0.01,0.15]$, $r \in \{0,\dots, 5\}$ and $P_i=x^{\ell_i^*}$ with $\beta \leq \ell_i^*\leq 0.5+\varepsilon$,  $P_i \geq P_{i+1}$ and $P_1 \dots P_r \leq x^{0.75}$.
Suppose that every $\{\ell_i\}_{i=1}^j \in \Xi^\star(\ell_1^*, \dots, \ell_r^*, \beta)$ satisfies one of the options {\rm (1)}, {\rm (2)} or {\rm (3)} for the given $a$.

\vspace{3mm}
Then there  exists  $\mathcal{I} \subseteq [x,3x]\cap \mathbb{N}$ such that $\#\mathcal{I} = O( \tau x^{0.23+3\varepsilon/4})$ and such that 
 $y \in ([x,3x]\cap\mathbb{N} )\setminus\mathcal{I} $ implies
\begin{align*}
& \Bigg|\sum_{p_i \sim P_i} S(\mathcal{A}(y)_{p_1 \dots p_r}, x^{\beta}) - 
\dfrac{x^b}{\tau} \sum_{p_i \sim P_i}S(\mathcal{B}(y)_{p_1 \dots p_r}, x^{\beta})  \Bigg| 
\leq  \left( \dfrac{\log\log(x)^{O(1)} y}{\tau \log(x)^{2+r}} \right). 
\end{align*}

Alternatively, let  $r \in \{2,4,6\}$ and $P_i=x^{\ell_i^*}$ with $0.01-\varepsilon \leq \ell_i^*\leq 0.5+\varepsilon$,  $P_i \geq P_{i+1}$ and $P_1 \dots P_r \leq x^{0.99}$.
Suppose that every $\{\ell_i\}_{i=1}^j \in \Xi^{\star\star}(\ell_1^*, \dots, \ell_r^*)$ satisfies one of the options {\rm (1)}, {\rm (2)} or {\rm (3)} for the given $a$. 

\vspace{3mm}
Then there again exists  $\mathcal{I} \subseteq [x,3x]\cap \mathbb{N}$ such that $\#\mathcal{I} = O( \tau x^{0.23+3\varepsilon/4})$ and 
 $y \in ([x,3x]\cap\mathbb{N} )\setminus\mathcal{I} $ implies
\begin{align*}
& \Bigg|\sum_{p_i \sim P_i} S(\mathcal{A}(y)_{p_1 \dots p_r}, p_r) - 
\dfrac{x^b}{\tau} \sum_{p_i \sim P_i}S(\mathcal{B}(y)_{p_1 \dots p_r}, p_r)  \Bigg| 
\leq  \left( \dfrac{\log\log(x)^{O(1)} y}{\tau \log(x)^{2+r}} \right). 
\end{align*}
Here the constants implied by big O notation  depend only on $\varepsilon$.
\end{lemma}

\subsection{About Proposition 3} The final proposition concerns  minorants of the prime indicator function, which are constructed via Harman's sieve, using as input the comparison of sifted sets from Lemma~\ref{lemma1}. We work with the following regions: 

\begin{defn} \label{def:Rstar}
Let $\mathcal{R}^*(a)$ denote the set of tuples $(\ell_1^*, \dots, \ell_r^*, \beta)$ which have $\beta \in [0.01,0.15]$, $r \in \{0,\dots, 5\}$, $\ell_1^*, \dots,  \ell_r^* \in [\beta, 0.5+\varepsilon]$, $\ell_i^* \geq \ell_{i+1}^*$  and $\ell_1^*+ \dots + \ell_r^* \leq 0.75$ and for which every  $\{\ell_i\}_{i=1}^j \in \Xi^{\star}(\ell_1^*, \dots, \ell_r^*,\beta)$ satisfies one of the options {\rm (1)}, {\rm (2)} or {\rm (3)} for $a$.  

\vspace{3mm}
Let $\mathcal{R}^{**}(a)$ denote the set of tuples $(\ell_1^*, \dots, \ell_r^*)$ which have  $r \in \{2,4,6\}$, $\ell_1^*, \dots,  \ell_r^* \in [0.01-\varepsilon, 0.5+\varepsilon]$, $\ell_i^* \geq \ell_{i+1}^*$  and $\sum_{i=1}^r\ell_i^* \leq 0.99$ and  for which every  $\{\ell_i\}_{i=1}^j \in \Xi^{\star\star}(\ell_1^*, \dots, \ell_r^*)$ satisfies   {\rm (1)}, {\rm (2)} or {\rm (3)} for $a$. 
\end{defn}

\begin{prop} \label{proposition3}
Let $\varepsilon>0$  and $a \in [0.475-\varepsilon,0.77-\varepsilon]$.
 If $x$ is sufficiently large, there exists a function $\rho: [x,6x] \cap \mathbb{N} \rightarrow \mathbb{R}$ with the following properties:
 
 \vspace{1mm}
\begin{enumerate}
    \item[{\rm ($\alpha$)}] For any $n \in [x,6x] \cap \mathbb{N}$, $\rho(n) \leq 1_{\mathbb{P}}(n)$. 
    \item[{\rm ($\beta$)}]  $\rho$ can be written as a  sum  $\rho = \sum_{t=1}^s \rho_t$, where each $\rho_t: [x,6x] \cap \mathbb{N} \rightarrow \mathbb{R}$ satisfies  one of
    \begin{align*}
    &\rho_t(n) = (-1)^r \sum_{\substack{n = p_1 \dots p_r m \\ p_i \sim x^{\ell_i^*}}} \psi(m,x^\beta) \quad \mbox{ for some } r \in \{0, \dots, 5\} \mbox{ and } (\ell_1^*, \dots, \ell_r^*,\beta) \in \mathcal{R}^{*}(a), \\  & \mbox{or } \quad \rho_t(n) =  \sum_{\substack{n = p_1 \dots p_r m \\ p_i \sim x^{\ell_i^*}}} \psi(m,p_r)   \quad \mbox{ for some } r \in \{2,4,6\} \mbox{ and } (\ell_1^*, \dots, \ell_r^*) \in \mathcal{R}^{**}(a).
    \end{align*}
    Here only $O(\log(x)^{u})$ choices of $t$ shall have $r=u$. 
    
    \item[{\rm ($\gamma$)}] For any $y \in [x,3x] \cap \mathbb{N}$, 
    \begin{align*}
    \sum_{n \in \mathcal{B}(y)} (1_{\mathbb{P}}(n)- \rho(n) ) \leq \dfrac{0.99999y}{x^b \log(x)}.
    \end{align*}
\end{enumerate}
The required size of $x$ depends only on $\varepsilon$. 
\end{prop}

\subsection{Final steps}

We now show how to deduce Theorem~\ref{thm:theorem1} from Lemma~\ref{lemma1} and Proposition~\ref{proposition3}. Recall that our goal is to prove bound $\sum_{p_n \leq x} (p_{n+1}-p_n)^2 \ll_\varepsilon  x^{1.23 + \varepsilon}$ for every $\varepsilon>0$. 
 By dyadic decomposition,  it is sufficient to show that for any $\tau>0$,
\begin{equation}\label{equ:mainsum-dyadic}
\sum_{ \substack{x \leq p_n \leq 2x \\ 6x/\tau \leq (p_{n+1}-p_n) \leq 12x/\tau   }} (p_{n+1}-p_n)^2 \ll_\varepsilon x^{1.23+\varepsilon}.
\end{equation} 
 This statement holds trivially when $\tau \geq x^{0.77-\varepsilon}$. A result of Baker, Harman and Pintz~\cite{Baker:2001:DCP} also tells us that $(p_{n+1}-p_n) \ll x^{0.525}$ and so (\ref{equ:mainsum-dyadic}) holds when $\tau \leq x^{0.475-\varepsilon}$. Hence we restrict our attention to $\tau$ with  $x^{0.475-\varepsilon} \leq \tau \leq x^{0.77-\varepsilon}$. Following an argument of Peck~\cite{Peck:1996:Thesis}, we observe that 
\begin{align*}
\sum_{ \substack{x \leq p_n \leq 2x \\ 6x/\tau \leq (p_{n+1}-p_n) \leq 12x/\tau  }} (p_{n+1}-p_n)^2 \ll  \dfrac{x}{\tau}
\sum_{\substack{x\leq p_n \leq 2x,\\ p_{n+1}-p_n \geq 6x/\tau}}\# \left\{ y \in \mathbb{N}: y \in \left(p_n, \dfrac{p_n+p_{n+1}}{2}\right)  \right\}.
\end{align*}
If $y \in (p_n, (p_n+p_{n+1})/2)$ and $p_n \leq 2x$, then $y \leq 3x$ by Bertrand's postulate and $y+y/\tau \leq y+ 3x/\tau$. If $p_{n+1}-p_n \geq 6x/\tau$, then also  $y+y/\tau \leq y+(p_{n+1}-p_n)/2< p_{n+1}$. Hence $[y,y+y/\tau]$ is free of primes. To obtain bound~(\ref{equ:mainsum-dyadic}) it thus suffices  to show that
there are only $O(\tau x^{0.23+\varepsilon})$ integers $y \in \mathbb{N} \cap [x,3x]$ for which $[y,y+y/\tau]$ contains no primes.

\vspace{3mm}
We count the number of primes in short interval $\mathcal{A}(y) = [y,y+y/\tau] \cap \mathbb{N}$ by comparing it to the number of primes in the longer interval $\mathcal{B}(y) = [y,y+y/x^b] \cap \mathbb{N}$, where $b =10^{-5}$. Using Proposition~\ref{proposition3} with $\tau =x^a$, there is a minorant $\rho(n) = \sum_{t=1}^s \rho_t(n)$ of the prime indicator function $1_{\mathbb{P}}(n)$ which satisfies  properties ($\alpha$), ($\beta$) and ($\gamma$). Using ($\alpha$) and ($\gamma$), we obtain lower bound 
\begin{align} \label{equ:usingminorant}
\sum_{n \in \mathcal{A}(y)} 1_{\mathbb{P}}(n) - \dfrac{x^b}{\tau} \sum_{n \in \mathcal{B}(y)} 1_{\mathbb{P}}(n)  &\geq 
\sum_{n \in \mathcal{A}(y)} \rho(n) - \dfrac{x^b}{\tau} \sum_{n \in \mathcal{B}(y)} \rho(n) + (1_{\mathbb{P}}(n) - \rho(n)) \\ \nonumber
&\geq \sum_{t=1}^s \left(
\sum_{n \in \mathcal{A}(y)} \rho_t(n) - \dfrac{x^b}{\tau} \sum_{n \in \mathcal{B}(y)} \rho_t(n) \right) - \dfrac{0.99999y}{\tau \log(x)}.
\end{align}
Recall that $S(\mathcal{C}_d,z) = \sum_{dm \in \mathcal{C}} \psi(m,z)$ for any $\mathcal{C} \subseteq \mathbb{N}$, $d \in \mathbb{N}$ and $z>0$. By property $(\beta)$ we thus have that  $\sum_{n \in \mathcal{A}(y)} \rho_t(n) - (x^b/\tau) \sum_{n \in \mathcal{B}(y)} \rho_t(n)$ equals one of 
\begin{align*}
&(-1)^r \Bigg(\sum_{p_i \sim P_i} S(\mathcal{A}(y)_{p_1 \dots p_r}, x^{\beta}) - 
\dfrac{x^b}{\tau} \sum_{p_i \sim P_i}S(\mathcal{B}(y)_{p_1 \dots p_r}, x^{\beta})  \Bigg) \\
&\mbox{or } \quad \,
 \Bigg(\sum_{p_i \sim P_i} S(\mathcal{A}(y)_{p_1 \dots p_r}, p_r) - 
\dfrac{x^b}{\tau} \sum_{p_i \sim P_i}S(\mathcal{B}(y)_{p_1 \dots p_r}, p_r)\Bigg),
\end{align*}
where $P_i = x^{\ell_i^*}$ and where every $ \{\ell_i\}_{i=1}^j \in \Xi^\star(\ell_1^*, \dots, \ell_r^*, \beta)$ satisfies one of the options {\rm (1)}, {\rm (2)} or {\rm (3)} (when $z = x^\beta$) or every  $\{\ell_i\}_{i=1}^j \in \Xi^{\star\star}(\ell_1^*, \dots, \ell_r^*)$ satisfies one of the options {\rm (1)}, {\rm (2)} or {\rm (3)} (when $z = p_r$). Here $\Xi^\star(\ell_1^*, \dots, \ell_r^*, \beta)$, $\Xi^{\star\star}(\ell_1^*, \dots, \ell_r^*)$ and {\rm (1)}, {\rm (2)} and {\rm (3)} (with $a = \log_x(\tau)$) are as given in Definition~\ref{def:combinatorialconditions}.

\vspace{3mm} Hence by Lemma~\ref{lemma1} there exists  for each $t \in \{1, \dots, s\}$  a set $\mathcal{I}_t$ such that $\# \mathcal{I}_t = O( \tau x^{0.23+3\varepsilon/4})$ and  every $y \in ([x,3x] \cap \mathbb{N})\setminus \mathcal{I}_t$ satisfies  $\sum_{n \in \mathcal{A}(y)} \rho_t(n) - (x^b/\tau) \sum_{n \in \mathcal{B}(y)} \rho_t(n) \geq -(\log\log(x)^{O(1)})y/(\tau \log(x)^{2+r}).$  However, as part of property ($\beta$) we also have that  only $O(\log(x)^{u})$ choices of $t$ have $r=u$. Thus when $y \in ([x,3x] \cap \mathbb{N})\setminus (\cup_{t=1}^s\mathcal{I}_t )$, also $\sum_{t=1}^s (\sum_{n \in \mathcal{A}(y)} \rho_t(n) - (x^b/\tau) \sum_{n \in \mathcal{B}(y)} \rho_t(n)) \geq - (\log\log(x)^{O(1)})y/(\tau \log(x)^2) $. Rearranging (\ref{equ:usingminorant}), we find that $\mathcal{I} = \bigcup_{t=1}^s \mathcal{I}_t$ has $\#\mathcal{I} = O(\tau x^{0.23+\varepsilon})$ and $y \in ([x,3x] \cap \mathbb{N})\setminus \mathcal{I}$ implies
\begin{align} \label{equ:usingminorant2}
\sum_{n \in \mathcal{A}(y)} 1_{\mathbb{P}}(n)  
&\geq \dfrac{x^b}{\tau} \sum_{n \in \mathcal{B}(y)} 1_{\mathbb{P}}(n)  - \dfrac{0.999999y}{\tau \log(x)} = \dfrac{(0.000001+o(1))y}{\tau \log(x)}.
\end{align}
In the final inequality we used that $\mathcal{B}(y)$ is an interval of length $y/x^b$, where  $y \in [x,3x]$ and $b$ is very small, and thus must contain $(1+o(1))y/(x^b \log(x))$ primes. Inequality (\ref{equ:usingminorant2}) implies that $\mathcal{A}(y)$ contains primes, provided that $x$ is sufficiently large. We have shown that there are only $O(\tau x^{0.23+\varepsilon})$ integers $y \in \mathbb{N} \cap [x,3x]$ for which $[y,y+y/\tau]$ contains no primes.  As discussed above, this implies Theorem~\ref{thm:theorem1}.

\vspace{3mm}
Hence to complete the proof of Theorem~\ref{thm:theorem1}, it remains to show that Proposition~\ref{proposition1} and Proposition~\ref{proposition2} are true (see Section~\ref{sec:sums} and Section~\ref{sec:values}, respectively), to derive Lemma~\ref{lemma1} from Proposition~\ref{proposition1} and Proposition~\ref{proposition2} (see Section~\ref{sec:comparison}) and to construct minorants which satisfy the requirements of Proposition~\ref{proposition3} (see Section~\ref{sec:harman}).

\section{From sums to Dirichlet polynomials} \label{sec:sums}

The goal of this section is to prove Proposition~\ref{proposition1}, which roughly states that  
\begin{align} \label{difference}
\sum_{\substack{k_0k_1 \dots k_{r_2} \in \mathcal{A}(y) \\ k_i \sim N_i  }} \left( 1_{\mathbb{N}}(k_0) \prod_{1 \leq i \leq r_1} 1_\mathbb{P}(k_i) \prod_{i > r_1} a^{({i})}_{k_{i}} \right) - \dfrac{x^b}{\tau}\sum_{\substack{k_0k_1 \dots k_{r_2} \in \mathcal{B}(y) \\ k_i \sim N_i  }} \left( 1_{\mathbb{N}}(k_0) \prod_{1 \leq i \leq r_1} 1_\mathbb{P}(k_i) \prod_{i > r_1} a^{({i})}_{k_{i}} \right)
\end{align}
is $O_A(x/(\tau \log(x)^A))$ for most  $y \in \mathbb{N} \cap [x,3x]$ and that the number of $y$ for which (\ref{difference}) is larger can be bounded by counting how often the values of  corresponding Dirichlet polynomials are of a certain size. 

\vspace{3mm} Throughout Section~\ref{sec:sums} we work with a fixed $\varepsilon>0$ and $a \in [0.475-\varepsilon,0.77-\varepsilon]$ and set $\tau = x^a$. We also recall that $\varepsilon_1 = \varepsilon/10^5 $,  $T_0= \tau  x^{\varepsilon_1}$ and $T_1= \gamma x^{b/2-\varepsilon_1}$, where $\gamma \in [1,2]$ is chosen such that $\log(T_0/T_1)/\log(2) \in \mathbb{Z}$ and $[T_1,T_0]$ can be split into dyadic intervals. 

\subsection{Comparison Lemma} \label{ssec:comparison}

We begin with  a simple consequence of Perron's theorem.
 
\begin{lemma} \label{lem:per}
Let $A\in \mathbb{N}$ and $J \in \mathbb{N}$. Suppose $(a_n)$ is a complex sequence with $|a_n| \leq x^{o(1)}$. Suppose also that  $a_n=0$ unless $2^{-J}x \leq n \leq 2^{J} x$. Write $F(s) = \sum a_n n^{-s}$.   Then for any $y \in [x,3x ]$,
\begin{align*}
\sum_{n \in \mathcal{A}(y)} a_n - \dfrac{x^b}{\tau} \sum_{n \in \mathcal{B}(y)} a_n &= 
O\!\left(\left|\int_{c+iT_1}^{c+iT_0} y^s C_\tau(s)F(s)\, \mbox{{\rm d}}s\right| \right) + O\!\left(\dfrac{x}{\tau \log(x)^{A }}  \right)\\
&+O\!\left( \dfrac{x^b}{\tau} \left| \int_{c+iT_1}^{c+iT_0} y^s C_{x^b}(s)F(s)\,\mbox{{\rm d}}s \right|\right), 
\end{align*}
where $c=1+1/\log(x)$ and $C_z(s)=((1+1/z)^s-1)/s$. 

The constants implied by big $O$ notation only depend on $A$, $J$ and $\varepsilon$.
\end{lemma}

\begin{proof}
Let $I(z) = [y, y+y/z] \cap \mathbb{N}$. We consider $z \in \{\tau, x^b\}$. By Perron's theorem,  since  $a_n = O(x^{\varepsilon_1/5})$,
\begin{align}
\sum_{n \in I(z)} \!\!\! a_n &= \dfrac{1}{2\pi i} \int_{c-iT_0}^{c+iT_0} y^s C_z(s)F(s)\,\mbox{d}s + O\!\left( \dfrac{x^{1+\varepsilon_1/2}}{T_0}\right)\label{equ:per1}\\
&= \dfrac{1}{2\pi i} \int_{c-iT_1}^{c+iT_1} \!\! y^s C_z(s)F(s)\,\mbox{d}s + O\!\left(\left|\int_{c+iT_1}^{c+iT_0} \!\! y^s C_z(s)F(s)\,\mbox{d}s \right|\right)+O\left( \dfrac{x}{\tau x^{\varepsilon_1/2}}\right). \nonumber
\end{align}
We further note that
\begin{align*}
\dfrac{(y+Y)^s-y^s}{s} - y^{s-1}Y &= \int_{y}^{y+Y} t_1^{s-1} - y^{s-1} \,\mbox{d}t_1= (s-1) \int_{y}^{y+Y} \int^{t_1}_{y} t_2^{s-2} \,\mbox{d}t_2\,\mbox{d}t_1
\end{align*}
and hence obtain $((y+Y)^s-y^s)/s = y^{s-1} Y + O(|s-1|y^{\operatorname{Re}(s)-2}Y^2)$ for  $Y=y/z$ with $z \in \{\tau, x^b\}$. Since $a_n = O(x^{\varepsilon_1/5})$ and $T_1 = \gamma x^{b/2-\varepsilon_1}$ with $\gamma \in [1,2]$, we thus have
\begin{align}
\dfrac{1}{2\pi i} \int_{c-iT_1}^{c+iT_1} y^s C_z(s)F(s)\,\mbox{d}s &=\dfrac{1}{2\pi i} \int_{c-iT_1}^{c+iT_1} \dfrac{y^s}{z} F(s)\,\mbox{d}s+ O\!\left(\dfrac{T_1^2x^{1+\varepsilon_1}}{z^2}\right) \label{equ:per2}\\
&=\dfrac{1}{2\pi i z} \int_{c-iT_1}^{c+iT_1} y^sF(s)\,\mbox{d}s+ O\!\left(\dfrac{x}{z x^{ \varepsilon_1}}\right). \nonumber
\end{align}
Combining (\ref{equ:per1}) and (\ref{equ:per2}), once with $z= \tau$ and once with $z=x^b$, and multiplying the latter equation by $x^b/\tau$, we obtain Lemma~\ref{lem:per}.
\end{proof}

Applying Lemma~\ref{lem:per}  to the choice of  $(a_n)$ given at the start of Section~\ref{ssec:proposition1}, we  then have to work with error terms $O\left(\left|\int_{c+iT_1}^{c+iT_0} y^s C_z(s)F(s)\, \mbox{{\rm d}}s\right|\right)$, where
\begin{align*}
F(s) = \left(\sum_{k_0 \sim N_0} \dfrac{1}{k_0^s}  \right) \prod_{i=1}^{r_1} \left(\sum_{k_i \sim N_i} \dfrac{1_{\mathbb{P}}(k_i)}{k_i^s}  \right)\prod_{i>r_1} \left(\sum_{k_i \sim N_i} \dfrac{a_{k_i}^{(i)}}{k_i^s}  \right).
\end{align*}
But $F(s)$ has long factors with coefficients $1_{\mathbb{P}}(n)$. We do not have good Dirichlet polynomial bounds in this case - instead we would rather  work with short factors or factors with smooth coefficients $1$ or $\log(n)$. To fix this problem, we decompose   $1_{\mathbb{P}}(n)$ via Heath-Brown's identity.

 \subsection{Heath-Brown's Identity} \label{ssec:heathbrownidentity}
  Heath-Brown's identity states the following: 
 
\begin{lemma}\label{lem:heath-brown} For $n \leq 6x$ and $K \in \mathbb{N}$
\begin{align*}
\Lambda(n) = \sum^K_{j=1} (-1)^{j+1} \binom{K}{j} \sum_{\substack{ \prod_{i=1}^{2j} n_i =n\\ n_i \leq (6x)^{1/K} \,\mbox{\small \rm for } i \leq j}} \mu(n_1) \dots \mu(n_j) \log n_{2j}.
\end{align*}
\end{lemma}
\begin{proof}
This is a simple consequence of equation (6) of~\cite{Heath-Brown:1982:HBI}.
\end{proof}

\begin{cor} \label{cor:hbidentity}
Let $z \in \{\tau, x^b\}$, $J, K \in \mathbb{N}$ and $T \in [T_1,T_0]$.

Suppose $F(s)= \prod_{i=1}^r S_i(s)$ where $r \leq J$ and where $S_i(s) = \sum_{n \sim N_i} a_n^{(i)} n^{-s}$ for some $a_n^{(i)} \in \mathbb{C}$ with $a_n^{(i)} = O(\tau_{J}(n) \log(n+3))$ and some $N_i \geq 1/2$ with $2^{-J} x \leq \prod_{i=1}^{ r} N_i \leq 2^J x$.

Suppose $S_1(s) = \sum_{n \sim N_1} 1_{\mathbb{P}}(n) n^{-s}$ with $N_1 \geq x^{1/K}$. 

\vspace{3mm} Denote by $S_k$ the set of $n \in \mathbb{N}$ with $n = p^k$ for some prime $p$. Denote by $\mathcal{M}(N,j,K)$ the set of $(M_i)_{i=1}^{2j}$ with $M_{i} \in \left\{ \frac{ (6x)^{1/K}}{2^{m}}: 1 \leq m \leq \frac{\log( (6x)^{1/K})}{\log(2)}+1 \right\}$ for $i  \leq j$ and $M_{i} \in \left\{ \frac{N}{2^{m}}:  0 \leq m \leq \frac{\log(N)}{\log(2)}+1 \right\}$ for $i  > j$ and $2^{-2K} N \leq \prod_{i=1}^{2j} M_{i} \leq 2^{2K} N$.

\vspace{3mm}
 Then   there exists $U_1 \in (N_1, 2N_1]$ with
\begin{align} \label{equ:hbbound22}
\left|\int_{c+iT}^{c+i2T} y^s C_z(s)F(s)\, \mbox{{\rm d}}s\right| &\ll \left|\int_{c+iT}^{c+i2T} y^s C_z(s)  \sum_{n \sim N_1}\left(  \sum_{k=2}^\infty \dfrac{1_{S_k}(n)}{k n^s} \right) \prod_{ i \geq 2} S_i(s)\, \mbox{{\rm d}}s\right|   \\ \nonumber
 &+ \sum_{j=1}^K \sum_{(M_i) \in \mathcal{M}(N_1,j,K)} \left|\int_{c+iT}^{c+i2T} y^s C_z(s)  \left(\sum_{n \sim N_1}   \dfrac{b_n^{(j,(M_i))}}{ n^s} \right) \prod_{ i \geq 2} S_i(s)\, \mbox{{\rm d}}s\right|
  \end{align}
  where 
 \begin{align}
 b_n^{(j,(M_i))}=\sum_{\substack{\prod_{i=1}^{2j} m_{i}  =n \\ m_{i} \sim M_{i} \\ \prod_{i=1}^{2j} m_{i}  \in (N_1, U_1] }} \left(\log(m_{2j}) \prod_{i=1}^{j} \mu(m_{i}) \prod_{i=j+1}^{2j-1} 1_{\mathbb{N}}(m_{i})\right). \nonumber
\end{align}
Here the constant implied by big O notation depends only on   $K$.  
\end{cor}

\begin{proof} 
By some rearrangements and partial summation, we have 
\begin{align*}
\sum_{\substack{n \sim N_1  }} \dfrac{1_{\mathbb{P}}(n)}{n^s} &= \sum_{\substack{n \sim  N_1  }} \dfrac{\Lambda(n)}{n^s\log(n)}-\sum_{k=2}^\infty\sum_{\substack{n \sim  N_1 \\ n=p^k \mbox{ \footnotesize for a prime } p }} \dfrac{1}{kn^s}\\
&= \dfrac{1}{\log 2 N_1} \sum_{\substack{n \sim N_1 }} \dfrac{ \Lambda(n)}{n^s}
+\int_{N_1}^{2 N_1}\dfrac{1}{t (\log t)^2}  \sum_{\substack{n \sim N_1 \\ n \leq t}} \dfrac{ \Lambda(n)}{n^s}\,{\rm d}t
-\sum_{n \sim N_1}\left(  \sum_{k=2}^\infty \dfrac{1_{S_k}(n)}{k n^s} \right) 
\end{align*}
Using Heath-Brown's identity, stated in Lemma~\ref{lem:heath-brown}, then 
\begingroup
\allowdisplaybreaks
\begin{align*}
\sum_{\substack{n \sim N_1  }} \dfrac{1_{\mathbb{P}}(n)}{n^s} &= \dfrac{1}{\log (2 N_1)}  \sum^K_{j=1} (-1)^{j+1} \binom{K}{j} \sum_{n \sim N_1} \sum_{\substack{  \prod_{i=1}^{2j} m_i = n \\ m_i \leq (6x)^{1/K} \,\mbox{\small \rm for } i \leq j}} \dfrac{\mu(m_1) \dots \mu(m_j) \log m_{2j}}{(m_1 \dots m_{2j})^s}  \\
&+\int_{N_1}^{2 N_1}\dfrac{1}{t (\log t)^2}   \sum^K_{j=1} (-1)^{j+1} \binom{K}{j}  \sum_{\substack{n \sim N_1 \\ n \leq t}} \sum_{\substack{  \prod_{i=1}^{2j} m_i = n \\ m_i \leq (6x)^{1/K} \,\mbox{\small \rm for } i \leq j}} \dfrac{\mu(m_1) \dots \mu(m_j) \log m_{2j}}{(m_1 \dots m_{2j})^s} \,{\rm d}t \\ 
&-\sum_{\substack{n \sim N_1}} \left( \sum_{k=2}^\infty  \dfrac{1_{S_k}(n)}{k n^s} \right).  
\end{align*}
\endgroup
We now substitute this identity into $I=\left|\int_{c+iT}^{c+i2T} y^s C_z(s)F(s)\, \mbox{{\rm d}}s\right|$ and split sums over $m_1, \dots, m_{2j}$ up by restricting  the ranges of $m_i$ to $m_i \sim M_i$  where $M_{i} \in\left\{ \frac{ (6x)^{1/K}}{2^{m}}: 1 \leq m \leq \frac{\log( (6x)^{1/K})}{\log(2)}+1 \right\}$ for $i  \leq j$ and $M_{i} \in \left\{ \frac{N_1}{2^{m}}:  0 \leq m \leq \frac{\log(N_1)}{\log(2)}+1 \right\}$ for $i  > j$. The following  bound holds: 
\begin{align}  
I &\ll  \sum_{j=1}^K \sum_{(M_i) \in \mathcal{M}(N_1,j,K)}  \Bigg|\int_{c+iT}^{c+i2T} y^s C_z(s)  \sum_{\substack{n \sim N_1 \\ \prod_{i=1}^{2j} m_i =n \\ m_i \sim M_i}}\left(   \dfrac{\log(m_{2j}) \prod_{i=1}^j \mu(m_i)}{(m_1 \dots m_{2j})^s} \right) \prod_{ i \geq 2} S_i(s)\, \mbox{{\rm d}}s\Bigg| \label{equ:hbound25}\\ 
&+   \int_{N_1}^{2 N_1}\!\!\!\dfrac{1}{t (\log t)^2} \sum_{j=1}^K \sum_{(M_i) \in \mathcal{M}} \Bigg| \int_{c+iT}^{c+i2T} y^s C_z(s)  \!\!\!\!\!\! \sum_{\substack{n \in (N_1,t] \\ \prod_{i=1}^{2j} m_i =n \\ m_i \sim M_i}} \!\!\!\left(   \dfrac{\log(m_{2j}) \prod_{i=1}^j \mu(m_i)}{(m_1 \dots m_{2j})^s} \right) \prod_{ i \geq 2} S_i(s)\, \mbox{{\rm d}}s\Bigg| \,\mbox{{\rm d}}t    \label{equ:hbound26}  \\
&+ \left|\int_{c+iT}^{c+i2T} y^s C_z(s)  \sum_{n \sim N_1}\left(  \sum_{k=2}^\infty \dfrac{1_{S_k}(n)}{k n^s} \right) \prod_{ i \geq 2} S_i(s)\, \mbox{{\rm d}}s\right|. \label{equ:hbound27} 
\end{align}   
   Now choose $U_1 \in (N_1,2N_1]$ so that 
   \begin{align*}
 \sum_{j=1}^K \sum_{(M_i) \in \mathcal{M}(N_1,j,K)}  \Bigg| \int_{c+iT}^{c+i2T} y^s C_z(s)  \!\!\! \sum_{\substack{n \in (N_1,U_1] \\ \prod_{i=1}^{2j} m_i =n \\ m_i \sim M_i}} \!\!\!\left(   \dfrac{\log(m_{2j}) \prod_{i=1}^j \mu(m_i)}{(m_1 \dots m_{2j})^s} \right) \prod_{ i \geq 2} S_i(s)\, \mbox{{\rm d}}s\Bigg|
   \end{align*}
   is maximized. Then the RHS of  (\ref{equ:hbound25}),  (\ref{equ:hbound26}) and (\ref{equ:hbound27}) each is bounded above by the RHS of (\ref{equ:hbbound22}). 
\end{proof}

\subsection{Removal of inequality conditions} \label{ssec:removal}

Applying Lemma~\ref{lem:per} to the choice of $(a_n)$ given at the start of Section~\ref{ssec:proposition1} and repeatedly applying Corollary~\ref{cor:hbidentity} to the  error terms,  we eventually end up working with bounds $O\left( \log(x)^C \left|\int_{c+iT_1}^{c+iT_0} y^s C_z(s)F(s)\, \mbox{{\rm d}}s\right|\right)$ where $F(s)$  factorises as $F(s) = \prod_{i=1}^j S_i(s)$ for some  $S_i(s)$ which are either short or have property (\ref{propertyz1}) or (\ref{propertyz2}) (with $w = 0$) or which are of the form 
\begin{align}
S_i(s) = \sum_{\substack{\prod_{k=1}^{2j_i} m_{i,k}  \in ( N_i, U_i] \\ m_{i,k} \sim M_{i,k} }} \left(\dfrac{\log(m_{i,2j_i})}{m_{i,2j_i}^s} \prod_{k=1}^{j_i} \dfrac{\mu(m_{i,k})}{m_{i,k}^s} \prod_{k=j_i+1}^{2j_i-1} \dfrac{1_{\mathbb{N}}(m_{i,k})}{m_{i,k}^s}\right).
\end{align}
If we could omit the condition $\prod_{k=1}^{2j_i} m_{i,k}  \in ( N_i, U_i]$, then $S_i(s)$ would factorise very nicely. The goal of this section is thus to remove certain inequality conditions which prevent further factorisation of a Dirichlet polynomial. More precisely, we will prove the following result:

\begin{lemma} \label{lem:mainperremoval}
Let $F(s)= \prod_{i=1}^j S_i(s)$ where $j \leq J$ and where $S_i(s) = \sum_{n \sim N_i} a_n^{(i)} n^{-s}$ for some $a_n^{(i)} \in \mathbb{C}$ with $a_n^{(i)} = O(\tau_{J}(n) \log(n+3))$ and some $N_i \geq 1/2$ with $2^{-J} x \leq \prod_{i=1}^j N_i \leq 2^J x$. 

\vspace{3mm}
Suppose there exist $\ell \in \{1, \dots, j\}$, $U_\ell \in (N_\ell, 2 N_\ell]$, $b_{n}^{(i)} \in \mathbb{C}$ with $b_n^{(i)} =O( \log(n+3))$ for $1 \leq i \leq r \leq J$ and $(M_i)$ with $2^{-J}x^{1/J} \leq 2^{-J} N_\ell \leq \prod_{i=1}^r M_i \leq 2^J N_\ell$ such that 
\begin{align*}
 a_n^{(\ell)}= \sum_{ \substack{ n=k_1 \dots k_r  \\ k_i \sim M_i  \\ k_1 \dots k_r \in (N_\ell, U_\ell]}} \!\!\! b^{(1)}_{k_1} \dots b^{(r)}_{k_r}.
\end{align*} 
Let $T \in [T_1,T_0]$. Suppose also that there exists $i_0 \neq \ell$ such that $N_{i_0} \geq x^{1/\log\log(x)}$ and such that  $S_{i_0}(s)$ satisfies property $(\ref{propertyd})$ with respect to $T$. 

\vspace{3mm}
Let $z \in \{\tau, x^b\}$,  $C_z(s)=((1+1/z)^s-1)/s$ and $c_0=1/\log(x)$. Then 
\begin{align*}
\left|\int_{c+iT}^{c+i2T} y^{s} C_z(s)F(s)\, \mbox{{\rm d}}s\right|&= O\!\left(\int^{\frac{T}{4}}_{-\frac{T}{4}} \int^{\frac{T}{4}}_{-\frac{T}{4}} \frac{1}{|\frac{c_0}{4}+i t_1|}\frac{1}{|\frac{c_0}{4}+i t_2|}\left|\int_{c+iT}^{c+i2T} y^{s} C_z(s)F^*(s;t_1-t_2)\, \mbox{{\rm d}}s\right|\mbox{{\rm d}}t_1\mbox{{\rm d}}t_2 \right)\\ &+ O\left( \dfrac{x}{z \log(x)^A}\right), \\
&\mbox{where }  F^*(s;t) =  \prod_{n=1}^r \left(\sum_{\substack{ k_n \sim M_n}} b^{(n)}_{k_n} k_n^{-s-it} \right)\prod_{i \neq \ell} S_i(s).
\end{align*}
The constants implied by big O notation depend only on   $J$ and $\varepsilon$. 
\end{lemma}

\subsubsection{Large values of Dirichlet polynomials} 
 
 We first record upper bounds on the size of Dirichlet polynomials which satisfy property $(\ref{propertyd})$. These results will be very important when we bound error terms in the proof 
of Lemma~\ref{lem:mainperremoval} and will also make multiple other appearances in Section~\ref{sec:values}.

 \begin{lemma} \label{lem:largevaluesall}
 Let  $T \in [T_1, T_0]$ and $ t\in [T,2T]$.   Suppose $S(s) = \sum_{n \sim N} a_n n^{-s}$ for some $N \in[  x^{1/\log \log x},6x]$ and $S(s)$ has property $(\ref{propertyd})$ with respect to $T$.  Let  $c=1+1/\log(x)$
 and $\mu = (\log x)^{-7/10}$.
 
  Then $|S(c+it)| \ll N^{- \mu}$ and, for $D>0$,
\begin{align*}
|S(c+it)| \ll_{D} \dfrac{1}{\log(x)^{D}}.
\end{align*}
 \end{lemma}
 
 \begin{proof}
 Since $S(s)$ has property $(\ref{propertyd})$ with respect to $T$, there exists $w \in [-T/2,T/2]$ with $S(s+iw) = \sum_{n \sim N} b_n n^{-s}$ such that $b_n=1$ or $b_n = \log(n)$ or $b_n = \mu(n)$ or $b_n = 1_{\mathbb{P}}(n)$ or $b_n = \frac{\Lambda(n)}{\log(n)} - 1_{\mathbb{P}}(n)$ for all $n$.
 
 \vspace{3mm} We first consider the case where
 \begin{align*}
 S(c+it) = \sum_{n \sim N} \dfrac{1}{n^{c+i(t-w)}} \quad &\mbox{ or }  \quad S(c+it) = \sum_{n \sim N} \dfrac{\log(n)}{n^{c+i(t-w)}}  \quad \mbox{ or } \quad S(c+it) = \sum_{n \sim N} \dfrac{\mu(n)}{n^{c+i(t-w)}}.
\end{align*}   
Note $t-w \in \left[\frac{T}{2},\frac{5T}{2}\right]$. Let $\mu_1= (\log x)^{-2/3}(\log \log x)^{-1/3}$. By the proof of Lemma~5.5 of~\cite{Maynard:2012:DCP},
 \begin{align*}
 |S(c+it)| \ll \log(x)^3(N^{-2C_0 \mu_1} +T_1^{-1})
 \end{align*}
 for a suitable constant $C_0>0$. This is a consequence of the Vinogradov-Korobov estimate. Then since $N \geq x^{1/\log \log x}$, also 
$ N^{-C_0 \mu_1} \leq x^{-C_0 ((\log x)^{-2/3} (\log \log x)^{-4/3})} \leq x^{-3(\log \log x) (\log x)^{-1}} = \log(x)^ {-3}
$
 and so $|S(c+it)| \ll N^{-C_0\mu_1}$. Since $\mu \leq C_0 \mu_1$ for large $x$, further 
 $|S(c+it)| \ll N^{-\mu}$.
 
 \vspace{3mm}
Next we consider $S(c+it) = \sum_{n \sim N} 1_{\mathbb{P}}(n) n^{-c-i(t-w)}$. Lemma~1.5 of~\cite{Harman:2007:PDS} and $N \geq x^{1/\log \log x}$ give 
 \begin{align*}
 |S(c+it)| &\ll \exp\left(-\dfrac{\log(N)}{\log(5T/2)^{7/10}} \right) + \dfrac{\log(x)^3}{T}
 \ll N^{-\log(x)^{-7/10}}=N^{-\mu}.
 \end{align*}
Finally, for $S(c+it) = \sum_{n \sim N} (\Lambda(n)/\log(n)-1_{\mathbb{P}}(n)) n^{-c-i(t-w)} = \sum_{k=2}^\infty \sum_{p^k \sim N} (1/k) n^{-c-i(t-w)}$ we have
$|S(c+it)| \ll \log(N)N^{-1/2} \ll N^{-\mu}$. We conclude the proof with observation 
  \begin{align*}
 N^{-\mu}
 &\ll \exp\left(-\dfrac{\log(x)^{3/10}}{ \log \log(x)}\right)
 \ll_{D} \exp\left(-D \log \log(x) \right)= \dfrac{1}{\log(x)^{D}}. \qedhere
 \end{align*}
 \end{proof}

\subsubsection{Comparison of integrals of Dirichlet polynomials} 

Now we give a proof of Lemma~\ref{lem:mainperremoval}.

\begin{proof}[Proof of Lemma~{\rm\ref{lem:mainperremoval}}]
Recall that we wish to remove condition $k_1 \dots k_r \in (N_\ell, U_\ell]$ from $S_\ell(s)$. To do so, we  use the following  identity, a precursor to Perron's formula: 
\begin{align} \label{equ:hb.perron}
\dfrac{1}{2\pi i} \int^{(c_0+iT)/4}_{(c_0-i T)/4} \dfrac{y^{w} }{w} \,\mbox{d}w = H(y) + O\left( \dfrac{y^{c_0/4}}{T|\log y|}\right) \!\!\quad \mbox{ where } \!\! \quad H(y)=\begin{cases}
1 \quad \mbox{if }\, y>1,\\
0 \quad \mbox{if }\, 0<y<1.
\end{cases}
\end{align} 
 We first  apply~(\ref{equ:hb.perron}) with $y=(\lfloor U_\ell \rfloor+1/2)/(\prod_{i=1}^r k_i)$. For $k_i \sim M_i$ we have $\prod_{i =1}^r k_i \in [2^{-2J}N_\ell,2^{2J}N_\ell]$, while $U_\ell \in [N_\ell,2N_\ell]$. So  $y \in [2^{-3J},1-1/(2\lfloor U_\ell\rfloor+2)]$  or $y\in[1+1/(2 \lfloor U_\ell\rfloor) , 2^{3J}]$ and $y^{c_0/4} = O(1)$.
Combining error terms $O(1/(T |\log(y)|)$ via standard summation arguments, for $s=c+it$,
\begin{align}  
  S_\ell(s) &=  \sum_{\substack{ k_i \sim M_i  \\ k_1 \dots k_r \in (N_\ell, U_\ell]}} \dfrac{b^{(1)}_{k_1} \dots b^{(r)}_{k_r} }{k_1^s\dots k_r^s} \nonumber  \\
&= \dfrac{1}{2\pi i} \sum_{\substack{ k_i \sim M_i  \\ k_1 \dots k_r > N_\ell}} \dfrac{b^{(1)}_{k_1} \dots b^{(r)}_{k_r} }{k_1^s\dots k_r^s}  \Bigg( \int_{(c_0-iT)/4}^{(c_0+iT)/4} \dfrac{(\lfloor U_\ell \rfloor + 1/2)^{w}}{w(\prod_{i=1}^r k_i)^{w}} \,\mbox{d}w  + O\Bigg(\left| T \log\left( \dfrac{\lfloor U_\ell \rfloor + 1/2}{\prod_{i=1}^r k_i}\right) \right|^{-1} \Bigg) \Bigg)\nonumber\\
&= \dfrac{1}{2\pi i}\int^{(c_0+iT)/4}_{(c_0-iT)/4}\dfrac{(\lfloor U_\ell \rfloor +1/2)^{w}}{w}\Bigg(\sum_{\substack{ k_i\sim M_i\\ k_1 \dots k_r > N_\ell}}\dfrac{b^{(1)}_{k_1} \dots b^{(r)}_{k_r} }{k_1^{s+w}\dots k_r^{s+w}}\Bigg) \,\mbox{d}w \label{equ:removalstep1} \\ 
&+ O\Bigg( \sum_{\substack{k_l \sim M_l }} \dfrac{\log(x)^{J}}{T M_1 \dots M_r} +  \sum_{i=1}^r \sum_{\substack{k_l \sim M_l \\ l \neq i} } \dfrac{M_i \log(x)^{J+1}}{T M_1 \dots M_r}\Bigg). \label{equ:removalstep1error} 
\end{align} 
For fixed $t_1 \in [-T/4,T/4]$ we define Dirichlet polynomial 
\begin{align*}
 S_\ell^{(1)}(s;t_1) = \sum_{\substack{ k_i\sim M_i\\ k_1 \dots k_r > N_\ell}}\dfrac{b^{(1)}_{k_1} \dots b^{(r)}_{k_r} }{(k_1\dots k_r)^{s+c_0/4+it_1}}.
\end{align*}
Substituting (\ref{equ:removalstep1}) and (\ref{equ:removalstep1error}) for $S_\ell(s)$ and changing the order of integration, we then find 
\begingroup
\allowdisplaybreaks
\begin{align}
\left|\int_{c+iT}^{c+i2T} y^{s} C_z(s)F(s)\, \mbox{{\rm d}}s\right|&= O\!\left(\left|\int_{c+iT}^{c+i2T} \int^{\frac{T}{4}}_{-\frac{T}{4}} \frac{(\lfloor U_\ell \rfloor +1/2)^{\frac{c_0}{4}+it_1}}{\left(\frac{c_0}{4}+it_1\right)} y^{s} C_z(s)   S_\ell^{(1)}(s;t_1)  \prod_{i \neq \ell} S_i(s)\,\mbox{{\rm d}}t_1 \,\mbox{{\rm d}}s\right| \right)  \nonumber \\
&+ O\!\left(\int_{c+iT}^{c+i2T} \left|y^{s} C_z(s) \prod_{i \neq \ell} S_i(s) \left(\dfrac{\log(x)^{J+1}}{T}\right)\right|\mbox{{\rm d}}s\right)\nonumber  \\
&= O\!\left(\int^{\frac{T}{4}}_{-\frac{T}{4}} \frac{1}{\left|\frac{c_0}{4}+it_1\right|}\left|\int_{c+iT}^{c+i2T}  y^{s} C_z(s)   S_\ell^{(1)}(s;t_1)  \prod_{i \neq \ell} S_i(s)\,\mbox{{\rm d}}s\right|\mbox{{\rm d}}t_1  \right) \label{equ:removalstep1new}  \\
&+ O\!\left( \int_{c+iT}^{c+i2T} \left|y^{s} C_z(s) \prod_{i \neq \ell} S_i(s) \left(\dfrac{\log(x)^{J+1}}{T}\right)\right|\mbox{{\rm d}}s\right). \label{equ:removalstep1newerror}
\end{align}
\endgroup
To treat the error term  (\ref{equ:removalstep1newerror}), we recall that there exists $i_0 \neq  \ell$ such that $N_{i_0} \geq x^{1/\log\log(x)}$ and such that $S_{i_0}(s)$ satisfies property (\ref{propertyd}) with respect to $T$. By Lemma~\ref{lem:largevaluesall} this implies  \begin{align} \label{equ:D}
 |S_{i_0}(s)| = \Bigg| \sum_{n \sim N_{i_0} } \dfrac{a_n^{(i_0)}}{n^{c+it}} \Bigg|\ll_D \dfrac{1}{\log(x)^{D}}
\end{align}
for any $D >0$ and $t \in [T,2T]$. Also, $|y^s| = O(x)$ and $C_z(s) = O(1/z)$.  Since
$a_n^{(i)}=O(\tau_J(n)\log(n+3))$, we have $|S_i(c+it)| = O(\log(8N_i)^J)$ by Shiu's theorem~\cite{Shiu:1980:BTT}. Taking $D = J^2+J+A+1$  in (\ref{equ:D}), error term (\ref{equ:removalstep1newerror}) is thus bounded by $O(x/(z \log(x)^A))$. 

\vspace{3mm}
The main term  (\ref{equ:removalstep1new})  still involves $S_\ell^{(1)}(s;t_1)$, which is a sum over $k_i \sim M_i$ with restriction $k_1 \dots k_r > N_\ell$. Repeating the above argument, but choosing $y= (\prod_{i=1}^r k_i)/(\lfloor N_\ell \rfloor +1/2)$, we also remove this condition and upper bound  (\ref{equ:removalstep1new})  by \begin{align*}
&O\!\left(\int\displaylimits^{\frac{T}{4}}_{-\frac{T}{4}}\int\displaylimits^{\frac{T}{4}}_{-\frac{T}{4}} \prod_{j=1}^2\frac{1}{\left|\frac{c_0}{4}+it_j\right|}\left|\int\displaylimits_{c+iT}^{c+i2T}  \!\!\!\! y^{s} C_z(s)    \prod_{n=1}^r \left(\sum_{\substack{ k_n \sim M_n}} \dfrac{b^{(n)}_{k_n}}{ k_n^{s+i(t_1-t_2)}} \right)\prod_{i \neq \ell} S_i(s)\,\mbox{{\rm d}}s\right|\mbox{{\rm d}}\mathbf{t} \! \right) \!\!+O\!\left(\dfrac{x}{z \log(x)^A} \!\right).   \qedhere
\end{align*}
\end{proof}

	\subsection{Summary and Sampling}

Next we combine our results from Section~ \ref{ssec:comparison}, Section~\ref{ssec:heathbrownidentity} and Section~\ref{ssec:removal}, giving  a first bound on the number of $y \in \mathbb{N} \cap [x,3x]$ for which difference (\ref{difference}) is large. This bound involves the integral of a Dirichlet polynomial which factorises nicely into short factors and long factors with smooth coefficients and another Dirichlet polynomial which has no more than $T_0$ non-zero coefficients. Both of these properties will be very important for further bounds in Section~\ref{sec:values}.
 
 \begin{lemma} \label{lem:intermediatesummary} 
  Let $A \in \mathbb{N}$ and $K \in \mathbb{N}$ and $0 \leq r_1  \leq r_2 \leq 10^5$.  Let $J \geq 10^6K$. 
  
Let $N_i \in [ \frac{1}{2}, \infty)$ with $2^{-r_2}x \leq \prod_{i=0}^{r_2} N_i \leq 6x$,  $N_i < x^{1/K}$ for $i>r_1$ and $N_i \geq x^{1/K}$ for $1 \leq i \leq r_1$.  

 Let $a_{n}^{(i)} \in \mathbb{C}$ with $a_n^{(i)}=O( \log(n+3))$ and consider $(a_n)$ with
\begin{align*}
a_n= \sum_{\substack{k_0k_1 \dots k_{r_2} =n \\ k_i \sim N_i  }} \left( 1_{\mathbb{N}}(k_0) \prod_{1 \leq i \leq r_1} 1_\mathbb{P}(k_i) \prod_{i > r_1} a^{({i})}_{k_{i}} \right).
\end{align*}
 Suppose also  that one of the following three options holds: 
 \begin{enumerate}[{\rm (i)}]
 \item $r_1\geq 2$ or $N_0 \geq x^{0.95}$.
 \item $r_1 =1$  and  $N_0 \geq x^{1/\log\log x}$.
 \item $r_1 \in \{0,1\}$  and there exists $i>r_1$ with $a_n^{(i)} = 1_{\mathbb{P}}(n)$ for all  $n \in \mathbb{N}$  and $N_i \geq x^{1/\log\log x}$.
 \end{enumerate} 
 
 \vspace{3mm}
Let $\mathcal{I}$ denote the set of  $y \in \mathbb{N} \cap [x,3x]$ for which
 \begin{align*}
\left| \sum_{n \in \mathcal{A}(y)} a_n - \dfrac{x^b}{\tau} \sum_{n \in \mathcal{B}(y)} a_n \right| > \dfrac{x}{\tau \log(x)^A}.
 \end{align*}
 Then if $x \geq C$ and $\#\mathcal{I} > 10\tau x^{0.23}$, there exist $F \in \mathcal{F}((N_i),T,J,K)$, $G(s) = \sum_{n \in \mathcal{J}} \xi_n n^{-s} \in \mathcal{G}(\tau,\varepsilon)$, $z \in \{ \tau, x^b\}$, $T \in [T_1,T_0]$ and $u \in \{1, \dots, (2\log(x))^{10^6K}\}$ such that 
 \begin{align*}
\left|\int_{T}^{2T} \left(\dfrac{C_{z}(c+it)}{H^{-it}u^{c+it}} \right) \overline{G(it)}  F(c+it)  \mbox{{\rm d}}t\right|&\geq  \dfrac{ \# \mathcal{J}}{z \log(x)^{A+5 \cdot 10^6K}} 
\end{align*}
and  $\# \mathcal{I}\leq 10(x/T_0)\# \mathcal{J}$.   

\vspace{3mm}
Here the required size of $C$ depends only on $A$, $J$,  $K$ and~$\varepsilon$. 
\end{lemma} 
Recall here that sets $\mathcal{F}((N_i),T,J,K)$ and $\mathcal{G}(\tau,\varepsilon)$ were defined  in Definition~$\ref{def:polynomialcollections}$.

\begin{proof}
 First note that $y \in \mathcal{I}$ if and only if $y \in \mathbb{N} \cap [x,3x]$ and 
 \begin{align*}
\mbox{Diff}(\tau,(a_n),y)=\left| \sum_{n \in \mathcal{A}(y)} a_n - \dfrac{x^b}{\tau} \sum_{n \in \mathcal{B}(y)} a_n \right| > \dfrac{x}{\tau \log(x)^{A}}.
 \end{align*}
 Set $H = \lceil 3x/T_0 \rceil$ and let $\mathcal{I}^*$ be the set of $y \in H\mathbb{Z} \cap [x,3x]$ with $\mbox{Diff}(\tau,(a_n),y) > x/(\tau \log(x)^{A+1})$.
 Consider $Hm + d $, where $Hm \in H\mathbb{Z}\cap [x,3x]$ and $|d| \leq H-1$.  
  Using the $J$-fold divisor bound of Shiu~\cite{Shiu:1980:BTT}, noting that $a_n =  O(\tau_{10^5+1}(n)\log(x)^{10^5})$  and recalling $H \leq 4x/(\tau x^{\varepsilon_1})$,  we have 
\begin{align*}
|\mbox{Diff}(\tau,(a_n),Hm)-\mbox{Diff}(\tau,(a_n),Hm+d)| &\ll 
 \sum_{n \in \mathcal{A}(Hm) \setminus \mathcal{A}(Hm+d)} |a_n|+ \sum_{n \in  \mathcal{A}(Hm+d) \setminus \mathcal{A}(Hm)} |a_n|\\ &+ \dfrac{x^b}{\tau} \sum_{n \in \mathcal{B}(Hm) \setminus \mathcal{B}(Hm+d)} |a_n| + \dfrac{x^b}{\tau} \sum_{n \in  \mathcal{B}(Hm+d) \setminus \mathcal{B}(Hm)  } |a_n| \\
&\ll \dfrac{x}{\tau x^{\varepsilon_1}}  \log(x)^{ 2\cdot 10^5} \\ &\ll \dfrac{x}{\tau \log(x)^{A+1}}. 
\end{align*}
Hence if $Hm + d \in \mathcal{I}$, also  $\mbox{Diff}(\tau,(a_n),Hm) > x/(\tau \log(x)^{A+1})$ and $Hm \in \mathcal{I}^*$.  Conversely, $Hm \notin \mathcal{I}^*$ implies $Hm+d \notin \mathcal{I}$ for $|d| \leq H-1$ and thus $\# \mathcal{I} \leq 2H \#\mathcal{I}^*$. Write $\mathcal{J} = \{m \in \mathbb{N}: Hm \in \mathcal{I}^* \} \subseteq [1,T_0]$.
 
\vspace{3mm} For each $y \in \mathcal{I}^*$,  we now  apply Corollary~\ref{cor:hbidentity} to $(a_n)$ (with $A+2$ in the place $A$). The resulting error terms involve integrals over $[c+iT_1, c+iT_0]$, but  $T_1$ was chosen  so that interval $(T_1,T_0]$ can be split exactly into $O(\log(x))$ intervals $(T,2T]$. Then if $x$ is sufficiently large, there must exist $z_y \in \{\tau, x^b\}$ and $T_y \in \{ 2^{-n}T_0: 1 \leq n \leq \log(T_0/T_1)/\log(2) \}$ with   
\begin{align} \label{equ:cor1step}
\dfrac{x}{\tau \log(x)^{A+3}} < \dfrac{z_y}{\tau} \left| \int_{c+iT_y}^{c+i2T_y} y^s C_{z_y}(s)\left( \sum_{k_0 \sim N_0} \dfrac{1_{\mathbb{N}}(k_0)}{k_0^s} \right)\prod_{i=1}^{r_1}\left( \sum_{k_i \sim N_i} \dfrac{1_{\mathbb{P}}(k_i)}{k_i^s} \right)\prod_{i>r_1} \left(\sum_{k_i \sim N_i} \dfrac{a_{k_i}^{(i)}}{k_i^s} \right)\,\mbox{{\rm d}}s \right|.
\end{align}
Next we alternate between applying Corollary~\ref{cor:hbidentity} and Lemma~\ref{lem:mainperremoval} to the RHS of (\ref{equ:cor1step}), removing coefficients $1_{\mathbb{P}}(k_i)$ for $1 \leq i \leq r_1$. The first application of Lemma~\ref{lem:mainperremoval} is possible because each of the options (i), (ii) and (iii) implies that the relevant Dirichlet polynomial has a factor $\sum_{n \sim N_0} n^{-s}$ with $N_0 \geq x^{1/\log\log(x)}$ or $\sum_{n \sim N_i} 1_{\mathbb{P}}(n) n^{-s}$ with  $N_i \geq x^{1/\log\log(x)}$. This factor satisfies property (\ref{propertyd}) with respect to $T$ (with $w=0$).  Each  application of Lemma~\ref{lem:mainperremoval} then produces at least one further factor with  property (\ref{propertyd}), allowing for further  applications of Lemma~\ref{lem:mainperremoval}.  By  Corollary~\ref{cor:hbidentity} and Lemma~\ref{lem:mainperremoval}, $\frac{x}{z_y \log(x)^{A+4}}$ is bounded by
\begin{align*}
 &\sum_{\substack{ I \in \mathcal{P}(r_1) \\ j_k \in \{1, \dots, K\} \\ (M_{k,i}) \in \mathcal{M}(N_k,j_k,K)}} \int\displaylimits^{\frac{T}{4}}_{-\frac{T}{4}} \dots \int\displaylimits^{\frac{T}{4}}_{-\frac{T}{4}} \prod_{k \in I} \frac{1}{|\frac{c_0}{4}+i t_{k}|}\frac{1}{|\frac{c_0}{4}+i s_{k}|}\left|\,\,\int\displaylimits_{c+iT_y}^{c+i2T_y} y^{s} C_{z_y}(s)F^*(s;\mathbf{t}-\mathbf{s},I,(j_k), (M_{k,i}))\, \mbox{{\rm d}}s\right|\mbox{{\rm d}}\mathbf{t}\mbox{{\rm d}}\mathbf{s} , \\
&\mbox{where } F^*(s;\mathbf{t},I,\dots) =  \left( \sum_{k_0 \sim N_0} \dfrac{1}{k_0^s} \right)\prod_{\substack{ k=1 \\ k \in I}}^{r_1} S^*(s;   t_k,j_k, (M_{k,i}))\prod_{\substack{ k=1 \\ k \notin I}}^{r_1}  \sum_{n \sim N_k}\left(  \sum_{\ell=2}^\infty \dfrac{1_{S_\ell}(n)}{\ell n^s} \right)\prod_{i>r_1} \left(\sum_{k_i \sim N_i} \dfrac{a_{k_i}^{(i)}}{k_i^s} \right)
\\
&\mbox{and } S^*(s;  t_k,\dots) =  \prod_{i=1}^{j_k}\left(\sum_{\substack{ n_i \sim M_{k,i}}} \dfrac{\mu(n_i)}{ n_i^{s+i   t_k}} \right)\prod_{i=j_k+1}^{2j_k-1} \left(\sum_{\substack{ n_i \sim M_{k,i}}} \dfrac{1}{ n_i^{s+i  t_k}} \right)\left(\sum_{\substack{ n_{2j_k} \sim M_{k,2j_k}}} \dfrac{\log(n_{2j_k})}{ n_{2j_k}}^{s+i t_k} \right).
\end{align*}
Here $\mathcal{P}(r_1)$ is the power set of $\{1, \dots, r_1\}$,  $\mathcal{M}(N,j,K)$ is as described in Corollary~\ref{cor:hbidentity} and we integrate with respect to variables $t_i$ and $s_i$ with $i \in I$, using $\mathbf{t}$ as shorthand  for $(t_i)_{i \in I}$ and $\mbox{d}\mathbf{t}$ as shorthand for $\prod_{i \in I} d\mathbf{t_i}$. To save space, $\dots$ was  used once as shorthand for $(j_k), (M_{k,i})$ and once for $j_k$, $(M_{k,i})$.

\vspace{3mm}
The number of possible combinations of $I$, $(j_k)$ and $(M_{k,i})$ in the above sum  is $O(\log(x)^{ 10^6K})$. There are  two choices for  $z_y$ and $O(\log(x))$ choices for $T_y$. Summing over all $y \in \mathcal{I}^*$, we thus find that  there exist  $I \subseteq \{1, \dots, r_1\}$, $j_k \in \{1, \dots, K\}$, $(M_{k,i}) \in \mathcal{M}(N_k, j_k, K)$ (for $k \in I$), $z \in \{\tau, x^b\}$ and $T \in [T_1,T_0]$ with
\begin{align*}
\dfrac{x \# \mathcal{I}^*}{z \log(x)^{A+10^6K+6}} \leq  \sum_{y \in \mathcal{I}^*}\! \int\displaylimits^{\frac{T}{4}}_{-\frac{T}{4}}\!\! \dots \!\! \int\displaylimits^{\frac{T}{4}}_{-\frac{T}{4}} \prod_{k \in I} \frac{1}{|\frac{c_0}{4}+i t_{k}|}\frac{1}{|\frac{c_0}{4}+i s_{k}|}\left|\int\displaylimits_{c+iT}^{c+i2T} \!\!\!\! y^{s} C_{z}(s)F^*(s;\mathbf{t}-\mathbf{s},I,(j_k), (M_{k,i})) \mbox{{\rm d}}s\right|\mbox{{\rm d}}\mathbf{t}\mbox{{\rm d}}\mathbf{s}.
\end{align*} 
Now choose $t_i, s_i \in [-T/4,T/4]$ such that $\mathbf{u} = (t_i-s_i)_{i \in I}$ maximises
 \begin{align*}
 \sum_{y \in \mathcal{I}^*}\left|\int\displaylimits_{c+iT}^{c+i2T} y^{s} C_{z}(s)F^*(s;\mathbf{u},I,(j_k), (M_{k,i}))\, \mbox{{\rm d}}s\right|.
\end{align*}
Recalling that $c_0 = 1/\log(x)$ and $\# I \leq r_1 \leq 10^5$, we get 
\begin{align} \label{equ:propmainproofstep2}
\dfrac{x \# \mathcal{I}^*}{z \log(x)^{A+10^6K+10^6}} \leq \sum_{y \in \mathcal{I}^*} \left|\int\displaylimits_{c+iT}^{c+i2T} y^{s} C_{z}(s)F^*(s;\mathbf{u},I,(j_k), (M_{k,i}))\, \mbox{{\rm d}}s\right|.
\end{align} 
In our definition of $F^*(s;\mathbf{u},I,(j_k), (M_{k,i}))$, we wrote $F^*(s)$ as a product $\prod_{i=1}^j S_i(s)$ where $j \leq 10^6K$   and $S_i(s) = \sum_{n \sim B_i}  b_n^{(i)} n^{-s}$ with $b_n^{(i)} =O( \log(n+3))$ and  $2^{-J}x \leq \prod_{i=1}^j B_i \leq 2^J x$.  The factorisation was chosen so that $B_\ell \geq x^{2/K}$ implies  that $S_\ell(s)$  has property $(\ref{propertyz1})$ or $(\ref{propertyz2})$ with respect to $T$ and so that there exist disjoint subsets $X_1, \dots, X_{r_1}$ of $\{1, \dots, j\}$ with $\# X_\ell \leq 2K$ and 
\begin{align*} 
2^{-2K} N_\ell \leq \prod_{i \in X_\ell} B_i \leq 2^{2K} N_\ell \quad \mbox{ \rm for } \quad 1 \leq \ell \leq r_1,
\end{align*} 
and so that one of the conditions (A), (B) or (C) holds. However, we might  have factors $S_i(s)$ with $B_i <\log(x)$ and so $F^*(s;\mathbf{u},I,(j_k), (M_{k,i}))$  is not necessarily in $\mathcal{F}((N_i),T,J,K)$. 

\vspace{3mm} Hence write $E= \{i \in \{1, \dots, j\} : B_i < \log(x)\}$ and  define 
\begin{align*}
F_1^*(s;\mathbf{u},I,(j_k), (M_{k,i})) = \prod_{i \in \{1, \dots, j\} \setminus E} S_i(s) \quad \mbox{ and } \quad F_2^*(s;\mathbf{u},I,(j_k), (M_{k,i})) = \prod_{ i \in E} S_i(s).
\end{align*}
There still  exist disjoint subsets $Y_1, \dots, Y_{r_1}$ of $\{1, \dots, j\} \setminus E$ with 
\begin{align*} 
\log(x)^{-2K-1} N_\ell \leq \prod_{i \in Y_\ell} B_i \leq \log(x)^{2K+1}  N_\ell \quad \mbox{ \rm for } \quad 1 \leq \ell \leq r_1
\end{align*} 
and now $F_1^*(s;\mathbf{u},I,(j_k), (M_{k,i})) \in \mathcal{F}((N_i),T,J,K)$ for $J \geq 10^6K$. 

\vspace{3mm}
On the other hand, write $F_2^*(s) = \sum d_n n^{-s}$. We have $d_n = 0$ for $n \geq (2\log(x))^{10^6K} $ and $d_n= O(\tau_{10^6K}(n) \log(n)^{10^6K})$ for $n <  (2\log(x))^{10^6K} $, so that $d_n = \log(x)^{o(1)}$. Continuing on from (\ref{equ:propmainproofstep2}), there is some $u \in \{1, \dots, (2\log(x))^{10^6K}\} $ with 
\begin{align} \label{equ:propmainproofstep3}
\dfrac{x \# \mathcal{I}^*}{z \log(x)^{A+2 \cdot 10^6K+10^6+1}} \leq \sum_{y \in \mathcal{I}^*} \left|\int\displaylimits_{c+iT}^{c+i2T} y^{s} u^{-s} C_{z}(s)F^*_1(s;\mathbf{u},I,(j_k), (M_{k,i}))\, \mbox{{\rm d}}s\right|.
\end{align} 
Write $F(s) =F^*_1(s;\mathbf{u},I,(j_k), (M_{k,i})) \in  \mathcal{F}((N_i),T,J,K)$. Now recall that $m \in \mathcal{J} \subseteq [1, T_0]$ if and only if $Hm \in \mathcal{I}^*$. Suppose first that $\#\mathcal{J}< \tau^2x^{-0.77}$.  Then $\#\mathcal{I} \leq 2H \#\mathcal{J} < 10\tau x^{0.23}$. Hence we now assume $\#\mathcal{J}\geq  \tau^2x^{-0.77}$.  For each $m \in \mathcal{J}$ choose $\xi_m \in \mathcal{C}$ with $|\xi_m| = 1$ and 
\begin{align*} 
\overline{\xi_m} \left(\int\displaylimits_{c+iT}^{c+i2T} (Hm)^{s} u^{-s} C_{z}(s)F(s)\, \mbox{{\rm d}}s\right) = \left|\int\displaylimits_{c+iT}^{c+i2T} (Hm)^{s} u^{-s} C_{z}(s)F(s)\, \mbox{{\rm d}}s\right| .
\end{align*} 
Then set $G(s) = \sum_{m \in \mathcal{J}} \xi_m m^{-s} \in \mathcal{G}(\tau,\varepsilon)$. Since $Hm \in [x,3x]$, inequality (\ref{equ:propmainproofstep3}) gives
\begin{align*} 
&\dfrac{ \# \mathcal{I}^*}{z \log(x)^{A+2 \cdot 10^6K+10^6+2}} \leq  \int_{T}^{2T} \left(\dfrac{C_{z}(c+it)}{H^{-it}u^{c+it}}\right) \overline{G(it)}  F(c+it)\, \mbox{{\rm d}}t. \qedhere
\end{align*} 

\end{proof}

\subsection{Introduction of  R, R* and Q}
\label{sssec:introrq}
 
 Now we complete the proof of Proposition~\ref{proposition1}. 
Let $\mathcal{I}$  denote the set of  $y \in \mathbb{N} \cap [x,3x]$ for which
 \begin{align*}
\left| \sum_{n \in \mathcal{A}(y)} a_n - \dfrac{x^b}{\tau} \sum_{n \in \mathcal{B}(y)} a_n \right| > \dfrac{x}{\tau \log(x)^A}.
 \end{align*}
 Recall that we wish to show  $\# \mathcal{I} =O( \tau x^{0.23+\varepsilon/2})$ for $(a_n)$  as given in Section~\ref{ssec:proposition1}. By Lemma~\ref{lem:intermediatesummary}, it is  enough to show the following: If  $F(s) = \prod_{i=1}^j S_i(s) \in \mathcal{F}((N_i),T,A+10^7K,K)$, $G(s) = \sum_{n \in \mathcal{J}} \xi_n n^{-s} \in \mathcal{G}(\tau,\varepsilon)$, $T\in [T_1,T_0]$ and  $u \in \{1, \dots, (2\log(x))^{10^6K}\}$ satisfy
  \begin{align*}
\left|\int_{T}^{2T} \left(\dfrac{C_{z}(c+it)}{H^{-it}u^{c+it}} \right) \overline{G(it)}  F(c+it)  \mbox{{\rm d}}t\right|&\geq  \dfrac{ \# \mathcal{J}}{z \log(x)^{A+5\cdot 10^6K}}
\end{align*}
 and  $F(s)$, $G(s)$ and $T$ satisfy one of  the conditions (\ref{cc1}), (\ref{cc2}),  (\ref{cc3}), (\ref{cc4.A})  or (\ref{cc4.B}) of Proposition~\ref{proposition1} for every choice of $\gamma \in (-\infty,1)$ and $\sigma_i \in \mathbb{R}$ and $x$ is large, then
 $\# \mathcal{J} \leq \tau^2 x^{0.23-1+\varepsilon/2}$. This implies   $10(x/T_0)\# \mathcal{J} \leq \tau x^{0.23+\varepsilon/2}$ and hence $\#\mathcal{I} \leq    \tau x^{0.23+\varepsilon/2}$.   
 
 \begin{proof}[Proof of Proposition~{\rm\ref{proposition1}}]
 We write $J=A+10^7K$ and consider  $F(s) = \prod_{i=1}^j S_i(s) \in \mathcal{F}((N_i)_{i=0}^{r_1},T,J,K)$, $G(s) = \sum_{n \in \mathcal{J}} \xi_n n^{-s} \in \mathcal{G}(\tau,\varepsilon)$, $z \in \{\tau, x^b\}$, $T\in [T_1,T_0]$ and  $u \in \{1, \dots, (2\log(x))^{10^6K}\}$ which satisfy
  \begin{align} \label{ourassumption}
\left|\int_{T}^{2T} D_z(c+it) \overline{G(it)}  F(c+it)  \mbox{{\rm d}}t\right|&\geq  \dfrac{ \#\mathcal{J}}{z \log(x)^{A+5\cdot 10^6K}} \quad \mbox{ where } \quad  D_z(c+it)=\dfrac{C_{z}(c+it)}{H^{-it}u^{c+it}}. 
\end{align} 
We assume that $F(s)$, $G(s)$ and $T$ satisfy one of  the conditions (\ref{cc1}), (\ref{cc2}),  (\ref{cc3}), (\ref{cc4.A})  or (\ref{cc4.B}) of Proposition~\ref{proposition1} for each choice of $(\sigma_i)$ and $\gamma$. 
Recall here that we defined  $\mathcal{S}_1(F,(S_i),(\sigma_i))$,  
$R(F,(S_i),(\sigma_i),T)$,  
$\mathcal{S}_1^*(F,(S_i),(\sigma_i))$,  
$R^*(F,(S_i),(\sigma_i),T)$, 
$\mathcal{S}_2(G,\gamma)$ and 
$Q(F,G,(S_i),(\sigma_i),\gamma,T)$  in Definition~\ref{def:counting}.

\vspace{3mm}
Since $S_i(s)= \sum_{n \sim B_i} b_n^{(i)} n^{-s}$ with $b_n^{(i)}=O(\log(n))$,  $|S_i(c+it)| \leq \log(2B_i) B_i^{-c+1}$. So if $\sigma_\ell \geq 1+\varepsilon_1$ for some $\ell$, then $\mathcal{S}_1(F,(S_i),(\sigma_i)) = \emptyset$. 
 Let $\Theta_i = \left\{\frac{-2\log(x)+ k \log(2)}{\log(B_i)} : 1 \leq k < \frac{2\log(x)}{\log(2)} +\frac{(1+\varepsilon_1)\log(B_i)}{\log(2)} \right\} $.  Since $|G(it)| \leq \#\mathcal{J}$,  $\mathcal{S}_2(G, \gamma) = \emptyset$ when $\gamma \geq 1$.   Write $\Theta_L = \left\{\frac{-2\log(x)+ k \log(2)}{\log(\#\mathcal{J})} : 1 \leq k < \frac{2\log(x)}{\log(2)} +\frac{\log(\#\mathcal{J})}{\log(2)} \right\} $. Set
\begin{align*}   
&\mathcal{M}(F,(S_i),G)= \bigcup_{k=1}^j\bigcup_{\substack{\sigma_i \in \mathbb{R} \mbox{ \scriptsize \rm for } i \neq k \\ \sigma_k \leq -\frac{2\log(x)}{\log(B_k) } }}\mathcal{S}_1(F,(S_i),(\sigma_i))\quad \cup \bigcup_{\substack{\gamma \leq -\frac{2\log(x)}{\log(\#\mathcal{J}) } }}\mathcal{S}_2(G,\gamma).
\end{align*}
Sets $\{\mathcal{M}(F,(S_i),G)\} \cup \{\mathcal{S}_1(F,(S_i),(\sigma_i)) \cap \mathcal{S}_2(G,\gamma) : \sigma_i \in \Theta_i, \gamma \in \Theta_L \}$  partition  $\mathbb{N} \cap [T, 2T]$. By  (\ref{ourassumption}),
\begin{align} 
\dfrac{\#\mathcal{J}}{z \log(x)^{A+5\cdot 10^6K}} &\leq \sum_{\substack{(\sigma_i):\sigma_i \in \Theta_i\\ \gamma \in \Theta_L}} \Bigg|\sum_{\substack{ n \in  \mathcal{S}_1(F,(S_i),(\sigma_i)) \\  n \in \mathcal{S}_2(G,\gamma) \\ n \in [T,2T]}} \int_{n}^{n+1} D_z(c+it) \overline{G(it)}  F(c+it)  \mbox{{\rm d}}t\Bigg|\label{equ:normalregion}\\
&+\sum_{\substack{ n \in  \mathcal{M}(F,(S_i),G)\\ n \in [T,2T]}}\int_{n}^{n+1}  \left| D_z(c+it) \overline{G(it)}  F(c+it)\right|  \mbox{{\rm d}}t \label{equ:Mregion}\\
&+\int_{[T,T+1] \cup [2T, 2T+1]}  \left|D_z(c+it) \overline{G(it)}  F(c+it)\right|  \mbox{{\rm d}}t .  \label{equ:extra}
\end{align} 
We have  $D_z(c+it) = O(1/z)$, $|G(it)| \leq \#\mathcal{J} $ and  $S_i(c+it)=O(\log(x))$. Recall that $F(s)$ also has factor $S_\ell(s)$ which satisfies property (\ref{propertyd}) with respect to $T$. Using Lemma~\ref{lem:largevaluesall},  $|S_\ell(c+it)| \ll \log(x)^{-(A+5\cdot 10^6K+J+1)}$ and so $|F(c+it)|\ll \log(x)^{-(A+5\cdot 10^6K+1)}$.  
So (\ref{equ:extra}) is $O(\#\mathcal{J}/(z \log(x)^{A+5\cdot 10^6K+1}))$. If $n \in \mathcal{M}(F,(S_i),G)$, we also have  $S_\ell(c+it)=O(1/x^2)$ for all $t \in [n,n+1]$ and  at least one choice of $\ell$ or $G(it) = O(1/x^2)$ for all $t \in [n,n+1]$.  Hence the contribution of (\ref{equ:Mregion}) is also $\ll T_0 (\#\mathcal{J}) \log(x)^{J}/(x^2 z) \ll (\#\mathcal{J})/(z \log(x)^{A+5\cdot 10^6K+1})$. Hence we focus on the RHS  of~(\ref{equ:normalregion}). Since $\# \Theta_i = O(\log (x))$ and $\#\Theta_L = O(\log(x))$, there exist $(\sigma_i)$ with $\sigma_i \in \Theta_i$ and $\gamma \in \Theta_L$ such that  
\begin{align} \label{equ:sec2whatwehave}
\Bigg|\sum_{\substack{ n \in  \mathcal{S}_1(F,(S_i),(\sigma_i)) \\  n \in \mathcal{S}_2(G,\gamma) \\ n \in [T,2T]}} \int_{n}^{n+1} D_z(c+it) \overline{G(it)}  F(c+it)  \mbox{{\rm d}}t\Bigg|\geq  \dfrac{\#\mathcal{J}}{z \log(x)^{A+5\cdot 10^6K+J+2}}. 
\end{align} 
Set $\sigma = (\sum_{i=1}^j \log(B_i) \sigma_i )/(\sum_{i=1}^j \log(B_i))$, so that $(\prod_{i=1}^j B_i)^\sigma = \prod_{i=1}^j B_i^{\sigma_i}$. 

\vspace{3mm}
Recall that we assumed that $F(s)$, $G(s)$ and $T$ satisfy one of  the conditions (\ref{cc1}), (\ref{cc2}),  (\ref{cc3}), (\ref{cc4.A})  or (\ref{cc4.B}) of Proposition~\ref{proposition1} for any given $\gamma \in (-\infty,1)$ and $(\sigma_i)$. We now split our argument into four different cases, depending on which of these conditions holds for the choice of $\gamma$ and $(\sigma_i)$ which satisfies (\ref{equ:sec2whatwehave}). The goal is to show that $\#\mathcal{J} \leq \tau^2 x^{0.23-1+\varepsilon/2}$. 

\vspace{3mm}
\textbf{Case 1: Condition (C1) holds. } Suppose $F(s)$, $(\sigma_i)$ and $T$ satisfy both (\ref{equ:sec2whatwehave}) and condition (\ref{cc1}) of Proposition~\ref{proposition1}. Recalling $D_z(c+it) = O(1/z)$ and $|G(it)| \leq \#\mathcal{J}$, using the size of $F(c+it)$  prescribed by  $\mathcal{S}_1(F,(S_i),(\sigma_i))$ (namely $O(x^{\sigma-c})$)  and rearranging, the LHS of (\ref{equ:sec2whatwehave}) is certainly of size
 \begin{align} \label{equ:consequence1}
O\left(\dfrac{x^{\sigma-c}R(F,(S_i),(\sigma_i),T)\#\mathcal{J}}{z} \right).
\end{align} 
But $R(F,(S_i),(\sigma_i),T) \leq x^{1-\sigma} \log(x)^{-2J}$ by (\ref{cc1}). Recall that  $J = A+10^7K.$ Hence (\ref{equ:consequence1}) is less than $\#\mathcal{J}/(z \log(x)^{A+5 \cdot 10^6K+J+3})$. This contradicts  (\ref{equ:sec2whatwehave}) when $\#\mathcal{J}>0$. Hence $\#\mathcal{J}=0\leq \tau^2 x^{0.23-1+\varepsilon/2}$.

\vspace{3mm}
\textbf{Case 2: Condition (C2) holds. }
Suppose $F(s)$, $(\sigma_i)$ and $T$ satisfy both (\ref{equ:sec2whatwehave}) and condition (\ref{cc2}) of Proposition~\ref{proposition1}. By Cauchy-Schwarz and  Montgomery's mean value estimate for Dirichlet polynomials~\cite{Montgomery:1971:Topics},
\begin{align*}
&\sum_{\substack{ n \in  \mathcal{S}_1(F,(S_i),(\sigma_i)) \cap \mathcal{S}_2(G,\gamma) \cap [T,2T]}} \int_{n}^{n+1} \left|D_z(c+it) \overline{G(it)}  F(c+it)  \right|\mbox{{\rm d}}t\\
&\leq \left(\int^{2T}_{T} \left|G(it)\right|^2 \,\mbox{d}t \right)^{1/2} 
\left(\sum_{\substack{ n \in  \mathcal{S}_1(F,(S_i),(\sigma_i)) \cap \mathcal{S}_2(G,\gamma) \cap [T,2T]}} \int_{n}^{n+1} \left| D_z(c+it)F(c+it)\right|^2 \,\mbox{d}t  \right)^{1/2} \\
&\ll \dfrac{(T_0 \#\mathcal{J})^{1/2} R(F,(S_i),(\sigma_i),T)^{1/2}(  x^{\sigma-c})}{z}. 
\end{align*}
(When applying Montgomery's mean value estimate, we used that $\mathcal{J} \subseteq [1,T_0]$.) Comparing with (\ref{equ:sec2whatwehave}), we get $\#\mathcal{J} \leq \tau x^{2\sigma-2+2\varepsilon_1} R(F,(S_i),(\sigma_i),T)$. But by (\ref{cc2}) we have  $R(F,(S_i),(\sigma_i),T)\leq  \tau x^{1.23-2\sigma+2\varepsilon_1} $ and so $\#\mathcal{J} \leq \tau^2 x^{0.23-1+4\varepsilon_1}$.

\vspace{3mm}
\textbf{Case 3: Condition (C3) holds. } 
Suppose $F(s)$, $(\sigma_i)$ and $T$ satisfy (\ref{equ:sec2whatwehave}) and condition (\ref{cc3}). Let 
\begin{align*} 
K(y) = \sum_{\substack{ n \in  \mathcal{S}_1(F,(S_i),(\sigma_i)) \\  n \in \mathcal{S}_2(G,\gamma) \\ n \in [T,2T]}} \int_n^{n+1} y^{it} u^{-(c+it)} C_z(c+it) F(c+it) \, \mbox{d}t.
\end{align*}
Note that $K(w) - K(y) = \int_y^w K'(u) \mbox{d}u$ and $\frac{1}{H} \int_{y}^{y+H} K(w)\mbox{d}w - K(y) = \frac{1}{H}\int_y^{y+H} K'(u) (y+H-u) \mbox{d}u$. So
\begin{align} 
  \dfrac{\#\mathcal{J}}{z \log(x)^{A+5 \cdot 10^6K+J+2}} &\leq \Bigg| \sum_{\substack{ n \in  \mathcal{S}_1(F,(S_i),(\sigma_i))  \\  n \in \mathcal{S}_2(G,\gamma) \\ n \in [T,2T]}} \int_{n}^{n+1} D_z(c+it) \overline{G(it)}  F(c+it)  \mbox{{\rm d}}t\Bigg| \nonumber \\
&\leq \sum_{m \in \mathcal{J}}\Bigg| \sum_{\substack{ n \in  \mathcal{S}_1(F,(S_i),(\sigma_i))  \nonumber\\  n \in \mathcal{S}_2(G,\gamma) \\ n \in [T,2T]}} \int_{n}^{n+1} (Hm)^{it} u^{-(c+it)} C_z(c+it) F(c+it)  \mbox{{\rm d}}t\Bigg| \\ &= \sum_{m \in \mathcal{J}} | K(Hm)|
\leq  \dfrac{1}{H} \sum_{m \in \mathcal{J}} \int_{Hm}^{Hm+H}  | K(w)| \mbox{d}w +  \sum_{m \in \mathcal{J}} \int_{Hm}^{Hm+H}  | K'(w)| \mbox{d}w \nonumber \\ 
&\leq  \dfrac{(H\#\mathcal{J})^{3/4}}{H}  \left( \int_x^{4x}  | K(w)|^4 \mbox{d}w \right)^{1/4} +  (H\#\mathcal{J})^{3/4} \left( \int_x^{4x}  | K'(w)|^4 \mbox{d}w \right)^{1/4}. \label{equ:r*introduction1}
\end{align} 
The purpose of these estimates was to remove $G(s)$ and to introduce a factor $w^{it}$ with $w$ varying over all $w \in [x,4x]$. This change is needed for the introduction of $R^*(F,(S_i),(\sigma_i),T)$. To be precise, we have 
\begin{align}
&\int_x^{4x}  | K(w)|^4 \mbox{d}w = \int_x^{4x} \Bigg| \sum_{\substack{ n \in  \mathcal{S}_1(F,(S_i),(\sigma_i))  \cap \mathcal{S}_2(G,\gamma) \cap [T,2T]}} \int_n^{n+1} w^{it} \left( \dfrac{ C_z(c+it)}{u^{c+it}} \right) F(c+it) \, \mbox{d}t \Bigg|^4 \mbox{d}w, \label{equ:interstep31}\\
&\int_x^{4x}  | K'(w)|^4 \mbox{d}w = \int_x^{4x} \Bigg| \sum_{\substack{ n \in  \mathcal{S}_1(F,(S_i),(\sigma_i)) \cap \mathcal{S}_2(G,\gamma) \cap [T,2T]}} \int_n^{n+1}   \!\!\!\!\!\! w^{it-1}  \left( \dfrac{ t C_z(c+it)}{u^{c+it}} \right) \! F(c+it) \, \mbox{d}t \Bigg|^4 \mbox{d}w. \label{equ:interstep32}
\end{align}
An upper bound for (\ref{equ:interstep31}) and  (\ref{equ:interstep32}) was given in the proof of Lemma~4.2 of~\cite{Maynard:2012:DCP}: The LHS of (\ref{equ:interstep31}) is $O(x^{-3+4\sigma+\varepsilon_1} z^{-4}R^*(F,(S_i),(\sigma_i),T))$, while the LHS of (\ref{equ:interstep32}) is $O(\tau^4 x^{-7+4\sigma+5\varepsilon_1} z^{-4}R^*(F,(S_i),(\sigma_i),T))$. But (\ref{cc3}) holds    and  $R^*(F,(S_i),(\sigma_i),T) \leq \tau x^{3.23-4\sigma+2\varepsilon_1} $. Hence 
\begin{align} 
(H\#\mathcal{J})^{3/4} \left(\dfrac{1}{H}  \left( \int_x^{4x}  \!\! | K(w)|^4 \mbox{d}w \right)^{1/4} \!\!\! \!\!+ \,  \left( \int_x^{4x}  \!\!| K'(w)|^4 \mbox{d}w \right)^{1/4} \right) \ll (\#\mathcal{J})^{3/4}  \left(  z^{-4} \tau^2 x^{-0.77+4\varepsilon_1} \right)^{1/4} \label{equ:r*introduction2}
\end{align} 
and by combining (\ref{equ:r*introduction1}) and (\ref{equ:r*introduction2}) we find that $\#\mathcal{J} \leq \tau^2 x^{0.23-1+\varepsilon/2}$, provided $x$ is sufficiently large.

\vspace{3mm}
\textbf{Case 4: Condition (C4.A) or (C4.B) holds. }
Finally, suppose $F(s)$, $(\sigma_i)$, $G(s)$, $\gamma$ and $T$ satisfy (\ref{equ:sec2whatwehave}) and one of condition (\ref{cc4.A}) or  condition (\ref{cc4.B}) of Proposition~\ref{proposition1}.
Using the sizes of $F$ and $G$ prescribed by  $\mathcal{S}_1(F,(S_i),(\sigma_i)) \cap \mathcal{S}_2(G,\gamma)$, 
 \begin{align*} 
\Bigg|\sum_{\substack{ n \in  \mathcal{S}_1(F,(S_i),(\sigma_i)) \\  n \in \mathcal{S}_2(G,\gamma) \\ n \in [T,2T]}} \int_{n}^{n+1} D_z(c+it) \overline{G(it)}  F(c+it)  \mbox{{\rm d}}t\Bigg| \ll \left((\#\mathcal{J})^\gamma x^{\sigma-1}\right)\dfrac{Q(F,G,(S_i),(\sigma_i),\gamma,T)}{z}.
\end{align*} 
Hence  by (\ref{equ:sec2whatwehave}),  $(\#\mathcal{J})^{1-\gamma} < x^{\sigma-1+\varepsilon_1} Q(F,G,(S_i),(\sigma_i),\gamma,T)$. 

\vspace{3mm} If (\ref{cc4.A}) holds,  then $  Q(F,G,(S_i),(\sigma_i),\gamma,T) \leq (\#\mathcal{J})^{7/8-\gamma}\tau^{1/4} x^{7/8-\sigma+0.23/8+\varepsilon_1}.$ We find 
\begin{align}
\#\mathcal{J} &<  \left(\dfrac{ x^{\sigma-1+\varepsilon_1} Q(F,G,(S_i),(\sigma_i),\gamma,T)}{(\#\mathcal{J})^{7/8-\gamma}} \right)^8 \leq \tau^2 x^{0.23-1+\varepsilon/2}.
\end{align} 
 Alternatively, if (\ref{cc4.B}) holds, then $ Q(F,G,(S_i),(\sigma_i),\gamma,T) \leq (\#\mathcal{J})^{1-\gamma} x^{1-\sigma-\varepsilon_1}.$ 
In particular, 
\begin{align}
1 &< \dfrac{x^{\sigma-1+\varepsilon_1} Q(F,G,(S_i),(\sigma_i),\gamma,T)}{(\#\mathcal{J})^{1-\gamma}}  
\leq 1.
\end{align}
This gives  a contradiction when $\#\mathcal{J}\neq0$. So $\#\mathcal{J}=0\leq \tau^2 x^{0.23-1+\varepsilon/2}$.

\vspace{3mm} Hence each of the conditions (\ref{cc1}), (\ref{cc2}),  (\ref{cc3}),  (\ref{cc4.A})  and (\ref{cc4.B}) implies $\# \mathcal{J} \leq \tau^2 x^{0.23-1+\varepsilon/2}$ and  $10(x/T_0)\# \mathcal{J} \leq \tau x^{0.23+\varepsilon/2}$. Continuing on from Lemma~\ref{lem:intermediatesummary}, this completes our proof of Proposition~\ref{proposition1}. \end{proof}

\section{Values of Dirichlet polynomials}  \label{sec:values}

The goal of this section is to prove Proposition~\ref{proposition2}. We will give  various bounds on $R(F,(S_i),(\sigma_i),T)$, $R^*(F,(S_i),(\sigma_i),T)$ and $Q(F,G,(S_i),(\sigma_i),\gamma,T)$, dependent on the factorisation of $F(s)$. 

\vspace{3mm}
Let $K=2000$ and $J>10^7K$. Throughout Section~\ref{sec:values} we assume for given  $\tau$, $\varepsilon$ and $T \in [T_1,T_0]$ that  $F(s) = \prod_{i=1}^j S_i(s)$, where 
\begin{align*} 
  S_i(s) = \sum_{n \sim N_i}  \dfrac{a_n^{(i)} }{n^{s}}
\end{align*} 
 for  some  $j \leq J$ and some $a_n^{(i)} \in \mathbb{C}$ with  $a_n^{(i)} =O( \log(n))$.  Further, $N_i \geq \log(x)$ for all $i \in \{1, \dots, j\}$ and  $2^{-J}x \leq \prod_{i=1}^j N_i \leq 2^J x$.  Additionally, we have the following criteria: 
\begin{defn}\label{def:typea}
Let $F(s) = \prod_{i=1}^j S_i(s)$ be as described at the start of Section~{\rm\ref{sec:values}}, with $N_i <x^{0.95}$ for all $i$.

\vspace{3mm}
 Assume that $S_i(s)$ has property {\rm (\ref{propertyz1})} or {\rm(\ref{propertyz2})} with respect to $T$ whenever $N_i \geq x^{2/K}$. Assume also that there exists $\ell \in \{1, \dots, j\}$ with $N_\ell \geq x^{2/K}$ and some $i_0 \neq \ell$ for which $N_{i_0} \geq x^{1/\log\log(x)}$ and for which $S_{i_0}(s)$ has property {\rm (\ref{propertyd})} with respect to $T$.

\vspace{3mm} Then we say that $F(s) = \prod_{i=1}^j S_i(s)$ is a Dirichlet polynomial of Type $A$ at $T$. 
\end{defn}

Here (\ref{propertyd}), (\ref{propertyz1}) and (\ref{propertyz2}) are as described in Definition~\ref{def:propertiesz1z2d}.

\begin{defn}\label{def:typeb}
Let $F(s) = \prod_{i=1}^j S_i(s)$ be as described at the start of Section~{\rm\ref{sec:values}}, with $N_i <x^{2/K}$ for all $i$.

\vspace{3mm}
Assume that for some $i_0 \in \{1, \dots, j\}$, $S_{i_0}(s)$ has property~{\rm (\ref{propertyd})} with respect to $T$ and $N_{i_0} \geq x^{1/\log\log(x)}$.

\vspace{3mm}
Then we say $F(s)=\prod_{i=1}^j S_i(s)$ is a Dirichlet polynomial of Type B at $T$.
\end{defn}

\begin{defn}\label{def:typec}
Let $F(s) = \prod_{i=1}^j S_i(s)$ be as described at the start of Section~{\rm\ref{sec:values}}.

\vspace{3mm}
Assume that for some $\ell \in \{1, \dots, j\}$, $S_\ell(s)$ has property {\rm (\ref{propertyz1})} or {\rm(\ref{propertyz2})} with respect to $T$ and $N_\ell \geq x^{0.95}$.
 
 \vspace{3mm} 
Then we say $F(s)=\prod_{i=1}^j S_i(s)$ is a Dirichlet polynomial of Type C at $T$.
\end{defn}

Throughout Section~\ref{sec:values} we only work with $F(s)$ which are of Type $A$, $B$ or $C$ at $T$. Note in particular that all choices of $F(s) \in \mathcal{F}^*(T,J,K)$, described in Definition~\ref{def:polynomialcollections}, satisfy  one of the above criteria.

\vspace{3mm}
 We also assume that  $G(s) = \sum_{n \in \mathcal{J}} \xi_n n^{-s}$, where $\mathcal{J} \subseteq [1, T_0]$ with  $\#\mathcal{J} \geq  \tau^2x^{-0.77}$  and  $|\xi_n| = 1$.

 \vspace{3mm}
We work with    $\gamma \in (-\infty,1)$ and $(\sigma_i)$ with  $\sigma_i \in  \mathbb{R}$ and corresponding  $\sigma = \frac{\sum_{i=1}^j \log(N_i) \sigma_i}{\sum_{i=1}^j \log(N_i) }$. 
 If $x$ is large and  $\sigma_j > 1+\varepsilon_1$ for some $j$, then $R(F, (S_i), (\sigma_i),T) = 0$. So we assume $\sigma_i \leq 1 +\varepsilon_1$ for all $i$. If $\sigma < 0.6$, then 
 \begin{align*}
 R(F,(S_i),(\sigma_i),T) \leq T \leq \tau x^{\varepsilon_1} \leq \tau x^{1.23-2\sigma+2\varepsilon_1}
 \end{align*}
 and condition (\ref{cc2}) of Proposition~\ref{proposition1} holds. Hence we restrict our attention to $(\sigma_i)$ with $\sigma \in [0.6,1+\varepsilon_1]$.

\vspace{3mm}
We set $N= \prod_{i=1}^j N_i$ and write  $N_i = N^{\ell_i}$, 
 so that $\sum_{i=1}^j \ell_i=1$ and $\sum_{i=1}^j \ell_i \sigma_i =\sigma$.

\subsection{Standard bounds on R} \label{ssec:standnew}

We now consider a non-empty set $I \subseteq \{1, \dots, j\}$  and let $\{k_i : i\in I\}$ be  a corresponding set of positive integers, with $k_i$ bounded by an absolute constant. For simplicity's sake, we assume that $k_i \leq J$. 
We then set
 $M= \prod_{i \in I} N_i^{k_i}$ and $M^\beta = \prod_{i \in I} N_i^{k_i\sigma_i}$, so that
 \begin{align*}
R(F,(S_i), (\sigma_i), T) &\leq \#\Big\{n \in \mathbb{Z} \cap [T,2T]: M^{-c+\beta} \leq \sup_{t\in [n,n+1]} \Big|\prod_{i \in I} S_i(c+it)^{k_i}\Big| \leq 2^{J^2}M^{-c+\beta}\Big\}.
\end{align*}
For the set on the RHS, concerning the size of Dirichlet polynomial $\prod_{i \in I} S_i(s)^{k_i}$, a number of standard bounds are available.  We will frequently use the following three bounds:

\begin{lemma} \label{lem:montgomery}
Uniformly for $T_1 \leq T \leq T_0$ and for $\delta>0$
\begin{align} \label{R:montgomery}
R(F,(S_i), (\sigma_i), T) &\ll_{\delta} x^{\delta} M^{2-2\beta}+x^{\delta} T_0 M^{1-2\beta}, \\
\label{R:huxley}
R(F,(S_i), (\sigma_i), T) &\ll_{\delta} x^{\delta} M^{2-2\beta}+x^{\delta} T_0 M^{4-6\beta}, \\
R(F,(S_i), (\sigma_i), T) &\ll_{\delta} x^{\delta} M^{2-2\beta}+x^{\delta} T_0 M^{11-14\beta}.
\label{R:bound3}
\end{align}
\end{lemma}
Result {\rm (\ref{R:montgomery})} is a consequence of Theorem~7.3 of~\cite{Montgomery:1971:Topics} and  is known as Montgomery's mean value estimate.  {\rm (\ref{R:huxley})} is 
a consequence of Equation~2.9 of~\cite{Huxley:1971:DCP}, 
Huxley's large values estimate. All three bounds can also be deduced from Lemma~8.2 of  \cite{Ivic:1985:Book}.

\vspace{3mm} Next we list simple bounds on $R(F,(S_i), (\sigma_i), T)$ which require $F(s)$ to be of Type $A$, $B$ or $C$ at $T$. Here it is important that long factors of $F(s)$ satisfy property (\ref{propertyz1}) or (\ref{propertyz2}) at $T$.

\subsubsection{Long factors}

Dirichlet polynomials  of Type $C$ are  easy to deal with, as the next lemma shows.

\begin{lemma} \label{lem:typec}
Suppose $F(s)=\prod_{i=1}^j S_i(s)$ is a Dirichlet polynomial of Type $C$ at $T$, where $T \in [T_1,T_0]$. Then for large $x$ and for $(\sigma_i)$ with corresponding $\sigma \geq 0.6$ we have $$R(F,(S_i),(\sigma_i), T) = 0.$$
\end{lemma}

\begin{proof}
Let $S_\ell(s)$ denote the factor of $F(s)$ which has $N_\ell \geq x^{0.95}$. Suppose first that $S_\ell(s)$  satisfies (\ref{propertyz2}). Then $|S_\ell(c+it)| \ll \log(N_\ell) N_\ell^{-1/2} \ll x^{-0.47}$. Now suppose that $S_\ell(s)$ satisfies (\ref{propertyz1}). By Van-der-Corput's method of exponential sums, taking for instance $l=2$ on page~37 of~\cite{Maynard:2012:DCP}, we have for $t \in [T,2T],$
\begin{align*}
|S_\ell(c+it)| \ll_{\delta} \dfrac{T^{1/2+\delta}}{N_\ell}.
\end{align*}
But we are working with $T \leq T_0 = \tau  x^{\varepsilon_1} \leq x^{1-\nu}$ and hence $\nu = 0.23$ and $N_\ell \geq x^{0.95}$ ensure  again that $|S_\ell(c+it)| \ll x^{-0.47}$. Our general assumptions on $F(s)$ also guarantee $|S_i(c+it)|\ll_{\delta} x^{\delta}$. So
\begin{align*}
|F(c+it)| =  |S_\ell(c+it)| \left(\prod_{i \neq \ell} |S_i(c+it)| \right)\ll x^{-0.45}.
\end{align*}
In particular, $R(F,(S_i),(\sigma_i), T) = 0$ for $\sigma \geq 0.6$ and $x$ large.
\end{proof} 

Very good bounds  on $R(F,(S_i),(\sigma_i),T) $ are also available when  a Dirichlet polynomial of Type A has a particularly long factor.
\begin{lemma} \label{lem:newsmooth}
Suppose $F(s)=\prod_{i=1}^j S_i(s)$ is as described at the start of Section~{\rm \ref{sec:values}}. Let $T \in [T_1,T_0]$.

 Suppose $S_\ell(s)$ has property~{\rm (\ref{propertyz1})} or {\rm(\ref{propertyz2})} with respect to $T$. Suppose also that there is some $i_0 \neq \ell$ for which $S_{i_0}(s)$  has property~{\rm(\ref{propertyd})} with respect to $T$ and for which $N_{i_0} \geq x^{1/\log\log(x)}$.
 
  Then if $x \geq C$,  one of the following two inequalities holds:
\begin{align*}
&R(F,(S_i),(\sigma_i),T)  \leq \dfrac{x^{1-\sigma}}{\log(x)^{B}}\quad \mbox{ or } \\
&R(F,(S_i),(\sigma_i),T)  \leq x^{\delta}\min\left\{T N_\ell^{(2-4\sigma_\ell)}, T^2 N_\ell^{(6-12\sigma_\ell)} \right\} 
\end{align*}
Here $C$ is a large constant dependent only on $J \in \mathbb{N}$, $B>0$ and $\delta>0$.
\end{lemma}
 
\begin{proof}
We begin by assuming that $S_\ell(s)$  satisfies property~(\ref{propertyz1}) with respect to $T$.

\vspace{3mm}
It follows from small modifications of the proof of Lemma 5.5 of~\cite{Maynard:2012:DCP} that either
\begin{align*}
R(F,(S_i),(\sigma_i),T)  \ll_{\delta} x^{\delta/2}\min\left\{T N_\ell^{(2-4\sigma_\ell)}, T^2 N_\ell^{(6-12\sigma_\ell)} \right\},
\end{align*}
in which case we are done, or
\begin{align} \label{equ:smoothmiddle}
 T \ll (\log x)^4 N_\ell^{1-\sigma_\ell} \ll (\log x)^4 x^{1-\sigma} \prod_{i \neq \ell} N_i^{\sigma_i-1}.
\end{align}
(This is (i) and (ii) on page 22 of~\cite{Maynard:2012:DCP}.)

\vspace{3mm} Since $S_{i_0}(s)$ has property~(\ref{propertyd}) and $N_{i_0} \geq x^{1/\log\log(x)}$, Lemma~\ref{lem:largevaluesall} gives 
\begin{align*}
|S_{i_0}(c+it)| \leq \dfrac{C_{B,J}}{\log(x)^{B+5+{J}}} = N_{i_0}^{-(B+5+{J})\log\log(x)/\log(N_{i_0})+\log(C_{B,J})/\log(N_{i_0})}
\end{align*}
for $t \in [T,2T]$, where $C_{B,J}$ is a suitable large constant, dependent on $B$ and $J$. Hence $R(F,(S_i),(\sigma_i),T) = 0$ when $\sigma_{i_0} >c-(B+5+{J})\log \log(x)/\log(N_{i_0})+\log(C_{B,J})/\log(N_{i_0})$. 
But $|S_i(c+it)| \ll \log(N_i)$ for all other $i$. By~(\ref{equ:smoothmiddle}) we must have $R(F,(S_i),(\sigma_i),T) = 0$ or 
\begin{align*}
R(F,(S_i),(\sigma_i),T)  \leq T \ll_{B,J} (\log x)^{4+{J}} x^{1-\sigma} \dfrac{1}{\log(x)^{B+5+{J}}}.
\end{align*}
This concludes the proof for the case where $S_\ell(s)$ has property (\ref{propertyz1}). 

\vspace{3mm}
Now suppose  that $S_\ell(s)$  has property (\ref{propertyz2}) with respect to $T$, so that  $S_\ell(c+it)=O(\log(N_\ell) N_\ell^{1/2-c})$. This implies that for $\sigma_\ell > 1/2+2\log\log(x)/\log N_\ell$ and $x$ large, we have $R(F,(S_i),(\sigma_i),T)=0$. On the other hand, if
$\sigma_\ell \leq  1/2+2\log\log(x)/\log N_\ell$, then trivially  
\begin{align*}
&R(F,(S_i),(\sigma_i),T) \leq T \leq x^{\delta}\min\left\{T N_\ell^{(2-4\sigma_\ell)}, T^2 N_\ell^{(6-12\sigma_\ell)} \right\}. \qedhere
\end{align*}
\end{proof}

\subsubsection{Very large sigma}

We also observe that extremely large values of $\sigma$  occur only rarely.

\begin{lemma} \label{lem:verylarge}
Suppose $F(s)=\prod_{i=1}^j S_i(s)$ is of Type A or Type B at $T$. 

Suppose $(\sigma_i)$ corresponds to $\sigma$ with $1-10^{-500} \leq \sigma \leq 1+\varepsilon_1$. Then  for $B>0$,
\begin{align*}
R(F, (S_i), (\sigma_i), T) \ll_B \dfrac{x^{1-\sigma}}{\log(x)^{B}}.
\end{align*}
\end{lemma}

\begin{proof}
First we suppose that $F(s)$ is a Dirichlet polynomial of Type $A$.  Then $F(s)$ has a factor $S_\ell(s)$ with $N_\ell \geq x^{2/K}$ which satisfies (\ref{propertyz1}) or (\ref{propertyz2}). If $S_\ell(s)$ has property (\ref{propertyz1}), the computations of Case 1B on page~37 and~38 of~\cite{Maynard:2012:DCP}, using Van-der-Corput's method of exponential sums, apply in slightly amended form with $k=1000$ and give $R(F,(S_i),(\sigma_i),T) =0$ when $x$ is large and $\sigma \geq 1-10^{-500}$.

\vspace{3mm}
If $S_\ell(s)$ has property (\ref{propertyz2}), then $R(F,(S_i),(\sigma_i),T)=0$ whenever $\sigma_\ell >1/2+2\log\log(x)/\log N_\ell$. But if $\sigma_\ell \leq 1/2+2\log\log(x)/\log N_\ell$, then  $\sum_{i=1}^j  \ell_i = 1$ , $\sum_{i=1}^j \ell_i \sigma_i = \sigma$, $\sigma_i \leq 1+\varepsilon_1$  and $N_\ell \geq x^{0.001}$ certainly give $\sigma < 1-10^{-500}$. 

\vspace{3mm}
Now  suppose that $F(s)$ is a Dirichlet polynomial of Type $B$. There exists $i_0$ such that $S_{i_0}(s)$ has property~(\ref{propertyd}) and such that $N_{i_0} \geq x^{1/\log\log(x)}$. By Lemma~\ref{lem:largevaluesall}, then   
\begin{align*}
|F(c+it)| \ll (\log x)^{J} x^{-(\log x)^{-7/10} (\log \log x)^{-1}} 
\end{align*}
and we may work with $\sigma \leq 1- (\log x)^{-8/10}$. Since  $\nu = 0.23$, a few small numerical changes to the computations of Case 1A on page~37 of~\cite{Maynard:2012:DCP} give 
\begin{align*}
R(F,(S_i),(\sigma_i),T) 
&\ll T_0^{5(1-\sigma)/4} \log(x)^{C}
 \end{align*}
 for a very large constant $C$, dependent only on $J$. However, $T_0^{5/4} \leq x^{1-0.01}$ and $-(1-\sigma) \leq -(\log x)^{-8/10}$. 
 \begin{align*}
R(F,(S_i),(\sigma_i),T) 
&\ll x^{(1-\sigma)-0.01 \log(x)^{-8/10}} \log(x)^{C}
\\
&\ll_B 
 \dfrac{x^{1-\sigma}}{(\log x)^{B }}
 \end{align*}
  for $1-10^{-500} \leq \sigma \leq 1- (\log x)^{-8/10},$ as proposed.
\end{proof}

\subsubsection{Simple general bounds}

Finally, we also give some bounds on $R(F,(S_i),(\sigma_i),T)$ which do not depend on the lengths $N_i$ of factors $S_i(s)$. Often they are   convenient for simple computations.

\begin{lemma} \label{lem:large2}
Suppose $F(s)=\prod_{i=1}^j S_i(s)$ is of Type $A$ or $B$ at $T$,  where $T \in [T_1,T_0]$. 

Let $I \subseteq \{1, \dots, j\}$ and suppose $\sum_{i \in I} \ell_i \geq 0.01$ and $\sigma_I = (\sum_{ i \in I} \ell_i \sigma_i)/ ( \sum_{i \in I} \ell_i)\geq 0.6$.

  Let $B >0$ and  $\delta>0$. Suppose 
\begin{align*} 
R(F,(S_i),(\sigma_i),T) > \dfrac{x^{1-\sigma}}{\log(x)^{B }}.
\end{align*}
 Then the following four bounds hold:
\begin{align} \label{equ:large2.1}
&R(F,(S_i),(\sigma_i),T) \ll T_0^{(3-3\sigma_I)/(2-\sigma_I) + \delta} \,\,\,\, \quad \mbox{ if } \quad  \sigma_I \geq 6/10,  \\\label{equ:large2.2}
&R(F,(S_i),(\sigma_i),T) \ll T_0^{(3-3\sigma_I)/(3\sigma_I-1) + \delta}  \,\,\quad \mbox{ if } \quad \sigma_I \geq 6/10, \\
\label{equ:large2.3}
&R(F,(S_i),(\sigma_i),T) \ll T_0^{(3-3\sigma_I)/(10\sigma_I-7) + \delta} \quad \mbox{ if } \quad 7/10 < \sigma_I \leq 25/28, \\ \label{equ:large2.4}
&R(F,(S_i),(\sigma_i),T) \ll T_0^{(4-4\sigma_I)/(4\sigma_I-1) + \delta} \quad \,\;\mbox{ if } \quad \sigma_I \geq 25/28.
\end{align}
Here the implied constants depend on $B$ and $\delta$.
\end{lemma}
\begin{proof}   
This statement follows from the proof of  Lemma~5.6 of~\cite{Maynard:2012:DCP}. Key in  that proof was the observation that bounds~(\ref{equ:large2.1}), (\ref{equ:large2.2}), (\ref{equ:large2.3}) and (\ref{equ:large2.4}) hold if there exist $I_2 \subseteq I$ and $k \in \mathbb{N}$ with $(\prod_{i \in I_2} N^{ \ell_i})^k \in [Y^{1/2}, Y]$ and $\sum_{i \in I_2}  \ell_i \sigma_i \geq \sigma_I (\sum_{i \in I_2} \ell_i)$ and $k=O(1)$, where $$Y = \min\{T_0^{3/(8-4\sigma_I)}, T_0^{3/(12\sigma_I-4)}, T_0^{3/(40\sigma_I-28)}, T_0^{3/(4\sigma_I-1)}\}.$$
 Since $T_0 \geq x^{0.47}$ and $\sigma_I \in [0.6,1+\varepsilon_1]$, we got $Y \geq x^{0.11}$.

\vspace{3mm} If $F(s)$ is a Dirichlet polynomial of Type $B$, all $S_i(s)$ have $N_i  < x^{0.001}$ and we can easily choose $I_2$  such that  $\prod_{i \in I_2} N^{ \ell_i} \in [x^{0.01}, Y]$ and  $\sum_{i \in I_2}  \ell_i \sigma_i \geq \sigma_I (\sum_{i \in I_2} \ell_i)$ are satisfied.

\vspace{3mm}
On the other hand, in the presence of longer factors, the proof of~Lemma~5.6 of~\cite{Maynard:2012:DCP} additionally required that the conclusion of Lemma~\ref{lem:newsmooth} holds whenever $i\in I$ and $N_i \geq Y x^{-0.01} \geq x^{0.1}$. But if $F(s)$ is a Dirichlet polynomial of Type A, then $N_i \geq x^{0.1}$ implies that $S_i(s)$ satisfies property~{\rm (\ref{propertyz1})} or {\rm(\ref{propertyz2})}, so that    Lemma~\ref{lem:newsmooth} is indeed applicable. Hence Lemma~\ref{lem:large2} follows from the proof of~(134), (135), (136) and (137)  of Lemma~5.6  of~\cite{Maynard:2012:DCP}.
\end{proof}

We often use the following version of Lemma~\ref{lem:large2}:
\begin{cor} \label{cor:large2}
Suppose $F(s)=\prod_{i=1}^j S_i(s)$ is of Type $A$ or $B$ at $T$,  where $T \in [T_1,T_0]$. 

Let $I \subseteq \{1, \dots, j\}$ and suppose $\sum_{i \in I} \ell_i \geq 0.01$. Write $\sigma_I = (\sum_{ i \in I} \ell_i \sigma_i)/( \sum_{i \in I} \ell_i)$.

 Let $B >0$. Suppose that $\sigma_I \geq \sigma$ or $\sum_{i \in I} \ell_i \geq 0.2$. Then one of  the following three statements  holds:
\begin{align}  
&R(F,(S_i),(\sigma_i),T) \leq \dfrac{x^{1-\sigma}}{\log(x)^B}, \label{equ:shortcut000}\\
&R(F,(S_i),(\sigma_i),T) \leq \tau x^{1-2\sigma+\nu +2\varepsilon_1}, \label{equ:shortcut00} \\
&R(F,(S_i),(\sigma_i),T) \leq  \min\left\{ x^{2\varepsilon_1 }\tau^{(3-3\sigma_I)/(2-\sigma_I)},
x^{2\varepsilon_1 }\tau^{(3-3\sigma_I)/(3\sigma_I-1)}\right\} \, \mbox{ and } \,\, \sigma_I \geq 0.35. \label{equ:shortcut0}
\end{align}
Let $B(a,\sigma)$ be a non-negative function dependent on $a \in [0.475-\varepsilon,0.77-\varepsilon]$ and $\sigma \in [0.6,1-10^{-500}]$.

 If  {\rm (\ref{equ:shortcut0})} holds, we also have $R(F,(S_i),(\sigma_i),T) \leq x^{B(a,\sigma)}$, provided
\begin{align}
\sigma_I \geq \min\left\{\dfrac{3a -2B(a,\sigma) + 4\varepsilon_1}{3a-B(a,\sigma)+ 2\varepsilon_1},  \dfrac{3a+B(a,\sigma)-2\varepsilon_1}{3a+3B(a,\sigma) -6\varepsilon_1 }
\right\}.
\end{align} 
Finally, if  $\sigma_I>0.7$, we also  have {\rm (\ref{equ:shortcut000})} or  $R(F,(S_i),(\sigma_i),T) \leq x^{B(a,\sigma)}$ whenever
\begin{align}\label{equ:shortcutlarge}
\sigma_I \geq \max\left\{\dfrac{3a+7B(a,\sigma)-14\varepsilon_1}{3a+10B(a,\sigma) -20\varepsilon_1 },  \dfrac{4a+B(a,\sigma)-2\varepsilon_1}{4a+4B(a,\sigma) -8\varepsilon_1 }
\right\}.
\end{align} 
\end{cor} 
 
\begin{proof} 
Assume $R(F,(S_i),(\sigma_i),T)>x^{1-\sigma}\log(x)^{-B}$. We always take $\sigma \geq 0.6$, so if $\sigma_I \geq \sigma$,  (\ref{equ:shortcut0}) follows directly from Lemma~\ref{lem:large2}. (We use  that  $(3-3x)/(3x-1)$ and $(3-3x)/(2-x)$  are bounded by $1.5$ on $[0.6,1]$ to replace $T_0$ by $\tau$.)
If $\sigma_I < \sigma$, Lemma~\ref{lem:large2}, applied with $\{1,\dots, j\}$ in the place of $I$, gives  
\begin{align} \label{equ:shortcut1}
R(F,(S_i), (\sigma_i),T) \leq \min\left\{x^{2\varepsilon_1 } \tau^{(3-3\sigma)/(2-\sigma)}, x^{2\varepsilon_1 }\tau^{(3-3\sigma)/(3\sigma-1)} \right\}.
\end{align}
But  $(3-3x)/(2-x)$ is decreasing on $x<2$ and $(3-3x)/(3x-1)$ is decreasing on $x \geq 0.35$. Thus if $0.35 \leq \sigma_I < \sigma$, then (\ref{equ:shortcut1}) implies (\ref{equ:shortcut0}). Finally, if $\sigma_I < 0.35$, then $I' = \{1, \dots, j \} \setminus I$ satisfies 
\begin{align*}
\sigma_{I'}  = \dfrac{\sigma - \sum_{i \in I} \ell_i \sigma_i }{1- \sum_{i \in I} \ell_i} \geq \dfrac{\sigma - 0.35 \cdot 0.2}{1-0.2},
\end{align*}
 provided $\sum_{i \in I} \ell_i \geq 0.2$. In this case, Lemma~\ref{lem:large2}, applied with $I'$ in the place of $I$, gives 
\begin{align} \label{equ:shortcut2}
R(F,(S_i), (\sigma_i),T) \leq\min\left\{x^{2\varepsilon_1 } \tau^{(3.2625-3.75\sigma)/(2.0875-1.25\sigma)}, x^{2\varepsilon_1 }\tau^{(3.2625-3.75\sigma)/(3.75\sigma-1.2625)} \right\}.
\end{align}
For $a \in [0.475-\varepsilon,0.77-\varepsilon]$ and $\sigma \geq 0.6$, the RHS of (\ref{equ:shortcut2}) is less than $\max\{\tau x^{1-2\sigma + 0.23+2\varepsilon_1}, x^{1-\sigma-\varepsilon_1}\}$. 
 
\vspace{3mm}
Now suppose we have (\ref{equ:shortcut0}). If $B(a,\sigma) > a+\varepsilon_1$, then $R(F,(S_i),(\sigma_i),T) \leq T_0 < x^{B(a,\sigma)}$ holds trivially. Thus we assume $0 \leq B(a,\sigma) \leq a+\varepsilon_1$.  Inequality $x^{2\varepsilon_1}\tau^{(3-3\sigma_I)/(2-\sigma_I) } \leq x^{B(a,\sigma)}$ is satisfied if
\begin{align*}
\sigma_I \geq \dfrac{3a -2B(a,\sigma) + 4\varepsilon_1}{3a-B(a,\sigma)+ 2\varepsilon_1},
\end{align*} 
while inequality  $x^{2\varepsilon_1}\tau^{(3-3\sigma_I)/(3\sigma_I-1) } \leq x^{B(a,\sigma)}$  holds if
\begin{align*}
\sigma_I \geq \dfrac{3a+B(a,\sigma)-2\varepsilon_1}{3a+3B(a,\sigma) -6\varepsilon_1 }.
\end{align*}
Finally, if  $\sigma_I > 0.7$, then $R(F,(S_i),(\sigma_i),T)\leq  \max\{x^{2\varepsilon_1} \tau^{(3-3\sigma_I)/(10\sigma_I-7) }, x^{2\varepsilon_1} \tau^{(4-4\sigma_I)/(4\sigma_I-1)} \}$ by bounds  (\ref{equ:large2.3}) and (\ref{equ:large2.4}) of Lemma~\ref{lem:large2}. The RHS of this inequality is at most $x^{B(a,\sigma)}$ if  (\ref{equ:shortcutlarge}) holds.
\end{proof}

\subsection{Heath-Brown's R* bound} 
We now have plenty of bounds on $R$ and move on to $R^*$. We use the following result of Heath-Brown~\cite{Heath-Brown:1979:ZDE}, stated here like in Lemma~5.4 of~\cite{Maynard:2012:DCP}:

\begin{lemma}[Heath-Brown's $R^*$ bound]
 \label{lem:r*bound}
 Let $F(s)=\prod_{i=1}^j S_i(s)$ be as described at the start of Section~{\rm \ref{sec:values}}.
 
 Let $I \subseteq \{1, \dots, j\}$  and $k_i \in \mathbb{N} \cap [1,J]$. Set
 $M= \prod_{i \in I} N_i^{k_i}$ and $M^\beta = \prod_{i \in I} N_i^{k_i\sigma_i}$. For $T \in [T_1,T_0]$,
\begin{align*} 
R^*(F,(S_i),(\sigma_i),T) &\ll_{\delta} x^{\delta}M^{1-2\beta} ( RM+R^2+R^{5/4}T_0^{1/2})^{1/2}(R^*M+R^4+R(R^*)^{3/4}T_0^{1/2})^{1/2},
\end{align*}
where $R=R(F,(S_i), (\sigma_i), T)$ and $R^*=R^*(F,(S_i), (\sigma_i), T)$.
\end{lemma}
 
To obtain a good bound on $R^*$, we thus again need to bound $R$. In particular, recalling that  condition~(\ref{cc3}) of Proposition~\ref{proposition1} requires $R^*(F,(S_i),(\sigma_i),T)   \leq \tau x^{3-4\sigma+\nu+2\varepsilon_1}$, we desire the following bounds:

\begin{cor} \label{cor:amendc3} 
Let $F(s)=\prod_{i=1}^j S_i(s)$ be as described at the start of Section~{\rm \ref{sec:values}}.

Suppose  the following condition  {\rm(\ref{c3'})} holds: 
\begin{equation} \tag{C3$^*$} \label{c3'}
  \parbox{400pt}{%
   There exist $I \subseteq \{1, \dots, j\}$ and $k_i \in \mathbb{N} \cap [1,J]$ such that  
 $M= \prod_{i \in I} N_i^{k_i}$ and $M^\beta = \prod_{i \in I} N_i^{k_i\sigma_i}$
 satisfy each of the following nine inequalities:
}
\end{equation}
\begin{align} 
&R(F,(S_i),(\sigma_i),T) \leq \tau M^{4\beta-4} x^{3-4\sigma+\nu+\varepsilon_1/2}, \label{equ:9inequ1} \\
&R(F,(S_i),(\sigma_i),T) \leq \tau^{1/2} M^{2\beta-3/2} x^{3/2-2\sigma+(1/2)\nu+\varepsilon_1/2}, \\
&R(F,(S_i),(\sigma_i),T) \leq \tau^{2/5} M^{(16/5)\beta-12/5} x^{12/5-(16/5)\sigma+(4/5)\nu+\varepsilon_1/2}, \\
&R(F,(S_i),(\sigma_i),T) \leq \tau^{2/5} M^{(4/5)\beta-3/5} x^{6/5-(8/5)\sigma+(2/5)\nu+\varepsilon_1/2}, \\
&R(F,(S_i),(\sigma_i),T) \leq \tau^{1/3} M^{(2/3)\beta-1/3} x^{1-(4/3)\sigma+(1/3)\nu+\varepsilon_1/2}, \\
&R(F,(S_i),(\sigma_i),T) \leq \tau^{2/7} M^{(16/21)\beta-8/21} x^{8/7-(32/21)\sigma+(8/21)\nu+\varepsilon_1/2}, \\
&R(F,(S_i),(\sigma_i),T) \leq \tau^{3/8} M^{2\beta-3/2} x^{15/8-(5/2)\sigma+(5/8)\nu+\varepsilon_1/2}, \\
&R(F,(S_i),(\sigma_i),T) \leq \tau^{1/4} M^{(4/3)\beta-2/3} x^{5/4-(5/3)\sigma+(5/12)\nu+\varepsilon_1/2}, \\
&R(F,(S_i),(\sigma_i),T) \leq \tau^{1/9} M^{(16/9)\beta-8/9} x^{5/3-(20/9)\sigma+(5/9)\nu+\varepsilon_1/2}. \label{equ:9inequ9}
\end{align} 
Then $R^*(F,(S_i),(\sigma_i),T) \leq \tau x^{3-4\sigma+\nu+2\varepsilon_1}$.
\end{cor}

\begin{proof} 
We  rearrange bound $R^* \ll_{\delta} x^{\delta}M^{1-2\beta} ( RM+R^2+R^{5/4}T_0^{1/2})^{1/2}(R^*M+R^4+R(R^*)^{3/4}T_0^{1/2})^{1/2}$, given in  Lemma~\ref{lem:r*bound}. For a fixed $\delta>0$ and sufficiently large $x$, one of the following nine inequalities holds:
\begin{align} \label{equ:9inequinter1}
&R^*(F,(S_i),(\sigma_i),T) \leq x^{\delta} M^{4-4\beta} R, \\
&R^*(F,(S_i),(\sigma_i),T) \leq x^{\delta} M^{3-4\beta} R^2,  \\ 
&R^*(F,(S_i),(\sigma_i),T) \leq x^{\delta} M^{3-4\beta} R^{5/4} T_0^{1/2},  \\ 
&R^*(F,(S_i),(\sigma_i),T) \leq x^{\delta} M^{{3/2}-2\beta} R^{5/2},  \\ 
&R^*(F,(S_i),(\sigma_i),T) \leq x^{\delta} M^{1-2\beta} R^{3},  \\ 
&R^*(F,(S_i),(\sigma_i),T) \leq x^{\delta} M^{1-2\beta} R^{21/8} T_0^{1/4},  \\ 
&R^*(F,(S_i),(\sigma_i),T) \leq x^{\delta} M^{12/5-(16/5)\beta} R^{8/5} T_0^{2/5},  \\ 
&R^*(F,(S_i),(\sigma_i),T) \leq x^{\delta} M^{8/5-(16/5)\beta} R^{12/5} T_0^{2/5},  \\ 
&R^*(F,(S_i),(\sigma_i),T) \leq x^{\delta} M^{8/5-(16/5)\beta} R^{9/5} T_0^{4/5}, \label{equ:9inequinter9}
\end{align}
where $R = R(F,(S_i),(\sigma_i),T)$. However, if $R^*(F,(S_i),(\sigma_i),T) \leq x^{\delta} M^{A_1-B_1\beta} R^{C_1}T_0^{D_1}$ with $\delta$  sufficiently small compared to $\varepsilon_1$, then $
R^*(F,(S_i),(\sigma_i),T) \leq \tau x^{3-4\sigma+\nu+2\varepsilon_1}$ whenever
\begin{align} \label{equ:transform}
R(F,(S_i),(\sigma_i),T) \leq \tau^{(1-D_1)/C_1} M^{(B_1/C_1)\beta-A_1/C_1} x^{3/C_1-(4/C_1)\sigma+\nu/C_1+\varepsilon_1/2}.
\end{align}
Applying (\ref{equ:transform}) to inequalities (\ref{equ:9inequinter1}) -- (\ref{equ:9inequinter9}), we obtain inequalities (\ref{equ:9inequ1}) -- (\ref{equ:9inequ9}).
\end{proof} 

\subsection{Sparse Dirichlet polynomials - Heath-Brown's bound on Q} 

To bound $Q$, we use recent work of Heath-Brown~\cite{Heath-Brown:2019:CP5}. It  provides good bounds on long Dirichlet polynomials with few non-zero coefficients:

\begin{lemma} \label{lem:hbn}
Let $G(s) = \sum_{n \in \mathcal{J}} \xi_n n^{-s}$ be a Dirichlet polynomial with $\mathcal{J} \subseteq [1,T_0]$ and $|\xi_n|=1$.   

Let $F(s) = \sum_{n \leq N} a_n n^{-s} $ be a Dirichlet polynomial with $a_n \in \mathbb{C}$ and $|a_n| \ll_\delta n^{\delta}.$ Write $L = \#\mathcal{J}$. Then
 \begin{align*}
 \int^{T_0}_0 \Big|\sum_{m \in \mathcal{J}} \xi_m m^{-it} \Big|^2 \Big| \sum_{n \leq N} a_n n^{-it} \Big|^2 \, \mbox{d}t \ll_{\delta} (N^2L^2 + (NT_0)^{\delta} (NLT_0 + NL^{7/4} T_0^{3/4}))N^{\delta}
 \end{align*}
 holds for any fixed  $\delta>0$.
\end{lemma} 
This is Proposition~1 of~\cite{Heath-Brown:2019:CP5}. This result was also used by Heath-Brown in~\cite{Heath-Brown:2018:CSN} to give a bound on the squares of gaps between smooth numbers and in~\cite{Heath-Brown:2019:CP5} to improve a bound on the number of very long prime gaps. Simultaneously to our work, the latter bound was further improved by Järviniemi~\cite{Jarviniemi:2022:LDP}, who showed that  at most $O(x^{0.57+\varepsilon})$ intervals $[y,y+y^{1/2}]$ with $y \in [x,2x]\cap \mathbb{N}$  
do not contain primes. Recently, Matomäki and Teräväinen~\cite{Matomaki:2022:APS} also  used Heath Brown's sparse Dirichlet polynomial bound to improve a bound on the length of intervals which almost surely contain products of exactly two primes. 

\vspace{3mm} We use Lemma~\ref{lem:hbn} to  deduce the following bound on $Q(F,G,(S_i),(\sigma_i),\gamma,T)$:

\begin{cor} \label{cor:boundonq}
Let $F(s)=\prod_{i=1}^jS_i(s)$ and $G(s)$ both be as described at the start of Section~{\rm\ref{sec:values}}.  

In particular, $G(s) = \sum_{m \in \mathcal{J}} \xi_m m^{-s}$ with $L = \#\mathcal{J} \geq \tau^2 x^{-0.77}$. 
Let $I \subseteq \{1, \dots, j\}$  and $k_i \in \mathbb{N} \cap [1,J]$.

 Set
 $M= \prod_{i \in I} N_i^{k_i}$ and $M^\beta = \prod_{i \in I} N_i^{k_i\sigma_i}$. Then for every $\delta>0$,
\begin{align*}
Q(F,G,(S_i),(\sigma_i),\gamma,T) \ll_{\delta} x^{\delta} (M^{2-2\beta}L^{2-2\gamma} +  M^{1-2\beta}L^{7/4-2\gamma} T_0^{3/4}).
\end{align*} 

\end{cor}

\begin{proof} Let $\mathcal{S}_1(F,(S_i),(\sigma_i))$ and $\mathcal{S}_2(G,\gamma)$ be as given in Definition~\ref{def:counting}. For each $k \in \mathcal{S}_1\cap \mathcal{S}_2\cap [T,2T]$, select $t_k^{(i)} \in [k,k+1]$ and $t_k^{(G)} \in [k,k+1]$ with $|S_i(c+it_k^{(i)})| \geq N_i^{-c+\sigma_i}$ and $|G(it_k^{(G)})| \geq L^\gamma$. By partial summation we find that for $t \in [k,k+1]$
\begin{align*}
S_i(c+it_k^{(i)}) &=  \sum_{n \sim N_i} a_n^{(i)} n^{-c-it_{k}^{(i)}}
= (2N_i)^{-c-i(t_{k}^{(i)}-t)} \sum_{n \sim N_i} a_n^{(i)} n^{-it} + O\left(\dfrac{1}{N_i^2} \int^{2N_i}_{N_i}\left| \sum_{n <y} a_n^{(i)} n^{-it}\right|\mbox{d}y\right).
\end{align*}
We obtain a similar expression for $G(it_k^{(G)})$. Multiplying these expressions together, we find 
\begin{align}
L^\gamma M^{-c+\beta} &\leq \left| G(it_k^{(G)}) \prod_{i\in I}S_i(c+it_k^{(i)})^{k_i} \right| \label{equ:smallproduct}\\ \nonumber
&\ll \left(  \left| \sum_{m \in \mathcal{J}}  \xi_m m^{-it} \right| + \int^{T_0 }_{1} \dfrac{1}{y}\Bigg| \sum_{\substack{m<y \\ m \in \mathcal{J}}}\xi_m m^{-it} \Bigg|\,\mbox{d}y  \right) \\
&\cdot
\prod_{i \in I} \left( \dfrac{1}{N_i} \left| \sum_{n \sim N_i} a_n^{(i)} n^{-it} \right| + \dfrac{1}{N_i^2}\int^{2N_i}_{N_i} \left| \sum_{n <y} a_n^{(i)} n^{-it} \right|\mbox{d}y  \right)^{k_i} \label{equ:bigproduct}
\end{align}
and thus work with summands of the form
\begin{align*}
\left(\prod_{i \in I} \dfrac{1}{N_i^{k_i}} \prod_{ j=1}^r \dfrac{1}{N_{i_j}} \right)\int^{2N_{i_1}}_{N_{i_1}} \dots\int^{2N_{i_r}}_{N_{i_r}} \left| \sum_{m \in \mathcal{J}}\xi_m m^{-it} \right|\prod_{j =1}^{r}  \left| \sum_{n <y_j} a_n^{(i_j)} n^{-it} \right| \prod_{j=r+1}^l\left|\sum_{n \sim N_{i_j}} a_n^{(i_j)} n^{-it}\right| \,\mbox{d}y_1 \dots\,\mbox{d}y_r
\end{align*}
 or similar products which involve $(1/y)\sum_{m<y} \xi_m m^{-it}$ integrated with respect to $y$ rather than the full Dirichlet polynomial $\sum \xi_m m^{-it}$. Here $i_1, \dots, i_l$ are such that each $i \in I$ appears exactly $k_i$ times. 

\vspace{3mm}
 Next, we integrate~(\ref{equ:bigproduct}) over $t \in [k,k+1]$ and sum $k$ over all elements of $\mathcal{S}_1 \cap \mathcal{S}_2 \cap [T,2T]$. On the LHS we get $L^\gamma M^{-c+\beta} Q(F,G,(S_i),(\sigma_i),\gamma,T)$, while on the RHS we are now looking
 at summands of the form
\begin{align} \label{equ:multipleint}
&\left(\prod_{i \in I} \dfrac{1}{N_i^{k_i}} \prod_{ j=1}^r \dfrac{1}{N_{i_j}} \right)  \int^{2N_{i_1}}_{N_{i_1}} \!\! \!\! \!\! \!\! \dots\int^{2N_{i_r}}_{N_{i_r}} \sum_{\substack{k \in \mathcal{S}}}  \int_k^{k+1}\left| \sum_{m \in \mathcal{J}} \xi_m m^{-it} \right| \left|\sum_{n \leq 2^\alpha M} b_n(\mathbf{y}) n^{-it}\right| \,\mbox{d}t\,\mbox{d}y_1 \dots\,\mbox{d}y_r, \\
&\mbox{where } \sum_{n \leq 2^\alpha M} b_n(\mathbf{y}) n^{-it} =\prod_{j =1}^{r}  \left( \sum_{n <y_j} a_n^{(i_j)} n^{-it} \right) \prod_{j=r+1}^l\left(\sum_{n \sim N_{i_j}} a_n^{(i_j)} n^{-it}\right), \nonumber 
\end{align}
 or similar expressions in which  $\sum \xi_m m^{-it}$ is replaced by  $ (1/y)\sum_{m<y} \xi_m m^{-it}$ and integrated over $[1,T_0]$. Here  $\alpha = \sum_{i \in I} k_i$ and  $|b_n(\mathbf{y})| \ll_\delta n^{\delta}$ for $\delta>0$ and $\mathcal{S}= \mathcal{S}_1(F,(S_i),(\sigma_i)) \cap \mathcal{S}_2(G,\gamma) \cap [T,2T]$. 
Focusing only on the innermost integral, the  Cauchy-Schwarz inequality gives 
\begin{align*}
\sum_{k \in \mathcal{S}} \int _k^{k+1}
\Bigg| \sum_{ \substack{ m <y \\ m \in \mathcal{J} }} \xi_m m^{-it} \Bigg|\Bigg| \sum_{n \leq 2^\alpha M} b_n(\mathbf{y}) n^{-it} \Bigg| \,\mbox{d}t \leq
Q^{1/2}\left(\int _0^{2T_0} 
\Bigg| \sum_{ \substack{ m <y \\ m \in \mathcal{J}  }}  \xi_m m^{-it} \Bigg|^2\Bigg| \sum_{n \leq 2^\alpha M} b_n(\mathbf{y}) n^{-it} \Bigg|^2 \,\mbox{d}t\right)^{1/2}\!\!\!, 
\end{align*}
where $Q=Q(F,G,(S_i), (\sigma_i),\gamma,T)$. 
Lemma~\ref{lem:hbn} applies to the RHS and hence  (\ref{equ:multipleint}) is bounded above by $O_\delta( Q^{1/2} x^\delta (1/M) (M^2L^2 +  MLT_0 + ML^{7/4} T_0^{3/4})^{1/2})$. Returning to (\ref{equ:smallproduct}), which we integrated over $[k,k+1]$ with $k \in \mathcal{S}$, we have $L^\gamma M^{-c+\beta}Q$ on the LHS and a finite number of multiple integrals of the form given in (\ref{equ:multipleint}) on the RHS. Hence 
\begin{align*} 
Q(F,G,(S_i),(\sigma_i),\gamma,T)L^{2\gamma}M^{-2c+2\beta} \ll_{\delta} x^{\delta} M^{-2}(M^2L^2 + MLT_0 + ML^{7/4} T_0^{3/4}).
\end{align*}
Since $\tau>x^{0.475}$ and $L>\tau^2 x^{-0.77}$, we have $L>\tau^{1/3}$ and $MLT_0<ML^{7/4}T_0^{3/4}$. Then
\begin{align*} 
Q(F,G,(S_i),(\sigma_i),\gamma,T) \ll_{\delta} x^{\delta} (M^{2-2\beta}L^{2-2\gamma} +  M^{1-2\beta}L^{7/4-2\gamma} T_0^{3/4})
\end{align*}
for $\delta>0$. This is what we wanted to show.
\end{proof}

We now use Corollary~\ref{cor:boundonq} to replace conditions (\ref{cc4.A}) and (\ref{cc4.B}) by bounds on $R(F,(S_i), (\sigma_i),T)$. 

\begin{cor} \label{cor:amendc4}
Let $F(s)=\prod_{i=1}^j S_i(s)$ and $G(s) = \sum_{m \in \mathcal{J}}\xi_m m^{-s}$  be as described at the start of Section~{\rm \ref{sec:values}}.

Suppose  the following condition  {\rm(\ref{c4'})} holds: 
\begin{equation} \tag{C4$^*$} \label{c4'}
  \parbox{400pt}{%
   There exist $I \subseteq \{1, \dots, j\}$ and $k_i \in \mathbb{N} \cap [1,J]$ such that  
 $M= \prod_{i \in I} N_i^{k_i}$ and $M^\beta = \prod_{i \in I} N_i^{k_i\sigma_i}$
 satisfy each of the following two inequalities:
}
\end{equation}
\begin{align} 
&R(F,(S_i),(\sigma_i),T) \leq \tau^{-1/4}M^{2\beta-1}  x^{7/4-2\sigma+\nu/4-4\varepsilon_1}, \label{equ:2inequ1} \\
&R(F,(S_i),(\sigma_i),T) \leq M^{2\beta-2} x^{2-2\sigma-4\varepsilon_1}. \label{equ:2inequ2}
\end{align} 

\vspace{3mm}
Then $Q(F,G,(S_i),(\sigma_i),\gamma,T) \leq \max\{ (\#\mathcal{J})^{7/8-\gamma}\tau^{1/4} x^{7/8-\sigma+\nu/8+\varepsilon_1}, (\#\mathcal{J})^{1-\gamma} x^{1-\sigma-\varepsilon_1}\}$.
\end{cor}

\begin{proof}
 Corollary~\ref{cor:boundonq} tells us that one of  the following two inequalities holds when $x$ is large: 
\begin{align}  
&Q(F,G,(S_i),(\sigma_i),\gamma,T) \leq x^{\varepsilon_1}  M^{1-2\beta}(\#\mathcal{J})^{7/4-2\gamma} T_0^{3/4}, \label{equ:qcomp1}\\
&Q(F,G,(S_i),(\sigma_i),\gamma,T) \leq x^{\varepsilon_1} M^{2-2\beta}(\#\mathcal{J})^{2-2\gamma} . \label{equ:qcomp2}
\end{align} 
 First assume we have (\ref{equ:qcomp1}). Note that $Q=Q(F,G,(S_i),(\sigma_i),\gamma,T) \leq R=R(F,(S_i),(\sigma_i),T)  $. Hence 
\begin{align} \label{equ:qcomp11}
Q(F,G,(S_i),(\sigma_i),\gamma,T) \leq  Q^{1/2}R^{1/2} \leq  x^{\varepsilon_1/2}  M^{1/2-\beta}(\#\mathcal{J})^{7/8-\gamma} T_0^{3/8} R^{1/2}.
 \end{align}
 But the RHS of (\ref{equ:qcomp11}) is at most $ (\#\mathcal{J})^{7/8-\gamma}\tau^{1/4} x^{7/8-\sigma+\nu/8+\varepsilon_1}$ when 
 \begin{align*}
   R(F,(S_i),(\sigma_i),T) \leq \tau^{-1/4}M^{2\beta-1}  x^{7/4-2\sigma+\nu/4+\varepsilon_1/4}.
 \end{align*}
 Now assume we have (\ref{equ:qcomp2}), so that 
 \begin{align}\label{equ:qcomp21}
Q(F,G,(S_i),(\sigma_i),\gamma,T) \leq  Q^{1/2}R^{1/2} \leq  x^{\varepsilon_1/2} M^{1-\beta}(\#\mathcal{J})^{1-\gamma} R^{1/2}.
 \end{align}
   But the RHS of (\ref{equ:qcomp21}) is at most  $(\#\mathcal{J})^{1-\gamma} x^{1-\sigma-\varepsilon_1}$  when
   \begin{align*}
  &R(F,(S_i),(\sigma_i),T) \leq  M^{2\beta-2} x^{2-2\sigma-4\varepsilon_1}. \qedhere
  \end{align*}
\end{proof}

\subsection{Computations} \label{ssec:partition} 

Let $F(s) = \prod_{i=1}^j S_i(s)$, $G(s)$,  $(\sigma_i)$ and $\gamma$ be as described  at the start of Section~\ref{sec:values} with $F(s)$  of Type $A$, $B$ or $C$ at $T$.   Recall that $S_i(s) = \sum_{n \sim N_i} a_n^{(i)} n^{-s}$ with $N= \prod_{i=1}^j N_i$ and $N_i = N^{\ell_i}$. We  now consider $F(s)$ which satisfy one of the following three criteria: 
 \begin{enumerate} [{\rm (1)}]
\item There exists $k \in \{1, \dots, j\}$  with $\ell_k \geq \chi_0(a)$. 
\item There exists $I_1 \subseteq \{1,\dots,j\}$ with  $\sum_{i \in I_1} \ell_i  \in \chi_1(a) $.
\item  There exist $I_2 \subseteq \{1,\dots,j\}$ with  $\sum_{i \in I_2} \ell_i  \in\chi_2(a) $ and 
$I_3 \subseteq \{1,\dots,j\}$ with  $\sum_{i \in I_3} \ell_i  \in  \chi_3(a)   $. 
\end{enumerate}
Here  $\chi_0(a)$, $\chi_1(a)$, $\chi_2(a)$ and $\chi_3(a)$  are as given in  $(\ref{chi0})$, $(\ref{chi1})$, $(\ref{chi2})$ and $(\ref{chi3})$.

\vspace{3mm}
To prove Proposition~\ref{proposition2}, we need to show that one of the inequalities (\ref{cc1}), (\ref{cc2}), (\ref{cc3}), (\ref{cc4.A}) or (\ref{cc4.B}) of Proposition~\ref{proposition1} holds. If $F(s)$ is of Type~$C$ at $T$, then Lemma~\ref{lem:typec} immediately tells us that $R(F,(S_i), (\sigma_i), T)=0$ and these bounds are trivial. Hence we focus on $F(s)$ of Type $A$ or $B$. We use  Corollary~\ref{cor:amendc3} and Corollary~\ref{cor:amendc4} to replace (\ref{cc3}), (\ref{cc4.A}) and (\ref{cc4.B}) by conditions (\ref{c3'}) and (\ref{c4'}). Write  \begin{align}
\mathcal{Z} = &\big\{(1,4,4,3,4,1), (1/2,2,3/2,3/2,2,1/2), (2/5,16/5,12/5,12/5,16/5,4/5), \label{def:curlyz} \\
&(2/5,4/5,3/5,6/5,8/5,2/5), (1/3,2/3,1/3,1,4/3,1/3),  (2/7,16/21,8/21,8/7,32/21,8/21), \nonumber\\
 &(3/8,2,3/2,15/8,5/2,5/8), (1/4,4/3,2/3,5/4,5/3,5/12), (1/9,16/9,8/9,5/3,20/9,5/9)\big\}.  \nonumber
  \end{align}
The proof of Proposition~\ref{proposition2} is complete if we can show for large $x$ and given $F(s) = \prod_{i=1}^j S_i(s)$ and $(\sigma_i)$  that there exist $I \subseteq \{1, \dots, j\}$ and $k_i \in \mathbb{N} \cap [1,J]$ such that  
 $M= \prod_{i \in I} N_i^{k_i}$ and $M^\beta = \prod_{i \in I} N_i^{k_i\sigma_i}$
 satisfy one of the following four conditions: 
\begin{align} 
& R(F,(S_i),(\sigma_i),T) \leq x^{1-\sigma} \log(x)^{-2J},  \tag{C1} \\
&
R(F,(S_i),(\sigma_i),T) \leq \tau x^{1-2\sigma+\nu+2\varepsilon_1},
 \tag{C2} 
 \\ 
&  R(F,(S_i),(\sigma_i),T)\leq \min_{(U,V,W,X,Y,Z) \in \mathcal{Z}} \tau^{U}M^{V\beta-W}  x^{X-Y\sigma+Z\nu
+\varepsilon_1}, \tag{C3$^*$} \\
 &R(F,(S_i),(\sigma_i),T) \leq  \min\{\tau^{-1/4}M^{2\beta-1}  x^{7/4-2\sigma+\nu/4-4\varepsilon_1},M^{2\beta-2} x^{2-2\sigma-4\varepsilon_1}\}. \tag{C4$^*$} \end{align} 
To accomplish this goal we use the many bounds on $R(F,(S_i),(\sigma_i),T)$ introduced in Section~\ref{ssec:standnew}.

\subsubsection{Long factors}

When $F(s)=\prod_{i=1}^jS_i(s)$ is of  Type $B$ at $T$, all its factors $S_i(s)$ are very short. However, if it is of Type $A$, there are some longer factors and these have  property~(\ref{propertyz1})  or (\ref{propertyz2}). We now show that one of (\ref{cc1}), (\ref{cc2}) or (\ref{c4'}) holds if one of these factors is very long.

\begin{lemma} \label{lem:smoothfull}
Let $F(s)=\prod_{i=1}^jS_i(s)$ be as described at the start of Section~{\rm\ref{sec:values}}. 
 Recall $S_i(s) = \sum_{n \sim N_i} a_n^{(i)} n^{-s}$ and $N_i=N^{\ell_i}$ with $x \ll N = \prod_{i=1}^j N_i \ll x$.  
Let $T\in [T_1,T_0]$ and suppose that  $F(s)$ is  of Type~$A$ at $T$.

Write $\tau = x^a$. Suppose  there exists $k \in \{1, \dots, j\}$  such that one of the following three options holds:
 \begin{enumerate}[{\rm (1) $\quad$ }]
 \item $\ell_k \geq 0.335-\varepsilon_1$ and $a \leq 0.57$,
 \item $\ell_k \geq 0.330-\varepsilon_1$ and  $a \in [ 0.57,0.61]$,
  \item $\ell_k \geq 0.320-\varepsilon_1$ and $a  \geq 0.61$.
 \end{enumerate}
Then if $x \geq C$, one of the following three inequalities holds: 
\begin{align} 
& R(F,(S_i),(\sigma_i),T) \leq x^{1-\sigma} \log(x)^{-2J},  \label{equ:smoothopt1}\\
&
R(F,(S_i),(\sigma_i),T) \leq \tau x^{1-2\sigma+\nu+2\varepsilon_1},
\label{equ:smoothopt2} \\ 
&
R(F,(S_i),(\sigma_i),T) \leq  \min\{\tau^{-1/4}N_k^{2\sigma_k-1}  x^{7/4-2\sigma+\nu/4-4\varepsilon_1},N_k^{2\sigma_k-2} x^{2-2\sigma-4\varepsilon_1}\}. \label{equ:smoothopt3}
\end{align} 
Here $C$ is a large constant dependent only on $J$ and $\varepsilon_1 = \varepsilon/10^5 $.
\end{lemma}
This completes the proof of Proposition~\ref{proposition2} when $a> 0.545$ and  $F(s) = \prod_{i=1}^j S_i(s)$ has  $\ell_k \geq \chi_0(a)$ for some $k$. For smaller values of $a$ we still need to cover the case $ \ell_k \in [\chi_0(a) , 0.335]$ in a later lemma.

\begin{proof} If $\sigma \geq 1-10^{-500}$, we are done by Lemma~\ref{lem:verylarge}. Thus we consider $\sigma \in [0.6, 1-10^{-500}]$. 
Since $F(s)$ is of Type $A$ at $T$ and  $S_k(s)$ has $\ell_k \geq 0.32-\varepsilon_1$, $S_k(s)$ has property  {\rm (\ref{propertyz1})} or {\rm(\ref{propertyz2})} with respect to $T$. Hence  Lemma~\ref{lem:newsmooth} applies.  We  assume $R(F,(S_i),(\sigma_i),T) > x^{1-\sigma} \log(x)^{-2J}$, as otherwise (\ref{equ:smoothopt1}) holds. Taking $B=2J$ in  Lemma~\ref{lem:newsmooth}, we then get 
\begin{align} \label{equ:intersmoothbound1}
R(F,(S_i),(\sigma_i),T)  \leq x^{4\varepsilon_1}\min\left\{\tau N_k^{(2-4\sigma_k)}, \tau^2 N_k^{(6-12\sigma_k)} \right\}.
\end{align}
Write $B(a,\sigma) = \max\{ 1-\sigma-\varepsilon_1, a+1-2\sigma+\nu+2\varepsilon_1\}$.  If one of the inequalities
\begin{align}\label{equ:intersmoothbound2}
N_k^{\sigma_k}\geq \left(\dfrac{x^{4\varepsilon_1}\tau N_k^2 }{x^{B(a,\sigma)}}  \right)^{1/4} \quad \mbox{ or } \quad N_k^{\sigma_k}\geq \left(\dfrac{x^{4\varepsilon_1}\tau^2 N_k^6 }{x^{B(a,\sigma)}}  \right)^{1/12}
\end{align}
is satisfied, then (\ref{equ:intersmoothbound1}) tells us that $R(F,(S_i),(\sigma_i),T) \leq x^{B(a,\sigma)} = \max\{ x^{1-\sigma-\varepsilon_1}, x^{a+1-2\sigma+\nu+2\varepsilon_1}\}$. Suppose first that $\sigma_k \geq \sigma$. Then (\ref{equ:intersmoothbound2}) holds if $\ell_k \geq (a/4 -B(a,\sigma)/4 +2\varepsilon_1)/(\sigma -1/2)$. The RHS of this inequality is at most $0.5$ whenever $a \in [0.47,0.77]$ and $\sigma \in [0.6,1-10^{-500}]$. Hence we may assume $1-\ell_k \geq 0.2$ or $\sigma_k < \sigma$ and $\sigma_I \geq \sigma$ for $I = \{1, \dots, j \} \setminus \{k\}$. Either way, 
Corollary~\ref{cor:large2} is applicable with $I = \{1, \dots, j \} \setminus \{k\}$. We thus have (\ref{equ:smoothopt1}) or (\ref{equ:smoothopt2}) if 
\begin{align*}
\sigma_I \geq \min\left\{\dfrac{3a -2B(a,\sigma) + 4\varepsilon_1}{3a-B(a,\sigma)+ 2\varepsilon_1},  \dfrac{3a+B(a,\sigma)-2\varepsilon_1}{3a+3B(a,\sigma) -6\varepsilon_1 }
\right\}.
\end{align*}
So
 one of (\ref{equ:smoothopt1}) or (\ref{equ:smoothopt2}) holds whenever one of the following four inequalities is satisfied:
 \begin{align} 
 &\ell_k \sigma_k \geq \dfrac{a}{4}+\dfrac{\ell_k}{2}-\dfrac{B(a,\sigma)}{4}+2\varepsilon_1, \label{equ:smoothinequ1} \\
  &\ell_k \sigma_k \geq \dfrac{a}{6}+\dfrac{\ell_k}{2}-\dfrac{B(a,\sigma)}{12}+2\varepsilon_1, \label{equ:smoothinequ2} \\
  &(1-\ell_k) \sigma_I \geq (1-\ell_k) \left(\dfrac{3a+B(a,\sigma)-2\varepsilon_1}{3a+3B(a,\sigma)-6\varepsilon_1}\right), \label{equ:smoothinequ3} \\
    &(1-\ell_k) \sigma_I \geq (1-\ell_k) \left(\dfrac{3a-2B(a,\sigma)+4\varepsilon_1}{3a-B(a,\sigma)+2\varepsilon_1}\right). \label{equ:smoothinequ4} 
 \end{align}
 Note $\ell_k \sigma_k + (1-\ell_k) \sigma_I =\sigma$. Adding ((\ref{equ:smoothinequ1}) and (\ref{equ:smoothinequ3})) or ((\ref{equ:smoothinequ1}) and (\ref{equ:smoothinequ4})) or ((\ref{equ:smoothinequ2}) and (\ref{equ:smoothinequ3})), we find that one of (\ref{equ:smoothinequ1}),  (\ref{equ:smoothinequ2}),  (\ref{equ:smoothinequ3}) or  (\ref{equ:smoothinequ4}) certainly holds whenever one of 
 \begin{align}
 &\sigma \geq \left(\dfrac{a}{4}-\dfrac{B(a,\sigma)}{4}+\left(\dfrac{3a+B(a,\sigma)-2\varepsilon_1}{3a+3B(a,\sigma)-6\varepsilon_1}\right)+2\varepsilon_1\right) + \left(\dfrac{1}{2}-\left(\dfrac{3a+B(a,\sigma)-2\varepsilon_1}{3a+3B(a,\sigma)-6\varepsilon_1}\right) \right)\ell_k, \label{equ:smoothinequ11} \\
  &\sigma \geq \left(\dfrac{a}{4}-\dfrac{B(a,\sigma)}{4}+\left(\dfrac{3a-2B(a,\sigma)+4\varepsilon_1}{3a-B(a,\sigma)+2\varepsilon_1}\right)+2\varepsilon_1\right) + \left(\dfrac{1}{2}-\left(\dfrac{3a-2B(a,\sigma)+4\varepsilon_1}{3a-B(a,\sigma)+2\varepsilon_1}\right) \right)\ell_k,  \label{equ:smoothinequ12} \\
   &\sigma \geq \left(\dfrac{a}{6}-\dfrac{B(a,\sigma)}{12}+\left(\dfrac{3a+B(a,\sigma)-2\varepsilon_1}{3a+3B(a,\sigma)-6\varepsilon_1}\right)+2\varepsilon_1\right) + \left(\dfrac{1}{2}-\left(\dfrac{3a+B(a,\sigma)-2\varepsilon_1}{3a+3B(a,\sigma)-6\varepsilon_1}\right) \right)\ell_k,  \label{equ:smoothinequ13} 
 \end{align}
 is satisfied. Note that the coefficient of $\ell_k$ in  (\ref{equ:smoothinequ11}) and (\ref{equ:smoothinequ13}) is negative if $-3a/2+B/2-\varepsilon_1<0$. This is certainly true for $\sigma \in [0.6,1]$. The coefficient of $\ell_k$ in  (\ref{equ:smoothinequ12}) is negative if $-3a/2+3B/2-3\varepsilon_1<0$ and hence if  $B(a,\sigma)<a+2\varepsilon_1$. But if $B(a,\sigma) \geq a+2\varepsilon_1$, then certainly $R(F,(S_i),(\sigma_i),T) \leq T_0 = x^{a+\varepsilon_1} \leq x^{B(a,\sigma)}$. Thus we may assume that all coefficients of $\ell_k$ are negative. Rearranging, we get that one of (\ref{equ:smoothinequ11}), (\ref{equ:smoothinequ12}) or (\ref{equ:smoothinequ13}) holds whenever $\ell_k \geq \min\{X_1(a,\sigma), X_2(a,\sigma), X_3(a,\sigma)\}$,  where
 \begin{align*}
&X_1(a,\sigma)= \left(\sigma - \dfrac{a}{4}+\dfrac{B(a,\sigma)}{4}-\left(\dfrac{3a+B(a,\sigma)-2\varepsilon_1}{3a+3B(a,\sigma)-6\varepsilon_1}\right)-2\varepsilon_1\right)  \left(\dfrac{1}{2}-\dfrac{3a+B(a,\sigma)-2\varepsilon_1}{3a+3B(a,\sigma)-6\varepsilon_1} \right)^{-1}, \\
&X_2(a,\sigma)= \left(\sigma - \dfrac{a}{4}+\dfrac{B(a,\sigma)}{4}-\left(\dfrac{3a-2B(a,\sigma)+4\varepsilon_1}{3a-B(a,\sigma)+2\varepsilon_1}\right)-2\varepsilon_1\right)  \left(\dfrac{1}{2}-\dfrac{3a-2B(a,\sigma)+4\varepsilon_1}{3a-B(a,\sigma)+2\varepsilon_1}\right)^{-1}, \\
&X_3(a,\sigma)= \left(\sigma - \dfrac{a}{6}+\dfrac{B(a,\sigma)}{12}-\left(\dfrac{3a+B(a,\sigma)-2\varepsilon_1}{3a+3B(a,\sigma)-6\varepsilon_1}\right)-2\varepsilon_1\right)  \left(\dfrac{1}{2}-\dfrac{3a+B(a,\sigma)-2\varepsilon_1}{3a+3B(a,\sigma)-6\varepsilon_1} \right)^{-1}.
 \end{align*}
 Since condition $\ell_k \geq \min\{X_1(a,\sigma), X_2(a,\sigma), X_3(a,\sigma)\}$ only involves the two variables $a$ and $\sigma$, it is easy to check explicitly that $\min\{X_1(a,\sigma), X_2(a,\sigma), X_3(a,\sigma)\} \leq 0.3-\varepsilon_1$ if $a \geq 0.65$. For smaller $a$  we still need to do a bit more work. Using~(\ref{equ:intersmoothbound1}), we note that (\ref{equ:smoothopt3}) holds whenever
\begin{align}  \label{equ:smoothQ1}
 x^{4\varepsilon_1}\tau N_k^{(2-4\sigma_k)} \leq  \min\{\tau^{-1/4}N_k^{2\sigma_k-1}  x^{7/4-2\sigma+\nu/4-4\varepsilon_1},N_k^{2\sigma_k-2} x^{2-2\sigma-4\varepsilon_1}\}.
\end{align}
Rearranging a bit, we find that (\ref{equ:smoothQ1})  is satisfied if both inequalities
\begin{align}
&\ell_k \sigma_k \geq \dfrac{5a}{24} - \dfrac{7}{24}+\dfrac{\sigma}{3}-\dfrac{\nu}{24}+2\varepsilon_1 +\dfrac{\ell_k}{2},  \label{equ:smoothQ2} \\ 
&\ell_k \sigma_k \geq \dfrac{a}{6}-\dfrac{1}{3}+\dfrac{\sigma}{3}+2\varepsilon_1+\dfrac{2\ell_k}{3} \label{equ:smoothQ3}
\end{align}
hold. Adding ((\ref{equ:smoothQ2}) and (\ref{equ:smoothinequ3})) and ((\ref{equ:smoothQ3}) and (\ref{equ:smoothinequ3})), we are done if the following is true:
\begin{align}
&\sigma \geq \left(\dfrac{5a}{24} - \dfrac{7}{24}+\dfrac{\sigma}{3}-\dfrac{\nu}{24}+2\varepsilon_1+\dfrac{3a+B(a,\sigma)-2\varepsilon_1}{3a+3B(a,\sigma)-6\varepsilon_1}\right) + \left(\dfrac{1}{2}-\left(\dfrac{3a+B(a,\sigma)-2\varepsilon_1}{3a+3B(a,\sigma)-6\varepsilon_1}\right) \right)\ell_k, \label{equ:smoothQ4}\\
&\sigma \geq \left(\dfrac{a}{6}-\dfrac{1}{3}+\dfrac{\sigma}{3}+2\varepsilon_1+\dfrac{3a+B(a,\sigma)-2\varepsilon_1}{3a+3B(a,\sigma)-6\varepsilon_1}\right) + \left(\dfrac{2}{3}-\left(\dfrac{3a+B(a,\sigma)-2\varepsilon_1}{3a+3B(a,\sigma)-6\varepsilon_1}\right) \right)\ell_k. \label{equ:smoothQ5}
\end{align}
The coefficient of $\ell_k$ in (\ref{equ:smoothQ5}) is only non-negative if $B(a,\sigma) \geq a+2\varepsilon_1$ and hence we may again assume that all coefficients of $\ell_k$ are negative. Rearranging, (\ref{equ:smoothopt3})  holds if $\ell_k \geq \max\{Y_1(a,\sigma), Y_2(a,\sigma)\}$ where 
\begin{align*}
&Y_1(a,\sigma)= \dfrac{\left(\sigma - \left(\dfrac{5a}{24} - \dfrac{7}{24}+\dfrac{\sigma}{3}-\dfrac{\nu}{24}+2\varepsilon_1+\dfrac{3a+B(a,\sigma)-2\varepsilon_1}{3a+3B(a,\sigma)-6\varepsilon_1}\right)\right) }{ \left(\dfrac{1}{2}-\dfrac{3a+B(a,\sigma)-2\varepsilon_1}{3a+3B(a,\sigma)-6\varepsilon_1} \right)}, \\
&Y_2(a,\sigma)= \dfrac{\left(\sigma - \left(\dfrac{a}{6}-\dfrac{1}{3}+\dfrac{\sigma}{3}+2\varepsilon_1+\dfrac{3a+B(a,\sigma)-2\varepsilon_1}{3a+3B(a,\sigma)-6\varepsilon_1}\right)\right) }{ \left(\dfrac{2}{3}-\dfrac{3a+B(a,\sigma)-2\varepsilon_1}{3a+3B(a,\sigma)-6\varepsilon_1} \right)}.
 \end{align*} 
 In summary, we have one of (\ref{equ:smoothopt1}), (\ref{equ:smoothopt2})  or (\ref{equ:smoothopt3})  whenever 
 \begin{align*}
 \ell_k \geq \min\{X_1(a,\sigma), X_2(a,\sigma), X_3(a,\sigma), \max\{Y_1(a,\sigma), Y_2(a,\sigma)\}\}.
 \end{align*}
 Evaluating $\min\{X_1(a,\sigma), X_2(a,\sigma), X_3(a,\sigma), \max\{Y_1(a,\sigma), Y_2(a,\sigma)\}\}$ explicitly, we observe that the quantity is strictly less than $0.335$ when $a \leq 0.57$, strictly less than $0.33$ when $a \in [0.57,0.61]$ and strictly less than $0.32$  when $a > 0.61$. 
\end{proof}

\subsubsection{Large $\tau$} Next we deal with particularly large values of $\tau$.

\begin{lemma} \label{lem:largetau}
Let $F(s)=\prod_{i=1}^jS_i(s)$ be as described at the start of Section~{\rm\ref{sec:values}}. 
 Recall $S_i(s) = \sum_{n \sim N_i} a_n^{(i)} n^{-s}$ and $N_i=N^{\ell_i}$ with $x \ll N = \prod_{i=1}^j N_i \ll x$.  
Let $T\in [T_1,T_0]$ and suppose that  $F(s)$ is  of Type~$A$ or $B$ at $T$. 
Write $\tau = x^a$. 
Suppose $a \geq 0.685$ and  there exists $I \subseteq \{1, \dots, j\}$ such that 
   \begin{align*}
  \sum_{i \in I} \ell_i  \in [0.35-\varepsilon_1,0.48+\varepsilon_1] \cup [0.52-\varepsilon_1,0.65+\varepsilon_1].
   \end{align*}
Then if $x \geq C$, one of the following two inequalities holds: 
\begin{align} 
& R(F,(S_i),(\sigma_i),T) \leq x^{1-\sigma} \log(x)^{-2J},  \label{equ:largeopt11}\\
&
R(F,(S_i),(\sigma_i),T) \leq  \tau x^{1-2\sigma+\nu+2\varepsilon_1}. \label{equ:largeopt2}
\end{align} 
Here $C$ is a large constant dependent only on  $J$ and $\varepsilon_1$, and $\varepsilon_1$ is assumed to be very small.
\end{lemma}

Lemma~\ref{lem:largetau} completes the proof of Proposition~\ref{proposition2} when $a \geq 0.685$. 

\begin{proof}
If $\sigma \geq 1-10^{-500}$, we are done by Lemma~\ref{lem:verylarge}. Thus we consider $\sigma \in [0.6, 1-10^{-500}]$. Suppose $R(F,(S_i),(\sigma_i),T) > x^{1-\sigma} \log(x)^{-2J}$.  Write $B(a,\sigma) = \max\{1-\sigma-\varepsilon_1, a+1-2\sigma+\nu+2\varepsilon_1\}$. By Corollary~\ref{cor:large2} we have $R(F,(S_i), (\sigma_i), T) \leq x^{B(a,\sigma)}$ if
\begin{align} \label{equ:largetau1}
\sigma \geq \min\left\{\dfrac{3a -2B(a,\sigma) + 4\varepsilon_1}{3a-B(a,\sigma)+ 2\varepsilon_1},  \dfrac{3a+B(a,\sigma)-2\varepsilon_1}{3a+3B(a,\sigma) -6\varepsilon_1 }
\right\}.
\end{align}
This inequality condition can be rewritten in terms of four simple quadratic equations in $\sigma$. Solving these quadratic equations, we find that  (\ref{equ:largetau1}) holds whenever $a \in [0.685,0.77]$ and $\sigma \in [0.6,0.88]$. 

\vspace{3mm}
 For $\sigma > 0.88$, we use bound (\ref{R:bound3}) of Lemma~\ref{lem:montgomery}:  For $\ell_I = \sum_{i \in I} \ell_i$, $\sigma_I = (\sum_{i\in I} \ell_i \sigma_i )/\ell_I$ and large $x$,
\begin{align*}
R(F,(S_i), (\sigma_i), T) &\leq \max\left\{ x^{\ell_I(2-2\sigma_I)+2\varepsilon_1}, x^{a+\ell_I(11-14\sigma_I)+2\varepsilon_1} \right\}.
\end{align*}
We thus have $R(F,(S_i), (\sigma_i), T)  \leq x^{B(a,\sigma)}$ if 
\begin{align} \label{equ:largetau11}
&\sigma_I
\geq  \dfrac{1}{\ell_I}\max\left\{ \ell_I-\dfrac{B(a, \sigma)}{2}, \, \dfrac{11}{14}\ell_I+\dfrac{a}{14}-\dfrac{B(a, \sigma)}{14} \right\}+\varepsilon_1.
\end{align}
However, for $\ell_I \in [0.35-\varepsilon_1,0.48+\varepsilon_1]$, the RHS of (\ref{equ:largetau11}) is at most $\max\{\sigma, 0.92\}$. Hence we may assume $\sigma_I < \max\{\sigma, 0.92\}$, so that  $(\sigma - \ell_I \sigma_I)/(1-\ell_I) \geq \min\{\sigma, 2(\sigma-0.46)\} \geq 0.84$ for $\sigma \geq 0.88$.  But Corollary~\ref{cor:large2}, applied with $K = \{1, \dots, j\} \setminus I$ in the place of $I$, then gives $R(F,(S_i),(\sigma_i),T) \leq x^{B(a,\sigma)}$ if 
\begin{align}\label{equ:largetau22}
\sigma_K \geq \max\left\{\dfrac{3a+7B(a,\sigma)-14\varepsilon_1}{3a+10B(a,\sigma)- 20\varepsilon_1 },
 \dfrac{4a+B(a,\sigma)-2\varepsilon_1}{4a+4B(a,\sigma)- 8\varepsilon_1 }\right\} = F(a,\sigma).
\end{align}
Multiplying (\ref{equ:largetau11}) by $\ell_I$ and (\ref{equ:largetau22}) by $(1-\ell_I)$ and adding up, we are hence done if
\begin{align*}
& \ell_I \leq \dfrac{1}{1-F(a,\sigma)}\left(\sigma+\dfrac{B(a, \sigma)}{2}-F(a,\sigma)-\varepsilon_1 \right)  \quad \mbox{ and } \\
&\ell_I\geq \dfrac{1}{11/14-F(a,\sigma)}\left(\sigma-\dfrac{a}{14}+\dfrac{B(a, \sigma)}{14}- F(a,\sigma)-\varepsilon_1 \right). 
\end{align*}
Both of these inequalities are satisfied if $a \geq 0.685$, $\sigma \geq 0.88$ and $\ell_I \in [0.35-\varepsilon_1,0.48+\varepsilon_1]$.
\end{proof}
 
\subsubsection{Small $\tau$} Next we work with particularly small $\tau$, say $\tau \leq x^{0.53}$, or large $\sigma$, say $\sigma \geq a+\nu-\varepsilon_1$. Here condition (\ref{c4'}), which was derived using $Q(F,G,(S_i),(\sigma_i), \gamma, T)$, is very useful.

\begin{lemma} \label{lem:smalltau}
Let $F(s)=\prod_{i=1}^jS_i(s)$ be as described at the start of Section~{\rm\ref{sec:values}}. 
 Recall $S_i(s) = \sum_{n \sim N_i} a_n^{(i)} n^{-s}$ and $N_i=N^{\ell_i}$.  
Let $T\in [T_1,T_0]$ and suppose that  $F(s)$ is  of Type~$A$ or $B$ at $T$. 
Write $\tau = x^a$. 

\vspace{3mm}
Suppose  there exists $I \subseteq \{1, \dots, j\}$ such that $I$, $a$ and $\sigma$ satisfy one of the following five options:
  \begin{enumerate}[{\rm (1) $\quad$ }]
 \item $  a \in [0.47, 0.53]$ and $\sum_{i \in I} \ell_i \in [0.29-\varepsilon_1,0.36+\varepsilon_1] \cup [0.64-\varepsilon_1,0.71+\varepsilon_1],$
 \item $  a \in [ 0.53, 0.545]$ and $\sum_{i \in I} \ell_i \in [0.315-\varepsilon_1,0.345+\varepsilon_1] \cup [0.655-\varepsilon_1,0.685+\varepsilon_1] $,
  \item $  a \in [ 0.53, 0.545]$, $\sum_{i \in I} \ell_i \in [0.285-\varepsilon_1 ,0.375+\varepsilon_1] \cup  [0.625-\varepsilon_1 ,0.715+\varepsilon_1]$   and $\sigma \geq a +\nu -\varepsilon_1$,
  \item $a \in [0.53,0.685]$ and $\sigma \geq \min\{a+\nu-\varepsilon_1, 0.85\}$ and
    \begin{align} \label{equ:applyingqtolarge}
&\sum_{i \in I} \ell_i \in \left[1-m(a,\sigma)+10^{-100},1-M(a,\sigma)-10^{-100}\right],
  \end{align}
where $M(a,\sigma)$ and $m(a,\sigma)$ are given by: 
\begin{align}
&M(a,\sigma) =\min \left\{  \dfrac{3a}{6\sigma-2},\max\left\{ \dfrac{3a}{20\sigma-14},  \dfrac{2a}{4\sigma-1} \right\}  \right\},    \label{equ:giveM}           \\
&m(a,\sigma) = \max\{m_1(a,\sigma), \min\{ m_2(a,\sigma), m_3(a,\sigma) \}\}, \\
&m_1(a,\sigma) = \min\left\{ \dfrac{a-3-\nu}{4(1-2\sigma)} +
\dfrac{a(3-3\sigma)}{(3\sigma-1)(1-2\sigma)},  \dfrac{a-3-\nu}{4(1-2\sigma^{\circ}_1)} +
\dfrac{a(3-3\sigma^{\circ}_1)}{(3\sigma^{\circ}_1-1)(1-2\sigma^{\circ}_1)}\right\}, \label{equ:givem1}   \\
&m_2(a,\sigma) = \min\left\{ \dfrac{a-3-\nu}{4(1-2\sigma)} +
\dfrac{a(3-3\sigma)}{(10\sigma-7)(1-2\sigma)},  \dfrac{a-3-\nu}{4(1-2\sigma^\circ_2)} +
\dfrac{a(3-3\sigma^\circ_2)}{(10\sigma^\circ_2-7)(1-2\sigma^\circ_2)}\right\}, \label{equ:givem2}  \\
&m_3(a,\sigma) = \min\left\{ \dfrac{a-3-\nu}{4(1-2\sigma)} +
\dfrac{a(4-4\sigma)}{(4\sigma-1)(1-2\sigma)},  \dfrac{a-3-\nu}{4(1-2\sigma^\circ_2)} +
\dfrac{a(4-4\sigma^\circ_2)}{(4\sigma^\circ_2-1)(1-2\sigma^\circ_2)}\right\} \label{equ:givem3}   \\ \nonumber
&\mbox{and } \quad \sigma^\circ_1 = \min\left\{ 1,\, a+\dfrac{1}{3} \right\} \quad \mbox{ and } \quad \sigma^\circ_2 = \max\left\{ \dfrac{3a}{10} + \dfrac{7}{10},\, a+\dfrac{1}{4} \right\},
\end{align} 
\item $a \in [0.53,0.685]$ and  $\sigma$ satisfies $\sigma \geq \min\{a+1/3, \max\{ 3a/10+7/10, a+1/4 \} \} +10^{-100}$.
 \end{enumerate}

\vspace{3mm}
Write $M =\prod_{i \in I} N_i$ and $M^{\beta} = \prod_{i \in I} N_i^{\sigma_i}$. Then if $x \geq C$, one of the following two inequalities holds: 
\begin{align} 
& R(F,(S_i),(\sigma_i),T) \leq x^{1-\sigma} \log(x)^{-2J},  \label{equ:smoothopt11}\\
&
R(F,(S_i),(\sigma_i),T) \leq  \min\{\tau^{-1/4}M^{2\beta-1}  x^{7/4-2\sigma+\nu/4-4\varepsilon_1},M^{2\beta-2} x^{2-2\sigma-4\varepsilon_1}\}. \label{equ:smoothopt33}
\end{align}  
Here $C$ is a large constant dependent only on $J$ and $\varepsilon_1$, and $\varepsilon_1$ is assumed to be very small.
\end{lemma} 

In particular, Lemma~\ref{lem:smalltau} completes the proof of Proposition~\ref{proposition2} for $a \leq 0.53$ and also  combines with Lemma~\ref{lem:smoothfull} to fully cover the case  $\ell_k \geq \chi_0(a)$ when $a \leq 0.545$.

\begin{proof} We first consider an arbitrary $a \in [0.47,0.685]$.  
If $\sigma \geq 1-10^{-500}$, we have $  R(F,(S_i),(\sigma_i),T) \leq x^{1-\sigma} \log(x)^{-2J}$ by  Lemma~\ref{lem:verylarge}. Thus we consider $\sigma \in [0.6, 1-10^{-500}]$ and also assume $R(F,(S_i),(\sigma_i),T) > x^{1-\sigma} \log(x)^{-2J}$. (Otherwise (\ref{equ:smoothopt11}) holds.)  We ask when (\ref{equ:smoothopt33}) is satisfied.

\vspace{3mm} \textbf{Case 1:}  $
\tau^{-1/4}M^{2\beta-1}  x^{7/4-2\sigma+\nu/4-4\varepsilon_1} \geq M^{2\beta-2} x^{2-2\sigma-4\varepsilon_1}.$

\vspace{3mm}
Write $\ell_I = \sum_{i \in I} \ell_i$ and $(\sum_{i \not \in I} \ell_i \sigma_i)/(1-\ell_I) =  \alpha $, so that $N^\sigma =(N/M)^\alpha M^\beta  = N^{(1-\ell_I)\alpha}N^{\ell_I \beta}$. Since $x \ll N \ll x$ and $M = N^{\ell_I}$,   further
$
M^{2\beta-2} x^{2-2\sigma-4\varepsilon_1} \geq  x^{(2-2\alpha)(1-\ell_I)-5\varepsilon_1}$,
provided $x$ is sufficiently large.

\vspace{3mm}
Suppose $\alpha \geq \sigma$. By Lemma~\ref{lem:large2}, applied with $\{1,\dots, j\}\setminus I$ in the place of $I$, we have
\begin{align} \label{equ:alpha1}
R(F,(S_i), (\sigma_i),T) \leq x^{2\varepsilon_1}\min\left\{ \tau^{(3-3\alpha)/(2-\alpha)},  \tau^{(3-3\alpha)/(3\alpha-1)} \right\}.
\end{align}
Hence $R(F,(S_i), (\sigma_i),T) \leq x^{(2-2\alpha)(1-\ell_I)-5\varepsilon_1}$ if $(2/3)(1-\ell_I)\max\{(2-\alpha), (3\alpha-1)\}-a-C^*\varepsilon_1 \geq 0$, where $C^*$ is a very large absolute constant. (Here we assumed that $\alpha<1-10^{-1000}$. This assumption is  valid as  (\ref{equ:alpha1}) immediately gives (\ref{equ:smoothopt11}) if $\alpha>1-10^{-1000}$ and $\sigma \leq 1-10^{-500}$.) Rearranging, we are done if  
\begin{align*}
&\ell_I \leq 1-\min\left\{\dfrac{3a}{4-2\alpha}, \dfrac{3a}{6\alpha-2}\right\}-C^{**}\varepsilon_1. 
\end{align*}
$C^{**}$ is another large constant. On the other hand, if $\alpha < \sigma$, then  $x^{(2-2\alpha)(1-\ell_I)-5\varepsilon_1} \geq x^{(2-2\sigma)(1-\ell_I)-5\varepsilon_1}$. 
Lemma~\ref{lem:large2}, applied with $\{1,\dots, j\}$ in the place of $I$,  tells us that
\begin{align}\label{equ:alpha01}
R(F,(S_i), (\sigma_i),T) \leq x^{2\varepsilon_1}\min\left\{ \tau^{(3-3\sigma)/(2-\sigma)},  \tau^{(3-3\sigma)/(3\sigma-1)} \right\}.
\end{align} 
By the same computation as above, with $\sigma$ in the place of $\alpha$,  $R(F,(S_i), (\sigma_i),T) \leq x^{(2-2\sigma)(1-\ell_I)-5\varepsilon_1}$ if
\begin{align*}
&\ell_I \leq 1-\min\left\{\dfrac{3a}{4-2\sigma}, \dfrac{3a}{6\sigma-2}\right\}-C^{**}\varepsilon_1. 
\end{align*}
Hence we have (\ref{equ:smoothopt33}) if 
$\ell_I \leq 1-\min\{\frac{3a}{4-2\gamma}, \frac{3a}{6\gamma-2}\}-C^{**}\varepsilon_1$
 holds for every $\gamma \in [\sigma, 1-10^{-1000}]$.
 
\vspace{3mm} \textbf{Case 1A:} Now suppose that $a \in [0.47,0.545]$. 

\vspace{3mm}
   Note that $ 1-\min\{\frac{3a}{4-2\gamma}, \frac{3a}{6\gamma-2}\}-C^{**}\varepsilon_1$  is strictly greater than $0.36$ if $a \leq 0.53$. It is also strictly greater than $0.345$ if $a \in [0.53,0.545]$. Further, for $\varepsilon$ sufficiently small and $\sigma \geq \sigma^*$, quantity $ 1-\min\{\frac{3a}{4-2\gamma}, \frac{3a}{6\gamma-2}\}-C^{**}\varepsilon_1$ is greater than $1-3a/(6\sigma^*-2) -10^{-100}$. If $a \in [0.53,0.545]$ and $\sigma^* = a+\nu-\varepsilon_1$,  then  $1-3a/(6\sigma^*-2) -10^{-100}$ is greater than $0.375$. This covers  options (1), (2) and (3) for Case 1.
 
\vspace{3mm} \textbf{Case 1B:} Now consider $a \in [0.53,0.685]$ and $\sigma \geq \min\{a+\nu-\varepsilon_1,0.85\}$.

\vspace{3mm}
 Inequality $\ell_I \leq 1-\min\{\frac{3a}{4-2\gamma}, \frac{3a}{6\gamma-2}\}-C^{**}\varepsilon_1$ holds for all $\gamma \in [\sigma,1-10^{-1000}]$ if $\ell_I \leq 1-3a/(6\sigma-2) -10^{-100}$ and in that case  (\ref{equ:smoothopt33}) is satisfied. 
  On the other hand,  Lemma~\ref{lem:large2} also gives 
\begin{align}\label{equ:alpha345}
R(F,(S_i), (\sigma_i),T) \leq x^{2\varepsilon_1}\max\left\{ \tau^{(3-3\alpha)/(10\alpha-7)},  \tau^{(4-4\alpha)/(4\alpha-1)} \right\}.
\end{align}
We still desire $R(F,(S_i), (\sigma_i),T)  \leq x^{(2-2\alpha)(1-\ell_I)-5\varepsilon_1}$ and (\ref{equ:alpha345}) provides this bound whenever 
\begin{align}\label{equ:alpha333}
&\ell_I \leq 1-\max\left\{\dfrac{3a}{20\alpha-14}, \dfrac{2a}{4\alpha-1}\right\}-C^{**}\varepsilon_1. 
\end{align}
The RHS of (\ref{equ:alpha333}) is increasing in $\alpha$, so if $\alpha \geq \sigma $, we only need to assume that this inequality holds with $\alpha = \sigma$ to obtain (\ref{equ:smoothopt33}). For $\alpha < \sigma$ we replace $\alpha$ by $\sigma$ in the application of Lemma~\ref{lem:large2} and arrive at the same conclusion. In summary, upper bound $\ell_I \leq 1-M(a, \sigma) -10^{-100}$ suffices to give (\ref{equ:smoothopt11}) or (\ref{equ:smoothopt33}) in the case $\sigma  \geq \min\{a+\nu-\varepsilon_1,0.85\}$ and $\tau^{-1/4}M^{2\beta-1}  x^{7/4-2\sigma+\nu/4-4\varepsilon_1} \geq  M^{2\beta-2} x^{2-2\sigma-4\varepsilon_1}$. This covers option (4) for Case 1.

\vspace{3mm} \textbf{Case 2:} $
\tau^{-1/4}M^{2\beta-1}  x^{7/4-2\sigma+\nu/4-4\varepsilon_1} < M^{2\beta-2} x^{2-2\sigma-4\varepsilon_1}$.

\vspace{3mm}
Note $\tau^{-1/4}M^{2\beta-1}  x^{7/4-2\sigma+\nu/4-4\varepsilon_1} \geq x^{(1-2\alpha)(1-\ell_I)-a/4+3/4+\nu/4-5\varepsilon_1}$. If $\alpha \geq \sigma \geq 0.6$, we apply  Lemma~\ref{lem:large2} with $\{1,\dots, j\}\setminus I$ in the place of $I$ to obtain (\ref{equ:alpha1}) and additionally note $x^{1-\sigma}\log(x)^{-2J} \geq x^{1-\alpha}\log(x)^{-2J}$. If $\alpha < \sigma$, we note $x^{(1-2\alpha)(1-\ell_I)-a/4+3/4+\nu/4-5\varepsilon_1} \geq x^{(1-2\sigma)(1-\ell_I)-a/4+3/4+\nu/4-5\varepsilon_1}$  and apply Lemma~\ref{lem:large2} with $\{1,\dots, j\}$ in the place of $I$ to obtain (\ref{equ:alpha01}). 

\vspace{3mm}
Either way,  we have one of (\ref{equ:smoothopt11}) or  (\ref{equ:smoothopt33}) if
\begin{align*}
x^{2\varepsilon_1}\min\left\{ \tau^{(3-3\gamma)/(2-\gamma)},  \tau^{(3-3\gamma)/(3\gamma-1)} \right\} \leq \max\left\{x^{(1-2\gamma)(1-\ell_I)-a/4+3/4+\nu/4-5\varepsilon_1},  x^{1-\gamma-\varepsilon_1} \right\}
\end{align*}
for every $\gamma \in [\sigma, 1-10^{-1000}]$. Rearranging, we require one of the following   two inequalities to hold: 
\begin{align} \label{equ:smallafinished}
& \ell_I  \geq 1-\dfrac{a-3-\nu}{4(1-2\gamma)}-\max\left\{ \dfrac{a(3-3\gamma)}{(2-\gamma)(1-2\gamma)}, 
\dfrac{a(3-3\gamma)}{(3\gamma-1)(1-2\gamma)}\right\}+35\varepsilon_1 , \\
&\gamma \geq a+\dfrac{1}{3}+C^{**}\varepsilon_1. \label{equ:smallafinished2}
\end{align}
\textbf{Case 2A:} Again we first consider $a \in [0.47,0.545]$. The RHS of (\ref{equ:smallafinished}) is strictly less than $0.29$ if $a \leq 0.53$ and $\gamma \leq a+0.34$. It is also strictly less than $0.315$ if $a \in [0.53,0.545]$ and $\gamma \leq a+0.34$. For $a \in [0.53,0.545]$ and $\gamma \in [a+\nu-\varepsilon_1,a+0.34]$, it is less than $0.285$.  This covers options (1), (2) and (3) for Case 2.

\vspace{3mm}
\textbf{Case 2B:}  
For $a \in [0.53,0.685]$ and $ \sigma \geq   \min\{a+\nu-\varepsilon_1,0.85\}$, the maximum of the RHS of (\ref{equ:smallafinished}) over $\gamma \in [\sigma, \min\{1-10^{-1000}, a+1/3+C^{**}\varepsilon_1\}]$ is attained at one of the endpoints of the interval. Hence, for $\varepsilon_1$ sufficiently small, (\ref{equ:smallafinished}) holds if
\begin{align} 
& \ell_I  \geq 1-\min\left\{ \dfrac{a-3-\nu}{4(1-2\sigma)} +
\dfrac{a(3-3\sigma)}{(3\sigma-1)(1-2\sigma)},  \dfrac{a-3-\nu}{4(1-2\sigma^\circ_1)} +
\dfrac{a(3-3\sigma^\circ_1)}{(3\sigma^\circ_1-1)(1-2\sigma^\circ_1)}\right\}+10^{-100}.
\end{align}
Alternatively,  using Lemma~\ref{lem:large2} with  bounds (\ref{equ:large2.3}) and (\ref{equ:large2.4}) rather than (\ref{equ:large2.1}) and (\ref{equ:large2.2}), we also have one of (\ref{equ:smoothopt11}) or  (\ref{equ:smoothopt33}) for $\sigma \geq  \min\{a+\nu-\varepsilon_1, 0.85\}$ if
\begin{align*}
x^{2\varepsilon_1}\max\left\{ \tau^{(3-3\gamma)/(10\gamma-7)},  \tau^{(4-4\gamma)/(4\gamma-1)} \right\} \leq \max\left\{x^{(1-2\gamma)(1-\ell_I)-a/4+3/4+\nu/4-5\varepsilon_1},  x^{1-\gamma-\varepsilon_1} \right\}
\end{align*}
for every $\gamma \in [\sigma, 1-10^{-1000}]$. Rearranging, we require one of the following   two inequalities to hold: 
\begin{align} \label{equ:smallafinishednew}
& \ell_I  \geq 1-\dfrac{a-3-\nu}{4(1-2\gamma)}-\min\left\{ \dfrac{a(3-3\gamma)}{(10\gamma-7)(1-2\gamma)}, 
\dfrac{a(4-4\gamma)}{(4\gamma-1)(1-2\gamma)}\right\}+35\varepsilon_1 , \\
&\gamma \geq \max\left\{\dfrac{3a}{10} +\dfrac{7}{10}, a+\dfrac{1}{4} \right\}+C^{**}\varepsilon_1. \label{equ:smallafinishednew2}
\end{align}
The RHS of (\ref{equ:smallafinishednew}) also attains its maximum over $\gamma \in [\sigma, \sigma^\circ_2 +C^{**}\varepsilon_1]$ at one of the endpoints  of the interval and this gives the desired lower bound $1- \min\{m_2(a,\sigma), m_3(a,\sigma)\}+10^{-100}$ for the size of $\ell_I$, covering  option~(4) for Case 2. 

\vspace{3mm}
\textbf{Case C:} Finally, we consider  $a \in [0.53,0.685]$ and  $\sigma \geq \min\{a+\frac{1}{3}, \max\{ \frac{3a}{10}+\frac{7}{10}, a+\frac{1}{4} \} \} +10^{-100}$. 

\vspace{3mm} Here we take $I = \emptyset$, so that $\tau^{-1/4}M^{2\beta-1}  x^{7/4-2\sigma+\nu/4-4\varepsilon_1}< M^{2\beta-2} x^{2-2\sigma-4\varepsilon_1}$ and we are back in Case~2. But every $\gamma \in [\sigma, 1- 10^{-1000}]$ satisfies one of (\ref{equ:smallafinished2}) or (\ref{equ:smallafinishednew2}). This concludes the treatment of option~(5) and the proof of Lemma~\ref{lem:smalltau}.
 \end{proof}

\subsubsection{Large $\sigma$} Having dealt with large and small $\tau$,  we now focus on large $\sigma$ and midsized $\tau$.

\begin{lemma} \label{lem:largesigmamediuma}
Let $F(s)=\prod_{i=1}^jS_i(s)$ be as described at the start of Section~{\rm\ref{sec:values}}. 
 Recall $S_i(s) = \sum_{n \sim N_i} a_n^{(i)} n^{-s}$ and $N_i=N^{\ell_i}$.  
Let $T\in [T_1,T_0]$ and suppose that  $F(s)$ is  of Type~$A$ or $B$ at $T$. 
Write $\tau = x^a$.

 Finally, suppose we have one of the following two options:
\begin{enumerate}[{\rm (1)}]
\item  $a \in [0.545,0.57]$ and $\sigma \geq a +\nu -\varepsilon_1,$
\item $a \in [0.53,0.545]$,  $\sigma \geq a +\nu -\varepsilon_1$ and   $\exists I \subseteq \{1, \dots, j\}$ with $\sum_{i \in I} \ell_i \in [0.427-\varepsilon_1,0.474+\varepsilon_1].$
\end{enumerate}

\vspace{3mm}
Let $\mathcal{Z}$ be as defined in {\rm (\ref{def:curlyz})}. 
Then if $x \geq C$, there exist $K \subseteq \{1, \dots, j\}$ and $r \leq J$ such that   $M =\prod_{i \in K} N_i^r$ and $M^{\beta} = \prod_{i \in K} N_i^{r \sigma_i}$ satisfy one of the following three inequalities: 
\begin{align} 
& R(F,(S_i),(\sigma_i),T) \leq x^{1-\sigma} \log(x)^{-2J},  \label{equ:mopt11}\\
&R(F,(S_i),(\sigma_i),T) \leq  \min\{\tau^{-1/4}M^{2\beta-1}  x^{7/4-2\sigma+\nu/4-4\varepsilon_1},M^{2\beta-2} x^{2-2\sigma-4\varepsilon_1}\} \label{equ:mopt44}, \\ 
&R(F,(S_i),(\sigma_i),T) \leq  \min_{(U,V,W,X,Y,Z) \in \mathcal{Z}} \tau^{U}M^{V\beta-W}  x^{X-Y\sigma+Z\nu
+\varepsilon_1}. \label{equ:mopt33}
\end{align} 
Here $C$ is a large constant dependent only on  $J$ and $\varepsilon_1$, and $\varepsilon_1$ is assumed to be very small.
\end{lemma}
 
\begin{proof} We first work with an arbitrary $a \in [0.53,0.57]$ and $\sigma \geq a +\nu-\varepsilon_1$. 
If $\sigma \geq 1-10^{-500}$, we are done by Lemma~\ref{lem:verylarge}. Thus we consider $\sigma \in [0.6, 1-10^{-500}]$.  We  also assume $R(F,(S_i),(\sigma_i),T) > x^{1-\sigma} \log(x)^{-2J}$. Recall $F(s) = \prod_{i=1}^j S_i(s)$ with $S_i(s) = \sum_{n \sim N_i} a_n^{(i)} n^{-s}$, $N_i = N^{\ell_i}$ and $x \ll N \ll x$. 
Combining factors  of $F(s)$ which have $\ell_i < 0.0001$, we note that $\{1, \dots, j\}$  can be partitioned into sets $U_1, \dots, U_k$ such that: 
\begin{enumerate}[(I)]
\item $\ell_s^* = \sum_{i \in U_s} \ell_i < 0.0001$ holds for at most one element $s$ of $\{1, \dots, k\}$, 
\item $\ell_s^* = \sum_{i \in U_s} \ell_i \geq  0.0002$ implies that set $U_s$ is a singleton,
\item $\sigma^*_s = (\sum_{i \in U_s} \ell_i \sigma_i )/ (\sum_{i \in U_s} \ell_i )$ are such that $\sigma_1^* \geq \dots \geq \sigma_k^*$.
\end{enumerate}  
Since $F(s)$ is assumed to be of Type A or B, we have $k \geq 2$. Note that $\sigma_1^* \geq \sigma$. 

\vspace{3mm}
\textbf{Case 1: Long factors.}  We first assume that $\ell_1^* \geq 0.249$ or ($\ell_1^* < 0.0001$ and $\ell_2^* \geq 0.2489$).

\vspace{3mm}
 Since $\sigma_2^* \geq \sigma_s^*$  for $s \geq 2$, the latter option implies $\sigma_2^* \geq (\sigma - \ell_1^* \sigma_1^*)/(1-\ell_1^*) \geq \sigma - 0.0002$. For both options we have some $S_i(s)$ with $\ell_i \geq 0.2489$ and $\sigma_i \geq \sigma -0.0002$. Then $F(s)$ must be of Type A and $S_i(s)$ must have property (\ref{propertyz1}) or (\ref{propertyz2}) with respect to $T$, so that Lemma~\ref{lem:newsmooth} applies and
\begin{align}
R(F,(S_i),(\sigma_i),T)  &\leq x^{4\varepsilon_1} \tau^2 N_i^{(6-12\sigma_i)}\leq  x^{2a+0.2489(6-12\sigma+0.0024)+4\varepsilon_1}. \label{equ:largesmooth}
\end{align}
The RHS of (\ref{equ:largesmooth}) is bounded above by $x^{1-\sigma-\varepsilon_1}$ if $a \in [0.53,0.57]$ and $\sigma \in [a+\nu+0.027, 1-10^{-500}]$.
On the other hand, if additionally $\ell_i > 0.27$, we have 
\begin{align}
R(F,(S_i),(\sigma_i),T)  &\leq x^{4\varepsilon_1} \tau^2 N_i^{(6-12\sigma_i)}\leq  x^{2a+0.27(6-12\sigma+0.0024)+4\varepsilon_1}. \label{equ:largesmooth2}
\end{align}
 and the RHS of (\ref{equ:largesmooth2}) is bounded above by $x^{1-\sigma-\varepsilon_1}$ for all $\sigma \in [a+\nu-\varepsilon_1, 1-10^{-500}]$. Now the only case left to consider is $a \in [0.53,0.57]$, $\sigma \in [a+\nu-\varepsilon_1,a+\nu+0.027]$ and $\ell_i \in [0.248,0.27]$. Here we take $M = N_i$ in condition~(\ref{equ:mopt44}). We require
\begin{align} \label{equ:mwithq}
R(F,(S_i),(\sigma_i),T) \leq  \min\{\tau^{-1/4}N_i^{2\sigma_i-1}  x^{7/4-2\sigma+\nu/4-4\varepsilon_1},N_i^{2\sigma_i-2} x^{2-2\sigma-4\varepsilon_1}\}.
\end{align}
But the RHS of (\ref{equ:mwithq}) is greater than $x^{(2-2\sigma)-\ell_i(2-2\sigma+0.0004)+\ell_i-(a+1-\nu)/4-5\varepsilon_1}$. On the other hand, by bound (\ref{R:huxley}) of Lemma~\ref{lem:montgomery}, with $I =\{i\}$ and $k_i=2$, we have
\begin{align}
R(F,(S_i),(\sigma_i),T)  &\leq  x^{ 2\varepsilon_1}\max \left\{ N_i^{2(2-2\sigma_i)}, \tau N_i^{2(4-6\sigma_i)}\right\} \label{equ:mwithq000} \\ \nonumber
&\leq \max \left\{ x^{2\ell_i(2-2\sigma+0.0004)+3\varepsilon_1}, x^{a+2\ell_i(4-6\sigma+0.0012)+3\varepsilon_1}\right\}. 
\end{align}
The RHS of (\ref{equ:mwithq000}) is less than $x^{(2-2\sigma)-\ell_i(2-2\sigma+0.0004)+\ell_i-(a+1-\nu)/4-5\varepsilon_1}$ for $\ell_i \in [0.2489,0.27]$ and  $\sigma \in [a+\nu-\varepsilon_1, a+\nu+0.027]$. We indeed have~(\ref{equ:mopt11}) or~(\ref{equ:mopt33}) if $\ell_1^* \geq 0.249$ or ($\ell_1^* < 0.0001$ and $\ell_2^* \geq 0.2489$).
 
\vspace{3mm} 
\textbf{Case 2: Short factors.}  We now assume that $\ell_1^* \in [0.0001,0.249]$ or ($\ell_1^* < 0.0001$ and $\ell_2^* < 0.2489$).

\vspace{3mm}  We can choose $U= U_1$ or $U= U_1 \cup U_2$ such that $S(s) = \prod_{i \in U} S_i(s)$ satisfies $\ell_U = \sum_{i \in U} \ell_i \in [0.0001,0.249]$ and $\sigma_U = (\sum_{i \in U} \ell_i \sigma_i)/ \ell_U \geq \sigma$.  Take $m \in \mathbb{N}$ such that $2^m \ell_U \in [0.1245,0.249]$.   We now describe three different approaches to obtaining regions in which one of~(\ref{equ:mopt11}), (\ref{equ:mopt44}) or~(\ref{equ:mopt33}) holds.

\vspace{3mm} 
\textbf{Approach 1: Use of (C1).}  Choose $r_1 \in \{2,3\}$. By bound (\ref{R:huxley}) of Lemma~\ref{lem:montgomery} with $I =U$ and $k_i =2^m r_1$, we have $R(F,(S_i),(\sigma_i),T) \leq x^{1-\sigma-\varepsilon_1}$ whenever
\begin{align*}
 \max \left\{ r_1 2^m\ell_U(2-2\sigma_U), a+ r_1 2^m\ell_U(4-6\sigma_U)\right\} + 2\varepsilon_1 \leq 1-\sigma-\varepsilon_1. 
\end{align*} 
But we assumed $\sigma_U \geq \sigma$. Thus  the inequality certainly holds when
\begin{align*}
&\dfrac{-1+\sigma+a+ 3\varepsilon_1}{ r_1 (6\sigma-4) } \leq   2^m\ell_U \leq \dfrac{1-\sigma- 3\varepsilon_1}{r_1(2-2\sigma) } .
\end{align*}
Write $S=\{(\sigma,\ell): \sigma \in [a+\nu-\varepsilon_1, 1-10^{-500}], \ell \in [0.1245,0.249] \}$  and
\begin{align*}
&S_{r}(a)=\left\{(\sigma,\ell) \in S : \dfrac{-1+\sigma+a+ 3\varepsilon_1}{ r (6\sigma-4) } \leq   \ell \leq \dfrac{1-\sigma- 3\varepsilon_1}{r(2-2\sigma) }\right\} .
\end{align*}
If $(\sigma, 2^m \ell_U) \in S_r(a)$ for some $r \in \{2,3\}$, then (\ref{equ:mopt11}) holds.

\vspace{3mm} 
\textbf{Approach 2: Use of (C4).}
Let $r_2 \in \mathbb{N}$.  Then $R \leq  \min\{\tau^{-1/4}M^{2\beta-1}  x^{7/4-2\sigma+\nu/4-4\varepsilon_1},M^{2\beta-2} x^{2-2\sigma-4\varepsilon_1}\}$ with 
$M = \prod_{ i \in U} N_i^{2^m r_2}$ if $R \leq  x^{-5\varepsilon_1 }\min\{ x^{r_2 2^m\ell_U(2\sigma-1)+7/4-a/4-2\sigma+\nu/4}, x^{r_2 2^m \ell_U(2\sigma-2)+2-2\sigma}\}$. 
Applying once more Lemma~\ref{lem:montgomery} with $I =U$ and $k_i =2^m r_1$, this condition is certainly satisfied whenever
\begin{align} \label{equ:applyingq1}
 &  \ell(r_1(2-2\sigma)+r_2(1-2\sigma)) \leq 7/4-a/4-2\sigma+\nu/4-7\varepsilon_1, \\ \label{equ:applyingq2}
 &\ell(r_1(4-6\sigma)+ r_2(1-2\sigma)) \leq 7/4-5a/4-2\sigma+\nu/4-7\varepsilon_1, \\\label{equ:applyingq3}
  & \ell(r_1 + r_2)(2-2\sigma) \leq 2-2\sigma-7\varepsilon_1, \\\label{equ:applyingq4}
  & \ell(r_1(4-6\sigma)+r_2 (2-2\sigma))  \leq 2-2\sigma-a-7\varepsilon_1
\end{align} 
all hold for $\ell = 2^m\ell_U$. Denote by $S^Q_{r_1,r_2}(a)$ the region of $S$ in which inequalities (\ref{equ:applyingq1}), (\ref{equ:applyingq2}), (\ref{equ:applyingq3}) and (\ref{equ:applyingq4}) are all satisfied.  If $(\sigma, 2^m \ell_U) \in S^{ Q}_{r_1,r_2}(a)$, then (\ref{equ:mopt44}) holds.

\vspace{3mm} 
We compute that the union $S_2(a) \cup S_3(a) \cup S^Q_{3,2}(a) \cup S^Q_{4,3}(a)$ covers $S \cap \{(\sigma,\ell): \ell  < 0.2-C^*\varepsilon_1\}$.
Here $C^*$ is a large absolute constant, not dependent on any variables.   Hence we only need to consider 
 $\ell \geq 0.2-C^*\varepsilon_1$. For that case, we  also make use of condition~(\ref{equ:mopt33}):
 
 \vspace{3mm} 
\textbf{Approach 3: Use of (C3).} Let $r_2 \in \mathbb{N}$ and  $M = \prod_{ i \in U} N_i^{2^m r_2}$. Condition~(\ref{equ:mopt33}) holds if 
\begin{align*}
R(F,(S_i),(\sigma_i),T)  \leq x^{r_2 2^m \ell_U (V\sigma-W)+aU +X-Y\sigma+Z \nu}
\end{align*}
for every $(U,V,W,X,Y,Z) \in \mathcal{Z}$.  Applying Lemma~\ref{lem:montgomery} with $I =U$ and $k_i =2^m r_1$, this is true if 
\begin{align}\label{equ:r*cutout}
 & \ell(r_1 (2-2\sigma) + r_2(W-V\sigma)) \leq aU +X-Y\sigma+Z \nu - 2\varepsilon_1 \,\, \mbox{ and } \\
 &\ell(r_1(4-6\sigma)+r_2(W-V\sigma) )  \leq a(U-1) +X-Y\sigma+Z \nu- 2\varepsilon_1 \nonumber
\end{align}
for $\ell = 2^m\ell_U$ and  every $(U,V,W,X,Y,Z) \in \mathcal{Z}$. Denote by $S^*_{r_1,r_2}(a)$ the intersection over $\mathcal{Z}$ of the regions of $S$ cut out by~(\ref{equ:r*cutout}). 

\vspace{3mm} We now distinguish between option (1) and (2) from the statement of the lemma. 

\vspace{3mm}
\textbf{Case 2A: Treatment of option (1).} Suppose $a \in [0.545,0.57]$. Careful computations tell us that $$S \cap \{(\ell,\sigma): \ell  \geq  0.2-C^*\varepsilon_1\}\subseteq S_2(a) \cup S^*_{2,2}(a)\cup S^*_{3,2}(a).$$ Hence we have one (\ref{equ:mopt11}), (\ref{equ:mopt44}) or (\ref{equ:mopt33}) for any choice of $2^m\ell_U \in [0.1245,0.249]$ and this concludes the proof for option (1).

\vspace{3mm}
\textbf{Case 2B: Treatment of option (2).} 
Suppose  $a<0.545$ and suppose there exists  $I \subseteq \{1, \dots, j\}$ with $\sum_{i \in I} \ell_i \in [0.427-\varepsilon_1, 0.474+\varepsilon_1]$.  
Now 
$S_2(a) \cup S^*_{2,2}(a)\cup S^*_{3,2}(a) \cup S^Q_{2,1}(a)$ only covers the region $$(S \cap \{(\ell,\sigma): \ell  \geq 0.2-C^*\varepsilon_1\}) \setminus \{(\sigma, \ell): \sigma \in [a+\nu-\varepsilon_1,a+\nu+0.025], \ell \in [0.2-C^*\varepsilon_1,0.215]\}.$$ We need to consider separately the case $2^m \ell_U \in [0.2-C^*\varepsilon_1,0.215]$ when $\sigma \leq a+\nu+0.025$.

\vspace{3mm} Suppose first that $m \geq 1$. Note $2^{m-1} \ell_U \in [0.1-C^*\varepsilon_1/2,0.1075]$. We take $M= \prod_{i \in U} N_i^{2^{m-1}r_2}$ in (\ref{equ:mopt44}) and note that  (\ref{equ:applyingq1}), (\ref{equ:applyingq2}), (\ref{equ:applyingq3}) and (\ref{equ:applyingq4}) are satisfied if $r_1=5$, $r_2=4$, $\ell \in  [0.1-C^*\varepsilon_1/2,0.1075]$,   $a \in [0.53,0.545]$ and $\sigma \in [a+\nu-\varepsilon_1,a+\nu+0.025]$. This gives (\ref{equ:mopt44}) for $m \geq 1$ and  thus we now assume $m=0$, so that $\ell_U \in [0.2-C^*\varepsilon_1,0.215]$.  Either $U = U_1$ and $U_1$ is a singleton with $\ell_1^* \in [0.2-C^*\varepsilon_1, 0.215]$ or $U = U_1 \cup U_2$ and $\ell_2^* \in [0.2-0.0001-C^*\varepsilon_1, 0.215]$, $\sigma_2^* \geq \sigma -0.0002$ and $U_2$ is a singleton. Thus we may assume that there exists $S_k(s)$ with $\ell_k \in [0.2-0.0001-C^*\varepsilon_1, 0.215]$ and $\sigma_k \geq \sigma -0.0002$. 

\vspace{3mm} Recall  that there exists $I$ with $\sum_{i \in I} \ell_i \in [0.427-\varepsilon_1, 0.474+\varepsilon_1]$. Suppose first that $k \not \in I$. Then $\ell_k+\sum_{i \in I} \ell_i \in [0.627-0.0001-2C^*\varepsilon_1, 0.7+\varepsilon_1]$ and 
\begin{align*}
\sum_{i \not\in I \cup \{k\}} \ell_i \in [0.3-\varepsilon_1, 0.375+\varepsilon_1].
\end{align*} But for $a \in [0.53,0.545]$, this sum then satisfies option (3) of Lemma~\ref{lem:smalltau}. Hence by  Lemma~\ref{lem:smalltau} we have (\ref{equ:mopt11}) or (\ref{equ:mopt44}) for $\sigma \geq a+\nu-\varepsilon_1$.  So suppose instead that $k \in I$. Assume first that $ \frac{\sum_{i \in I} \ell_i \sigma_i}{\sum_{i \in I} \ell_i} < \sigma$. Then we replace the partition of $\{1, \dots, j\}$ into $U_1, \dots, U_k$, carried out at the start of our proof, by a partition of $\{1,\dots, j\} \setminus I$ with the same properties (I), (II) and (III). Since $  \frac{\sum_{i \not\in I} \ell_i \sigma_i}{\sum_{i \not\in I} \ell_i}  \geq \sigma$, all arguments still work the same way and give one of (\ref{equ:mopt11}), (\ref{equ:mopt44}) or (\ref{equ:mopt33})  unless there is some $k_2 \in \{1,\dots, j\} \setminus I$ with $\ell_{k_2} \in [0.2-0.0001-C^*\varepsilon_1, 0.215]$. But then  $\ell_{k_2}+\sum_{i \in I} \ell_i \in [0.627-0.0001-2C^*\varepsilon_1, 0.7+\varepsilon_1]$ and we are again done by Lemma~\ref{lem:smalltau}. Hence we may finally assume that $ \frac{\sum_{i \in I} \ell_i \sigma_i}{\sum_{i \in I} \ell_i}  \geq \sigma$. Taking $M = \prod_{i \in I} N_i$ in (\ref{equ:mopt33}) and noting that conditions (\ref{equ:r*cutout})  are satisfied for each $(U,V,W,X,Y,Z) \in \mathcal{Z}$, $a \in [0.53,0.545]$,  $\sigma \in [a+\nu-\varepsilon_1,a+\nu+0.025]$ and $\ell \in [0.427-\varepsilon_1, 0.474+\varepsilon_1]$ when $r_1 = r_2 =1$, we conclude the proof.
\end{proof} 
 
 \subsubsection{Midsized $\tau$, small $\sigma$} Now we move on to midsized $\tau$ and small $\sigma$. Here condition (\ref{c3'}), which was derived using $R^*(F,(S_i),(\sigma_i), T)$, is used extensively. However, (\ref{c4'}) also makes another appearance, together with all lemmas of Section~\ref{ssec:standnew} -- this is the most complicated computation.

\begin{lemma} \label{lem:smallsigma} 
Let $F(s)=\prod_{i=1}^jS_i(s)$ be as described at the start of Section~{\rm\ref{sec:values}}. 
 Recall $S_i(s) = \sum_{n \sim N_i} a_n^{(i)} n^{-s}$ and $N_i=N^{\ell_i}$.  
Let $T\in [T_1,T_0]$ and suppose that  $F(s)$ is  of Type~$A$ or $B$ at $T$. 
Write $\tau = x^a$. 
 
 Finally, suppose  there exists $I \subseteq \{1, \dots, j\}$ such that $I$ and $a$ satisfy one of the following conditions:
  \begin{enumerate}[{\rm (1) $\quad$ }]
 \item $\sum_{i \in I} \ell_i \in [0.405-\varepsilon_1,0.485+\varepsilon_1] \cup [0.515-\varepsilon_1,0.595+\varepsilon_1] $ and $  a \in [0.53, 0.545]$,
 \item $\sum_{i \in I} \ell_i \in [0.400-\varepsilon_1,0.475+\varepsilon_1] \cup [0.525-\varepsilon_1,0.600+\varepsilon_1] $ and $  a \in [ 0.545, 0.57]$,
  \item $\sum_{i \in I} \ell_i \in [0.380-\varepsilon_1,0.455+\varepsilon_1] \cup [0.545-\varepsilon_1,0.620+\varepsilon_1] $ and $  a \in [ 0.57, 0.59]$,
   \item $\sum_{i \in I} \ell_i \in [0.365-\varepsilon_1,0.435+\varepsilon_1] \cup [0.565-\varepsilon_1,0.635+\varepsilon_1] $ and $  a \in [ 0.59, 0.61]$,
   \item $\sum_{i \in I} \ell_i \in [0.355-\varepsilon_1,0.420+\varepsilon_1] \cup [0.580-\varepsilon_1,0.645+\varepsilon_1] $ and $  a \in [ 0.61, 0.685]$.
 \end{enumerate}
 Write $M =\prod_{i \in I} N_i$ and $M^{\beta} = \prod_{i \in I} N_i^{\sigma_i}$. Let $\mathcal{Z}$ be as given in {\rm (\ref{def:curlyz})}.  
Then if $x \geq C$ and $(\sigma_i)$ has corresponding $\sigma \leq a+\nu-\varepsilon_1$, one of the following four inequalities holds: 
\begin{align} 
& R(F,(S_i),(\sigma_i),T) \leq x^{1-\sigma} \log(x)^{-2J},  \label{equ:r*opt11}\\ 
& R(F,(S_i),(\sigma_i),T) \leq \tau x^{1-2\sigma+\nu+2\varepsilon_1},  \label{equ:r*opt22}\\
&R(F,(S_i),(\sigma_i),T) \leq  \min\{\tau^{-1/4}M^{2\beta-1}  x^{7/4-2\sigma+\nu/4-4\varepsilon_1},M^{2\beta-2} x^{2-2\sigma-4\varepsilon_1}\} \label{equ:r*opt44}, \\ 
&R(F,(S_i),(\sigma_i),T) \leq  \min_{(U,V,W,X,Y,Z) \in \mathcal{Z}} \tau^{U}M^{V\beta-W}  x^{X-Y\sigma+Z\nu
+\varepsilon_1}. \label{equ:r*opt33}
\end{align} 
Here $C$ is a large constant dependent only on $J$ and $\varepsilon_1$, and $\varepsilon_1$ is assumed to be very small.
\end{lemma}

Lemma~\ref{lem:smallsigma} combines with  Lemma~\ref{lem:largesigmamediuma} to complete the proof of Proposition~\ref{proposition2} when $a \in (0.545,0.57]$.  It also combines with Lemma~\ref{lem:smalltau} and Lemma~\ref{lem:largesigmamediuma} to complete the proof of Proposition~\ref{proposition2} when $a \in (0.53,0.545]$. For $a \in (0.57,0.685]$, it covers the case $\sigma \leq a+\nu-\varepsilon_1$, explaining our choice of $\chi_2(a)$, while $\sigma > a+\nu-\varepsilon_1$ is treated in a later lemma.
  
\begin{proof}
 We  assume $R(F,(S_i),(\sigma_i),T) > x^{1-\sigma} \log(x)^{-2J}$. Suppose $(U,V,W,X,Y,Z) \in \mathcal{Z}$ is such that 
\begin{align} \label{equ:r*condition}
 \tau^{U}M^{V\beta-W}  x^{X-Y\sigma+Z\nu
+\varepsilon_1} = \min_{(U_1,V_1,W_1,X_1,Y_1,Z_1) \in \mathcal{Z}} \tau^{U_1}M^{V_1\beta-W_1}  x^{X_1-Y_1\sigma+Z_1\nu
+\varepsilon_1}.
\end{align}
Write $\ell_I = \sum_{i \in I} \ell_i$ and $(\sum_{i \not \in I} \ell_i \sigma_i)/(1-\ell_I) =  \alpha $, so that $N^\sigma =(N/M)^\alpha M^\beta  = N^{(1-\ell_I)\alpha}N^{\ell_I \beta}$. Since $x \ll N \ll x$ and $M = N^{\ell_I}$,   then $\tau^{U}M^{V\beta-W}  x^{X-Y\sigma+Z\nu
+\varepsilon_1}  \geq  x^{(1-\ell_I)(W-V\alpha)+aU+ (X-W)-(Y-V)\sigma+Z\nu}$,
provided $x$ is large.
By Lemma~\ref{lem:montgomery} with $\{1, \dots,j\} \setminus I$ in the place of $I$ and $k_i =1$, we have 
\begin{align*}
R(F,(S_i),(\sigma_i),T) \leq x^{2\varepsilon_1}\max \left\{ x^{(1-\ell_I)(2-2\alpha)}, \min\left\{\tau x^{(1-\ell_I)(1-2\alpha)}, \tau x^{(1-\ell_I)(4-6\alpha)}\right\} \right\}. 
\end{align*}
Hence if $\max\{(1-\ell_I)(2-2\alpha), a+(1-\ell_I)(\gamma-\delta \alpha)\} +2\varepsilon_1 \leq (1-\ell_I)(W-V\alpha)+aU+ (X-W)-(Y-V)\sigma+Z\nu$ holds for some $(\gamma, \delta) \in \{ (1,2), (4,6)\}$, then (\ref{equ:r*opt33}) is satisfied. Rearranging, we need the following: 
\begin{align}
&(V-2)(1-\ell_I)\alpha \leq (1-\ell_I)(W-2)+aU+ (X-W)-(Y-V)\sigma+Z\nu-2\varepsilon_1, \label{equ:2r*1} \\
&(V-\delta )(1-\ell_I)\alpha \leq (1-\ell_I)(W-\gamma)+a(U-1)+ (X-W)-(Y-V)\sigma+Z\nu -2\varepsilon_1. \label{equ:2r*2}
\end{align}
We will now  derive many conditions which ensure that (\ref{equ:2r*1}) and (\ref{equ:2r*2}) hold.

\vspace{3mm}
\textbf{Step 1:} Treatment of (\ref{equ:2r*2}) with  $\delta=6$.

\vspace{3mm}
Since $(U,V,W,X,Y,Z) \in \mathcal{Z}$, we have $V<6$. Hence for $\delta=6$,  (\ref{equ:2r*2}) holds whenever 
\begin{align} \label{equ:intermediateconvenient0}
(1-\ell_I)\alpha \geq \dfrac{a(U-1)+ (X-4)-(Y-V)\sigma+Z\nu -2\varepsilon_1}{(V-6 )} -\dfrac{(W-4)}{(V-6 )}\ell_I.
\end{align}
On the other hand, taking $B(a,\sigma)=a+1-2\sigma+\nu+2\varepsilon_1$, Corollary~\ref{cor:large2} tells us that (\ref{equ:r*opt22})  holds if 
 \begin{align} \label{equ:intermediateconvenient4}
  &\ell_I \beta \geq \ell_I \min\left\{ \left(\frac{3a+B(a,\sigma)-2\varepsilon_1}{3a+3B(a,\sigma)-6\varepsilon_1}\right), \left(\frac{3a-2B(a,\sigma)+4\varepsilon_1}{3a-B(a,\sigma)+2\varepsilon_1}\right)\right\} = \ell_I E(a,\sigma). 
 \end{align}
  Adding  (\ref{equ:intermediateconvenient0}) and (\ref{equ:intermediateconvenient4}), we have  (\ref{equ:r*opt22}) or ((\ref{equ:2r*2}) with  $\delta=6$) if
\begin{align} \label{equ:intermediateconvenient00}
\sigma \geq \dfrac{a(U-1)+ (X-4)-(Y-V)\sigma+Z\nu -2\varepsilon_1}{(V-6 )} -\left(\left(\dfrac{W-4}{V-6 }\right) -E(a,\sigma)\right)\ell_I.
\end{align}
But $(W-4)/(V-6)-E(a,\sigma) \leq 14/19-(4a+1-3/2+\nu)/(6a+3-9/2+3\nu)<0$ whenever $(U,V,W,X,Y,Z) \in \mathcal{Z}$, $a \in [0.53,0.68]$ and $\sigma \in [0.75,1]$. Rearranging, (\ref{equ:intermediateconvenient00}) becomes
\begin{align} \label{equ:delta6}
\sigma \geq 0.75 \,\, \mbox{ and }  \,\, \ell_I \leq \Upsilon_6(a,\sigma;(U,V,W,X,Y,Z))=  \dfrac{a(U-1)+ (X-4)-(Y-6)\sigma+Z\nu -2\varepsilon_1}{(W-4) -(V-6) E(a,\sigma) }.
\end{align}

\textbf{Step 2:} Treatment of (\ref{equ:2r*1}) and ((\ref{equ:2r*2}) with  $\delta=2$). 
\vspace{3mm}

We now split into three cases, depending on whether  $V=2$, $V>2$ or $V<2$. 

\vspace{3mm} \textbf{Case 1:} $V= 2$. Here $(U,V,W,X,Y,Z)$ is $(1/2,2,3/2,3/2,2,1/2)$ or $(3/8,2,3/2,15/8,5/2,5/8)$.

\vspace{3mm} 
Then (\ref{equ:2r*1}) holds if $\ell_I(W-2) \leq aU+ (X-2)-(Y-V)\sigma+Z\nu-2\varepsilon_1$. But $W =3/2$ and so we need
\begin{align*}
\ell_I \geq -2aU- 2(X-2)+2(Y-V)\sigma-2Z\nu+4\varepsilon_1.
\end{align*}
Further, (\ref{equ:2r*2}) with $\delta=2$ holds if $(W-1) \ell_I \leq a(U-1)+ (X-1)-(Y-V)\sigma+Z\nu -2\varepsilon_1$ and so if
\begin{align*}
 \ell_I \leq 2a(U-1)+ 2(X-1)-2(Y-V)\sigma+2Z\nu -4\varepsilon_1.
\end{align*}

Write $\mathcal{Z}_1 = \{ (1/2,2,3/2,3/2,2,1/2), (3/8,2,3/2,15/8,5/2,5/8)\}$ and define
\begin{align*}
&\Lambda_*(a,\sigma;(U,V,W,X,Y,Z)) = -2aU- 2(X-2)+2(Y-V)\sigma-2Z\nu+4\varepsilon_1, \\
&\Upsilon_2(a,\sigma;(U,V,W,X,Y,Z)) = 2a(U-1)+ 2(X-1)-2(Y-V)\sigma+2Z\nu -4\varepsilon_1,
\end{align*} 
for $(U,V,W,X,Y,Z) \in \mathcal{Z}_1$.

\vspace{3mm} \textbf{Case 2:} $V> 2$. Here $(U,V,W,X,Y,Z)$ is  $(1,4,4,3,4,1)$ or  $(2/5,16/5,12/5,12/5,16/5,4/5)$.
 
\vspace{3mm}
Then (\ref{equ:2r*1}) and ((\ref{equ:2r*2}) with $\delta =2$) hold if the following inequalities are satisfied, respectively: 
\begin{align} \label{equ:r1cond}
&(1-\ell_I)\alpha \leq \dfrac{aU+ (X-2)-(Y-V)\sigma+Z\nu-2\varepsilon_1}{(V-2)}-\dfrac{(W-2)}{(V-2)}\ell_I, \\ \label{equ:r2cond}
&(1-\ell_I)\alpha \leq \dfrac{a(U-1)+ (X-1)-(Y-V)\sigma+Z\nu -2\varepsilon_1}{(V-2 )}-\dfrac{(W-1)}{(V-2)}\ell_I.
\end{align}
On the other hand, by Corollary~\ref{cor:large2} with $ \{1, \dots, j \} \setminus I$ in the place of $I$,   $R(F,(S_i),(\sigma_i),T) \leq x^{B(a,\sigma)}$ if
 \begin{align} \label{equ:simplebound1}
  &(1-\ell_I) \alpha \geq (1-\ell_I) \min\left\{ \left(\dfrac{3a+B(a,\sigma)-2\varepsilon_1}{3a+3B(a,\sigma)-6\varepsilon_1}\right), \left(\dfrac{3a-2B(a,\sigma)+4\varepsilon_1}{3a-B(a,\sigma)+2\varepsilon_1}\right)\right\}. 
 \end{align}
Let $E(a,\sigma) = \min\left\{ \left(\frac{3a+B(a,\sigma)-2\varepsilon_1}{3a+3B(a,\sigma)-6\varepsilon_1}\right), \left(\frac{3a-2B(a,\sigma)+4\varepsilon_1}{3a-B(a,\sigma)+2\varepsilon_1}\right)\right\}$. (\ref{equ:r1cond}) and (\ref{equ:simplebound1}) combined then give that one of (\ref{equ:r*opt11}),  (\ref{equ:r*opt22}) or (\ref{equ:2r*1}) holds whenever 
\begin{align} \label{equ:simpleboundnew1} 
 \dfrac{aU+ (X-2)-(Y-V)\sigma+Z\nu-2\varepsilon_1}{(V-2)} -E(a,\sigma) \geq \left(\dfrac{(W-2)}{(V-2)}- E(a,\sigma)\right)\ell_I . 
\end{align}
If $(U,V,W,X,Y,Z) = (1,4,4,3,4,1)$, then $(W-2)/(V-2)-E(a,\sigma)=1-E(a,\sigma)$ is positive for all $a \in [0.53, 0.77]$ and $\sigma \in [0.6,1-10^{-500}]$. If $(U,V,W,X,Y,Z) = (2/5,16/5,12/5,12/5,16/5,4/5)$, then 
$(W-2)/(V-2)-E(a,\sigma)=1/3-E(a,\sigma)$ is instead  negative for all $a \in [0.53, 0.77]$ and $\sigma \in [0.6,1-10^{-500}]$. 
Thus to satisfy (\ref{equ:simpleboundnew1}), we require
\begin{align*}
&\dfrac{2a+ 2+4\nu-6E(a,\sigma)-10\varepsilon_1}{6\left(1/3- E(a,\sigma)\right)} \leq \ell_I \leq \dfrac{a+ 1+\nu -2E(a,\sigma)-2\varepsilon_1}{2\left(1- E(a,\sigma)\right)}.
\end{align*}
We proceed  similarly for (\ref{equ:2r*2}): (\ref{equ:r2cond}) and (\ref{equ:simplebound1}) combined give  that one of (\ref{equ:r*opt11}),  (\ref{equ:r*opt22}) or ((\ref{equ:2r*2}) with $\delta=2$) holds whenever  
\begin{align*} 
  & \dfrac{a(U-1)+ (X-1)-(Y-V)\sigma+Z\nu -2\varepsilon_1}{(V-2 )}-E(a,\sigma) \geq  \left(\dfrac{(W-1)}{(V-2)}- E(a,\sigma)\right)\ell_I. 
 \end{align*} 
 But $(W-1)/(V-2)$ is either $3/2$ or $7/6$ and $(W-  1 )/(V-2)-E(a,\sigma)$ is positive for all $a \in [0.53, 0.77]$ and $\sigma \in [0.6,1-10^{-500}]$ in both cases. So we require
 \begin{align*} 
  & \ell_I \leq  \dfrac{a(U-1)+ (X-1)+Z\nu -(V-2 )E(a,\sigma)-2\varepsilon_1}{(W-1)- (V-2 )E(a,\sigma)}. 
 \end{align*}  
 Next, a second approach to lower bounds on $(1-\ell_I)\alpha$: So far, we have made use of Lemma~\ref{lem:large2}, but if we instead use Lemma~\ref{lem:montgomery}, we have $R(F,(S_i),(\sigma_i),T) \leq x^{B(a,\sigma)}$ whenever
\begin{align} \label{equ:monhux}
 \max \left\{(1-\ell_I)(2-2\alpha), a+\min\left\{(1-\ell_I)(1-2\alpha), (1-\ell_I)(4-6\alpha)\right\} \right\} + 2\varepsilon_1 \leq B(a,\sigma). 
\end{align} 
We have $\xi a+(1-\ell_I)(\lambda-\mu \alpha) +2\varepsilon_1 \leq B(a, \sigma)$ if and only if  
$ (1-\ell_I)\alpha \geq (1-\ell_I)\left(\frac{\lambda}{\mu} \right) -\frac{B(a, \sigma)}{\mu} +\frac{\xi a}{\mu}+\frac{2\varepsilon_1}{\mu}$. Combining this inequality with (\ref{equ:r1cond}) and (\ref{equ:r2cond}), we need 
\begin{align}
&\left(\dfrac{(W-2)}{(V-2)} - \left(\dfrac{\lambda}{\mu} \right) \right)\ell_I\leq \dfrac{aU+ (X-2)+Z\nu-2\varepsilon_1}{(V-2)} +\dfrac{B(a, \sigma)}{\mu}-\dfrac{\xi a}{\mu}-\left(\dfrac{\lambda}{\mu} \right)-\dfrac{2\varepsilon_1}{\mu}, \label{equ:rstarintermediate1} \\ 
&\left(\dfrac{(W-1)}{(V-2)} -\left(\dfrac{\lambda}{\mu} \right) \right) \ell_I \leq \dfrac{a(U-1)+ (X-1)+Z\nu -2\varepsilon_1}{(V-2 )}+\dfrac{B(a, \sigma)}{\mu}-\dfrac{\xi a}{\mu}-\left(\dfrac{\lambda}{\mu} \right) -\dfrac{2\varepsilon_1}{\mu}. \label{equ:rstarintermediate2}
\end{align}
for both ($(\xi,\lambda, \mu) = (0,2,2)$ and $(\xi,\lambda, \mu) = (1,1,2)$) or for both ($(\xi,\lambda, \mu) = (0,2,2)$ and $(\xi,\lambda, \mu) = (1,4,6)$). Then one of ((\ref{equ:r*opt11}),  (\ref{equ:r*opt22}) or (\ref{equ:2r*1})) and one of ((\ref{equ:r*opt11}),  (\ref{equ:r*opt22}) or (\ref{equ:2r*2})) is satisfied.

\vspace{3mm} But $(W-2)/(V-2)$ is either $1$ or $1/3$, while $(W-1)/(V-2)$ is either $3/2$ or $7/6$ and $\lambda/\mu$ is $1$, $1/2$ or $2/3$. We first consider the case   $(W-2)/(V-2)=1$, where $(U,V,W,X,Y,Z) = (1,4,4,3,4,1)$. If $(\xi,\lambda, \mu) = (0,2,2)$, then  $(W-2)/(V-2)-\lambda/\mu=0$. Since $B(a, \sigma) = a+1-2\sigma+\nu+2\varepsilon_1$,  (\ref{equ:rstarintermediate1}) becomes
\begin{align*}
\sigma\leq a+\nu -\varepsilon_1.
\end{align*}
If $(\xi,\lambda, \mu) = (1,1,2)$ or $(\xi,\lambda, \mu) = (1,4,6)$, then $(W-2)/(V-2)-\lambda/\mu$ is positive and (\ref{equ:rstarintermediate1})  becomes
\begin{align*}
&\ell_I\leq \left(1 - \dfrac{\lambda}{\mu}  \right)^{-1} \left(\dfrac{a+ 1+\nu-2\varepsilon_1}{2} +\frac{B(a, \sigma)}{\mu}-\dfrac{ a}{\mu}-\dfrac{\lambda}{\mu} -\dfrac{2\varepsilon_1}{\mu} \right),
\end{align*}
For the case $(W-2)/(V-2)=1/3$, we note that $(W-2)/(V-2)-\lambda/\mu$ is certainly negative. We  have $(U,V,W,X,Y,Z) = (2/5,16/5,12/5,12/5,16/5,4/5)$ and (\ref{equ:rstarintermediate1})  becomes
\begin{align*}
&\ell_I\geq \left(\dfrac{1}{3} - \dfrac{\lambda}{\mu}  \right)^{-1} \left( \dfrac{2a+ 2+4\nu-10\varepsilon_1}{6} +\frac{B(a, \sigma)}{\mu}-\dfrac{\xi a}{\mu}-\dfrac{\lambda}{\mu}-\dfrac{2\varepsilon_1}{\mu} \right).
\end{align*}
On the other hand, $(W-1)/(V-2) - \lambda/\mu$ is certainly positive and (\ref{equ:rstarintermediate2}) rearranges to 
\begin{align*}
& \ell_I \leq \left(\dfrac{(W-1)}{(V-2)} -\dfrac{\lambda}{\mu} \right)^{-1} \left(\dfrac{a(U-1)+ (X-1)+Z\nu -2\varepsilon_1}{(V-2 )}+\dfrac{B(a, \sigma)}{\mu}-\dfrac{\xi a}{\mu}-\dfrac{\lambda}{\mu}  -\dfrac{2\varepsilon_1}{\mu} \right).
\end{align*}

Finally, for yet another treatment of (\ref{equ:2r*2}), recall from Step 1 that we have (\ref{equ:2r*2}) with $\delta =6$ if 
\begin{align} \label{equ:intermediateconvenient000}
(1-\ell_I)\alpha \geq \dfrac{a(U-1)+ (X-4)-(Y-V)\sigma+Z\nu -2\varepsilon_1}{(V-6 )} -\dfrac{(W-4)}{(V-6 )}\ell_I.
\end{align}
Combining (\ref{equ:intermediateconvenient000}) and (\ref{equ:r2cond}), we have  (\ref{equ:2r*2}) with $\delta =2$  or $\delta=6$ if 
\begin{align*}
&\ell_I \leq \left(\dfrac{(W-4)}{(V-6 )}-\dfrac{(W-1)}{(V-2)} \right)^{-1} \left(\dfrac{4a(U-1)+ 4X-3V+2+4Z\nu -8\varepsilon_1
}{(V-2)(V-6 )} \right).
\end{align*}

We summarize these conditions as follows: 

\vspace{3mm}

For $(U,V,W,X,Y,Z) = (1,4,4,3,4,1)$, we have one of (\ref{equ:r*opt11}),  (\ref{equ:r*opt22}) or (\ref{equ:2r*1}) if  $\sigma\leq a+\nu -\varepsilon_1$ and
\begin{align}\label{equ:bigcondition1}
\ell_I\leq \max\Bigg\{&\Bigg(\dfrac{a+ 1+\nu -2E(a,\sigma)-2\varepsilon_1}{2\left(1- E(a,\sigma)\right)} \Bigg),  \\
&\left(\max_{(\lambda, \mu) \in \{(1,2), (4,6)\}} \left(1 - \dfrac{\lambda}{\mu}  \right)^{-1} \left(\dfrac{a+ 1+\nu-2\varepsilon_1}{2} +\frac{B(a, \sigma)}{\mu}-\dfrac{ a}{\mu}-\dfrac{\lambda}{\mu} -\dfrac{2\varepsilon_1}{\mu} \right)\right) \Bigg\}. \nonumber
\end{align}
Denote the RHS of (\ref{equ:bigcondition1}) by $\Upsilon_*(a,\sigma;(1,4,4,3,4,1))$. Write $\mathcal{Z}_{2,1} = \{(1,4,4,3,4,1)\}$. 

\vspace{3mm}
 Write $\mathcal{Z}_{2,2} = \{ (2/5,16/5,12/5,12/5,16/5,4/5)\}$. For $(U,V,W,X,Y,Z) \in \mathcal{Z}_{2,2}$, we have one of (\ref{equ:r*opt11}),  (\ref{equ:r*opt22}) or (\ref{equ:2r*1}) if $\ell_I$ is greater than or equal to
\begin{align}\label{equ:bigcondition2}
\min\Bigg\{& \Bigg( \dfrac{2a+ 2+4\nu-6E(a,\sigma)-10\varepsilon_1}{6\left(1/3- E(a,\sigma)\right)} \Bigg),  \\
&\,\,\max \Bigg\{ \left(-\dfrac{2}{3}\right)^{-1} \left( \dfrac{2a+ 2+4\nu-10\varepsilon_1}{6} +\frac{B(a, \sigma)}{2}-1-\varepsilon_1 \right),  \nonumber\\ 
&\quad \quad \quad \,\,\left(\min_{(\lambda, \mu) \in \{(1,2), (4,6)\}}  \left(\dfrac{1}{3} - \dfrac{\lambda}{\mu}  \right)^{-1} \left( \dfrac{2a+ 2+4\nu-10\varepsilon_1}{6} +\frac{B(a, \sigma)}{\mu}-\dfrac{ a}{\mu}-\dfrac{\lambda}{\mu}-\dfrac{2\varepsilon_1}{\mu} \right)\right) \Bigg\}
\Bigg\}. \nonumber
\end{align}
Denote the RHS of (\ref{equ:bigcondition2}) by $\Lambda_*(a,\sigma;(2/5,16/5,12/5,12/5,16/5,4/5))$.

\vspace{3mm} 
For any $(U,V,W,X,Y,Z) \in \mathcal{Z}_2 = \mathcal{Z}_{2,1} \cup \mathcal{Z}_{2,2}$, we have  one of (\ref{equ:r*opt11}),  (\ref{equ:r*opt22})  or  ((\ref{equ:2r*2}) with $\delta=2$ or $\delta=6$)  if $\ell_I$ is less than or equal to
\begin{align} \label{equ:bigcondition3}
\max\Bigg\{& \Bigg(\dfrac{a(U-1)+ (X-1)+Z\nu -(V-2 )E(a,\sigma)-2\varepsilon_1}{(W-1)- (V-2 )E(a,\sigma)}\Bigg),  \\
&\nonumber \left(\dfrac{(W-4)}{(V-6 )}-\dfrac{(W-1)}{(V-2)} \right)^{-1} \left(\dfrac{4a(U-1)+ 4X-3V+2+4Z\nu -8\varepsilon_1
}{(V-2)(V-6 )} \right), \\
&\,\,\min \Bigg\{
\left(\dfrac{(W-1)}{(V-2)} -1 \right)^{-1} \left(\dfrac{a(U-1)+ (X-1)+Z\nu -2\varepsilon_1}{(V-2 )}+\dfrac{B(a, \sigma)}{2}-1 -\varepsilon_1 \right)
,  \nonumber\\ 
&\quad \quad \quad \left(\max_{\substack{\{(1,2), (4,6)\}}}  \dfrac{(a(U-1)+ (X-1)+Z\nu -2\varepsilon_1)/(V-2 )+(B(a, \sigma)- a-\lambda -2\varepsilon_1)/\mu }{ (W-1)/(V-2) -\lambda/\mu }\right) \Bigg\}
\Bigg\}. \nonumber
\end{align}
Denote the RHS of (\ref{equ:bigcondition3}) by $\Upsilon_2(a,\sigma;(U,V,W,X,Y,Z))$ when $(U,V,W,X,Y,Z) \in \mathcal{Z}_2$.
  
\vspace{3mm} \textbf{Case 3:} $ V < 2$. This case covers the remaining elements of $\mathcal{Z}$, namely  $(2/5,4/5,3/5,6/5,8/5,2/5)$, $(1/3,2/3,1/3,1,4/3,1/3)$,  $(2/7,16/21,8/21,8/7,32/21,8/21)$,
 $(1/4,4/3,2/3,5/4,5/3,5/12)$ and  also 
 
  $(1/9,16/9,8/9,5/3,20/9,5/9)$. Write $\mathcal{Z}_3 = \mathcal{Z} \setminus ( \mathcal{Z}_1 \cup \mathcal{Z}_2)$.

\vspace{3mm}
Then (\ref{equ:2r*1}) and ((\ref{equ:2r*2}) with $\delta =2$) hold if the following inequalities are satisfied, respectively: 
\begin{align} \label{equ:rr1cond}
&(1-\ell_I)\alpha \geq \dfrac{aU+ (X-2)-(Y-V)\sigma+Z\nu-2\varepsilon_1}{(V-2)}-\dfrac{(W-2)}{(V-2)}\ell_I, \\ \label{equ:rr2cond}
&(1-\ell_I)\alpha \geq \dfrac{a(U-1)+ (X-1)-(Y-V)\sigma+Z\nu -2\varepsilon_1}{(V-2 )}-\dfrac{(W-1)}{(V-2)}\ell_I.
\end{align}
On the other hand, Corollary~\ref{cor:large2} tells us that  $R(F,(S_i),(\sigma_i),T) \leq x^{B(a,\sigma)}$ whenever
 \begin{align} \label{equ:simpleboundnew}
  &\ell_I \beta \geq \ell_I \min\left\{ \left(\dfrac{3a+B(a,\sigma)-2\varepsilon_1}{3a+3B(a,\sigma)-6\varepsilon_1}\right), \left(\dfrac{3a-2B(a,\sigma)+4\varepsilon_1}{3a-B(a,\sigma)+2\varepsilon_1}\right)\right\} = \ell_I E(a,\sigma). 
 \end{align}
Summing up (\ref{equ:rr1cond}) and (\ref{equ:simpleboundnew}), we certainly have one of (\ref{equ:r*opt11}),  (\ref{equ:r*opt22}) or (\ref{equ:2r*1}) if 
\begin{align*}
 \sigma  \geq \dfrac{aU+ (X-2)-(Y-V)\sigma+Z\nu-2\varepsilon_1}{(V-2)}+\left(E(a,\sigma)-\dfrac{(W-2)}{(V-2)} \right)\ell_I.
\end{align*}
But $E(a,\sigma) - (W-2)/(V-2) \leq 1- 7/6 <0$ and so this rearranges to
\begin{align} \label{equ:simpleboundnew2}
 \ell_I \geq \left(E(a,\sigma)-\dfrac{(W-2)}{(V-2)} \right)^{-1} \left(\sigma -\left(\dfrac{aU+ (X-2)-(Y-V)\sigma+Z\nu-2\varepsilon_1}{(V-2)} \right) \right).
\end{align}

\vspace{3mm}
Summing up (\ref{equ:rr2cond}) and (\ref{equ:simpleboundnew}), we certainly have one of (\ref{equ:r*opt11}),  (\ref{equ:r*opt22}) or ((\ref{equ:2r*2}) with $\delta=2$) if 
\begin{align} \label{equ:simpleboundnew3}
 \sigma  \geq \dfrac{a(U-1)+ (X-1)-(Y-V)\sigma+Z\nu -2\varepsilon_1}{(V-2 )}+\left(E(a,\sigma)-\dfrac{(W-1)}{(V-2)} \right)\ell_I.
\end{align}  
 But $E(a,\sigma) - (W-1)/(V-2) \geq (3a-2B+4\varepsilon_1)/(3a-B+2\varepsilon_1)-1/2>0$ whenever $B(a,\sigma)<a+2\varepsilon_1$. Of course, if  $B(a,\sigma)\geq a+2\varepsilon_1$, we automatically have $R(F,(S_i),(\sigma_i),T) \leq T_0= x^{a+\varepsilon_1} < x^{B(a,\sigma)}$. So we may  assume $B(a,\sigma)< a+2\varepsilon_1$ and $E(a,\sigma) - (W-   1)/(V-2) >0$. Thus (\ref{equ:simpleboundnew3}) rearranges to
 \begin{align} \label{equ:simpleboundnew4}
 \ell_I \leq \left(E(a,\sigma)-\dfrac{(W-1)}{(V-2)} \right)^{-1} \left(\sigma -\left(\dfrac{a(U-1)+ (X-1)-(Y-V)\sigma+Z\nu -2\varepsilon_1}{(V-2 )}\right) \right).
\end{align}
For $a \geq 0.64$ and $\Delta =(U,V,W,X,Y,Z) \in \mathcal{Z}_3$ , we now simply set
\begin{align*}
&\Lambda_*(a,\sigma;\Delta) = \left(E(a,\sigma)-\dfrac{(W-2)}{(V-2)} \right)^{-1} \left(\sigma -\left(\dfrac{aU+ (X-2)-(Y-V)\sigma+Z\nu-2\varepsilon_1}{(V-2)} \right) \right), \\
&\Upsilon_2(a,\sigma;\Delta) = \left(E(a,\sigma)-\dfrac{(W-1)}{(V-2)} \right)^{-1}\! \left(\sigma -\left(\dfrac{a(U-1)+ (X-1)-(Y-V)\sigma+Z\nu -2\varepsilon_1}{(V-2 )}\right) \right).
\end{align*}
For $a < 0.64$, we also mix in condition (\ref{equ:r*opt44}): It is satisfied if 
\begin{align} \label{equ:qusingmonhux}
R(F,(S_i),(\sigma_i),T) \leq  x^{(2\beta-2)\ell_I} x^{2-2\sigma-5\varepsilon_1}\min\{ x^{\ell_I-(1+a-\nu)/4}, 1\}.
\end{align} 
If $a < 0.64$ then $(1+a-\nu)/4 <0.3525$ and $\ell_I \geq 0.355-\varepsilon_1$ gives $\min\{ x^{\ell_I-(1+a-\nu)/4}, 1\} = 1$. 
Applying Lemma~\ref{lem:montgomery}, we have (\ref{equ:qusingmonhux})  whenever
\begin{align*} 
\xi a+(\lambda-\mu \beta)\ell_I +2\varepsilon_1 \leq (2\beta-2)\ell_I+2-2\sigma-5\varepsilon_1
\end{align*} 
for both ($(\xi,\lambda, \mu) = (0,2,2)$ and $(\xi,\lambda, \mu) = (1,1,2)$) or ($(\xi,\lambda, \mu) = (0,2,2)$ and $(\xi,\lambda, \mu) = (1,4,6)$) or ($(\xi,\lambda, \mu) = (0,2,2)$ and $(\xi,\lambda, \mu) = (1,11,14)$). The inequality rearranges to
\begin{align} \label{equ:workingwithq}
\beta \ell_I \geq \dfrac{1}{\mu+2}\left(\xi a-(2-2\sigma)+(\lambda+2)\ell_I +7\varepsilon_1 \right).
\end{align} 
Summing up (\ref{equ:rr1cond}) and (\ref{equ:workingwithq})  and rearranging, we certainly have one of (\ref{equ:r*opt44}) or (\ref{equ:2r*1}) if  both
\begin{align} \label{equ:workingwithq1}
&\ell_I \geq \left(1-\dfrac{W-2}{V-2} \right)^{-1} \left( -\dfrac{aU+ (X-2)-(Y-V)\sigma+Z\nu-2\varepsilon_1}{V-2} +\dfrac{2+2\sigma -7\varepsilon_1 }{4} \right) \quad \mbox{ and } \\   \label{equ:workingwithq2}
&\ell_I \geq \left(\dfrac{\lambda+2}{\mu+2}-\dfrac{W-2}{V-2} \right)^{-1}\left( - \dfrac{aU+ (X-2)-(Y-V)\sigma+Z\nu-2\varepsilon_1}{V-2} + \dfrac{ 2+\mu \sigma-a -7\varepsilon_1 }{2+\mu} \right)
\end{align}
for some $(\lambda,\mu) \in \{ (1,2), (4,6), (11,14)\}$. Here we used that $(\lambda+2)/(\mu+2)-(W-2)/(V-2) \leq 1-7/6 <0$. For $a < 0.64$ and  $(U,V,W,X,Y,Z) \in \mathcal{Z}_3$, denote by $\Lambda_*(a,\sigma;(U,V,W,X,Y,Z))$ the quantity 
\begin{align*} 
\min\Bigg\{& \left(E(a,\sigma)-\dfrac{(W-2)}{(V-2)} \right)^{-1} \left(\sigma -\left(\dfrac{aU+ (X-2)-(Y-V)\sigma+Z\nu-2\varepsilon_1}{(V-2)} \right) \right),  \\
&\!\max \Bigg\{ \left(1-\dfrac{W-2}{V-2} \right)^{-1} \left( -\dfrac{aU+ (X-2)-(Y-V)\sigma+Z\nu-2\varepsilon_1}{V-2} +\dfrac{2+2\sigma -7\varepsilon_1 }{4} \right),  \nonumber\\ 
&\quad \quad \,  \min \Bigg\{\left(\dfrac{3}{4}-\dfrac{W-2}{V-2} \right)^{-1}\left( - \dfrac{aU+ (X-2)-(Y-V)\sigma+Z\nu-2\varepsilon_1}{V-2} + \dfrac{ 2+2 \sigma-a -7\varepsilon_1 }{4} \right), \nonumber \\
&\quad \quad \quad\quad \, \left(\dfrac{6}{8}-\dfrac{W-2}{V-2} \right)^{-1}\left( - \dfrac{aU+ (X-2)-(Y-V)\sigma+Z\nu-2\varepsilon_1}{V-2} + \dfrac{ 2+6 \sigma-a -7\varepsilon_1 }{8} \right),  \nonumber \\ 
&\quad \quad \quad\quad \, \left(\dfrac{13}{16}-\dfrac{W-2}{V-2} \right)^{-1}\!\!\left( - \dfrac{aU+ (X-2)-(Y-V)\sigma+Z\nu-2\varepsilon_1}{V-2} + \dfrac{ 2+14 \sigma-a -7\varepsilon_1 }{16} \right)\! \Bigg\} \Bigg\}
\Bigg\} \nonumber
\end{align*}
and note that one of (\ref{equ:r*opt11}), (\ref{equ:r*opt22}), (\ref{equ:r*opt44})  or (\ref{equ:2r*1}) holds if $\ell_I \geq \Lambda_*(a,\sigma;(U,V,W,X,Y,Z))$.

\vspace{3mm}
Instead summing up (\ref{equ:rr2cond}) and (\ref{equ:workingwithq})  and rearranging, we have one of (\ref{equ:r*opt44}) or (\ref{equ:2r*2}) if  both
\begin{align}  \label{equ:alternative0}
&\ell_I \leq \left(1-\dfrac{W-1}{V-2} \right)^{-1} \left( -\dfrac{a(U-1)+ (X-1)-(Y-V)\sigma+Z\nu-2\varepsilon_1}{V-2} +\dfrac{2+2\sigma -7\varepsilon_1 }{4} \right) \quad \mbox{ and } \\ \label{equ:alternative00}  
&\ell_I \leq \left(\dfrac{\lambda+2}{\mu+2}-\dfrac{W-1}{V-2} \right)^{-1}\left( - \dfrac{a(U-1)+ (X-1)-(Y-V)\sigma+Z\nu-2\varepsilon_1}{V-2} + \dfrac{ 2+\mu \sigma-a -7\varepsilon_1 }{2+\mu} \right). 
\end{align}
For $(U,V,W,X,Y,Z) \in \mathcal{Z}_3$  let $\Upsilon_2(a,\sigma;(U,V,W,X,Y,Z)) = \max\{(\ref{equ:simpleboundnew4}), \min\{(\ref{equ:alternative0}), \max(\ref{equ:alternative00})\}\}$, where (\ref{equ:simpleboundnew4}), (\ref{equ:alternative0}) and (\ref{equ:alternative00}) are short-hand for the RHS of the corresponding inequalities and the innermost maximum is taken over  $(\lambda, \mu) \in \{(1,2), (4,6), (11,14)\}$. One of (\ref{equ:r*opt11}), (\ref{equ:r*opt22}), (\ref{equ:r*opt44})  or (\ref{equ:2r*2}) holds if $a < 0.64$ and $0.355-\varepsilon_1 \leq \ell_I \leq \Upsilon_2(a,\sigma;(U,V,W,X,Y,Z))$. 

\vspace{3mm} \textbf{Step 3:} Treatment of large $a$.

\vspace{3mm} We are now done with the discussion of (\ref{equ:2r*1}) and (\ref{equ:2r*2}), which combine to give (\ref{equ:r*opt33}). But for large $a$ it is sometimes better to take a direct approach, working only towards conditions  (\ref{equ:r*opt11}) and (\ref{equ:r*opt22}). Recall that we already know from (\ref{equ:monhux}) in Step 2, Case 2 and (\ref{equ:simpleboundnew}) in Step 2, Case 3 that $R(F,(S_i),(\sigma_i),T) \leq x^{B(a,\sigma)}$ whenever one of the following three inequalities is satisfied:
\begin{align*}
&(1-\ell_I)\alpha \geq  \max\left\{(1-\ell_I) -\frac{B(a, \sigma)}{2}+\varepsilon_1, 
(1-\ell_I) \left(\frac{1}{2} \right) -\frac{B(a, \sigma)}{2} +\frac{ a}{2}+\varepsilon_1\right\}, \\
&(1-\ell_I)\alpha \geq  \max\left\{(1-\ell_I) -\frac{B(a, \sigma)}{2}+\varepsilon_1,
(1-\ell_I) \left(\frac{2}{3} \right) -\frac{B(a, \sigma)}{6} +\frac{ a}{6}+\frac{2\varepsilon_1}{6}\right\}, \\
  &\ell_I \beta \geq \ell_I \min\left\{ \left(\dfrac{3a+B(a,\sigma)-2\varepsilon_1}{3a+3B(a,\sigma)-6\varepsilon_1}\right), \left(\dfrac{3a-2B(a,\sigma)+4\varepsilon_1}{3a-B(a,\sigma)+2\varepsilon_1}\right)\right\} = \ell_I E(a,\sigma). 
 \end{align*} 
 Adding up and rearranging, we have (\ref{equ:r*opt11}) or (\ref{equ:r*opt22}) if $\sigma \geq 0.75$ and
 \begin{align*}
&\ell_I  \geq \Lambda_E(a,\sigma) = \dfrac{1}{(1-E(a,\sigma))} \left(1 -\frac{B(a, \sigma)}{2} -\sigma +\varepsilon_1 \right), \\
&\ell_I  \leq \Upsilon_{E,2}(a,\sigma)= \max\left\{ \left(\dfrac{1-B(a, \sigma) + a-2\sigma+2\varepsilon_1 }{ 1 - 2E(a,\sigma) }\right), 
 \left(\dfrac{ 4 -B(a, \sigma) +a-6\sigma+6\varepsilon_1 }{4 - 6E(a, \sigma) }\right) \right\}.
 \end{align*}
 (The assumption $\sigma \geq 0.75$ ensured that $2/3-E(a,\sigma) \leq 2/3 - E(a,0.75) \leq 2/3 - E(0.68,0.75) <0$.) 
 For $\sigma \geq 0.65$ we still have (\ref{equ:r*opt11}) or (\ref{equ:r*opt22}) if
 \begin{align*}
\Lambda_E(a,\sigma) \leq  \ell_I  \leq \Upsilon_{E,1}(a,\sigma) = \left(\dfrac{1-B(a, \sigma) + a-2\sigma+2\varepsilon_1 }{ 1 - 2E(a,\sigma) }\right).
 \end{align*}
 Alternatively, we can also use option (4) of  Lemma~\ref{lem:smalltau}: It 
 gives (\ref{equ:r*opt11}) or (\ref{equ:r*opt44}) if $\sigma \geq  \min\{a+\nu-\varepsilon_1, 0.85\}$ and $\Lambda_Q(a,\sigma) \leq \ell_I \leq \Upsilon_Q(a,\sigma)$, where 
  \begin{align*} 
& \Upsilon_Q(a,\sigma) = 1- \min\left\{\dfrac{3a}{6\sigma-2}, \max\left\{\dfrac{3a}{20\sigma-14},  \dfrac{2a}{4\sigma-1} \right\}\right\}  -10^{-100} \quad \mbox{ and }            \\
&\Lambda_Q(a,\sigma) = 1- \max\{m_1(a,\sigma), \min\{ m_2(a,\sigma), m_3(a,\sigma)\}\}+10^{-100} \quad \mbox{ where } \\
&m_1(a,\sigma) = \min\left\{ \dfrac{a-3-\nu}{4(1-2\sigma)} +
\dfrac{a(3-3\sigma)}{(3\sigma-1)(1-2\sigma)},  \dfrac{a-3-\nu}{4(1-2\sigma^{\circ}_1)} +
\dfrac{a(3-3\sigma^{\circ}_1)}{(3\sigma^{\circ}_1-1)(1-2\sigma^{\circ}_1)}\right\},    \\
&m_2(a,\sigma) = \min\left\{ \dfrac{a-3-\nu}{4(1-2\sigma)} +
\dfrac{a(3-3\sigma)}{(10\sigma-7)(1-2\sigma)},  \dfrac{a-3-\nu}{4(1-2\sigma^\circ_2)} +
\dfrac{a(3-3\sigma^\circ_2)}{(10\sigma^\circ_2-7)(1-2\sigma^\circ_2)}\right\},  \\
&m_3(a,\sigma) = \min\left\{ \dfrac{a-3-\nu}{4(1-2\sigma)} +
\dfrac{a(4-4\sigma)}{(4\sigma-1)(1-2\sigma)},  \dfrac{a-3-\nu}{4(1-2\sigma^\circ_2)} +
\dfrac{a(4-4\sigma^\circ_2)}{(4\sigma^\circ_2-1)(1-2\sigma^\circ_2)}\right\}   \\ \nonumber
&\mbox{and } \quad \sigma^\circ_1 = \min\left\{ 1,\, a+\dfrac{1}{3} \right\} \quad \mbox{ and } \quad \sigma^\circ_2 = \max\left\{ \dfrac{3a}{10} + \dfrac{7}{10},\, a+\dfrac{1}{4} \right\}.
\end{align*}
 
\vspace{3mm} \textbf{Summary:} We assume $ \ell_I \in [0.355-\varepsilon_1,0.645+\varepsilon_1]$, $a \in [0.53,0.685]$ and $\sigma \in [0.6, a+\nu-\varepsilon_1]$. If $B(a,\sigma) \geq a +2\varepsilon_1$, we immediately have $R(F,(S_i),(\sigma_i),T) < x^{B(a,\sigma)}$. Thus we also assume $\sigma>(1+\nu)/2.$

\vspace{3mm}
  At least one of (\ref{equ:r*opt11}), (\ref{equ:r*opt22}), (\ref{equ:r*opt44}) or ((\ref{equ:2r*1}) and (\ref{equ:2r*2})) holds if $\sigma \leq a+\nu-\varepsilon_1$ and 
 \begin{align} \label{equ:longsummary1}
&\max_{\triangle \in \mathcal{Z}_1\cup \mathcal{Z}_{2,2} \cup \mathcal{Z}_3} \Lambda_*(a,\sigma;\triangle) \leq 
\ell_I \leq 
\min\left\{ \min_{\triangle \in \mathcal{Z}_{2,1}} \Upsilon_*(a,\sigma;\triangle), \min_{\triangle \in \mathcal{Z}} \Upsilon_2(a,\sigma;\triangle)  \right\}.
\end{align}  
We also have one of (\ref{equ:r*opt11}), (\ref{equ:r*opt22}), (\ref{equ:r*opt44}) or ((\ref{equ:2r*1}) and (\ref{equ:2r*2})) if $0.75 \leq \sigma \leq a+\nu-\varepsilon_1$
  and 
   \begin{align}\label{equ:longsummary2}
&\max_{\triangle \in \mathcal{Z}_1\cup \mathcal{Z}_{2,2} \cup \mathcal{Z}_3} \Lambda_*(a,\sigma;\triangle) \leq 
\ell_I \leq 
\min\left\{ \min_{\triangle \in \mathcal{Z}_{2,1}} \Upsilon_*(a,\sigma;\triangle), \min_{\triangle \in \mathcal{Z}} \left( \max\left\{ \Upsilon_2(a,\sigma;\triangle), \Upsilon_6(a,\sigma;\triangle) \right\}   \right) \right\}.
\end{align}  
One of (\ref{equ:r*opt11}) or (\ref{equ:r*opt22}) holds if ($\sigma \geq 0.65$ and $\Lambda_E(a,\sigma) \leq \ell_I \leq \Upsilon_{E,1}(a,\sigma)$) or ($\sigma \geq 0.75$ and $\Lambda_E \leq \ell_I \leq \Upsilon_{E,2}(a,\sigma)$). One of (\ref{equ:r*opt11}) or (\ref{equ:r*opt44}) holds if $\sigma \geq \min\{a+\nu-\varepsilon_1, 0.85\}$ and $\Lambda_Q(a,\sigma) \leq \ell_I \leq \Upsilon_Q(a,\sigma)$.

  \vspace{3mm}
  But if (\ref{equ:2r*1}) and (\ref{equ:2r*2}) are satisfied for some $(\gamma,\delta) \in \{(1,2),(4,6)\}$, then (\ref{equ:r*opt33}) holds and we are done. So all we have got left to do is to give explicit lower and upper bounds on the minima and maxima in  (\ref{equ:longsummary1}) and (\ref{equ:longsummary2}). These expressions  may appear complicated at first, but quantities $\Upsilon_*(a,\sigma;\triangle)$,  $\Upsilon_2(a,\sigma;\triangle)$, $\Upsilon_6(a,\sigma;\triangle)$  and $\Lambda_*(a,\sigma;\triangle)$ individually are easy to bound for $a$ and $\sigma$ contained in some fixed  intervals. In particular, we numerically compute the following:
  \begin{align}
&\max_{\triangle \in \mathcal{Z}_1\cup \mathcal{Z}_{2,2} \cup \mathcal{Z}_3} \Lambda_*(a,\sigma;\triangle)  \leq 
\begin{cases}
0.405 -\varepsilon_1 \quad \mbox{ if } a \in [0.53,0.545], \sigma \in [(1+\nu)/2,a+\nu],\\ 
0.400 -\varepsilon_1 \quad \mbox{ if } a \in [0.545,0.57], \sigma \in [(1+\nu)/2,a+\nu], \\ 
0.380 -\varepsilon_1 \quad \mbox{ if } a \in [0.57,0.59], \sigma \in [(1+\nu)/2,a+\nu], \\ 
0.365 -\varepsilon_1 \quad \mbox{ if } a \in [0.59,0.61], \sigma \in [(1+\nu)/2,a+\nu], 
\end{cases} \\
&\min\left\{ \min_{\triangle \in \mathcal{Z}_{2,1}} \Upsilon_*(a,\sigma;\triangle), \min_{\triangle \in \mathcal{Z}} \Upsilon_2(a,\sigma;\triangle)  \right\}
 \geq 
\begin{cases}
0.485 +\varepsilon_1 \quad \mbox{ if } a \in [0.53,0.545], \sigma \leq 0.75,\\ 
0.475 +\varepsilon_1 \quad \mbox{ if } a \in [0.545,0.57], \sigma \leq 0.75, \\ 
0.455 +\varepsilon_1 \quad \mbox{ if } a \in [0.57,0.59], \sigma \leq 0.75, \\ 
0.435 +\varepsilon_1 \quad \mbox{ if } a \in [0.59,0.61], \sigma \leq 0.75, 
\end{cases}  \label{equ:cases2}
\end{align}
and  (\ref{equ:cases2}) also holds with $\min\{ \min_{\triangle \in \mathcal{Z}_{2,1}} \Upsilon_*(a,\sigma;\triangle), \min_{\triangle \in \mathcal{Z}} \left( \max\left\{ \Upsilon_2(a,\sigma;\triangle), \Upsilon_6(a,\sigma;\triangle) \right\}   \right) \}$ in the place of $\min\{ \min_{\triangle \in \mathcal{Z}_{2,1}} \Upsilon_*(a,\sigma;\triangle), \min_{\triangle \in \mathcal{Z}} \Upsilon_2(a,\sigma;\triangle)  \}$ and $\sigma \in [0.75, a+\nu]$ in the place of $\sigma \leq 0.75$.

\vspace{3mm}
 For $a \in [0.61,0.685]$, $\sigma \in [0.6,a+\nu]$ and $\ell_I \in [0.355-\varepsilon_1,0.42+\varepsilon_1]$, there are instances in which (\ref{equ:longsummary1}) and (\ref{equ:longsummary2}) fail to hold, particularly when ($\sigma$ is close to $0.75$ and $\ell_I$ is close to $0.42$) or ($\sigma$ is close to $a +\nu$ and $\ell_I$ is close to $0.355$). But these cases are instead easily covered by (($\sigma \geq 0.65$ and $\Lambda_E \leq \ell_I \leq \Upsilon_{E,1}$) or ($\sigma \geq 0.75$ and $\Lambda_E \leq \ell_I \leq \Upsilon_{E,2}$)) or ($\sigma > 0.85$ and $\Lambda_Q(a,\sigma) \leq \ell_I \leq \Upsilon_Q(a,\sigma)$), respectively.
\end{proof}

\subsubsection{Midsized $\tau$, large $\sigma$}

 Most parts of the proof of Lemma~\ref{lem:smallsigma} are actually also applicable when $\sigma \geq a+ \nu-\varepsilon_1$. In particular, combining the proof of Lemma~\ref{lem:smallsigma} with Lemma~\ref{lem:smalltau}, we get the following result, which is very important  for the treatment of midsized $\tau$ and large $\sigma$:

\begin{lemma} \label{lem:mediumtau}
Let $F(s)=\prod_{i=1}^jS_i(s)$ be as described at the start of Section~{\rm\ref{sec:values}}. 
 Recall $S_i(s) = \sum_{n \sim N_i} a_n^{(i)} n^{-s}$ and $N_i=N^{\ell_i}$.  
Let $T\in [T_1,T_0]$ and suppose that  $F(s)$ is  of Type~$A$ or $B$ at $T$. 
Write $\tau = x^a$.
 
 Finally, suppose  there exists $I \subseteq \{1, \dots, j\}$ such that $I$ and $a$ satisfy one of the following conditions:
  \begin{enumerate}[{\rm (1) $\quad$ }]
  \item $\sum_{i \in I} \ell_i \in [0.315-\varepsilon_1,0.420+\varepsilon_1] \cup [0.580-\varepsilon_1,0.685+\varepsilon_1] $ and $  a \in [ 0.57, 0.59]$,
   \item $\sum_{i \in I} \ell_i \in [0.330-\varepsilon_1,0.420+\varepsilon_1] \cup [0.580-\varepsilon_1,0.670+\varepsilon_1] $ and $  a \in [ 0.59, 0.61]$,
   \item $\sum_{i \in I} \ell_i \in [0.355-\varepsilon_1,0.420+\varepsilon_1] \cup [0.580-\varepsilon_1,0.645+\varepsilon_1] $ and $  a \in [ 0.61, 0.685]$.
 \end{enumerate}
 Write $M =\prod_{i \in I} N_i$ and $M^{\beta} = \prod_{i \in I} N_i^{\sigma_i}$. Let $\mathcal{Z}$ be as given in {\rm (\ref{def:curlyz})}. 
Then if $x \geq C$ and $(\sigma_i)$ has corresponding $\sigma >  a+\nu-\varepsilon_1$, one of the following four inequalities holds: 
\begin{align} 
& R(F,(S_i),(\sigma_i),T) \leq x^{1-\sigma} \log(x)^{-2J},  \label{equ:r*opt11new}\\ 
& R(F,(S_i),(\sigma_i),T) \leq \tau x^{1-2\sigma+\nu+2\varepsilon_1},  \label{equ:r*opt22new}\\
&R(F,(S_i),(\sigma_i),T) \leq  \min\{\tau^{-1/4}M^{2\beta-1}  x^{7/4-2\sigma+\nu/4-4\varepsilon_1},M^{2\beta-2} x^{2-2\sigma-4\varepsilon_1}\} , \label{equ:r*opt44new}\\ 
&R(F,(S_i),(\sigma_i),T) \leq  \min_{(U,V,W,X,Y,Z) \in \mathcal{Z}} \tau^{U}M^{V\beta-W}  x^{X-Y\sigma+Z\nu
+\varepsilon_1}. \label{equ:r*opt33new}
\end{align} 
Here $C$ is a large constant dependent only on  $J$ and $\varepsilon_1$, and $\varepsilon_1$ is assumed to be very small.
\end{lemma}

Lemma~\ref{lem:mediumtau} explains our choice of $\chi_3(a)$ and completes the proof of Proposition~\ref{proposition2} when $a \in (0.57,0.685]$.

\begin{proof}
If $\sigma \geq \max\{3a/10+7/10, a+1/4\}+10^{-100}$, we are done by  option (5) of Lemma~\ref{lem:smalltau}. Thus we only consider $\sigma \in [a+\nu-\varepsilon_1, \max\{3a/10+7/10, a+1/4\}+10^{-100}]$.  We  also assume $R(F,(S_i),(\sigma_i),T) > x^{1-\sigma} \log(x)^{-2J}$. Suppose $(U,V,W,X,Y,Z) \in \mathcal{Z}$ is such that 
\begin{align*} 
 \tau^{U}M^{V\beta-W}  x^{X-Y\sigma+Z\nu
+\varepsilon_1} = \min_{(U_1,V_1,W_1,X_1,Y_1,Z_1) \in \mathcal{Z}} \tau^{U_1}M^{V_1\beta-W_1}  x^{X_1-Y_1\sigma+Z_1\nu
+\varepsilon_1}.
\end{align*}
Write $\ell_I = \sum_{i \in I} \ell_i$ and $(\sum_{i \not \in I} \ell_i \sigma_i)/(1-\ell_I) =  \alpha $, so that $N^\sigma =(N/M)^\alpha M^\beta  = N^{(1-\ell_I)\alpha}N^{\ell_I \beta}$ and $\tau^{U}M^{V\beta-W}  x^{X-Y\sigma+Z\nu
+\varepsilon_1}  \geq  x^{(1-\ell_I)(W-V\alpha)+aU+ (X-W)-(Y-V)\sigma+Z\nu}$.
By Lemma~\ref{lem:montgomery},
\begin{align*}
R(F,(S_i),(\sigma_i),T) \leq x^{2\varepsilon_1}\max \left\{ x^{(1-\ell_I)(2-2\alpha)}, \tau x^{(1-\ell_I)(4-6\alpha)} \right\}. 
\end{align*}
 Rearranging, we thus have (\ref{equ:r*opt33new})  whenever 
\begin{align}
&(V-2)(1-\ell_I)\alpha \leq (1-\ell_I)(W-2)+aU+ (X-W)-(Y-V)\sigma+Z\nu-2\varepsilon_1,  \label{equ:need1} \\
&(V-6 )(1-\ell_I)\alpha \leq (1-\ell_I)(W-4)+a(U-1)+ (X-W)-(Y-V)\sigma+Z\nu -2\varepsilon_1. \label{equ:need2}
\end{align}
In the proof of Lemma~\ref{lem:smallsigma} we carefully derived conditions which ensure that (\ref{equ:need1}) and (\ref{equ:need2}) hold. We assumed $\sigma \leq a+\nu-\varepsilon_1$ and took $B(a,\sigma) = a+1-2\sigma+\nu+2\varepsilon_1$, but all but one of our arguments work just as well with $\sigma \in [a+\nu-\varepsilon_1,\max\{3a/10+7/10, a+1/4\}+10^{-100}]$ and $B(a,\sigma) =1-\sigma-\varepsilon_1$. The exception is the treatment of (\ref{equ:need1}) with  $(U,V,W,X,Y,Z)=(1,4,4,3,4,1)$, for which we now give a different approach.

\vspace{3mm}
Note that (\ref{equ:need1}) holds with $(U,V,W,X,Y,Z)=(1,4,4,3,4,1)$ whenever
\begin{align}\label{equ:comb1}
&(1-\ell_I)\alpha \leq \dfrac{a+ 1+\nu-2\varepsilon_1}{2}-\ell_I.
\end{align}
So suppose instead that $\alpha > (\frac{a+ 1+\nu-2\varepsilon_1}{2}-\ell_I)/(1-\ell_I)$. For $\ell_I < 0.5$, this gives $\alpha \geq 0.6$.  
By Lemma~\ref{lem:large2} with $ \{1, \dots, j \} \setminus I$ in the place of $I$, then
\begin{align}\label{equ:qtry1}
R(F,(S_i), (\sigma_i),T) \leq x^{2\varepsilon_1} \tau^{(3-3\alpha)/(3\alpha-1)}.
\end{align}
Assume now that $\ell_I > 0.365$. Then $\ell_I > (a+1-\nu)/4$ for $a \leq 0.685$. Condition (\ref{equ:r*opt44new}) is satisfied if 
\begin{align} \label{equ:qtry2}
R(F,(S_i),(\sigma_i),T) \leq  x^{(1-\ell_I)(2-2\alpha)-5\varepsilon_1}.
\end{align}
Combining (\ref{equ:qtry1}) and (\ref{equ:qtry2}) and rearranging, we thus require 
\begin{align}
\dfrac{2}{3}(1-\ell_I)(3\alpha-1)-a-C^*\varepsilon_1 \geq 0.
\end{align} 
Here $C^*$ is a large absolute constant. (We assumed $\alpha < 0.99$ as otherwise (\ref{equ:r*opt11new}) follows trivially from (\ref{equ:qtry1}) and $\sigma \leq \max\{3a/10+7/10,a+1/4\}+10^{-100} \leq 0.95$.) But   $\alpha > (\frac{a+ 1+\nu-2\varepsilon_1}{2}-\ell_I)/(1-\ell_I)$ and so 
\begin{align*}
\dfrac{2}{3}(1-\ell_I)(3\alpha-1)-a-C^*\varepsilon_1 >  \left(\dfrac{1}{3}+\nu-2\varepsilon_1-C^*\varepsilon_1\right)-\dfrac{4}{3}\ell_I.
\end{align*}
Therefore we have (\ref{equ:r*opt44new}) if $\ell_I \leq 1/4+3\nu/4-C^{**}\varepsilon_1$. It suffices to take $\ell_I \leq 0.42+\varepsilon_1$.

\vspace{3mm}
The requirement $0.365 \leq \ell_I \leq 0.42+\varepsilon_1$  replaces conditions $\sigma \leq a+\nu-\varepsilon_1$ and $\ell_I \leq \Upsilon_*(a,\sigma;(1,4,4,3,4,1))$ of the proof of Lemma~\ref{lem:smallsigma}. Leaving the rest of that proof unaltered, we thus have one of
 (\ref{equ:r*opt11new}), (\ref{equ:r*opt22new}), (\ref{equ:r*opt44new}) or (\ref{equ:r*opt33new})  if $\sigma \in [a+\nu-\varepsilon_1, \max\{3a/10+7/10,a+1/4\}+ 10^{-100}]$ 
  and 
   \begin{align} \label{equ:newoption1}
\max\left\{\max_{\triangle \in \mathcal{Z}_1\cup \mathcal{Z}_{2,2} \cup \mathcal{Z}_3} \Lambda_*(a,\sigma;\triangle), 0.365 \right\} \leq \ell_I \leq 
 \min\left\{\min_{\triangle \in \mathcal{Z}}  \Upsilon_6(a,\sigma;\triangle)   , 0.42+\varepsilon_1\right\} .
\end{align}  
Here  $\Lambda_*(a,\sigma;\triangle)$ and $\Upsilon_6(a,\sigma;\triangle)$ are  as defined in the proof of Lemma~\ref{lem:smallsigma}, except  $B(a,\sigma) =1-\sigma-\varepsilon_1$.
Alternatively, the proof of Lemma~\ref{lem:smallsigma} also gives (\ref{equ:r*opt11new}) or (\ref{equ:r*opt44new}) whenever \begin{align}\label{equ:newoption2}
\Lambda_Q(a,\sigma) \leq \ell_I \leq \Upsilon_Q(a,\sigma).
\end{align}

\vspace{3mm}
Computing the values of $\Lambda_*(a,\sigma;\triangle)$,  $\Upsilon_6(a,\sigma;\triangle)$, $\Lambda_Q(a,\sigma)$ and $\Upsilon_Q(a,\sigma)$, we find that one of (\ref{equ:newoption1}) or (\ref{equ:newoption2}) holds for any $\sigma \in [a+\nu-\varepsilon_1, \max\{3a/10+7/10,a+1/4\}+10^{-100}]$ provided we have     ($  a \in [ 0.57, 0.59]$ and $\ell_I \in [0.315-\varepsilon_1,0.420+\varepsilon_1]$) or ($  a \in [ 0.59, 0.61]$ and $\ell_I \in [0.33-\varepsilon_1,0.420+\varepsilon_1]$) or 
($  a \in [ 0.61, 0.685]$ and $\ell_I \in [0.355-\varepsilon_1,0.420+\varepsilon_1]$).
\end{proof}

\section{Comparison of sifted sets}~\label{sec:comparison}

The goal of this section is to prove Lemma~\ref{lemma1}.  For certain good sizes of $P_1, \dots, P_r$, we show that 
\begin{align*}
& \Bigg|\sum_{p_i \sim P_i} S(\mathcal{A}(y)_{p_1 \dots p_r}, z) - 
\dfrac{x^b}{\tau} \sum_{p_i \sim P_i}S(\mathcal{B}(y)_{p_1 \dots p_r}, z)  \Bigg|
\leq  \left( \dfrac{\log\log(x)^{O(1)}y}{\tau \log(x)^{2+r}} \right). 
\end{align*}
Here we use the following notation from sieve theory: For $\mathcal{C} \subseteq \mathbb{N}$ and $z>0$, we write $P(z)=\prod_{p< z} p$. 
Then 
$S(\mathcal{C},z) = \# \{n \in \mathcal{C}: (n,P(z))=1\}.$ We also write $V(z)=\prod_{p < z} (1-1/p)$ and $\mathcal{C}_d = \{ n \in \mathbb{N}: nd \in \mathcal{C}\}$. 

\subsection{Results so far} Lemma~\ref{lemma1} will be derived from Proposition~\ref{proposition1} and Proposition~\ref{proposition2}. These two propositions can be summarized as follows: 

\begin{cor} \label{cor:combining}
Let $\varepsilon>0$ and $a \in [0.475-\varepsilon, 0.77-\varepsilon]$. Set $\tau = x^a$. Let $K=2000$ and  $0 \leq r_1  \leq r_2 \leq 10^5$.  
Let $N_i \in [ \frac{1}{2}, \infty)$ with $2^{-r_2}x \leq \prod_{i=0}^{r_2} N_i \leq 6x$ and  $N_i < x^{1/K}$ for $i>r_1$ and $N_i \geq x^{1/K}$ for $1 \leq i \leq r_1$.   Let $a_{n}^{(i)} \in \mathbb{C}$ with $a_n^{(i)}=O( \log(n+3))$ for $i > r_1$. Then consider
\begin{align*}
a_n= \sum_{\substack{k_0k_1 \dots k_{r_2} =n \\ k_i \sim N_i  }} \left( 1_{\mathbb{N}}(k_0) \prod_{1 \leq i \leq r_1} 1_\mathbb{P}(k_i) \prod_{i > r_1} a^{({i})}_{k_{i}} \right).
\end{align*}

\vspace{3mm}
 Suppose also  that one of the following three options holds: 
 \begin{enumerate}[{\rm (i)}]
 \item $r_1\geq 2$ or $N_0 \geq x^{0.95}$.
 \item  $r_1 =1$  and  $N_0 \geq x^{1/\log\log x}$.
 \item $r_1 \in \{0,1\}$  and there exists $i>r_1$ with $a_n^{(i)} = 1_{\mathbb{P}}(n)$ for all  $n \in \mathbb{N}$  and $N_i \geq x^{1/\log\log x}$.
 \end{enumerate}

Let $A \in \mathbb{N}$. 
Write  $N_i = x^{\ell_i^*}$.    Denote by $\Xi(\ell_{1}^*, \dots, \ell_{r_1}^*)$ the set of finite sequences $\{\ell_i\}_{i=1}^j$ with $\ell_1, \dots, \ell_j \in [0,1]$, $\sum_{i=1}^j \ell_i =1$ and $j \leq A+10^7K$ for which there exist disjoint subsets $X_1, \dots, X_{r_1}$ of $\{1, \dots, j\}$ with
 \begin{align*}
 \sum_{i \in X_s} \ell_i \in \left[\ell_s^*-\dfrac{\varepsilon}{10^{100}}, \ell_s^*+\dfrac{\varepsilon}{10^{100}}\right] \quad \mbox{ for } \quad 1 \leq s \leq r_1
 \end{align*}
 and with $\ell_i \leq 2/K=0.001$ for all but at most one $i \in \{1, \dots, j \} \setminus \bigcup_{s=1}^{r_1} X_s$.

\vspace{3mm}
Let $\chi_0(a)$, $\chi_1(a)$, $\chi_2(a)$ and $\chi_3(a)$ be as given in equations $(\ref{chi0})$, $(\ref{chi1})$, $(\ref{chi2})$ and $(\ref{chi3})$ of Section~{\rm \ref{sec:key}}. Suppose every $\{\ell_i\}_{i=1}^j \in \Xi(\ell_1^*, \dots, \ell_{r_1}^*)$ satisfies  one of the following three options:
\begin{enumerate} [{\rm (1)}]
\item There exists $k \in \{1, \dots, j\}$  with $\ell_k \geq \chi_0(a)$. 
\item There exists $I_1 \subseteq \{1,\dots,j\}$ with  $\sum_{i \in I_1} \ell_i  \in \chi_1(a) $.
\item There exist $I_2 \subseteq \{1,\dots,j\}$ with  $\sum_{i \in I_2} \ell_i  \in\chi_2(a) $ and 
$I_3 \subseteq \{1,\dots,j\}$ with  $\sum_{i \in I_3} \ell_i  \in  \chi_3(a)   $. 
\end{enumerate}

\vspace{3mm}
Then there exists $\mathcal{I} \subseteq [x,3x] \cap \mathbb{N}$ such that  $\#\mathcal{I} = O( \tau x^{0.23+\varepsilon/2})$ and  such that   $y  \in ( [x,3x] \cap \mathbb{N}) \setminus \mathcal{I}$ implies
\begin{align*}
\left|\sum_{ n \in \mathcal{A}(y)} a_n - \dfrac{x^b}{\tau} \sum_{ n \in \mathcal{B}(y) } a_n \right|  \leq \dfrac{x}{\tau \log(x)^A} .
\end{align*}
Here the constant implied by big O notation only depends on $A$ and $\varepsilon$.
\end{cor}

\begin{proof}
Here $(a_n)$ is as given at the start of Section~\ref{ssec:proposition1}. Let $F(s) = \prod_{i=1}^j S_i(s) \in  \mathcal{F}((N_i),T,A+10^7K,K)$ with  $S_i(s) = \sum_{n \sim B_i} b_n^{(i)} n^{-s}$,   $B = \prod_{i=1}^j B_i$ and  $B_i = B^{\ell_i}$.  
 There exist disjoint sets $X_1, \dots, X_{r_1} \subseteq \{1, \dots, j\}$ with $B_i <  x^{1/K}$ for all but at most one  $ i \in \{1, \dots, j\} \setminus \bigcup_{\ell=1}^{r_1} X_\ell$ and 
\begin{align*} 
\log(x)^{-2K-1} N_\ell \leq \prod_{i \in X_\ell} B_i \leq \log(x)^{2K+1} N_\ell \quad \mbox{ \rm for } \quad 1 \leq \ell \leq r_1.
\end{align*}  
Hence we must have $\{ \ell_i \}_{i=1}^j \in \Xi(\ell_1^*, \dots, \ell_{r_1}^*)$.  But then one of the options (1), (2) or (3) of Proposition~\ref{proposition2} holds. Hence by Proposition~\ref{proposition2}, every choice of  $F(s)$, $G(s)$, $(\sigma_i)$ and $\gamma$  described in Proposition~\ref{proposition1} satisfies one of (\ref{cc1}), (\ref{cc2}), (\ref{cc3}), (\ref{cc4.A}) or (\ref{cc4.B}). But by  Proposition~\ref{proposition1}  the set $\mathcal{I}$  of $y \in  [x,3x] \cap \mathbb{N}$ with 
 \begin{align*}
\left|\sum_{ n \in \mathcal{A}(y)} a_n - \dfrac{x^b}{\tau} \sum_{ n \in \mathcal{B}(y) } a_n \right|  > \dfrac{x}{\tau \log(x)^A} 
\end{align*}
then has $\# \mathcal{I} = O( \tau x^{0.23+\varepsilon/2})$, where the implied constant only depends on $A$ and $\varepsilon$.
\end{proof}

\subsection{Fundamental Lemma}

We now show how Corollary~\ref{cor:combining} can be used to compare $\mathcal{S}(\mathcal{A}(y), w)$ and $\mathcal{S}(\mathcal{B}(y), w)$ when $w$ is very small. This, too, will be a key ingredient in the proof of Lemma~\ref{lemma1}. 

\begin{lemma} \label{lem:fundsing}  Let $ y \in [x,3x] \cap \mathbb{N}$, $K_2 \in \mathbb{N}$  and $A>0$. Let $w = \alpha x^{1/\log \log x}$ for some $\alpha \in [1,2]$. Let $(c_s)$ be a sequence with $0\leq c_s \leq \log(x)$ and $c_s=0$ when  $(s,P(w)) \neq 1$ or $s \geq x^{1-b-2/K_2}$.   

\vspace{3mm }  Then there exist sequences $(\xi_d^{(1)})$ and $(\xi_d^{(2)})$, dependent on $x$,  such that  $|\xi_d^{(i)}|\leq 1$ for $i=1,2$ and $\xi_d^{(i)}=0$ for $d \geq x^{1/K_2}$  and  such that the following is true:
\begin{align}
  &E_2 \leq \sum_{s} c_s \,S(\mathcal{A}(y)_s,w) - \dfrac{x^b}{\tau} \sum_{s}c_s \, S(\mathcal{B}(y)_s,w) \leq  E_1,
\\
\mbox{where } \quad \,\,&E_1 = \sum_{\substack{ dsn \in \mathcal{A}(y) \\ d\mid P(w)}} c_s\,\xi_d^{(1)}  - \dfrac{x^b}{\tau}\sum_{\substack{ dsn \in \mathcal{B}(y)\\ d \mid P(w)}} c_s\,\xi_d^{(1)}+
O\!\left(\dfrac{x}{\tau \log (x)^{A}}\right),\label{equ:funde3}
\\
&E_2 =   \sum_{\substack{ dsn \in \mathcal{A}(y)\\ d \mid P(w)}} c_s\,\xi_d^{(2)}  - \dfrac{x^b}{\tau}\sum_{\substack{ dsn \in \mathcal{B}(y) \\ d \mid P(w)}} c_s\,\xi_d^{(2)}+
O\!\left(\dfrac{x}{\tau \log (x)^{A }}\right). 
\label{equ:funde4}
\end{align}
Here the implied constants only depend on the choice of $A$ and $K_2$. 
\end{lemma}

\begin{proof}
We take  $\delta(p)=1$, $k=1$ and $s=(1/K_2)(\log \alpha/\log x +1/\log\log x)^{-1}$ and replace $y$ by $x^{1/K_2}$ in  Theorem~7 of~\cite{Motohashi:1983:LecturesOS}. We apply the theorem twice for each $q=s$, once with $A=\mathcal{A}(y)$, $X=y/\tau$, $R_d = |\mathcal{A}(y)|-y/(\tau d)$ and $\nu=1$, and once with $A=\mathcal{B}(y)$, $X=y/x^b$, $R_d = |\mathcal{B}(y)|-y/(x^b d)$ and $\nu=2$. 
\begin{align} 
S(\mathcal{A}(y)_s, w) &\leq \dfrac{y}{\tau s}V(w) + O\left(\dfrac{y}{\tau s}\exp\left(-\dfrac{ \log \log x \log\log \log x}{8K_2} \right)\right) \label{equ:funda} \\
&+\sum_{\substack{d | P(w)  \\ d < x^{1/K_2}}} \xi_d^{(1)} \left(|\mathcal{A}(y)_{ds}| - \dfrac{y}{\tau ds} \right),
\nonumber
\\
S(\mathcal{B}(y)_s, w) &\geq \dfrac{y}{x^b s}V(w) + O\left(\dfrac{y}{x^b s}\exp\left(-\dfrac{\log \log x \log \log \log x}{8K_2} \right)\right)  \label{equ:fundb}\\
&+\sum_{\substack{d | P(w)  \\ d < x^{1/K_2}}} \xi_d^{(2)} \left(|\mathcal{B}(y)_{ds}| - \dfrac{y}{x^b ds} \right). \nonumber
\end{align}
Now we note that $||\mathcal{B}(y)_{ds}|-y/(x^bds)|$ is bounded by $1$ and  that 
\begin{align*}
 |\mathcal{A}(y)_{ds}| -\dfrac{y}{\tau d s} &= |\mathcal{A}(y)_{ds}|-\dfrac{x^b}{\tau}|\mathcal{B}(y)_{ds}| + \left(\dfrac{x^b}{\tau}\right)\left(|\mathcal{B}(y)_{ds}|-\dfrac{ y  }{x^b ds}\right) \nonumber \\
 &= |\mathcal{A}(y)_{ds}|-\dfrac{x^b}{\tau}|\mathcal{B}(y)_{ds}|+O\!\left(\dfrac{x^{b}}{\tau}\right).
 \end{align*} Substituting this identity into~(\ref{equ:funda}) and summing (\ref{equ:funda}) and~(\ref{equ:fundb}) over $s$, we  find
\begin{align*}
&\sum_s c_s S(\mathcal{A}(y)_s,w) - \dfrac{x^b}{\tau} \sum_s c_s S(\mathcal{B}(y)_s,w) \\&\leq  O\Bigg(\dfrac{x}{\tau} (\log x)^{- \log\log \log x/(8K_2)} \sum_{s \leq x} \dfrac{1}{s} + \sum_{s} |c_s| \dfrac{x^{b+1/K_2}}{\tau}\Bigg)
+\sum_{s}\!\sum_{d \mid P(w)}\! c_s \xi_d^{(1)} \left( |\mathcal{A}(y)_{ds}| -\dfrac{x^b}{\tau} |\mathcal{B}(y)_{ds}|\right)\\
&=\sum_{s} \sum_{d \mid P(w)} c_s\xi_d^{(1)} \left( |\mathcal{A}(y)_{ds}| -\dfrac{x^b}{\tau} |\mathcal{B}(y)_{ds}|\right) + O\left(\dfrac{x}{\tau \log(x)^{A }} \right).
\end{align*}
Here we used that $c_s=0$ for $s \geq x^{1-b-2/K_2}$. By the same argument,
\begin{align*}
&\sum_s c_s S(\mathcal{A}(y)_s,w) - \dfrac{x^b}{\tau} \sum_s c_s S(\mathcal{B}(y)_s,w) \geq \sum_{s} \!\sum_{d \mid P(w)}\!  c_s \xi_d^{(2)} \left( |\mathcal{A}(y)_{ds}| -\dfrac{x^b}{\tau} |\mathcal{B}(y)_{ds}|\right), 
\end{align*}
where we omitted an error term of $O(x/(\tau \log(x)^{A}))$ on the RHS. 
\end{proof}

\subsection{Buchstab's identity and small prime factors} \label{ssec:smallprimes}

 We  now prove the first half of Lemma~\ref{lemma1}.

\begin{lemma} \label{lem:smallprimessmalltaufirststep}
Consider $\tau = x^a$ with $x^{19/40-\varepsilon} \leq \tau \leq x^{0.77-\varepsilon}$.
Let $\beta \in [0.01,0.15]$.

Let $r \in \{0,\dots, 5\}$ and $P_i=x^{\ell_i^*}$ with $\beta \leq \ell_i^*\leq 0.5+\varepsilon_1$,  $P_i \geq P_{i+1}$ and $P_1 \dots P_r \leq x^{0.75}$.

\vspace{3mm} Denote by $\Xi^\star(\ell_1^*, \dots, \ell_r^*, \beta)$ the set of finite sequences $\{\ell_i\}_{i=1}^j$ with $\ell_1, \dots, \ell_j \in [0,1]$, $\sum_{i=1}^j \ell_i =1$ and $j \leq 10^{20}$ for which there exist disjoint subsets $X_1, \dots, X_r$ of $\{1, \dots, j\}$ with
 \begin{align*}
 \sum_{i \in X_s} \ell_i \in \left[\ell_s^*-\dfrac{\varepsilon}{10^{100}}, \ell_s^*+\dfrac{\varepsilon}{10^{100}}\right]
 \end{align*}
 for $s \leq r$ and $\ell_i \leq \beta + 10^{-100}\varepsilon$ for all but at most one $i \in \{1, \dots, j \} \setminus \bigcup_{i=1}^r X_{ i} $. 
  
Suppose that $\{\ell_i\}_{i=1}^j \in \Xi^\star(\ell_1^*, \dots, \ell_r^*, \beta)$ implies that one of the following three options is satisfied: 
\begin{enumerate} [{\rm (1)}]
\item There exists $k \in \{1, \dots, j\}$  with $\ell_k \geq \chi_0(a)$. 
\item There exists $I_1 \subseteq \{1,\dots,j\}$ with  $\sum_{i \in I_1} \ell_i  \in \chi_1(a) $.
\item There exist $I_2 \subseteq \{1,\dots,j\}$ with  $\sum_{i \in I_2} \ell_i  \in\chi_2(a) $ and 
$I_3 \subseteq \{1,\dots,j\}$ with  $\sum_{i \in I_3} \ell_i  \in  \chi_3(a)   $. 
\end{enumerate}
Here $\chi_0(a)$, $\chi_1(a)$, $\chi_2(a)$ and $\chi_3(a)$ are as described in {\rm (\ref{chi0})}, {\rm (\ref{chi1})}, 
{\rm (\ref{chi2})} and {\rm (\ref{chi3})}, respectively.

\vspace{3mm}
Then there exists  $\mathcal{I} \subseteq [x,3x]\cap \mathbb{N}$ such that $\#\mathcal{I} = O( \tau x^{0.23+3\varepsilon/4})$ and 
 $y \in ([x,3x]\cap\mathbb{N} )\setminus\mathcal{I} $ implies
\begin{align}
& \left| \sum_{p_i \sim P_i} S(\mathcal{A}(y)_{p_1 \dots p_r}, x^\beta) - 
\dfrac{x^b}{\tau} \sum_{p_i \sim P_i}S(\mathcal{B}(y)_{p_1 \dots p_r}, x^\beta) \right|   
\leq  \left( \dfrac{\log\log(x)^{O(1)} y}{\tau \log(x)^{2+r}} \right). 
\end{align} 
Here the constant implied by big O notation only depends on $\varepsilon$. 
\end{lemma}

The key ingredients of the proof of this lemma are  Corollary~\ref{cor:combining}, Lemma~\ref{lem:fundsing}  and the Buchstab identity.  Lemma~\ref{lem:smallprimessmalltaufirststep} is related to Theorem~3.1 of Harman's book~\cite{Harman:2007:PDS}, but the sequences $(a_n)$ allowed in  Corollary~\ref{cor:combining} are of a significantly more restrictive form than those considered in Theorem~3.1 and thus we include a sketch of the proof.
   
\begin{proof} 
We first prove the following lower bound:
\begin{align}
& (-1)^r \left( \sum_{p_i \sim P_i} S(\mathcal{A}(y)_{p_1 \dots p_r}, x^\beta) - 
\dfrac{x^b}{\tau} \sum_{p_i \sim P_i}S(\mathcal{B}(y)_{p_1 \dots p_r}, x^\beta) \right) 
\geq - \left( \dfrac{\log\log(x)^{O(1)} y}{\tau \log(x)^{2+r}} \right). \label{equ:targetsmallprimesfirststep}
\end{align}

Write $D= \lfloor \log(x^{\beta}/x^{1/\log\log x}) /\log(2) \rfloor$ and $w=2^{-D} x^{\beta}$. We repeatedly apply the Buchstab identity $S(\mathcal{C},z) = S(\mathcal{C},w) - \sum_{w \leq p < z} S(\mathcal{C}_p,p)$, starting with the LHS of (\ref{equ:targetsmallprimesfirststep}). At intermediate steps, we look at 
\begin{align} \label{todecompose}
(-1)^{k}
\sum_{\substack{  p_\ell \sim P_\ell \\ p_k < p_{k-1}}} S(\mathcal{A}(y)_{p_1 \dots p_{k}}, p_{k}) -\dfrac{x^b}{\tau}S(\mathcal{B}(y)_{p_1 \dots p_{k}}, p_{k}),
\end{align}
where $P_i \in \{2^{-d}x^{\beta}: 1 \leq d \leq D\}$. The condition $p_k< p_{k-1}$ is only relevant when $P_k = P_{k-1}$. If $P_k < P_{k-1}$ and $p_1 \dots p_{k-1} \leq x^{0.999}$, we decompose (\ref{todecompose}) further via the Buchstab identity. If $P_k = P_{k-1}$ and $k$ is even, we bound (\ref{todecompose}) trivially from below by discarding the contribution of the $\mathcal{A}(y)$--terms and only recording the $\mathcal{B}(y)$--terms. If $P_k = P_{k-1}$ and $k$ is odd, we instead remove the condition $p_k < p_{k-1}$ from the $\mathcal{A}(y)$--terms, making their contribution more negative, and also add and subtract the corresponding $\mathcal{B}(y)$--terms. For $k$ odd, we then  decompose (\ref{todecompose}) further, but the restriction $p_k < p_{k-1}$ is omitted.

\vspace{3mm}
 Overall, we obtain  the following lower bound on the LHS of (\ref{equ:targetsmallprimesfirststep}):
\begin{align}
&(-1)^r\sum_{\substack{p_i \sim P_i  \mbox{ \scriptsize for } i\leq r}}S(\mathcal{A}(y)_{p_1 \dots p_r}, w)-\dfrac{x^b}{\tau}S(\mathcal{B}(y)_{p_1 \dots p_r}, w) \label{equ:comp10}\\
&+ \sum_{k=r+1}^{\lfloor \log\log(x) \rfloor} \sum^{*}_{\substack{P_{r+1}, \dots, P_k }} 
(-1)^k 
\sum_{\substack{p_\ell \sim P_\ell\\ p_1 p_2 p_3 \dots p_{k-1} \leq x^{0.999}}}  S(\mathcal{A}(y)_{p_1 \dots p_{k}}, w) -\dfrac{x^b}{\tau}S(\mathcal{B}(y)_{p_1 \dots p_{k}}, w) \label{equ:comp11}\\
&
+  \sum_{k=r+3}^{\lfloor \log\log(x) \rfloor}\sum_{\substack{P_{r+1}, \dots, P_k}}^{*}
(-1)^{k}
\sum_{\substack{  p_\ell \sim P_\ell \\ p_1 p_2 p_3 \dots p_{k-2} \leq x^{0.999}\\ p_1 p_2 p_3 \dots p_{k-1} > x^{0.999}}} S(\mathcal{A}(y)_{p_1 \dots p_{k}}, p_{k}) -\dfrac{x^b}{\tau}S(\mathcal{B}(y)_{p_1 \dots p_{k}}, p_{k})\label{equ:comp12}
\\
&
-\dfrac{x^b}{\tau}  \sum_{\substack{k=r+2}}^{\lfloor \log\log(x) \rfloor}\sum_{\substack{P_{r+1}, \dots, P_{k-1}}}^{*}
\sum_{ \substack{p_\ell \sim P_\ell \\ p_k \sim P_{k-1} \\ p_1 p_2 p_3 \dots p_{k-2} \leq x^{0.999}}} S(\mathcal{B}(y)_{p_1 \dots p_{k}}, p_{k}), \label{equ:comp13}
\end{align}
where sums $\sum^{*}_{P_{r+1}, \dots, P_j}$ range over all choices of $P_i \in \{2^{-d}x^{\beta}: 1 \leq d \leq D\}$ with $P_{r+1} \geq P_{r+2} \geq \dots  \geq P_j$ and $P_i > P_{i+1}$ for odd $i$. Recall here that $P_1, \dots, P_r$ are fixed. 

\vspace{3mm}
We first bound the contribution of (\ref{equ:comp13}).
Since $\mathcal{B}(y)$ is an interval of length about $x/x^b$, we can use standard results, such as~\cite{HeathBrown:1988:NPI}, to estimate $\sum S(\mathcal{B}(y)_{p_1 \dots p_k}, p_k)$ for $p_i$ ranging over some intervals, provided we impose condition $x/(x^b p_1 \dots p_k ) \geq (x/(p_1 \dots p_k))^{7/12}$. Since $b\leq 10^{-5}$, we only require say $p_1p_2 p_3 \dots p_k\leq x^{0.9999}$. We split (\ref{equ:comp13}) into $(\ref{equ:comp13})_1$+$(\ref{equ:comp13})_2$, where $(\ref{equ:comp13})_1$ is the component of $(\ref{equ:comp13})$ with $p_1p_2 p_3 \dots p_k\leq x^{0.9999}$ and $(\ref{equ:comp13})_2$ is the component with $p_1p_2 p_3 \dots p_k> x^{0.9999}$.

\vspace{3mm}
We fix $k$ and $\mathcal{M} \subseteq \{r+1, \dots, k-1\} \cap \mathbb{N}^{\mbox{\scriptsize odd}}$, set $\widehat{m} = \#\mathcal{M}$  and denote  by $\sum^{**}_{P_1, \dots, P_{k-1}}$ the sum over $P_1, \dots, P_{k-1}$  which appear in  $\sum^{*}_{P_1, \dots, P_{k-1}}$ and  additionally satisfy $P_i = P_{i-1}$ if and only if $i \in \mathcal{M}$. Then
\begin{align} 
\sum_{\substack{P_{r+1}, \dots, P_{k-1}}}^{**}
\sum_{\substack{  p_\ell \sim P_\ell \\ p_k \sim P_{k-1} \\ p_1 p_2 p_3 \dots p_{k-2} \leq x^{0.999}\\ p_1 p_2 p_3 \dots p_{k} \leq x^{0.9999}}}  S(\mathcal{B}(y)_{p_1 \dots p_{k}}, p_{k})
&\leq \dfrac{1}{(k-1-r-2\widehat{m})!} \!\!\!\!\! \! \!\!\! \!\!\!\!
\sum_{\substack{  p_1 \dots p_{k} \leq  x^{0.9999}, \\ p_\ell \sim P_\ell  \mbox{ \footnotesize for } \ell \leq r
\\w \leq p_\ell  \leq x^{\beta} \mbox{ \footnotesize for } \ell>r,\, \ell \notin \mathcal{M}\cup\{k\} \\ p_\ell \leq  2p_{\ell-1} \mbox{ \footnotesize for } \ell \in \mathcal{M}\cup\{k\} \\  p_\ell \geq  (1/2) p_{\ell-1} \mbox{ \footnotesize for } \ell \in \mathcal{M}\cup\{k\}}}\!\!\! \!\!\! \!\!\!\!\!\! \!\!\!  S(\mathcal{B}(y)_{p_1 \dots p_{k}}, p_{k}), \label{equ:comp134.2}
\end{align} 
where we used that in $\sum^{**}$, $P_{r+1} \geq \dots \geq P_{k-1}$ with $P_i=P_{i-1}$ for only $\widehat{m}$ choices of $i$,  and  removed the double counting introduced by ignoring this condition by dividing by $(k-1-r-2\widehat{m})!$. 

\vspace{3mm}
Let $\gamma(x)=\log(2)/\log(x)$. For $i \in \mathcal{M} \cup \{k\}$ we set $\beta^+_i = \alpha_{i-1}+\gamma(x)$ and $\beta^-_i=\alpha_{i-1}-\gamma(x)$, while for $i \notin \mathcal{M} \cup \{1, \dots, r, \}\cup\{k\}$ we set $\beta_i^+ =\beta$ and $\beta_i^- = 1/\log\log(x)$. For $i \in \{1, \dots r\}$ we set $\beta_i^+=\ell_i^*+\gamma(x)$ and $\beta_i^- = \ell_i^*$.  For $p_1 \dots p_{k} \leq x^{0.9999}$, standard computations give
\begin{align*}
\mbox{RHS of }(\ref{equ:comp134.2}) &\leq \dfrac{y(1+o(1))}{x^b (k-1-r-2\widehat{m})!} \int_{\beta_1^-}^{\beta_1^+} \dots \int_{\beta_k^-}^{\beta_k^+}  \dfrac{1}{\alpha_1 \dots \alpha_{k-1} \alpha_k^2 \log(x)} \mbox{d}\alpha_k\dots\mbox{d}\alpha_1 \nonumber \\
&\leq \dfrac{C^*y}{x^b(k-1-r-2\widehat{m})!} \dfrac{( \log \log \log(x))^{k-1-r-\widehat{m}} (\log\log(x))^{\widehat{m}+2}}{\log(x)^{\widehat{m}+2+r}},
\end{align*}
where $C^*$ is a large constant. Summing over all  choices of $k  \leq \log\log(x)$ and $\mathcal{M} \subseteq \{r+1, \dots, k-1\} \cap \mathbb{N}^{\mbox{\scriptsize odd}}$, $(\ref{equ:comp13})_1$ is thus bounded below by $-y \log\log(x)^5/(\tau \log(x)^{2+r})$.

\vspace{3mm}
 Next we consider $(\ref{equ:comp13})_2$, where $p_1p_2 \dots p_k > x^{0.9999}$. Since we also got $p_1 \dots p_{k-2} \leq  x^{0.999}$, the relevant products have $p_1 \geq p_2 \geq \dots \geq p_{k-1} \geq x^{0.0001}$. This implies that any $n \in [x,6x]$ with $m$ prime factors greater or equal to $x^{0.0001}$ is counted  at most $m^{k-1}$ times by quantities $S(\mathcal{B}(y) _{p_1 \dots p_k}, p_k)$ present in $(\ref{equ:comp13})_2$. Here we must have $m \leq 10001$ and $k \leq 10001$ and hence we count any single $n \in [x,6x]$ at most $O(1)$ times. For a lower  bound on  $(\ref{equ:comp13})_2$, it thus suffices to count the elements of $\mathcal{B}(y)$ which have two prime factors $p_{k-1}$ and $p_k$ which are very close together and also have  prime factors $p_i \sim P_i$ for $i \leq r$. Using again~\cite{HeathBrown:1988:NPI}, simple computations give a lower bound of $-y \log\log(x)^3/(\tau \log(x)^{2+r})$ and overall we find that (\ref{equ:comp13}) $\geq -2y \log\log(x)^5/(\tau \log(x)^{2+r})$.
 
 \vspace{3mm}  
To treat (\ref{equ:comp12}), we fix some $k$,  $t_0=r < t_1 < \dots < t_s=k-2$  and
 $
Q_\ell \in \{2^{-d} x^{\beta}:1 \leq d \leq D\}$ for  $\ell \in \{t_0+1,t_1,t_1+1,t_2, \dots,  t_{s-1}+1,t_s,k-1,k\}$ and denote by
 \begin{align*}
\mathcal{T}(t_0, \dots, t_{s}, Q_{t_0+1},Q_{t_1}, Q_{t_1+1}, Q_{t_2},  \dots,  Q_{t_{s-1}+1}, Q_{t_s},Q_{k-1},Q_k, k) 
\end{align*}
  the set of $(P_\ell)_{\ell = r+1}^k \in \{2^{-d} x^{\beta}:1 \leq d \leq D\}^{k-r}$ with the following properties:
  \begin{enumerate}
 \item $P_\ell =Q_\ell$ whenever   $\ell=t_{i-1}+1$ or $\ell=t_i$ for $1\leq i \leq s$ or $\ell=k-1$  or $\ell=k$. 
  \item  $P_{r+1} \geq \dots \geq P_k$ and  $P_\ell > P_{\ell+1}$ for odd $\ell$.
  \item $P_{t_i+1} \dots P_{t_{i+1}-1} < x^{0.00005}$  and 
$P_{t_i+1} \dots P_{t_{i+1}} \geq x^{0.00005}$ for $0 \leq i \leq s-2$.
\item $P_{t_{s-1}+1} \dots P_{t_s} \leq x^{0.0001}$ if $t_s	-t_{s-1} \geq 2$.
  \end{enumerate}

We can partition the set of $(P_\ell)_{\ell = r+1}^k$ which are under consideration in the sum $\sum^{*}_{P_{r+1}, \dots, P_k}$ in (\ref{equ:comp12}) via such sets $\mathcal{T}(t_0, \dots, t_{s}, Q_{t_0+1},Q_{t_1}, Q_{t_1+1}, Q_{t_2},  \dots,  Q_{t_{s-1}+1}, Q_{t_s},Q_{k-1},Q_k, k) $. Here $k \leq \log\log(x)$, $s \leq 10^5$, and there are at most $\log\log(x)$ choices for each $t_i$ and $\log(x)$ choices for each $Q_i$, so that less than $(\log x)^{10^6}$ sets are necessary to complete the partition.   
 
\vspace{3mm} For a  set  $\mathcal{T}(t_0, \dots, t_{s}, Q_{t_0+1},Q_{t_1}, Q_{t_1+1}, Q_{t_2},  \dots,  Q_{t_{s-1}+1}, Q_{t_s},Q_{k-1},Q_k, k) $ of interest, note
\begin{align} 
\sum_{(P_\ell)_{\ell=r+1}^k \in \mathcal{T}}& \,
\sum_{{ \substack{ p_\ell \sim P_\ell \\ p_1\dots p_{k-2} \leq x^{0.999} \\ p_1\dots p_{k-1} > x^{0.999}}}} S(\mathcal{C}_{p_1 \dots p_{k}}, p_{k}) = \sum_{\substack{n_1\dots  n_{r+s+2} \in \mathcal{C}  \\ n_i \sim P_{i}   \mbox{ \scriptsize for } i \leq  r \\ n_{r+s+1} \sim Q_{k-1}\\ n_1 \dots n_{r+s} \leq x^{0.999} \\ n_1 \dots n_{r+s+1} > x^{0.999} }} \!\!\!\!\! a_{n_1}^{(1)} \dots a_{n_{r+s+2}}^{(r+s+2)} \label{equ:st1} \\
\mbox{ where } \quad 
&a_n^{(i)} = 1_{\mathbb{P}}(n) \quad \mbox{ for } \quad i \in \{1, \dots, r\},  \nonumber
\\&a_n^{(r+i+1)}= \sum_{\substack{P_{t_i+1}, \dots, P_{t_{i+1}} \\  P_{t_i+1} \dots P_{t_{i+1}-1} < x^{0.00005}   \\
P_{t_i+1} \dots P_{t_{i+1}} \geq x^{0.00005} \\ P_{t_{i}+1} = Q_{t_i+1},\, P_{t_{i+1}}=Q_{t_{i+1}}   }}^{*} \sum_{\substack{ p_{t_i+1} \dots p_{t_{i+1}} =n \\ p_\ell \sim P_\ell}} 1 \qquad \mbox{ for } \quad  i \in \{0, \dots, s-2\},  \nonumber\\
&a_n^{(r+s)}= \sum_{\substack{P_{t_{s-1}+1}, \dots, P_{t_s} \\  P_{t_{s-1}+1} \dots P_{t_{s}} \leq x^{0.0001} \mbox{ \footnotesize if } t_s-t_{s-1} \geq 2 \\  P_{t_{s-1}+1} = Q_{t_{s-1}+1}, \,  P_{t_{s}}=Q_{t_{s}}}}^{*} \sum_{\substack{ p_{t_{s-1}+1} \dots p_{t_s} =n \\ p_\ell \sim P_\ell}} 1,  \nonumber\\
&a_n^{(r+s+1)} = 1_{\mathbb{P}}(n),  \nonumber\\
&a_n^{(r+s+2)} =\sum_{p_k\sim Q_k} 1_{\mathbb{N}}(n/p_k) \cdot 1_{\{1\}}(\gcd(n/p_k, P(p_k))).  \nonumber
\end{align}
We first observe that $a_n^{(i)} = O(\log(n))$:   For $i \in \{1, \dots, r, r+s+1\}$, clearly $|a_n^{(i)}| = |1_{\mathbb{P}}(n)| \leq 1$. Further, $a_n^{(r+s+2)}$  counts $p_k \sim Q_k$ for which  $p_k \mid n$  is the smallest prime factor of $n$. Thus again $|a_n^{(r+s+2)}| \leq 1$.   Finally, we consider $i \in \{r+1,\dots,r+s\}$.  Recall that $P_{r+1} \geq \dots \geq P_k \geq 2^{-D}x^\beta \geq x^{1/\log \log x}$.  This means $a_n^{(i)}$ is zero unless all prime factors of $n$ are at least $x^{1/\log \log x}$. But then the number of prime factors $\Omega(n)$ of $n$ is at most $O((\log \log x) \frac{\log n}{\log x})$. Further, recall $P_\ell > P_{\ell+1}$ for odd $\ell$. This means $|a_n^{(i)}|$ is bounded by the number of ways in which $n$ can be written as $n=p'_1 \dots p'_{\Omega(n)}$ with $p'_{\ell-1} ,p'_\ell > p'_{\ell'}$ for all odd $\ell$ and all $\ell'>\ell$. Since $\Omega(n) = O((\log \log x) \frac{\log n}{\log x})$, we   get $|a_n^{(i)}| \leq  2^{(\log \log x) \log n/\log x} = O(\log(n))$. 

\vspace{3mm} 
We now split~(\ref{equ:st1}) up by restricting the summation ranges  to $n_i \sim N_i$ for $r+1 \leq i \leq r+s$  and  $n_{r+s+2} \sim N_{r+s+2} $. The properties of $\mathcal{T}$ ensure that we only need to consider $(x^{0.00005}/2) \leq N_i \leq x^\beta$ for $r+1 \leq i \leq r+s$, while $N_{r+s+2} \leq x^{0.0001}$. Additionally, $a_n^{(i)} = 1 _{\mathbb{P}}(n)$ if $N_i \geq x^{0.0001}$. Set
\begin{align*}
a_n =  \sum_{\substack{n_1\dots  n_{r+s+2}  =n  \\ n_i \sim P_{i}   \mbox{ \scriptsize for } i \leq  r \\ n_i \sim N_{i}   \mbox{ \scriptsize for } r+1 \leq i \leq  r+s \\ n_{r+s+1} \sim Q_{k-1}\\ n_1 \dots n_{r+s} \leq x^{0.999} \\ n_1 \dots n_{r+s+1} > x^{0.999} }} \!\!\!\!\! a_{n_1}^{(1)} \dots a_{n_{r+s+2}}^{(r+s+2)}.
\end{align*}  
We split (\ref{equ:comp12}) into expressions  $(-1)^k(\sum_{n \in \mathcal{A}(y)} a_n - (x^b/\tau) \sum_{n \in \mathcal{B}(y)} a_n )$ with $a_n$ as described above. Conditions  $ n_1 \dots n_{r+s} \leq x^{0.999}$ and $n_1 \dots n_{r+s+1} > x^{0.999} $ are only relevant if $P_1 \dots P_r N_{r+1} \dots N_{r+s} \in [ 2^{-r-s} x^{0.999},x^{0.999}]$ or    $P_1 \dots P_r N_{r+1} \dots N_{r+s} Q_{k-1} \in [  2^{-r-s-1}x^{0.999},x^{0.999}]$. In that case, if $k$ is even, we bound  $(-1)^k(\sum_{n \in \mathcal{A}(y)} a_n - (x^b/\tau) \sum_{n \in \mathcal{B}(y)} a_n)  $ below by discarding the contribution of the sum over $\mathcal{A}(y)$ and bounding  below the contribution of long interval $\mathcal{B}(y)$. To do so, we use standard integral computations, very similar to the ones used in the treatment of (\ref{equ:comp13}).  In particular, for given $N_{r+1}, \dots, N_{r+s-1}$, the choice of $N_{r+s}$ is severely restricted by requirements $P_1 \dots P_r N_{r+1} \dots N_{r+s} \in [ 2^{-r-s} x^{0.999},x^{0.999}]$ or    $P_1 \dots P_r N_{r+1} \dots N_{r+s} Q_{k-1} \in [  2^{-r-s-1}x^{0.999},x^{0.999}]$  (effectively giving us an extra factor of $1/\log(x)$). We get a total lower bound of $-y \log\log(x)^5/(\tau \log(x)^{2+r})$ for the contribution of problematic  terms with even $k$. If $k$ is odd, we instead remove conditions $ n_1 \dots n_{r+s} \leq x^{0.999}$ and $n_1 \dots n_{r+s+1} > x^{0.999} $ from the $\mathcal{A}(y)$ term, making it smaller, and from the $\mathcal{B}(y)$ term, making it larger. What we have added to the $\mathcal{B}(y)$ term is again easy to bound from  above, as interval  $\mathcal{B}(y)$ is long. In total, the   error introduced by removal of these conditions for even $k$ is  bounded below by  $-y \log\log(x)^5/(\tau \log(x)^{2+r})$. 

\vspace{3mm} Hence we may exclusively work with $(a_n)$ for which the conditions $ n_1 \dots n_{r+s} \leq x^{0.999}$ and $n_1 \dots n_{r+s+1} > x^{0.999} $ have been removed or are not relevant. But then   $(a_n)$ is of the form given in Corollary~\ref{cor:combining}. The corresponding $\{\ell_i\}_{i=1}^j \in \Xi$ are all also contained in $\Xi^*(\ell_1^*, \dots, \ell_r^*, \beta)$ and so satisfy one of the options (1), (2) or (3). Hence  Corollary~\ref{cor:combining} applies.  We apply it once for each relevant choice of $\mathcal{T}$ and $N_i$ - no more than $\log(x)^{10^7}$ applications are needed in total.   We find that there is some $\mathcal{I}_1 \subseteq  [x,3x]\cap \mathbb{N}$ with  $\# \mathcal{I}_1 \ll (\log x)^{10^7}\tau x^{\nu+\varepsilon/2}$  such that  $ y \in  ( [x,3x]\cap \mathbb{N} ) \setminus \mathcal{I}_1$ implies that  (\ref{equ:comp12}) is at least $-y \log\log(x)^6/(\tau \log(x)^{2+r})$.

\vspace{3mm}
Next we apply Lemma~\ref{lem:fundsing} to (\ref{equ:comp10}), obtaining upper and lower  bounds $E_1$ and $E_2$, where 
\begin{align} \label{equ:smoothsum}
E_s=\sum_{\substack{ dp_1\dots p_rn \in \mathcal{A}(y) \\ p_i \sim P_i \mbox{ \footnotesize for } i \leq r \\ d\mid P(w)}} \xi_d^{(s)}  - \dfrac{x^b}{\tau}\sum_{\substack{ dp_1\dots p_r n \in \mathcal{B}(y)\\ p_i \sim P_i \mbox{ \footnotesize for } i \leq r \\ d \mid P(w)}} \xi_d^{(s)}+
O\!\left(\dfrac{x}{\tau \log (x)^{100}}\right),
\end{align}
and where $\xi_d^{(s)} =0$ for $d \geq x^{0.0001}$ and $s\in \{1,2\}$.  Decomposing into sums 
\begin{align} \label{equ:longsmoothcomponent}
\sum_{\substack{ n_0 \dots n_r d \in \mathcal{A}(y) \\ d\mid P(w) \\ n_0 \sim N, \,n_i \sim P_i, \,  d \sim D}} 1_{\mathbb{N}}(n_0) 1_\mathbb{P}(n_1) \dots 1_{\mathbb{P}}(n_r)\xi_d^{(s)} - \dfrac{x^b}{\tau}\sum_{\substack{ n_0 \dots n_r d \in \mathcal{B}(y) \\ d\mid P(w) \\ n_0 \sim N, \,n_i \sim P_i, \,  d \sim D}} 1_{\mathbb{N}}(n_0) 1_\mathbb{P}(n_1) \dots 1_{\mathbb{P}}(n_r)\xi_d^{(s)} 
\end{align}
with $x/2^{r+2} \leq NP_1 \dots P_rD \leq 6x$ and $D \leq x^{0.0001}$, so that $N \geq x^{0.9999-\ell_1^*-\dots-\ell_r^*}$, we see that Corollary~\ref{cor:combining} is applicable with option (i) or (ii).   The corresponding $\{\ell_i\}_{i=1}^j \in \Xi$ are again contained in $\Xi^*(\ell_1^*, \dots, \ell_r^*, \beta)$. Hence the assumptions of Lemma~\ref{lem:smallprimessmalltaufirststep} together with Corollary~\ref{cor:combining}  imply that there is some $\mathcal{I}_2 \subseteq  [x,3x]\cap \mathbb{N}$ with $\# \mathcal{I}_2 \ll \tau x^{\nu+\varepsilon/2}$  such that  $ y \in  ( [x,3x]\cap \mathbb{N} ) \setminus \mathcal{I}_2$ implies that  quantity (\ref{equ:comp10}) is $O(x/(\tau \log(x)^{100}))$.

\vspace{3mm}
To treat (\ref{equ:comp11}), we simply combine the techniques  used in the treatment of (\ref{equ:comp10})  and (\ref{equ:comp12}). To be a little more precise, we decompose (\ref{equ:comp11}) using a partition very similar to the previously described $\mathcal{T}$, splitting the expression into $\leq (\log x)^{10^6}$ parts. To each of the resulting parts we then apply Lemma~\ref{lem:fundsing} and the resulting upper and lower bounds  can be described by expressions similar to~(\ref{equ:smoothsum}), except now  $k$ rather than  $r$ prime factors appear in the sum and these are combined using $\mathcal{T}$ like in~(\ref{equ:st1}). Corollary~\ref{cor:combining}  is applicable to the resulting expressions and  gives us  some $\mathcal{I}_3 \subseteq  [x,3x]\cap \mathbb{N}$ with $\# \mathcal{I}_3 \ll (\log x)^{10^7}\tau x^{\nu+\varepsilon/2}$  such that  $ y \in  ( [x,3x]\cap \mathbb{N} ) \setminus \mathcal{I}_3$ implies that  quantity (\ref{equ:comp11}) is at least $-y \log\log(x)^6/(\tau \log(x)^{2+r})$.

\vspace{3mm} Combining our bounds on (\ref{equ:comp10}), (\ref{equ:comp11}), (\ref{equ:comp12}) and  (\ref{equ:comp13}), we now have that  $\mathcal{I}^* =\mathcal{I}_1 \cup \mathcal{I}_2 \cup \mathcal{I}_3$ satisfies $\#\mathcal{I}^* \ll \tau x^{\nu+3\varepsilon/4} $ and whenever $ y \in  ( [x,3x]\cap \mathbb{N} ) \setminus \mathcal{I}^*$, 
\begin{align*}
&(-1)^r \left(\sum_{p_i \sim P_i} S(\mathcal{A}(y)_{p_1 \dots p_r}, x^{\beta}) - 
\dfrac{x^b}{\tau} \sum_{p_i \sim P_i}S(\mathcal{B}(y)_{p_1 \dots p_r}, x^{\beta}) \right) 
\geq - \left( \dfrac{ \log\log(x)^7 y}{\tau \log(x)^{2+r}} \right).
\end{align*}
This assumes that $x$ is sufficiently large, but the required size of $x$ depends only on $\varepsilon$.

\vspace{3mm}
The corresponding upper bound 
\begin{align*}
& (-1)^r \left( \sum_{p_i \sim P_i} S(\mathcal{A}(y)_{p_1 \dots p_r}, x^\beta) - 
\dfrac{x^b}{\tau} \sum_{p_i \sim P_i}S(\mathcal{B}(y)_{p_1 \dots p_r}, x^\beta) \right) 
\leq  \left( \dfrac{\log\log(x)^{7} y}{\tau \log(x)^{2+r}} \right)
\end{align*}
can be derived in almost exactly the same way, we only need to  swap the words “odd” and “even” or  “lower” and “upper” in various places. We omit its proof here.
\end{proof}

\subsection{Products of larger primes}

 We  now prove the second half of Lemma~\ref{lemma1}.

\begin{lemma} \label{lem:acoupleofprimes}
Consider $\tau = x^a$ with $x^{19/40-\varepsilon} \leq \tau \leq x^{0.77-\varepsilon}$.

Let $r \in \{2,4,6\}$. Let $P_i=x^{\ell_i^*}$ with $0.01-\varepsilon \leq \ell_i^*\leq 0.5+\varepsilon$,  $P_i \geq P_{i+1}$ and $P_1 \dots P_r \leq x^{0.99}$.

Let $ \Xi^{\star\star}(\ell_1^*, \dots, \ell_r^*)$  be as described in Definition~{\rm\ref{def:combinatorialconditions}}. Suppose that every $\{\ell_i\}_{i=1}^j \in \Xi^{\star\star}(\ell_1^*, \dots, \ell_r^*)$ satisfies one of the options {\rm (1)}, {\rm (2)} or {\rm (3)}. 

\vspace{3mm}
Then there  exists  $\mathcal{I} \subseteq [x,3x]\cap \mathbb{N}$ such that $\#\mathcal{I} = O( \tau x^{0.23+3\varepsilon/4})$ and 
 $y \in ([x,3x]\cap\mathbb{N} )\setminus\mathcal{I} $ implies
\begin{align}
& \left|\sum_{p_i \sim P_i} S(\mathcal{A}(y)_{p_1 \dots p_r}, p_r) - 
\dfrac{x^b}{\tau} \sum_{p_i \sim P_i}S(\mathcal{B}(y)_{p_1 \dots p_r}, p_r) \right| 
\leq  \left( \dfrac{\log\log(x)^{O(1)} y}{\tau \log(x)^{2+r}} \right). 
\end{align}
Here the constant implied by big O notation  depends  only on $\varepsilon$.
\end{lemma}

\begin{proof}
Again we first derive the lower bound 
\begin{align}\label{equ:fewprimes}
& \sum_{p_i \sim P_i} S(\mathcal{A}(y)_{p_1 \dots p_r}, p_r) - 
\dfrac{x^b}{\tau} \sum_{p_i \sim P_i}S(\mathcal{B}(y)_{p_1 \dots p_r}, p_r)  
\geq - \left( \dfrac{\log\log(x)^{O(1)} y}{\tau \log(x)^{2+r}} \right). 
\end{align}
We split the LHS of (\ref{equ:fewprimes}) into sums
\begin{align}\label{equ:fewprimes2}
&\dfrac{1}{k!}\Bigg(\sum_{\substack{n_i \sim P_i \\ m_i \sim M_i \\ m_i > n_r \\ \prod n_i \prod m_i \in \mathcal{A}(y)}} \prod_{i=1}^r 1_\mathbb{P}(n_i)  \prod_{i=1}^k 1_\mathbb{P}(m_i) - 
\dfrac{x^b}{\tau} \sum_{\substack{n_i \sim P_i \\ m_i \sim M_i \\ m_i > n_r \\  \prod n_i \prod m_i \in \mathcal{B}(y)}}\prod_{i=1}^r 1_\mathbb{P}(n_i)  \prod_{i=1}^k 1_\mathbb{P}(m_i)\Bigg),
\end{align}
where we consider $1 \leq k \leq  (1-\sum_{i=1}^r \ell_i^*)/\ell_r^*$ and  $M_i \in \{2^d P_r: 0 \leq d \leq D\}$ for $D= \lceil \log(x/P_r)/\log(2)\rceil$. The factor $\frac{1}{k!}$ is included to account for double counting, which arises from rearrangements of $M_i$. 

\vspace{3mm}
Double counting is still an issue when $M_i =M_j$ for some $i \neq j$. However, the contribution of such terms to the LHS of  (\ref{equ:fewprimes}) is small:  We discard the $\mathcal{A}$-term in (\ref{equ:fewprimes2}) and bound below the $\mathcal{B}$-term, noting that it only involves $n \in \mathcal{B}(y)$ which have two prime factors $q_1,q_2$ with $q_1, q_2 \sim M_i$ as well as prime factors $p_\ell \sim P_\ell$ for $1 \leq \ell \leq r$. Since $\mathcal{B}(y)$ is an interval of length $y/x^b$ and we multiply by $x^b/\tau$, standard computations then give a bound of $O(y/(\tau \log(x)^{3+r}))$. Summing over the $O(\log(x))$ choices for $M_i$ then gives a total bound of $O(y/(\tau \log(x)^{2+r}))$. Hence we can ignore the contribution of case $M_i = M_j$. 

\vspace{3mm} Similarly, the condition $m_i > n_r$ is only relevant whenever $M_i=P_r$ for some $i$ and this case also only contributes $O(y/(\tau \log(x)^{r+2}))$. For $M_i \neq P_r$,  (\ref{equ:fewprimes2}) is of the form described in  Corollary~\ref{cor:combining} with option~(i) and its corresponding  $\Xi$ is contained in $\Xi^{\star \star}(\ell_1^*, \dots, \ell_r^{*})$.  Hence our assumptions ensure that  Corollary~\ref{cor:combining} is applicable and we obtain that there exists some  $\mathcal{I}^* \cap [x,3x]\subseteq \mathbb{N}$ with $\#\mathcal{I}^* = O( \tau x^{\nu+3\varepsilon/4})$ such that every $y \in ([x,3x]\cap\mathbb{N} )\setminus\mathcal{I}^* $ satisfies (\ref{equ:fewprimes}).

\vspace{3mm}
The corresponding upper bound $(\log\log(x)^{O(1)}y)/(\tau \log(x)^{2+r})$ is derived in almost exactly the same way and so we omit its proof.
\end{proof}

\section{Harman's sieve} \label{sec:harman}

The goal of this section is to prove  Proposition~\ref{proposition3}. We construct a minorant of the prime indicator function with various nice properties, linked to a comparison of sums over sets $\mathcal{A}(y)$ and $\mathcal{B}(y)$. 
 This approach is based on Harman's sieve~\cite{Harman:2007:PDS}. Relatedly, Harman's sieve was also used by Baker, Harman and Pintz~\cite{Baker:2001:DCP} to improve an upper bound on the lengths of prime gaps. It was also used by  Matomäki~\cite{Matomaki:2007:LD} and later Islam~\cite{Islam:2015:Thesis} to improve a bound on the number of very long prime gaps. Simultaneously to our work, these results were improved by Järviniemi~\cite{Jarviniemi:2022:LDP}, who showed that $\sum_{p_n \leq x, \, p_{n+1}-p_n \geq x^{1/2}} (p_{n+1}-p_n) \ll_\varepsilon x^{0.57+\varepsilon}$.
 
\subsection{Preparation} \label{ssec:preparation}

Before embarking on the proof of Proposition~\ref{proposition3}, we introduce a few more definitions. 
Write $X=(7x)^{1/2}$.  Throughout Section~\ref{sec:harman} we say that  $\Psi: [x,6x] \cap \mathbb{N} \rightarrow \mathbb{R}$ is good$^\star$ if it can be written as 
\begin{align*}
\Psi(n) = (-1)^r \sum_{\substack{n=p_1 \dots p_r m  \\ x^\beta \leq p_r < \dots < p_1 <X\\ \mathcal{G}}} \psi(m, x^\beta)
\end{align*}
for some $\beta \in [0.01,0.15]$, $r \in \{0,\dots,5\}$ and a set of conditions $\mathcal{G}$, with $\#\mathcal{G} \leq 100$ and each $g \in \mathcal{G}$ of the form $\sum_{i \in I_g} k_{(g,i)} \log_x(p_i)  \in [a_g,b_g]$ for some $I_g \subseteq \{1, \dots, r\}$, $k_{(g,i)} \in \mathbb{N}$ and $a_g, b_g \in [0,1]$. Additionally, we require that     condition ($0.5 +\varepsilon \geq  \ell_1^* \geq \dots \geq \ell_r^*  \geq \beta$ and  $\sum_{i \in I_g} k_{(g,i)} \ell_i^* \in [a_g- \frac{\varepsilon_1}{2}, b_g +\frac{\varepsilon_1}{2}]$ for all $g \in \mathcal{G}$) implies $(\ell_1^*, \dots, \ell_r^*,\beta) \in \mathcal{R}^\star(a)$. (We also allow $[a_g,b_g]$ to be replaced by $(a_g,b_g]$, $[a_g,b_g)$ or $(a_g,b_g)$.)
 
 \vspace{3mm}
We say a function $\Psi: [x,6x] \cap \mathbb{N} \rightarrow \mathbb{R}$ is good$^{\star\star}$ if it can be written as 
\begin{align*}
\Psi(n) = \sum_{\substack{n=p_1 \dots p_r m  \\ p_r < \dots < p_1 \\ \mathcal{G}}} \psi(m, p_r)
\end{align*} for some $r \in \{2,4,6\}$ and  a set of conditions $\mathcal{G}$, with $\#\mathcal{G} \leq 100$ and  $g \in \mathcal{G}$ of the form $\sum_{i \in I_g}  k_{(g,i)} \log_x(p_i)  \in [a_g,b_g]$ for some $I_g \subseteq \{1, \dots, r\}$, $k_{(g,i)} \in \mathbb{N}$ and $a_g, b_g \in [0,1]$. Additionally, we require  that  ($0.5 +\varepsilon \geq  \ell_1^* \geq \dots \geq \ell_r^*  \geq 0.01-\varepsilon$ and  $\sum_{i \in I_g} k_{(g,i)}\ell_i^* \in [a_g- \frac{\varepsilon_1}{2}, b_g +\frac{\varepsilon_1}{2}]$ for all $g \in \mathcal{G}$) implies $(\ell_1^*, \dots, \ell_r^*) \in \mathcal{R}^{\star\star}(a)$.

\vspace{3mm}
Here $\psi(n,z)$, $\mathcal{R}^\star(a)$ and $\mathcal{R}^{\star\star}(a)$ are as given in Section~\ref{ssec:notation} and  Definition~\ref{def:Rstar}, respectively. In particular, for any $n \in \mathbb{N}$ and $z > 0$, 
\begin{align*}
\psi(n,z) = \begin{cases}
&1 \quad  \mbox{ if } \, p \mid n \Rightarrow p \geq z\\
&0 \quad \mbox{ otherwise. }
\end{cases}
\end{align*}
Specific subsets of $\mathcal{R}^\star(a)$ and $\mathcal{R}^{\star\star}(a)$ are computed later on, in Section~\ref{ssec:suitable}. 

\vspace{3mm} Suppose we can show that there exist   $\Psi_1, \dots \Psi_{s_1}: [x,6x] \cap \mathbb{N} \rightarrow \mathbb{R}$ and  $\Theta_1, \dots, \Theta_{s_2}: [x,6x] \cap \mathbb{N} \rightarrow [0,\infty)$ such that $s_1, s_2 \leq 100$, such that each $\Psi_i$  is  good$^\star$ or good$^{\star\star}$, and
 such that \begin{align} \label{decompositiontoaimfor}
1_{\mathbb{P}}(n) = \sum_{i=1}^{s_1} \Psi_i(n) + \sum_{i=1}^{s_2} \Theta_i(n) \quad \quad \mbox{ and } \quad \quad \sum_{i=1}^{s_2}\sum_{n \in \mathcal{B}(y)} \Theta_i(n) \leq \dfrac{0.9999y}{x^b\log(x)}.
\end{align}
We then split each $\Psi_i(n)$ up, restricting the range of $p_j$ to $p_j \sim P_j$. 
If $\Psi_i(n)$ is good$^\star$, there are some $\beta \in [0.01,0.15]$, some $r \in \{0, \dots, 5\}$ and some  $\mathcal{G}$ such that $\Psi_i(n)$ can be written  as a sum of $O(\log(x)^r)$ functions $\rho_{i,t}: [x,6x] \cap \mathbb{N} \rightarrow \mathbb{R}$ of the form 
\begin{align*}
    &\rho_{i,t}(n) = (-1)^r \sum_{\substack{n = p_1 \dots p_r m,\, p_j \sim  x^{\ell_j^*} \\ x^\beta \leq p_r < \dots < p_1 < X \\ \mathcal{G}}} \psi(m,x^\beta).
    \end{align*}
Here $\ell_1^*, \dots, \ell_r^*$ depend on $t$ and satisfy $x^{\ell_j^*} = P_j \in \{2^{d} x^{\beta}: 0 \leq d \leq \log_2(\frac{X}{x^\beta})\}$. (For $d=0$ we may take $p_j \in [P_j, 2P_j]$ rather than $p_j \in (P_j, 2P_j]$ to ensure that the case $p_j =x^\beta$ is covered. Any of our results concerning  comparisons of $\sum_{p_j}\mathcal{S}(\mathcal{A}(y)_{p_1 \dots p_r},z)$ and $\sum_{p_j}\mathcal{S}(\mathcal{B}(y)_{p_1 \dots p_r},z)$ are still valid with this change.)

\vspace{3mm}  If $\sum_{j \in I_g} k_{(g,i)} [\ell_j^*, \ell_j^* + \log_x(2)] \cap [a_g,b_g] = \emptyset$ for some $g \in \mathcal{G}$ or $\ell_j^*< \ell_{j+1}^*$ for some $j \in \{1, \dots, r\}$, then $\rho_{i,t}(n) =0 $ for all $n$ and we can discard $\rho_{i,t}$. Denote by $U_{i}$ the set of $t$ for which $\rho_{i,t}$ has $\sum_{j \in I_g} k_{(g,i)}[\ell_j^*, \ell_j^* + \log_x(2)] \subseteq [a_g,b_g] $  for all $g \in \mathcal{G}$ and  for which $\ell_1^* + \log_x(2) < \log_x(X)$ and $\ell_1^*> \dots > \ell_r^*$. If $r$ is even, denote by $W_i$ the set of $t$  for which $\rho_{i,t}$ has ($\sum_{j \in I_g}k_{(g,i)}[\ell_j^*, \ell_j^* + \log_x(2)] \not\subseteq  [a_g,b_g]$ for some $g \in \mathcal{G}$ or $\ell_j^* = \ell_{j+1}^*$ for some $j \in \{1, \dots, r\}$ or $\ell_1^* + \log_x(2) \geq \log_x(X)$) and for which $\ell_1^* \geq \dots \geq \ell_r^*$. If $r$ is odd, instead denote this set by $V_i$. If $r$ is even, write $V_i = \emptyset$. If $r$ is odd, write $W_i = \emptyset$.  For $t \in V_i$, set 
    \begin{align*}
    &\widetilde{\rho}_{i,t}(n) = (-1) \sum_{\substack{n = p_1 \dots p_r m \\ p_j \sim  x^{\ell_j^*} }} \psi(m,x^\beta) \qquad \mbox{ and } \qquad \widehat{\rho}_{i,t}(n) = \rho_{i,t}(n)-\widetilde{\rho}_{i,t}(n).
    \end{align*}
    Observe that $\widehat{\rho}_{i,t}(n)$ counts the number of ways in which  $n$ can be written as $n = p_1 \dots p_r m$ with $p_j \sim x^{\ell_j^*}$ and ($p_j \leq p_{j+1}$ for some $j$, $g$ not satisfied for some $g \in \mathcal{G}$ or $p_1 \geq X$). In particular,  $\widehat{\rho}_{i,t}(n)$ is non-negative. 
    
    \vspace{3mm }
    If $\Psi_i(n)$ is good$^{\star\star}$, we proceed similarly. There are some $r \in \{2, 4, 6\}$ and some conditions $\mathcal{G}$ such that $\Psi_i(n)$ can be written  as a sum of $O(\log(x)^r)$ functions $\rho_{i,t}: [x,6x] \cap \mathbb{N} \rightarrow \mathbb{R}$ of the form  
\begin{align*}
    &\rho_{i,t}(n) =  \sum_{\substack{n = p_1 \dots p_r m,\, p_j \sim  x^{\ell_j^*} \\ p_r < \dots < p_1 \\ \mathcal{G}}} \psi(m,p_r).
    \end{align*}
     We denote by $U_{i}$ the set of $t$ for which $\rho_{i,t}$ has $\sum_{j \in I_g} k_{(g,i)}[\ell_j^*, \ell_j^* + \log_x(2)] \subseteq [a_g,b_g] $  for all $g \in \mathcal{G}$ and  for which $\ell_1^*> \dots > \ell_r^*$. We denote by $W_i$ the set of $t$  for which $\rho_{i,t}$ has ($\sum_{j \in I_g}k_{(g,i)}[\ell_j^*, \ell_j^* + \log_x(2)] \not\subseteq [a_g,b_g]$ for some $g \in \mathcal{G}$ or $\ell_j^* = \ell_{j+1}^*$ for some $j \in \{1, \dots, r\}$) and for which $\sum_{j \in I_g}k_{(g,i)}[\ell_j^*, \ell_j^* + \log_x(2)] \cap [a_g,b_g] \neq \emptyset$ and $\ell_1^* \geq \dots \geq \ell_r^*$. We also write $V_i = \emptyset$. 
     
     \vspace{3mm} We have  ensured that every $\rho_{i,t}(n)$ with $t \in U_i$ and every $\widetilde{\rho}_{i,t} (n)$ with $t \in V_i$ can be written as \begin{align*}
    &(-1)^r \sum_{\substack{n = p_1 \dots p_r m \\ p_j \sim x^{\ell_j^*}}} \psi(m,x^\beta) \quad \mbox{ for some } r \in \{0, \dots, 5\} \mbox{ and } (\ell_1^*, \dots, \ell_r^*,\beta) \in \mathcal{R}^{*}(a), \\  & \mbox{or } \quad   \sum_{\substack{n = p_1 \dots p_r m \\ p_i \sim x^{\ell_i^*}}} \psi(m,p_r)   \quad \mbox{ for some } r \in \{2,4,6\} \mbox{ and } (\ell_1^*, \dots, \ell_r^*) \in \mathcal{R}^{**}(a).
    \end{align*}
Hence every $\rho_{i,t}$ with $t \in U_i$ and every $\widetilde{\rho}_{i,t} $ with $t \in V_i$ is of the form given in property ($\beta$) of Proposition~\ref{proposition3}.    Furthermore, the prime indicator function $1_{\mathbb{P}}(n)$ has been decomposed as 
\begin{align*}
1_{\mathbb{P}}(n) = \sum_{i=1}^{s_1} \sum_{t \in U_i} \rho_{i,t} (n)  +  \sum_{i=1}^{s_1} \sum_{t \in V_i} \widetilde{\rho}_{i,t}(n)   +  \sum_{i=1}^{s_1} \sum_{t \in V_i} \widehat{\rho}_{i,t}(n)  + \sum_{i=1}^{s_1} \sum_{t \in W_i} \rho_{i,t}(n)   +\sum_{i=1}^{s_2} \Theta_i(n).
\end{align*}     
But every $\widehat{\rho}_{i,t}$ with $t \in V_i$, every $\rho_{i,t}$ with $t \in W_i$   and every $\Theta_i$ is non-negative and so the function $\rho(n) =  \sum_{i=1}^{s_1} \sum_{t \in U_i} \rho_{i,t}(n)   +  \sum_{i=1}^{s_{1}} \sum_{t \in V_i} \widetilde{\rho}_{i,t}(n)  $ is a minorant of $1_{\mathbb{P}}(n)$  (property ($\alpha$)) and decomposes into functions which satisfy property ($\beta$) of Proposition~\ref{proposition3}. Futhermore, we have simple estimates 
\begin{align*}
\sum_{n \in \mathcal{B}(y)} \rho_{i,t}(n) \leq   \sum_{n \in \mathcal{B}(y)} \sum_{\substack{n = p_1 \dots p_r m \\ p_j \sim  x^{\ell_j^*}}} \psi(m,z) = \sum_{\substack{ p_j \sim  P_j}} S(\mathcal{B}(y)_{p_1 \dots p_r}, z) = O\left( \dfrac{y}{x^b \log(x)^{r+1}} \right),
\end{align*}
where $P_j = x^{\ell_j^*}$ and $z \in \{x^\beta, p_r\}$. Similarly, $\sum_{n \in \mathcal{B}(y)} \widehat{\rho}_{i,t}(n) = O(y/(x^b \log(x)^{r+1}))$. But since $t \in V_i \cup W_i$ requires $\sum_{j \in I_g}[\ell_j^*, \ell_j^* + \log_x(2)] \cap [a_g,b_g] \neq \emptyset$ and $\sum_{j \in I_g}[\ell_j^*, \ell_j^* + \log_x(2)] \not\subseteq [a_g,b_g]$ for some $g \in \mathcal{G}$ or $\ell_j^*= \ell_{j+1}^*$ for some $j \in \{1, \dots, r\}$, we have $\# (\bigcup_{i=1}^r V_i \cup \bigcup_{i=1}^r  W_i ) = O(\log(x)^{r-1})$.  Summing up, 
\begin{align*}
\sum_{n \in \mathcal{B}(y)}( 1_{\mathbb{P}}(n) - \rho(n) ) =  \sum_{i=1}^{s_1} \sum_{t \in V_i} \widehat{\rho}_{i,t}(n)  + \sum_{i=1}^{s_1} \sum_{t \in W_i} \rho_{i,t}(n)   +\sum_{i=1}^{s_2} \Theta_i(n) \leq \dfrac{0.99999y}{x^b \log(x)}.
\end{align*}
     This is property $(\gamma)$. Hence to prove Proposition~\ref{proposition3}, it suffices to show that we can write $1_{\mathbb{P}}(n)$ as a sum $1_{\mathbb{P}}(n) = \sum_{i=1}^{s_1} \Psi_i(n) + \sum_{i=1}^{s_2} \Theta_i(n)$ with $s_1, s_2 \leq 100$ and each $\Psi_i: [x,6x] \cap \mathbb{N} \rightarrow \mathbb{R}$ good$^\star$ or good$^{\star \star}$ and $\Theta_i: [x,6x] \cap \mathbb{N} \rightarrow [0,\infty)$  such that  $\sum_{i=1}^s\sum_{n \in \mathcal{B}(y)} \Theta_i(n) \leq \frac{0.9999y}{x^b\log(x)}$, just like described in (\ref{decompositiontoaimfor}).
   
\vspace{3mm}
(Note: On a few  occasions, we will also run into functions of the form
\begin{align} \label{goodstaralt}
\Psi(n) =  -\sum_{\substack{n=p_1 \dots p_r m  \\ x^\beta \leq p_r < \dots < p_1 <X\\ \mathcal{G}}} \psi(m, p_r),
\end{align}
where  $\mathcal{G}$ is a set of conditions (with $g \in \mathcal{G}$ of the form $\sum_{i \in I_g} k_{(g,i)} \log_x(p_i)  \in [a_g,b_g])$  such that   the inequalities ($0.5 +\varepsilon \geq  \ell_1^* \geq \dots \geq \ell_r^*  \geq 0.01-\varepsilon$ and  $\sum_{i \in I_g} k_{(g,i)} \ell_i^* \in [a_g- \frac{\varepsilon_1}{2}, b_g +\frac{\varepsilon_1}{2}]$ for all $g \in \mathcal{G}$) imply $(\ell_1^*, \dots, \ell_r^*,\beta) \in \mathcal{R}^\star(a)$ for $\beta =\ell_r^*$. 

\vspace{3mm}
Restricting each $p_j$ to $p_j \sim x^{\ell_j^*}$, the function $\Psi(n)$ can be  decomposed into a sum $\sum_t \rho_t(n)+ {\rho}^*_t(n)$ with 
\begin{align*}
\rho_t(n)+ {\rho}^*_t(n)= &\sum_{\substack{n=p_1 \dots p_r m,\, p_j \sim x^{\ell_j^*}  \\ x^\beta \leq p_r < \dots < p_1 <X\\ \mathcal{G}}} (-1)\psi(m, x^{\ell_r^*})  +  \sum_{\substack{n=p_1 \dots p_r m,\, p_j \sim x^{\ell_j^*}  \\ x^\beta \leq p_r < \dots < p_1 <X\\ \mathcal{G}}} (\psi(m, x^{\ell_r^*}) - \psi(m, p_r)),
\end{align*}
where $\rho_{t}(n)$ denotes the first big sum and  ${\rho}^*_t(n)$ denotes the second big sum on the right and where $\ell_1^*, \dots, \ell_r^*$ depend on $t$. Here ${\rho}^*_t(n)$ counts $n$ which are divisible by some $p_j \sim x^{\ell_j^*}$ and have one additional factor in $[x^{\ell_r^*},2x^{\ell_r^*}]$, so that $ \sum_{n \in \mathcal{B}(y)} {\rho}^*_t (n) = O(y/(x^b \log(x)^{r+2}))$. Since there are only $O(\log(x)^r)$ choices for $\ell_1^*, \dots, \ell_r^*$, this error can  easily be discarded. On the other hand, just like in the treatment of good$^\star$ functions, the requirements placed on $\mathcal{G}$ ensure that $\sum_t\rho_t(n)$ can be written as a sum of functions with property $(\beta)$ and an error term $\Theta(n)$ with  $ \sum_{n \in \mathcal{B}(y)} \Theta (n) = O(y/(x^b \log(x)^{2}))$. In short, the above discussion of decompositions of $1_{\mathbb{P}}(n)$ also applies to functions of the form described in (\ref{goodstaralt}). These will only be relevant on rare occasions and for simplicity's sake we also call them good$^\star$.)

\subsection{Suitable ranges}\label{ssec:suitable}

Let $\Xi^\star(\ell_1^*, \dots, \ell_r^*, \beta)$ and $\Xi^{\star\star}(\ell_1^*, \dots, \ell_r^*)$ be as given in Definition~\ref{def:combinatorialconditions}. Recall that $\{\ell_i\}_{i=1}^j \in \Xi^\star(\ell_1^*, \dots, \ell_r^*, \beta) \cup \Xi^{\star\star}(\ell_1^*, \dots, \ell_r^*)$  implies that $\ell_1, \dots, \ell_j \in [0,1]$, $\sum_{i=1}^j \ell_i =1$ and $j \leq 10^{20}$ and that there exist disjoint subsets $X_1, \dots, X_r$ of $\{1, \dots, j\}$ with
 \begin{align*}
 \sum_{i \in X_s} \ell_i \in \left[\ell_s^*-\dfrac{\varepsilon }{10^{100}}, \, \ell_s^*+\dfrac{\varepsilon }{10^{100}}\right] \quad \mbox{ for } \quad s\leq r.
 \end{align*}
If  $\{\ell_i\}_{i=1}^j \in \Xi^\star(\ell_1^*, \dots, \ell_r^*, \beta)$,  additionally $\ell_i \leq \beta+10^{-100}\varepsilon$ for all but at most one $i \in \{1, \dots, j \} \setminus \bigcup_{i=1}^r X_{ i} $. 
 
\vspace{3mm}
Let $\chi_0(a)$, $\chi_1(a)$, $\chi_2(a)$ and $\chi_3(a)$ be as described in {\rm (\ref{chi0})}, {\rm (\ref{chi1})}, 
{\rm (\ref{chi2})} and {\rm (\ref{chi3})}, respectively. Recall that  we say that $\{\ell_i\}_{i=1}^j \in \Xi^\star(\ell_1^*, \dots, \ell_r^*, \beta) \cup \Xi^{\star\star}(\ell_1^*, \dots, \ell_r^*)$ satisfies one of the options $(1)$, $(2)$ or $(3)$  for a given $a$ if one of the following three conditions holds:
  
\begin{enumerate} [{\rm (1)}]
\item There exists $k \in \{1, \dots, j\}$  with $\ell_k \geq \chi_0(a)$. 
\item There exists $I_1 \subseteq \{1,\dots,j\}$ with  $\sum_{i \in I_1} \ell_i  \in \chi_1(a) $.
\item There exist $I_2 \subseteq \{1,\dots,j\}$ with  $\sum_{i \in I_2} \ell_i  \in\chi_2(a) $ and 
$I_3 \subseteq \{1,\dots,j\}$ with  $\sum_{i \in I_3} \ell_i  \in  \chi_3(a)   $. 
\end{enumerate}

Then recall that  $\mathcal{R}^*(a)$ denotes the set of tuples $(\ell_1^*, \dots, \ell_r^*, \beta)$ which have $\beta \in [0.01,0.15]$, $r \in \{0,\dots, 5\}$, $\ell_1^*, \dots,  \ell_r^* \in [\beta, 0.5+\varepsilon]$, $\ell_i^* \geq \ell_{i+1}^*$  and $\ell_1^*+ \dots + \ell_r^* \leq 0.75$ and for which every $\{\ell_i\}_{i=1}^j \in \Xi^\star(\ell_1^*, \dots, \ell_r^*, \beta) $ satisfies one of the options $(1)$, $(2)$ or $(3)$. 
Further, $\mathcal{R}^{**}(a)$ denotes the set of tuples $(\ell_1^*, \dots, \ell_r^*)$ which have  $r \in \{2,4,6\}$, $\ell_1^*, \dots,  \ell_r^* \in [0.01-\varepsilon, 0.5+\varepsilon]$, $\ell_i^* \geq \ell_{i+1}^*$  and $\ell_1^*+ \dots + \ell_r^* \leq 0.99$ and for which every $\{\ell_i\}_{i=1}^j \in \Xi^{\star\star}(\ell_1^*, \dots, \ell_r^*)$ satisfies one of the options $(1)$, $(2)$ or $(3)$.

\vspace{3mm} To find many good$^\star$ and good$^{\star\star}$ functions, we now wish to compute elements of $\mathcal{R}^{*}(a)$ and $\mathcal{R}^{**}(a)$. 
  We begin by listing a few simple criteria which ensure that $ \{\ell_i\}_{i=1}^j$ has some $I \subseteq \{1,\dots,j\}$ such that $\sum_{i \in I} \ell_i$ is contained in a fixed interval $[a_1-\frac{\varepsilon_1}{2},a_2+\frac{\varepsilon_1}{2}]$.

\begin{lemma} \label{lem:combinations}
 
Let $r \in \{0,\dots, 5\}$ and suppose $\ell_1^*, \dots, \ell_r^*$ satisfy  $0.01 \leq \ell_s^*\leq 0.5+\varepsilon_1$ and  $\ell_s^* \geq \ell^*_{s+1}$.

Let $\Xi^\star(\ell_1^*, \dots, \ell_r^*, \beta)$  and $\Xi^{\star \star}(\ell_1^*, \dots, \ell_r^*)$ be  as given in Definition~{\rm \ref{def:combinatorialconditions}.}

\vspace{3mm} Let $\chi = [a_1-\frac{\varepsilon_1}{2}, a_2+\frac{\varepsilon_1}{2}] \subseteq [0,1]$. Let $\rho>0$. Let $b_1 \in [0,1]$ and $b_2 \in [0,1]$.

 \vspace{3mm}
Assume that one of the following conditions holds:
\begin{enumerate} [{\rm (A)}]
\item $\sum_{s=1}^r \ell_s^* \leq a_2$ and $1-\rho \geq a_1$.
\item For some $k \in \{2, \dots, r\}$, $\sum_{s=1}^{k-1} \ell_s^* \leq a_2$  and $\ell_k^* \leq b_1$ and $1-\rho - (r-k+1) b_1 \geq a_1$.
\item $\sum_{s=2}^r \ell_s^* \leq a_2$ and $\ell_1^* \leq b_2$ and $1-\rho - b_2 \geq a_1$.
\item $a_1 \geq 0.5+\varepsilon_1$,  $\sum_{s=1}^r \ell_s^* \leq 1-\rho - (2a_1-1)$ and  $(\sum_{s=1}^k \ell_s^* \in [1-a_2-(2a_1-1),a_2]$ for some $k)$. 
\item There exists some $S \subseteq \{1, \dots, r\}$ with $\sum_{s \in S} \ell_s^* \in [a_1,a_2]$.
\end{enumerate}
For conditions {\rm(A)}, {\rm(B)}, {\rm(C)} and {\rm(D)}, additionally assume that $a_2-a_1 \geq \beta$. $($This assumption is not required when using condition {\rm(E)}.$)$

\vspace{3mm}
Then  if $\{\ell_i\}_{i=1}^j \in \Xi^\star(\ell_1^*, \dots, \ell_r^*, \beta)$ and $\ell_i \leq \rho$ for every $i \in \{1, \dots, j\}$, there exists $I_\chi \subseteq \{1, \dots, j\}$ with 
\begin{align} \label{equ:aiming}
\sum_{i \in I_\chi} \ell_i \in \chi.
\end{align}
Alternatively, assume some $S \subseteq \{1, \dots, r\}$ satisfies $\sum_{s \in S} \ell_s^* \in [a_1,a_2]$ and suppose that $\{\ell_i\}_{i=1}^j \in \Xi^{\star \star}(\ell_1^*, \dots, \ell_r^*)$. Then there exists some $I_\chi \subseteq \{1, \dots, j\}$ with  $\sum_{i \in I_\chi} \ell_i \in \chi$.
\end{lemma}

\begin{proof}
 Let $\{\ell_i\}_{i=1}^j \in \Xi^\star(\ell_1^*, \dots, \ell_r^*, \beta)$ and let $X_1, \dots, X_r$  be disjoint subsets of $\{1,\dots,j\}$ with 
 \begin{align*}
 \sum_{i \in X_s} \ell_i \in \left[\ell_s^*-\dfrac{\varepsilon }{10^{100}}, \, \ell_s^*+\dfrac{\varepsilon }{10^{100}}\right]
 \end{align*} 
 for $s \leq r$ and $\ell_i \leq \beta + 10^{-100}\varepsilon$ for all but at most one $i \in \{1, \dots, j \} \setminus \bigcup_{i=1}^r X_{ i} $. Assume that $\ell_i \leq \rho$ for every $i \in \{1, \dots, j\}$. Then there exists $Y_1 \subseteq \{1, \dots, j \} \setminus \bigcup_{i=1}^r X_{ i} $ with 
 \begin{align} \label{equ:startingpoint}
 \sum_{s=1}^r \sum_{i \in X_s} \ell_i + \sum_{i \in Y_1} \ell_i \geq 1-\rho
 \end{align}
and $\ell_i \leq \beta+10^{-100}\varepsilon$ for every $i \in Y_1$.  With condition (A) we have $\sum_{s=1}^r \ell_s^* \leq a_2$, while $\sum_{s=1}^r \sum_{i \in X_s} \ell_i + \sum_{i \in Y_1} \ell_i \geq 1- \rho \geq a_1$. Recall also that $a_2 -a_1 \geq \beta$. Thus, starting with $\sum_{s=1}^r \sum_{i \in X_s} \ell_i \leq a_2 +\varepsilon_1/2$ and adding  elements of $Y_1$, we must eventually hit interval $[a_1-\varepsilon_1/2, a_2+\varepsilon_1/2]$. There exists $Y_2 \subseteq Y_1$ with 
   \begin{align*}
 \sum_{s=1}^r \sum_{i \in X_s} \ell_i + \sum_{i \in Y_2} \ell_i  \in \left[a_1-\frac{\varepsilon_1}{2}, a_2+\frac{\varepsilon_1}{2} \right]. 
 \end{align*} 
 Taking $I_\chi = \bigcup_{s=1}^r X_s \cup Y_2$, (\ref{equ:aiming}) is satisfied for condition (A).
 
 \vspace{3mm} With condition (B) we have $\ell_k^* \leq b_1$ and since $\ell_s^* \geq \ell_{s+1}^*$, (\ref{equ:startingpoint}) gives 
  \begin{align*} 
 \sum_{s=1}^{k-1} \sum_{i \in X_s} \ell_i + \sum_{i \in Y_1} \ell_i \geq 1-\rho - (r-k+1) b_1 - \varepsilon_1/2 \geq a_1 -\varepsilon_1/2.
 \end{align*}
 On the other hand, we also have $ \sum_{s=1}^{k-1} \sum_{i \in X_s} \ell_i \leq \sum_{s=1}^{k-1} \ell_s^*+ \varepsilon_1/2 \leq a_2+ \varepsilon_1/2$. Since $\ell_i \leq \beta +10^{-100}\varepsilon \leq a_2-a_1+10^{-100}\varepsilon$ for $i \in Y_1$, there  exists some $Y_2 \subseteq Y_1$ such that $I_\chi = \bigcup_{s=1}^{k-1} X_s \cup Y_2$ satisfies (\ref{equ:aiming}).
 
 \vspace{3mm} With condition (C) we have $\ell_1^* \leq b_2$, so that (\ref{equ:startingpoint}) gives 
  \begin{align*} 
 \sum_{s=2}^{r} \sum_{i \in X_s} \ell_i + \sum_{i \in Y_1} \ell_i \geq 1-\rho - b_2 - \varepsilon_1/2 \geq a_1 -\varepsilon_1/2,
 \end{align*}
 while $\sum_{s=2}^r \ell_s^* \leq a_2$, and hence there exists some $Y_2 \subseteq Y_1$ such that $I_\chi = \bigcup_{s=2}^{r} X_s \cup Y_2$ satisfies (\ref{equ:aiming}).
 
 \vspace{3mm} With condition (D) we have  $\sum_{s=1}^{k} \ell_s^* \in [1-a_2-(2a_1-1), a_2]$. If $\sum_{s=1}^{k} \ell_s^* \in [a_1, a_2]$, we simply take $I_\chi = \bigcup_{s=1}^{k} X_s$ to satisfy (\ref{equ:aiming}). If $\sum_{s=1}^{k} \ell_s^* \in [1-a_2, 1-a_1]$, we recall that $\{\ell_i\}_{i=1}^j \in \Xi^\star(\ell_1^*, \dots, \ell_r^*, \beta)$  implies $\sum_{i=1}^j \ell_i =1$ and take $I_\chi = \{1, \dots, j\} \setminus (\bigcup_{s=1}^{k} X_s)$ to satisfy (\ref{equ:aiming}). This leaves the cases $\sum_{s=1}^{k} \ell_s^* \in [1-a_1, a_1]$ and $\sum_{s=1}^{k} \ell_s^* \in [1-a_2-(2a_1-1), 1-a_2]$. But condition (D) and (\ref{equ:startingpoint}) give 
  \begin{align*} 
  \sum_{i \in Y_1} \ell_i \geq 1-\rho -  \sum_{s=1}^r \sum_{i \in X_s} \ell_i \geq 2a_1-1 -\varepsilon_1/20.
 \end{align*}
So if $\sum_{s=1}^{k} \ell_s^* \in [1-a_1, a_1]$, we have $ \sum_{s=1}^{k} \sum_{i \in X_s} \ell_i + \sum_{i \in Y_1} \ell_i \geq (1-a_1)+(2a_1-1)-\varepsilon_1/10=a_1-\varepsilon_1/10$ and if $\sum_{s=1}^{k} \ell_s^* \in [1-a_2-(2a_1-1), 1-a_2]$, then $ \sum_{s=1}^{k} \sum_{i \in X_s} \ell_i + \sum_{i \in Y_1} \ell_i \geq 1-a_2-\varepsilon_1/10$.
 Since also $\ell_i-10^{-100}\varepsilon \leq \beta \leq a_2-a_1 = (1-a_1)-(1-a_2)$ for $i \in Y_1$, there exists $Y_2 \subseteq Y_1$ with 
 \begin{align*} 
 \sum_{s=1}^{k} \sum_{i \in X_s} \ell_i + \sum_{i \in Y_2} \ell_i \in [1-a_2-\varepsilon_1/10,1-a_1] \cup [a_1-\varepsilon_1/10,a_2]
 \end{align*} 
 and one of $I_\chi = \bigcup_{s=1}^k X_s \cup Y_2$ or $I_\chi = \{1, \dots, j\} \setminus (\bigcup_{s=1}^k X_s \cup Y_2)$ satisfies (\ref{equ:aiming}).
 
 \vspace{3mm}
 On the other hand, if  $\{\ell_i\}_{i=1}^j \in \Xi^\star(\ell_1^*, \dots, \ell_r^*, \beta)$ or  $\{\ell_i\}_{i=1}^j \in \Xi^{\star\star}(\ell_1^*, \dots, \ell_r^*)$ and there exists some $S\subseteq \{1, \dots, r\}$ with $\sum_{s \in S} \ell_s^* \in [a_1, a_2]$, the definitions of $\Xi^\star$ and $\Xi^{\star\star}$ tell us that there exist some disjoint subsets $X_1, \dots, X_r$ of $\{1,\dots,j\}$ with 
 \begin{align*}
\sum_{i \in X_s} \ell_i \in \left[\ell_s^*-\dfrac{\varepsilon }{10^{100}}, \, \ell_s^*+\dfrac{\varepsilon }{10^{100}}\right]
 \end{align*}
 and (\ref{equ:aiming}) is satisfied by taking $I_\chi = \bigcup_{s \in S} X_s$. This covers the final case of the proof, the only case for which assumption $\beta \leq a_2-a_1$ is not needed.
\end{proof}

\begin{cor} \label{cor:computations} 
Let $\varepsilon>0$ and $a \in [0.475-\varepsilon, 0.77-\varepsilon]$. 
Let $\beta \in [0.01,0.15]$ and  $r \in \{0,\dots, 5\}$ and suppose $\ell_1^*, \dots, \ell_r^*$ satisfy  $\beta \leq \ell_i^*\leq 0.5+\varepsilon$,  $\ell_i^* \geq \ell^*_{i+1}$ and $\ell_1^*+ \dots +\ell_r^* \leq 0.75$.

 \vspace{3mm}
Assume that one of the following conditions {\rm (I.i)  --  (VI.vi) }holds: 
\begin{enumerate} [{\rm (I) }]
\item  $a \leq 0.53$ and $\beta = 0.07$.
\begin{enumerate} [{\rm (i) }]
\item $(r =0$ or $r=1)$.
\item $(r=2$ or $r=3)$  and $\ell_2^* \leq (0.71-\ell_1^*)/2 + \frac{\varepsilon_1}{4}$.
\item $(r=4$ or $r=5)$ and  $\ell_4^* \leq (0.71-\ell_1^*-\ell_2^*-\ell_3^*)/2 + \frac{\varepsilon_1}{4}$. 
\end{enumerate}
\item $a \in (0.53,0.545]$.
\begin{enumerate} [{\rm (i)}]
\item $(r =0$ or $r=1)$ and $\beta = 0.08$.
\item $(r=2$ or $r=3)$, $\ell_1^* \geq  0.474 + \frac{\varepsilon_1}{2}$, $\ell_2^* \leq  \min\{0.595-\ell_1^*+\frac{\varepsilon_1}{2},(0.715-\ell_1^*)/2+\frac{\varepsilon_1}{4}\}$, $\beta \leq 0.09$.
\item $(r=2$ or $r=3)$ and $\ell_1^* \leq 0.427 +\frac{\varepsilon_1}{2}$,  $\ell_2^* \leq (0.655-\ell_1^*)/2 + \frac{\varepsilon_1}{4}$, $\beta = 0.08$.
\item $r \in \{4, 5\}$, $\ell_1^*\leq 0.285+\frac{\varepsilon_1}{2}$,  $\ell_2^* \leq (0.655-\ell_1^*)/2 + \frac{\varepsilon_1}{4}$, $\ell_4^* \leq (0.655-\sum_{i=1}^3\ell_{ i}^*)/2+ \frac{\varepsilon_1}{4}$, $\beta = 0.08$. 
\end{enumerate} 
\item $a \in ( 0.545, 0.57]$ and $\beta = 0.075$.
\begin{enumerate} [{\rm (i)}]
\item $(r =0$ or $r=1)$.
\item $(r=2$ or $r=3)$ and  $\ell_2^* \in [0.475-\ell_1^*-\frac{\varepsilon_1}{2}, \min\{0.525-\ell_1^*, 0.14\}+\frac{\varepsilon_1}{2}]$.
\item $(r=2$ or $r=3)$ and $\ell_2^* \leq ( 0.6-\ell_1^*)/2+\frac{\varepsilon_1}{4}$.
\item $(r=4$ or $r=5)$ and  $\ell_2^* \leq  0.4-\ell_1^*+\frac{\varepsilon_1}{2}$ and $\ell_4^* \leq (0.615-\sum_{i=1}^3\ell_{ i}^*)/2+\frac{\varepsilon_1}{4}$. 
\end{enumerate}
\item $a \in (0.57,0.59]$.
\begin{enumerate} [{\rm (i)}]
\item $(r =0$ or $r=1)$ and $\beta = 0.075$.
\item $(r=2$ or $r=3)$ and $\ell_1^* \geq 0.455 -\frac{\varepsilon_1}{2}$, $\ell_2^* \leq \min\{0.58-\ell_1^*+\frac{\varepsilon_1}{2},(0.685-\ell_1^*)/2+\frac{\varepsilon_1}{4}\}$, $\beta =0.075$.
\item $(r=2$ or $r=3)$ and $\ell_1^* \in  [0.42-\frac{\varepsilon_1}{2},0.455+\frac{\varepsilon_1}{2}]$, $\ell_2^* \leq (0.685-\ell_1^*)/2 +\frac{\varepsilon_1}{4}$, $\beta \leq 0.105$.
\item $r \in \{2,3\}$, $\ell_1^* \in [0.315-\frac{\varepsilon_1}{2},0.38+\frac{\varepsilon_1}{2}]$,  $\ell_2^* \leq \max\{0.145 +\frac{\varepsilon_1}{2},(0.62-\ell_1^*)/2+\frac{\varepsilon_1}{4}\} $, $\beta = 0.075$.
\item $(r=2$ or $r=3)$ and $\ell_1^* \in [0.29-\frac{\varepsilon_1}{2},0.315+\frac{\varepsilon_1}{2}]$,  $\ell_2^*  \leq (0.62-\ell_1^*)/2+\frac{\varepsilon_1}{4}$, $\beta = 0.075$.
\item $(r=2$ or $r=3)$ and $\ell_1^* \leq 0.29+\frac{\varepsilon_1}{2}$,  $\ell_2^* \leq \min\{0.2275+\frac{\varepsilon_1}{2}, (0.685-\ell_1^*)/2+\frac{\varepsilon_1}{4}\}]$, $\beta = 0.075$.
\item $(r=2$ or $r=3)$ and $\ell_1^* \leq  0.29+\frac{\varepsilon_1}{2}$,  $\ell_2^*  \in [0.42 -\ell_1^*-\frac{\varepsilon_1}{2},0.455-\ell_1^*+\frac{\varepsilon_1}{2}]$,  $\beta \leq 0.105$.
\item $(r=4$ or $r=5)$ and $\ell_1^* \leq 0.29+\frac{\varepsilon_1}{2}$,  $\ell_4^*\leq (0.62-\sum_{i=1}^3\ell_{ i}^*)/2+\frac{\varepsilon_1}{4}$, $\beta = 0.075$.  
\end{enumerate}
\item $a \in (0.59,0.61]$.
\begin{enumerate} [{\rm (i)}]
\item $(r =0$ or $r=1)$ and $\beta = 0.07$.
\item $(r=2$ or $r=3)$ and $\ell_1^* \geq 0.435-\frac{\varepsilon_1}{2}$, $\ell_2^* \leq \min\{0.105+\frac{\varepsilon_1}{2},(0.67-\ell_1^*)/2+\frac{\varepsilon_1}{4}\}$, $\beta =0.07$.
\item $(r=2$ or $r=3)$ and $\ell_1^* \in  [0.42-\frac{\varepsilon_1}{2},0.435+\frac{\varepsilon_1}{2}]$, $\ell_2^* \leq (0.67-\ell_1^*)/2+\frac{\varepsilon_1}{4}$, $\beta \leq 0.09$.
\item $(r=2$ or $r=3)$ and $\ell_1^* \in [0.33-\frac{\varepsilon_1}{2},0.365+\frac{\varepsilon_1}{2}]$,  $\ell_2^* \leq  0.1524+\frac{\varepsilon_1}{2}$, $\beta = 0.07$.
\item $(r=2$ or $r=3)$ and $\ell_1^* \in [0.305-\frac{\varepsilon_1}{2},0.33+\frac{\varepsilon_1}{2}]$,  $\ell_2^* \leq (0.635-\ell_1^*)/2 + \frac{\varepsilon_1}{4}$, $\beta = 0.07$.
\item $(r=2$ or $r=3)$ and $\ell_1^* \leq 0.305+\frac{\varepsilon_1}{2}$,  $\ell_2^* \leq  0.2099+\frac{\varepsilon_1}{2}$, $\beta = 0.07$.
\item $(r=4$ or $r=5)$ and $\ell_1^* \leq 0.305+\frac{\varepsilon_1}{2}$,  $\ell_4^* \leq (0.635-\sum_{i=1}^3\ell_{ i}^*)/2+\frac{\varepsilon_1}{4}$, $\beta = 0.07$.  
\end{enumerate}
\item $a > 0.61$ and $\beta = 0.065$.
\begin{enumerate} [{\rm (i)}]
\item $(r =0$ or $r=1)$.
\item $(r=2$ or $r=3)$ and $\ell_1^* \geq 0.42-\frac{\varepsilon_1}{2}$, $\ell_2^* \leq \min\{0.58-\ell_1^*, 0.1\}+\frac{\varepsilon_1}{2}$.
\item $(r=2$ or $r=3)$ and $\ell_1^* \in  [0.325-\frac{\varepsilon_1}{2},0.355+\frac{\varepsilon_1}{2}]$, $\ell_2^* \leq (0.645-\ell_1^*)/2 +\frac{\varepsilon_1}{4}$.
\item $(r=2$ or $r=3)$ and $\ell_1^* \leq 0.325+\frac{\varepsilon_1}{2}$, $\ell_2^*  \leq  0.2099+ \frac{\varepsilon_1}{2}$.
\item $(r=4$ or $r=5)$ and $\ell_1^* \leq 0.325+\frac{\varepsilon_1}{2}$,  $\ell_4^* \leq (0.645-\sum_{i=1}^3\ell_i^*)/2+\frac{\varepsilon_1}{4}$. 
\item $(r=4$ or $r=5)$ and $\ell_1^* \leq 0.325 + \frac{\varepsilon_1}{2}$,  $\ell_4^* \leq  (0.42-\ell_2^*-\ell_3^*)/2 + \frac{\varepsilon_1}{4}$. 
\end{enumerate}
\end{enumerate}
Then $(\ell_1^*, \dots, \ell_r^*, \beta)$ is contained in $\mathcal{R}^\star(a)$.

 \vspace{3mm}
Alternatively, let  $r \in \{2,4,6\}$, $0.01-\varepsilon \leq \ell_i^*\leq 0.5+\varepsilon$,  $\ell_i^* \geq \ell^\star_{i+1}$ and $\ell_1^* +  \dots + \ell_r^* \leq 0.99$. Assume that one of the following conditions {\rm (A.a)  --  (F.b) } holds:
\begin{enumerate}[{\rm (A)}] 
\item $a \leq 0.53$ and $\ell_i^* \geq 0.07-\frac{\varepsilon_1}{2}$ for all $i$. 
\begin{enumerate}
[{\rm (a)}]
\item $r=2$ and $\ell_1^* \in [0.29-\frac{\varepsilon_1}{2},0.36+\frac{\varepsilon_1}{2}]$.
\item $r=2$ and $\ell_2^* \in [0.64-\ell_1^*-\frac{\varepsilon_1}{2}, 0.71-\ell_1^*+\frac{\varepsilon_1}{2}]$.
\item $r=4$ and   $\ell_2^* \leq (0.71-\ell_1^*)/2 + \frac{\varepsilon_1}{4}$ and $\ell_3^* \geq 0.64-\ell_1^*-\ell_2^*-\frac{\varepsilon_1}{2}$.
\item $r=4$ and $\ell_1^* \in [0.22-\frac{\varepsilon_1}{2},0.29+\frac{\varepsilon_1}{2}]$ and $\ell_4^* \leq 0.36-\ell_1^*+\frac{\varepsilon_1}{2}$.
\item $r=6$ and $\ell_i^*+ \ell_j^* \in [0.29-\frac{\varepsilon_1}{2},0.36+\frac{\varepsilon_1}{2}]$ for some $i \neq j$. 
\item $r=6$ and $\ell_1^*+\dots +\ell_5^* \leq 0.71+\frac{\varepsilon_1}{2}$, $\ell_1^* \geq 0.18-\frac{\varepsilon_1}{2}$ and $\ell_3^* \in [0.145-\frac{\varepsilon_1}{2},0.29+\frac{\varepsilon_1}{2}]$.
\end{enumerate}
\item $a \in (0.53,0.545]$  and $\ell_i^* \geq 0.08-\frac{\varepsilon_1}{2}$ for all $i$. 
\begin{enumerate}[{\rm (a)}]
\item $r=2$ and $\ell_1^* \in [0.315-\frac{\varepsilon_1}{2},0.345+\frac{\varepsilon_1}{2}]\cup [0.427-\frac{\varepsilon_1}{2},0.474+\frac{\varepsilon_1}{2}]$.
\item $r=2$ and $\ell_1^*+\ell_2^* \in [0.427-\frac{\varepsilon_1}{2}, 0.474+\frac{\varepsilon_1}{2}] \cup [0.526-\frac{\varepsilon_1}{2}, 0.573+\frac{\varepsilon_1}{2}] \cup [0.655-\frac{\varepsilon_1}{2}, 0.685+\frac{\varepsilon_1}{2}]$.
\item $r=2$ and $\ell_1^* \in [0.285-\frac{\varepsilon_1}{2},0.375+\frac{\varepsilon_1}{2}]$, $\ell_1^*+\ell_2^* \in [0.405-\frac{\varepsilon_1}{2}, 0.485+\frac{\varepsilon_1}{2}] \cup [0.515-\frac{\varepsilon_1}{2}, 0.595+\frac{\varepsilon_1}{2}]$.
\end{enumerate}
\item $a \in (0.545,0.57]$  and $\ell_i^* \geq 0.075-\frac{\varepsilon_1}{2}$ for all $i$. 
 \begin{enumerate}[{\rm (a)}]
\item $r=2$ and $\ell_1^* \in [0.4-\frac{\varepsilon_1}{2},0.475+\frac{\varepsilon_1}{2}]$.
\item $r=2$ and  $\ell_2^* \in [0.4-\ell_1^*-\frac{\varepsilon_1}{2}, 0.475-\ell_1^*+\frac{\varepsilon_1}{2}] \cup [0.525-\ell_1^*-\frac{\varepsilon_1}{2}, 0.6-\ell_1^*+\frac{\varepsilon_1}{2}] $.
\item $r=4$ and  $\ell_4^* \in [0.4-\ell_1^*-\ell_2^*-\frac{\varepsilon_1}{2}, 0.475-\ell_1^*-\ell_2^*+\frac{\varepsilon_1}{2}]$.
\end{enumerate}
\item 
$a \in (0.57,0.59]$  and $\ell_i^* \geq 0.075-\frac{\varepsilon_1}{2}$ for all $i$. 
\begin{enumerate}[{\rm (a)}]
\item $r=2$ and $\ell_1^* \in [0.38-\frac{\varepsilon_1}{2},0.42+\frac{\varepsilon_1}{2}]$.
\item $r=2$ and $\ell_1^* \in [0.42-\frac{\varepsilon_1}{2},0.455+\frac{\varepsilon_1}{2}]$ and $\ell_2^* \in [0.58-\ell_1^*-\frac{\varepsilon_1}{2}, 0.685-\ell_1^*+\frac{\varepsilon_1}{2}]$.
\item $r \in \{2,4\}$, $\ell_1^* \in [0.315-\frac{\varepsilon_1}{2},0.38+\frac{\varepsilon_1}{2}]$, $\ell_1^*+\ell_i^* \in [0.38-\frac{\varepsilon_1}{2}, 0.455+\frac{\varepsilon_1}{2}] \cup  [0.545-\frac{\varepsilon_1}{2}, 0.62+\frac{\varepsilon_1}{2}]$ for some $i \neq 1$.
\item $r \in \{2,4\}$,  $\ell_i^* \in [0.38-\ell_1^*-\frac{\varepsilon_1}{2}, 0.42-\ell_1^*+\frac{\varepsilon_1}{2}] \cup [0.58-\ell_1^*-\frac{\varepsilon_1}{2}, 0.62-\ell_1^*+\frac{\varepsilon_1}{2}]$ for some $i \neq 1$.
\item $r=4$ and $\ell_2^* \in [0.42-\ell_1^*-\frac{\varepsilon_1}{2}, 0.455-\ell_1^*+\frac{\varepsilon_1}{2}]$, $\ell_4^* \in [0.315-\ell_1^*-\frac{\varepsilon_1}{2}, 0.42-\ell_1^*+\frac{\varepsilon_1}{2}]$.
\item $r=4$ and $\ell_1^*\leq 0.29+\frac{\varepsilon_1}{2}$,  $\ell_2^* \in [0.315-\ell_1^*-\frac{\varepsilon_1}{2}, 0.38-\ell_1^*+\frac{\varepsilon_1}{2}]$, $\ell_1^*+\ell_2^*+\ell_4^* \in [0.38-\frac{\varepsilon_1}{2}, 0.455+\frac{\varepsilon_1}{2}]$.
\item $r=4$ and $\ell_1^* \leq 0.29+\frac{\varepsilon_1}{2}$, $\ell_2^* \in [0.315-\ell_1^*-\frac{\varepsilon_1}{2}, 0.38-\ell_1^*+\frac{\varepsilon_1}{2}]$, $\ell_2^*+\ell_3^*+\ell_4^* \in [0.38-\frac{\varepsilon_1}{2}, 0.455+\frac{\varepsilon_1}{2}]$.
\end{enumerate}
\item $a \in (0.59,0.61]$  and $\ell_i^* \geq 0.07-\frac{\varepsilon_1}{2}$ for all $i$. 
\begin{enumerate}[{\rm (a)}]
\item $r \in \{2,4\}$ and $\exists$ $S \subseteq \{1,\dots,r\}$ with $\sum_{i\in S}\ell_i^* \in [0.365-\frac{\varepsilon_1}{2},0.42+\frac{\varepsilon_1}{2}] \cup [0.58-\frac{\varepsilon_1}{2},0.635+\frac{\varepsilon_1}{2}]$. 
\item $r =2$ and $\ell_1^* \in [0.42-\frac{\varepsilon_1}{2},0.435+\frac{\varepsilon_1}{2}]$ and $\ell_2^* \in [0.58-\ell_1^*-\frac{\varepsilon_1}{2},0.67-\ell_1^*+\frac{\varepsilon_1}{2}]$.
\item $r =2$ and $\ell_1^* \in [0.33-\frac{\varepsilon_1}{2},0.365+\frac{\varepsilon_1}{2}]$, $\ell_1^*+\ell_2^* \in [0.365-\frac{\varepsilon_1}{2},0.435+\frac{\varepsilon_1}{2}] \cup [0.565-\frac{\varepsilon_1}{2},0.635+\frac{\varepsilon_1}{2}]$.
\end{enumerate}
\item $a > 0.61$  and $\ell_i^* \geq 0.065-\frac{\varepsilon_1}{2}$ for all $i$. 
\begin{enumerate}[{\rm (a)}]
\item $r=2$ and $\ell_1^* \in [0.355-\frac{\varepsilon_1}{2},0.42+\frac{\varepsilon_1}{2}]$.
\item $(r=2$ or $r=4)$ and  $\ell_1^*+ \ell_i^* \in [0.355-\frac{\varepsilon_1}{2}, 0.42+\frac{\varepsilon_1}{2}] \cup [0.58-\frac{\varepsilon_1}{2}, 0.645+\frac{\varepsilon_1}{2}] $ for some $i \neq 1$.
\end{enumerate}
\end{enumerate}
Then $(\ell_1^*, \dots, \ell_r^*)$ is contained in $\mathcal{R}^{\star \star}(a)$. 
\end{cor}

\begin{proof}
Let $\{\ell_i\}_{i=1}^j \in \Xi^\star(\ell_1^*, \dots, \ell_r^*, \beta)$.  If there exists $k \in \{1, \dots, j \} $ with $\ell_k \geq \chi_0(a)$, then $\{\ell_i\}_{i=1}^j$ satisfies option (1). So we may assume that $\ell_i < \chi_0(a)$ for each  $i$ and take $\rho = \chi_0(a)$ in Lemma~\ref{lem:combinations}.  

\vspace{3mm}
Case (I): Suppose we have option (I), so that $a \leq 0.53$ and $\beta =0.07$. Then recall  $\chi_0(a) = 0.29-\varepsilon_1$ and $\chi_1(a) = [0.29-\varepsilon_1, 0.36+\varepsilon_1] \cup [0.64-\varepsilon_1, 0.71+\varepsilon_1]$.  
  Note that $\sum_{s=1}^r \ell_s^* \leq 0.71+\frac{\varepsilon_1}{2}$ in each of the conditions (i), (ii) and (iii). We also have $1-(0.29-\varepsilon_1) =0.71+\varepsilon_1 \geq 0.64$ and $(0.71+ \frac{\varepsilon_1}{2})-(0.64+ \frac{\varepsilon_1}{2}) = 0.07 + \varepsilon_1$. Hence we may apply Lemma~\ref{lem:combinations} with $\beta = 0.07$, $a_1 = 0.64- \frac{\varepsilon_1}{2}$ and $a_2 = 0.71+ \frac{\varepsilon_1}{2}$ (so that $\chi = [0.64-\varepsilon_1, 0.71+\varepsilon_1]$) and $\rho = 0.29- \varepsilon_1$. Condition (A) is satisfied and there is some $I_2 \subseteq \{1, \dots, j\}$ with $\sum_{i \in I_2} \ell_i \in [0.64-\varepsilon_1, 0.71+\varepsilon_1]$. But then  $\{\ell_i\}_{i=1}^j$ satisfies option (2).

  \vspace{3mm}
Case (II): If (II) holds,  then $a \in (0.53, 0.545]$,  $\chi_0(a) = 0.315-\varepsilon_1$, 
$
\chi_1(a) = [0.315-\varepsilon_1, 0.345+\varepsilon_1] \cup [0.427-\varepsilon_1, 0.474+\varepsilon_1] \cup [0.526-\varepsilon_1, 0.573+\varepsilon_1]  \cup [0.655-\varepsilon_1, 0.685+\varepsilon_1]$, 
$\chi_2(a) = [0.405-\varepsilon_1, 0.485+\varepsilon_1] \cup [0.515-\varepsilon_1, 0.595+\varepsilon_1]$ and
$\chi_3(a) = [0.285-\varepsilon_1, 0.375+\varepsilon_1] \cup [0.625-\varepsilon_1, 0.715+\varepsilon_1]$.

\vspace{3mm}
With each of the conditions (i), (ii), (iii) and (iv) we have $\sum_{s=1}^r \ell_s^* \leq 0.715 +\frac{\varepsilon_1}{2}$. Also, $1-0.315 = 0.685 \geq 0.625$, while $0.715-0.625=0.09$. Hence we may apply Lemma~\ref{lem:combinations} with $\beta \leq 0.09$, $\chi = [0.625-\varepsilon_1, 0.715+\varepsilon_1]$ and $\rho = 0.315-\varepsilon_1$. Condition (A) is satisfied and there is some $I_3 \subseteq \{1, \dots, j\}$ with $\sum_{i \in I_3} \ell_i \in \chi_3(a)$. 
 
\vspace{3mm} 
 On the other  hand, with condition (i) we also have  $\sum_{s=1}^r \ell_s^* \leq 0.595$, while $1-0.315 = 0.685 \geq 0.515$ and $0.595-0.515=0.08$. We may apply Lemma~\ref{lem:combinations} with $\beta = 0.08$, $\chi = [0.515-\varepsilon_1, 0.595+\varepsilon_1]$ and $\rho = 0.315$. Condition (A) is satisfied and there is some $I_2 \subseteq \{1, \dots, j\}$ with $\sum_{i \in I_2} \ell_i \in \chi_2(a)$.  With condition (ii) we instead have $\ell_1^* + \ell_2^* \in [0.554-\frac{\varepsilon_1}{2}, 0.595+\frac{\varepsilon_1}{2}]$ and by condition (E) of Lemma~\ref{lem:combinations} there  exists $I_2 \subseteq \{1, \dots, j\}$ with $\sum_{i \in I_2} \ell_i \in [0.515-\varepsilon_1, 0.595+\varepsilon_1] \subseteq \chi_2(a)$. (Note that $\beta \leq 0.09$ is allowed in this case.) Finally, with condition (iii) or (iv) we have $\sum_{s=1}^r \ell_s^* \leq 0.655+\frac{\varepsilon_1}{2} = 1-0.315-(2 \cdot 0.515-1)+\frac{\varepsilon_1}{2}$, while $\ell_1^* +\ell_2^* \leq \ell_1^*+(0.655-\ell_1^*)/2 +\frac{\varepsilon_1}{2}\leq 0.595+\frac{\varepsilon_1}{2}$.  If $\ell_1^*+\ell_2^* < 0.375$, we must have $\ell_i^* < 0.19$ for $i \geq 2$ and so either $\ell_1^* + \dots +\ell_r^* \leq 0.375$ or there exists $k$ with $\ell_1^* + \dots + \ell_{k}^* \in [0.405-0.03, 0.595]$. In the first case we use condition (A) and in the second case we use condition (D) of Lemma~\ref{lem:combinations}, taking $\beta = 0.08$, $\chi = [0.515-\varepsilon_1,0.595+\varepsilon_1]$ and $\rho = 0.315-\varepsilon_1$, to deduce that there  exists $I_2 \subseteq \{1, \dots, j\}$ with $\sum_{i \in I_2} \ell_i \in \chi_2(a)$. But then  $\{\ell_i\}_{i=1}^j$ satisfies option (3).

\vspace{3mm}
Case (III): 
If condition (III) holds, then $\beta =0.075$, $a \in (0.545, 0.57]$,  $\chi_0(a) = 0.335-\varepsilon_1$ and $
\chi_1(a) =  [0.4-\varepsilon_1, 0.475+\varepsilon_1] \cup [0.525-\varepsilon_1, 0.6+\varepsilon_1]$. With conditions (i) and (iii) we have $\ell_1^* + \dots + \ell_r^* \leq 0.6 +\frac{\varepsilon_1}{2}$. We apply Lemma~\ref{lem:combinations} with $\beta = 0.075$, $\chi = [0.525-\varepsilon_1,0.6+\varepsilon_1] \subseteq \chi_1(a)$, $\rho = 0.335-\varepsilon_1$ and condition (A) to deduce that option (2) of  Proposition~\ref{proposition3} is satisfied. With condition (ii)  we  have $\ell_1^*+\ell_2^* \leq 0.525+\frac{\varepsilon_1}{2}$, while $\ell_3^* \leq \ell_2^* \leq 0.14+\frac{\varepsilon_1}{2}$ and $r \in \{2,3\}$. Note $1-0.335-(3-3+1)0.14=0.525$. For $r=2$ we again use condition (A). For $r=3$ we apply Lemma~\ref{lem:combinations} with $\beta = 0.075-\varepsilon_1$, $\chi = [0.525-\varepsilon_1,0.6+\varepsilon_1]$, $\rho = 0.335-\varepsilon_1$ and $k=3$ in condition (B) to deduce that option (2) is satisfied.  Finally, in condition (iv) we have $ \ell_1^* + \dots +\ell_r^* \leq 0.615 +\frac{\varepsilon_1}{2}=1-0.335-(2\cdot 0.525-1)+\frac{\varepsilon_1}{2}$. We also have $\ell_1^*+\ell_2^* \leq 0.4+\frac{\varepsilon_1}{2}$ and $\ell_2^* \leq 0.2+\frac{\varepsilon_1}{4}$. If $ \ell_1^* + \dots +\ell_r^* \leq 0.4$, we again use condition (A) of Lemma~\ref{lem:combinations}. Otherwise there exists $k$ with $ \ell_1^* + \dots +\ell_{k}^* \in [0.4-\frac{\varepsilon_1}{2},0.6+\frac{\varepsilon_1}{2}]$. So we may apply Lemma~\ref{lem:combinations} with $\beta = 0.075$, $\chi = [0.525-\varepsilon_1,0.6+\varepsilon_1]$, $\rho = 0.335-\varepsilon_1$ and condition (D) to obtain that $\{\ell_i\}_{i=1}^j$ satisfies option (2).

\vspace{3mm}
Case (IV): If (IV) holds, we have  $a \in (0.57, 0.59]$,  $\chi_0(a) = 0.33-\varepsilon_1$, 
$
\chi_1(a) = [0.38-\varepsilon_1, 0.42+\varepsilon_1] \cup [0.58-\varepsilon_1, 0.62+\varepsilon_1]$, 
$\chi_2(a) = [0.38-\varepsilon_1, 0.455+\varepsilon_1] \cup [0.545-\varepsilon_1, 0.62+\varepsilon_1]$ and
$\chi_3(a) = [0.315-\varepsilon_1, 0.42+\varepsilon_1] \cup [0.58-\varepsilon_1, 0.685+\varepsilon_1]$.

\vspace{3mm}
 With  each of the conditions (i), (ii), (iii), (iv), (v), (vi), (vii) and (viii) we have $\ell_1^* + \dots + \ell_r^* \leq 0.685 +\frac{\varepsilon_1}{2}$. We apply Lemma~\ref{lem:combinations} with $\beta \leq 0.105$, $\chi = [0.58-\varepsilon_1,0.685+\varepsilon_1] \subseteq \chi_3(a)$, $\rho = 0.33-\varepsilon_1$ and condition (A) to deduce that there exists $I_3 \subseteq  \{1,\dots, j\}$ with $\sum_{i \in I_3} \ell_i \in \chi_3(a)$.

\vspace{3mm}
In condition (iii) we also have $\ell_1^* \in [0.42-\frac{\varepsilon_1}{2},0.455+\frac{\varepsilon_1}{2}]$ and in condition (vii) we have $\ell_1^*+\ell_2^* \in [0.42-\frac{\varepsilon_1}{2},0.455+\frac{\varepsilon_1}{2}]$. By condition (E) of Lemma~\ref{lem:combinations}, there exists  $I_2 \subseteq  \{1,\dots, j\}$ with $\sum_{i \in I_2} \ell_i \in \chi_2(a)$.   With condition (i), (v) and (viii) (as well as (iv), if $\max\{0.145+\frac{\varepsilon_1}{2},(0.62-\ell_1^*)/2+\frac{\varepsilon_1}{4}\} =(0.62-\ell_1^*)/2+\frac{\varepsilon_1}{4}$), we have $\ell_1^* + \dots + \ell_r^* \leq 0.62+\frac{\varepsilon_1}{2}$ and $\beta = 0.075$. Applying Lemma~\ref{lem:combinations} with $\beta =0.075$, $\chi = [0.545-\varepsilon_1,0.62+\varepsilon_1]$, $\rho = 0.33-\varepsilon_1$ and condition (A), we again find $I_2 \subseteq  \{1,\dots, j\}$ with $\sum_{i \in I_2} \ell_i \in \chi_2(a)$. With condition (ii) we observe that $\ell_2^* \leq (0.685-\ell_1^*)/2+\frac{\varepsilon_1}{4} \leq 0.125$ for $\ell_1^* \geq 0.455-\frac{\varepsilon_1}{2}$, while $\ell_1^*+\ell_2^* \leq 0.62+\frac{\varepsilon_1}{2}$. We note that $1-0.33-0.125 = 0.545$ and apply Lemma~\ref{lem:combinations} with $\beta =0.075$, $\chi = [0.545-\varepsilon_1,0.62+\varepsilon_1]$, $\rho = 0.33-\varepsilon_1$ and condition (A) or  (B) (where $k=3$ and $b_1 = 0.125$), depending on whether $r=2$ or $r=3$.  Similarly, if $\max\{0.145+\frac{\varepsilon_1}{2},(0.62-\ell_1^*)/2+\frac{\varepsilon_1}{4}\} =0.145+\frac{\varepsilon_1}{2}$, then condition (iv) is treated by observing that $1-0.33-2 \cdot 0.145= 0.38$ and taking $\beta =0.075$, $\chi = [0.38-\varepsilon_1,0.455+\varepsilon_1]$, $\rho = 0.33-\varepsilon_1$ and $k=2$ and $b_1=0.145+\frac{\varepsilon_1}{2}$ in condition (B) of Lemma~\ref{lem:combinations}. Finally, in condition (vi) we have $\ell_1^* \leq 0.29+\frac{\varepsilon_1}{2}$ and $\sum_{i=2}^3 \ell_i^* \leq 0.455+\frac{\varepsilon_1}{2}$. Since $1-0.33-0.29=0.38$, we can apply Lemma~\ref{lem:combinations} with $\beta =0.075$, $\chi = [0.38-\varepsilon_1,0.455+\varepsilon_1]$, $\rho = 0.33-\varepsilon_1$ and $b_2=0.29+\frac{\varepsilon_1}{2}$ in condition (C) to deduce that there exists $I_2 \subseteq  \{1,\dots, j\}$ with $\sum_{i \in I_2} \ell_i \in \chi_2(a)$. Hence 
$\{\ell_i\}_{i=1}^j$ satisfies option (3).

\vspace{3mm}
Case (V): Now $a \in (0.59, 0.61]$,  $\chi_0(a) = 0.33-\varepsilon_1$, 
$
\chi_1(a) = [0.365-\varepsilon_1, 0.42+\varepsilon_1] \cup [0.58-\varepsilon_1, 0.635+\varepsilon_1]$, 
$\chi_2(a) = [0.365-\varepsilon_1, 0.435+\varepsilon_1] \cup [0.565-\varepsilon_1, 0.635+\varepsilon_1]$ and
$\chi_3(a) = [0.33-\varepsilon_1, 0.42+\varepsilon_1] \cup [0.58-\varepsilon_1, 0.67+\varepsilon_1]$.

\vspace{3mm}
  With  each of the conditions (i), (ii), (iii), (iv), (v) and (vii) we have $\ell_1^* + \dots + \ell_r^* \leq 0.67+\frac{\varepsilon_1}{2}$. We apply Lemma~\ref{lem:combinations} with $\beta \leq 0.09$, $\chi = [0.58-\varepsilon_1,0.67+\varepsilon_1] \subseteq \chi_3(a)$, $\rho = 0.33-\varepsilon_1$ and condition (A) to deduce that there exists $I_3 \subseteq  \{1,\dots, j\}$ with $\sum_{i \in I_3} \ell_i \in \chi_3(a)$. On the other hand, with condition (vi) we have $\ell_1^* \leq 0.305+\frac{\varepsilon_1}{2}$ and $\ell_2^* \leq 0.2099+\frac{\varepsilon_1}{2}$, so that $\sum_{i=2}^3\ell_i^* \leq 0.42+\frac{\varepsilon_1}{2}$.  Note $1-0.33-0.305 = 0.365$. We apply Lemma~\ref{lem:combinations} with $\beta = 0.07$, $\chi = [0.33-\varepsilon_1,0.42+\varepsilon_1]$, $\rho = 0.33-\varepsilon_1$ and $b_2 = 0.305+\frac{\varepsilon_1}{2}$ in condition (C) to deduce  that there exists $I_3 \subseteq  \{1,\dots, j\}$ with $\sum_{i \in I_3} \ell_i \in \chi_3(a)$.

\vspace{3mm} 
In condition (iii) we have $\ell_1^* \in [0.42-\frac{\varepsilon_1}{2},0.435+\frac{\varepsilon_1}{2}]$. By condition (E) of Lemma~\ref{lem:combinations}, there exists  $I_2 \subseteq  \{1,\dots, j\}$ with $\sum_{i \in I_2} \ell_i \in \chi_2(a)$.   With condition (i), (v) and (vii), we have $\ell_1^* + \dots + \ell_r^* \leq 0.635+\frac{\varepsilon_1}{2}$ and $\beta = 0.07$. Applying Lemma~\ref{lem:combinations} with $\beta =0.07$, $\chi = [0.565-\varepsilon_1,0.635+\varepsilon_1]$, $\rho = 0.33-\varepsilon_1$ and condition (A), we again find $I_2 \subseteq  \{1,\dots, j\}$ with $\sum_{i \in I_2} \ell_i \in \chi_2(a)$. With condition (ii) we have  $\ell_1^*+\ell_2^* < 0.635$, while $\ell_3^* \leq 0.105+\frac{\varepsilon_1}{2}$. We note that $1-0.33-0.105 = 0.565$ and apply Lemma~\ref{lem:combinations} with $\beta =0.07$, $\chi = [0.565-\varepsilon_1,0.635+\varepsilon_1]$, $\rho = 0.33-\varepsilon_1$ and condition (A) or  (B) (where $k=3$ and $b_1 = 0.105+\frac{\varepsilon_1}{2}$), depending on whether $r=2$ or $r=3$.  Condition (iv) is treated by observing that $1-0.33-2 \cdot 0.1525= 0.365$ and taking $\beta =0.07$, $\chi = [0.365-\varepsilon_1,0.435+\varepsilon_1]$, $\rho = 0.33-\varepsilon_1$ and $k=2$ and $b_1=0.1525+\frac{\varepsilon_1}{2}$ in condition (B) of Lemma~\ref{lem:combinations}. Finally, in condition (vi) we have $\ell_1^* \leq 0.305+\frac{\varepsilon_1}{2}$ and $\sum_{i=2}^3 \ell_i^* \leq 0.42+\frac{\varepsilon_1}{2}$. Since $1-0.33-0.305=0.365$, we can apply Lemma~\ref{lem:combinations} with $\beta =0.07$, $\chi = [0.365-\varepsilon_1,0.435+\varepsilon_1]$, $\rho = 0.33-\varepsilon_1$ and $b_2=0.305+\frac{\varepsilon_1}{2}$ in condition (C) to deduce that there exists $I_2 \subseteq  \{1,\dots, j\}$ with $\sum_{i \in I_2} \ell_i \in \chi_2(a)$.  Hence 
$\{\ell_i\}_{i=1}^j$ satisfies option (3).

\vspace{3mm}
Case (VI):  If (VI) holds, then  $a >0.61$,  $\chi_0(a) = 0.32-\varepsilon_1$ and $
\chi_1(a) =  [0.355-\varepsilon_1, 0.42+\varepsilon_1] \cup [0.58-\varepsilon_1, 0.645+\varepsilon_1]$.  In conditions (i), (iii) and (v)  we have $\ell_1^* + \dots + \ell_r^* \leq 0.645+\frac{\varepsilon_1}{2}$. We apply Lemma~\ref{lem:combinations} with $\beta = 0.065$, $\chi = [0.58-\varepsilon_1,0.645+\varepsilon_1]$, $\rho =0.32-\varepsilon_1$ and condition (A) to deduce that option (2) holds.  With condition (ii)  we  have $\ell_1^*+\ell_2^* \in [0.42-\frac{\varepsilon_1}{2},0.58+\frac{\varepsilon_1}{2}]$, while $\ell_2^* \leq 0.1+\frac{\varepsilon_1}{2}$ and $r \in \{2,3\}$. For $r=2$ we again use condition (A), while for $r=3$ we note $1-0.32-(3-3+1)0.1 = 0.58$ and then apply Lemma~\ref{lem:combinations} with $\beta = 0.065$, $\chi = [0.58-\varepsilon_1,0.645+\varepsilon_1]$, $\rho =0.32-\varepsilon_1$ and $k=3$ and $b_1=0.1+\frac{\varepsilon_1}{2}$ in condition (B).
 Finally, with (iv) or (vi)  we  have  $\sum_{i=2}^r \ell_i^* \leq 0.42+\frac{\varepsilon_1}{2}$ and $\ell_1^* \leq 0.325+\frac{\varepsilon_1}{2}$ and note $1-0.32-0.325=0.355$. We apply Lemma~\ref{lem:combinations} with $\beta=0.065$, $\chi = [0.355-\varepsilon_1,0.42+\varepsilon_1] \subseteq \chi_1(a)$, $\rho = 0.32-\varepsilon_1$ and $b_2=0.325+\frac{\varepsilon_1}{2}$ in condition (C) to deduce that $\{\ell_i\}_{i=1}^j$ satisfies option (2).

\vspace{3mm}
  
Now let $\{\ell_i\}_{i=1}^j \in \Xi^{\star \star}(\ell_1^*, \dots, \ell_r^*)$.   

\vspace{3mm} 
Case (A):  Suppose we have  option (A), so that $\chi_1(a) = [0.29-\varepsilon_1, 0.36+\varepsilon_1] \cup [0.64-\varepsilon_1, 0.71+\varepsilon_1]$. Option  (a) gives $\ell_1^* \in [0.29-\frac{\varepsilon_1}{2},0.36+\frac{\varepsilon_1}{2}]$,  (b) gives $\ell_1^* + \ell_2^* \in [0.64-\frac{\varepsilon_1}{2},0.71+\frac{\varepsilon_1}{2}]$, (c) gives $\ell_1^* +\ell_2^*+ \ell_3^* \in [0.64- \frac{\varepsilon_1}{2},0.71+\frac{\varepsilon_1}{2}]$, (d) gives $\ell_1^* +\ell_4^* \in [0.29-\frac{\varepsilon_1}{2},0.36+\frac{\varepsilon_1}{2}]$,  (e) gives $\ell_i^* + \ell_j^* \in [0.29-\frac{\varepsilon_1}{2}, 0.36+\frac{\varepsilon_1}{2}]$ and (f) gives $\ell_1^*+ \dots+\ell_5^* \in [0.64-\frac{\varepsilon_1}{2},0.71+\frac{\varepsilon_1}{2}]$. In every case there is some $S \subseteq \{1, \dots, r\}$ with $\sum_{s \in S} \ell_s^* \in [0.29-\frac{\varepsilon_1}{2},0.36+\frac{\varepsilon_1}{2}]\cup [0.64-\frac{\varepsilon_1}{2},0.71+\frac{\varepsilon_1}{2}]$. Applying Lemma~\ref{lem:combinations} with $\chi =  [0.29-\varepsilon_1, 0.36+\varepsilon_1]$ or $\chi =  [0.64-\varepsilon_1, 0.71+\varepsilon_1]$, we  find that $\{\ell_i\}_{i=1}^j$ satisfies option~(2). 

\vspace{3mm} 
Case (B): Suppose now that we have  (B).  With  (a) we have $\ell_1^* \in [0.315 -\frac{\varepsilon_1}{2},0.345+\frac{\varepsilon_1}{2}]\cup [0.427-\frac{\varepsilon_1}{2},0.474+\frac{\varepsilon_1}{2}] \subseteq \chi_1(a)$, with  (b) we have $\ell_1^*+\ell_2^* \in [0.427-\frac{\varepsilon_1}{2},0.474+\frac{\varepsilon_1}{2}] \cup [0.526-\frac{\varepsilon_1}{2},0.573+\frac{\varepsilon_1}{2}] \cup [0.655-\frac{\varepsilon_1}{2},0.685+\frac{\varepsilon_1}{2}] \subseteq \chi_1(a)$ and with  (c) we have $\ell_1^* \in [0.285-\frac{\varepsilon_1}{2},0.375+\frac{\varepsilon_1}{2}] \subseteq \chi_3(a)$ and $\ell_1^* + \ell_2^* \in [0.405-\frac{\varepsilon_1}{2},0.485+\frac{\varepsilon_1}{2}] \cup [0.515-\frac{\varepsilon_1}{2},0.595+\frac{\varepsilon_1}{2}] \subseteq \chi_2(a)$.  Lemma~\ref{lem:combinations}  tells us that $\{\ell_i\}_{i=1}^j$ satisfies option~(2) or (3). 

\vspace{3mm}
Cases (C)(a)--(F)(b): All of these cases can be treated just like (A) and (B). The values of $\ell_i^*$ have been carefully restricted to ensure that there exists a set $S_1 \subseteq \{1, \dots, r\}$ with $\sum_{i \in S_1} \ell_i^*+\delta \in \chi_1(a) $ for all $\delta \in [-\frac{\varepsilon_1}{2},\frac{\varepsilon_1}{2}]$ or two sets $S_2, S_3 \subseteq \{1, \dots, r\}$ with $\sum_{i \in S_j} \ell_i^*+\delta \in \chi_j(a) $ for $\delta \in [-\frac{\varepsilon_1}{2},\frac{\varepsilon_1}{2}]$ and $j \in \{2,3\}$. Lemma~\ref{lem:combinations} again tells us that $\{\ell_i\}_{i=1}^j$ satisfies option~(2) or (3). 
\end{proof}
 
\subsection{Computations and the Buchstab identity}

To complete the proof of Proposition~\ref{proposition3}, we just need to use the elements of $\mathcal{R}^\star(a)$ and $\mathcal{R}^{\star\star}(a)$ given in Corollary~\ref{cor:computations} to decompose $1_{\mathbb{P}}(n)$  into a sum which is as described in (\ref{decompositiontoaimfor}):  $1_{\mathbb{P}}(n) = \sum_{i=1}^{s_1} \Psi_i(n) + \sum_{i=1}^{s_2} \Theta_i(n)$ with $s_1, s_2 \leq 100$ and each $\Psi_i: [x,6x] \cap \mathbb{N} \rightarrow \mathbb{R}$ good$^\star$ or good$^{\star \star}$ and $\Theta_i$ non-negative with  $\sum_{i=1}^s\sum_{n \in \mathcal{B}(y)} \Theta_i(n) \leq \frac{0.9999y}{x^b\log(x)}$.

\vspace{3mm} Recall the Buchstab identity: For $z_1 < z_2$, 
\begin{align*}
\psi(n,z_2) = \psi(n,z_1) - \sum_{\substack{n = p_1 m \\ z_1 \leq p_1 < z_2 }} \psi(m,p_1).
\end{align*}
Recall $X=(7x)^{1/2}$. We  only need to  decompose $1_{\mathbb{P}}(n)$ when $n \in [x,6x]$, where $1_{\mathbb{P}}(n) = \psi(n,X)$. The desired decomposition will be obtained via repeated use of the Buchstab identity, starting with $\psi(n,X)$.

\vspace{3mm} The case $a \in (0.53,0.545]$  corresponds to one of the more complicated situations and its treatment also demonstrates the full range of arguments required to deal with all other cases. For illustrative purposes, we thus begin with a detailed proof of (\ref{decompositiontoaimfor}) for  $a \in (0.53,0.545]$ and then gradually decrease the level of detail given in  proofs of other cases.

\vspace{3mm} In the remaining sections of this paper, we need to compute upper bounds for many different integrals. The corresponding Mathematica code can be found on arXiv.

\subsubsection{$a \in (0.53, 0.545]$}    

After two applications of the Buchstab identity with $z_1 = x^{0.08}$, and after using $\ell_i^*$   as  a shorthand for $\log_x(p_i)$, we have 
\begingroup
\allowdisplaybreaks
\begin{align}
\psi(n,X) &= \psi(n,x^{0.08}) - \sum_{\substack{n = p_1 m \\ x^{0.08} \leq p_1 < X }} \psi(m,p_1) \nonumber \\
 &= \psi(n,x^{0.08}) - \sum_{\substack{n = p_1 m \\ x^{0.08} \leq p_1 < X }} \psi(m,x^{0.08}) +  \sum_{\substack{n = p_1p_2 m \\ x^{0.08} \leq p_2 < p_1 < X }} \psi(m,p_2) \nonumber
 \\   \label{line1}
  &= \psi(n,x^{0.08}) - \sum_{\substack{n = p_1 m \\ x^{0.08} \leq p_1 < X }} \psi(m,x^{0.08}) \\  \label{line2}
  &+  \sum_{\substack{n = p_1p_2 m \\ x^{0.08} \leq p_2 < p_1 < X  \\ \ell_1^*>0.474 }} \psi(m,p_2)
  +  \sum_{\substack{n = p_1p_2 m \\ x^{0.08} \leq p_2 < p_1 < X \\ 0.427 \leq \ell_1^* \leq 0.474}} \psi(m,p_2) +  \sum_{\substack{n = p_1p_2 m \\ x^{0.08} \leq p_2 < p_1 < X \\ 0.375 < \ell_1^* < 0.427 }} \psi(m,p_2) \\  \label{line3}
    &+  \sum_{\substack{n = p_1p_2 m \\ x^{0.08} \leq p_2 < p_1 < X \\  \ell_1^* \in [0.285,0.315)\cup ( 0.345, 0.375] }} \psi(m,p_2)+  \sum_{\substack{n = p_1p_2 m \\ x^{0.08} \leq p_2 < p_1 < X \\ 0.315 \leq \ell_1^* \leq 0.345}} \psi(m,p_2)+  \sum_{\substack{n = p_1p_2 m \\ x^{0.08} \leq p_2 < p_1 < X \\  \ell_1^* <0.285}} \psi(m,p_2).
    \end{align}
    By condition (II)(i) of Corollary~\ref{cor:computations}, both $\psi(n, x^{0.08})$ and the sum $- \sum_{n = p_1 m, \, 0.08 \leq \ell_1^* \leq 0.5+\varepsilon_1 } \psi(m,x^{0.08}) $ in (\ref{line1}) are good$^\star$ functions.  Denote the first sum in  (\ref{line2}) by $\Upsilon_1(n)$ and the third sum by $\Upsilon_2(n)$.   The second sum in (\ref{line2})  is a  good$^{\star\star}$ function  by condition (B)(a) of Corollary~\ref{cor:computations}.  Denote the first sum in  (\ref{line3}) by $\Upsilon_3(n)$ and the third by $\Upsilon_4(n)$. The second sum is again a  good$^{\star\star}$ function  by condition (B)(a) of Corollary~\ref{cor:computations}. We now treat $\Upsilon_1(n)$, $\Upsilon_2(n)$, $\Upsilon_3(n)$ and $\Upsilon_4(n)$ further, beginning with $\Upsilon_1(n)$. 
\begin{align}
\Upsilon_1(n)=\sum_{\substack{n = p_1p_2 m \\ x^{0.08} \leq p_2 < p_1 < X  \\ \ell_1^*>0.474 \\ \ell_2^* \geq \min\{0.595-\ell_1^*,(0.715-\ell_1^*)/2\}}} \psi(m,p_2)+  \sum_{\substack{n = p_1p_2 m \\ x^{0.08} \leq p_2 < p_1 < X  \\ \ell_1^*>0.474 \\ \ell_2^* < \min\{0.595-\ell_1^*,(0.715-\ell_1^*)/2\}}} \psi(m,p_2).
  \label{line5}
\end{align}     
Denote the first large sum in (\ref{line5}) by $\Theta_1(n)$ and the second by $\Upsilon_{11}(n)$.  Using a generalised prime number theorem, we can bound   $\sum_{n \in \mathcal{B}(y)} \Theta_1(n)$ from above as follows: 
\begin{align*}
&\sum_{n \in \mathcal{B}(y)} \Theta_1(n) \leq  \dfrac{y(1+o(1))}{x^b \log(x)}\int_{0.474}^{0.5+\varepsilon_1} \int_{\min\{0.595-\alpha_1,(0.715-\alpha_1)/2\}}^{(1-\alpha_1)/2}\dfrac{\omega\left(\frac{1-\alpha_1-\alpha_2}{\alpha_2}\right)}{\alpha_1 \alpha_2^2}\mbox{d}\alpha_2\mbox{d}\alpha_1 < \dfrac{0.185 y}{x^b \log(x)}.
\end{align*} 
Here $\omega(u)$ is Buchstab's function, which satisfies $\omega(u) \leq \max\{0.6, \frac{1}{u}\}$ for $u \geq 1$.

\vspace{3mm}
 Next we apply two  Buchstab iterations to $\Upsilon_{11}(n)$, the first with $z_1 =x^{0.08}$, the second with $z_1 =x^{0.09}$.  
\begin{align}
\Upsilon_{11}(n)
&= \sum_{\substack{n = p_1p_2 m \\ x^{0.08} \leq p_2 < p_1 < X \\ \ell_1^*>0.474 \\ \ell_2^* < \min\{0.595-\ell_1^*,(0.715-\ell_1^*)/2\}}}  \psi(m,x^{0.08})-\sum_{\substack{n = p_1p_2p_3 m \\ x^{0.08} \leq p_3 < p_2 < p_1 < X \\ \ell_1^*>0.474, \, \ell_3^* \leq 0.09 \\ \ell_2^* < \min\{0.595-\ell_1^*,(0.715-\ell_1^*)/2\}}} \psi(m,p_3) \label{line14}\\
&-\sum_{\substack{n = p_1p_2p_3 m \\ x^{0.08} \leq p_3 < p_2 < p_1 < X \\ \ell_1^*>0.474, \, \ell_3^* > 0.09  \\ \ell_2^* < \min\{0.595-\ell_1^*,(0.715-\ell_1^*)/2\}}} \psi(m,x^{0.09}) +\sum_{\substack{n = p_1p_2p_3p_4 m \\ x^{0.09} \leq p_4 < p_3 < p_2 < p_1 < X \\ \ell_1^*>0.474  \\ \ell_2^* < \min\{0.595-\ell_1^*,(0.715-\ell_1^*)/2\}}} \psi(m,p_4). \label{line15}
\end{align}
By condition (II)(ii) of Corollary~\ref{cor:computations}, both sums in (\ref{line14}) are good$^{\star}$ functions. (The second sum is of the special type of good$^\star$ described at the end of Section~\ref{ssec:preparation}.) The first sum in $(\ref{line15})$ is also good by condition (II)(ii) of Corollary~\ref{cor:computations}. Denote the second sum in $(\ref{line15})$ by $\Theta_2(n)$. Then $\sum_{n \in \mathcal{B}(y)} \Theta_2(n)$ is bounded above by the following quantity:  
\begin{align*}
&\dfrac{y(1+o(1))}{x^b \log(x)}\int_{0.474}^{0.5+\varepsilon_1} \int_{0.09}^{\min\{0.595-\alpha_1,\frac{(0.715-\alpha_1)}{2}\}}\!\! \int_{0.09}^{\alpha_2} \int_{0.09}^{\min\{\alpha_3,\frac{(1-\sum_{i=1}^3\alpha_i)}{2}\}} \dfrac{\omega\left(\frac{1-\sum_{i=1}^4 \alpha_i}{\alpha_4}\right)}{\alpha_1 \alpha_2 \alpha_3 \alpha_4^2}\mbox{d}\underline{\alpha} < \dfrac{0.001 y}{x^b \log(x)}.
\end{align*}

Next we split $\Upsilon_2(n)$  up as follows: 
\begin{align}
\Upsilon_2(n)&=  \sum_{\substack{n = p_1p_2 m \\ x^{0.08} \leq p_2 < p_1 < X \\ 0.375 < \ell_1^* < 0.427  \\ \ell_2^* \in (0.573- \ell_1^*, 0.655 - \ell_1^*) \cup ( 0.685 - \ell_1^*, \ell_1^*)}} \!\!\!\!\!\!\psi(m,p_2)+ \!\!\!\!\!\!\!\!\!\!\!\! \sum_{\substack{n = p_1p_2 m \\ x^{0.08} \leq p_2 < p_1 < X \\  0.375 < \ell_1^* < 0.427  \\  \ell_1^*+\ell_2^* \in [0.427,  0.474 ] \cup [0.526, 0.573]\cup[0.655, 0.685]}}\!\!\!\!\!\!\!\!\!\!\!\! \psi(m,p_2)
    \label{line7}\\
    &+  \sum_{\substack{n = p_1p_2 m \\ x^{0.08} \leq p_2 < p_1 < X \\  0.375 < \ell_1^* < 0.427 
\\ \ell_2^* \in (0.474 - \ell_1^*, 0.526 - \ell_1^*) \\ \ell_2^* > (0.655-\ell_1^*)/2}} \psi(m,p_2)+  \sum_{\substack{n = p_1p_2 m \\ x^{0.08} \leq p_2 < p_1 < X \\  0.375 < \ell_1^* < 0.427 
\\ \ell_2^* \in (0.474 - \ell_1^*, 0.526 - \ell_1^*) \\ \ell_2^* \leq (0.655-\ell_1^*)/2}} \psi(m,p_2).
\label{line8}
\end{align}
 Denote the first sum in  (\ref{line7}) by $\Theta_3(n)$. By condition (B)(b) of Corollary~\ref{cor:computations}, the second sum in  (\ref{line7}) is a sum of three good$^{\star\star}$ functions.  
Denote the first sum in  (\ref{line8}) by $\Theta_4(n)$ and the second sum by $\Upsilon_{21}(n)$.  Then $\sum_{n \in \mathcal{B}(y)} \Theta_3(n)$ and  $\sum_{n \in \mathcal{B}(y)} \Theta_4(n)$ are bounded above by the following quantities: 
\begin{align*}
&\sum_{n \in \mathcal{B}(y)} \Theta_3(n) \leq  \dfrac{y(1+o(1))}{x^b \log(x)}\int_{0.375}^{0.427} \!\!\int_{[0.573-\alpha_1,0.655-\alpha_1]\cup [0.685-\alpha_1,(1-\alpha_1)/2]}\!\!\!\dfrac{\omega\left(\frac{1-\alpha_1-\alpha_2}{\alpha_2}\right)}{\alpha_1 \alpha_2^2}\mbox{d}\alpha_2\mbox{d}\alpha_1 < \dfrac{0.175 y}{x^b \log(x)},\\
&\sum_{n \in \mathcal{B}(y)} \Theta_4(n) \leq  \dfrac{y(1+o(1))}{x^b \log(x)}\int_{0.375}^{0.427} \int_{\max\{0.08,0.474-\alpha_1,(0.655-\alpha_1)/2\}}^{0.526-\alpha_1}\dfrac{\omega\left(\frac{1-\alpha_1-\alpha_2}{\alpha_2}\right)}{\alpha_1 \alpha_2^2}\mbox{d}\alpha_2\mbox{d}\alpha_1 < \dfrac{0.01 y}{x^b \log(x)}.
\end{align*}
In the above expression (and throughout the rest of the paper) we have adopted the convention that $\int_c^d=0$ whenever $c>d$ to simplify notation.
  
 \vspace{3mm}
Now we apply two more Buchstab iterations to $\Upsilon_{21}(n)$: 
\begin{align}
\Upsilon_{21}(n) &=  
   \sum_{\substack{n = p_1p_2 m \\ x^{0.08} \leq p_2 < p_1 < X \\  0.375 < \ell_1^* < 0.427 
\\ \ell_2^* \in (0.474 - \ell_1^*, 0.526 - \ell_1^*) \\ \ell_2^* \leq (0.655-\ell_1^*)/2}} \psi(m,x^{0.08}) -
 \sum_{\substack{n = p_1p_2p_3 m \\ x^{0.08} < p_3 < p_2 < p_1 < X \\  0.375 < \ell_1^* < 0.427 
\\ \ell_2^* \in (0.474 - \ell_1^*, 0.526 - \ell_1^*) \\ \ell_2^* \leq (0.655-\ell_1^*)/2}} \psi(m,x^{0.08})  \label{line16} \\
&+ \sum_{\substack{n = p_1p_2p_3p_4 m \\ x^{0.08} \leq p_4 < p_3 < p_2 < p_1 < X \\  0.375 < \ell_1^* < 0.427 
\\ \ell_2^* \in (0.474 - \ell_1^*, 0.526 - \ell_1^*) \\ \ell_2^* \leq (0.655-\ell_1^*)/2}} \psi(m,p_4). \label{line17}
\end{align}
By condition (II)(iii) of  Corollary~\ref{cor:computations}, both sums in (\ref{line16}) are good$^{\star}$ functions.   Denote the sum  in (\ref{line17}) by $\Theta_5(n)$. Then $\sum_{n \in \mathcal{B}(y)} \Theta_5(n)$ is bounded above by the following quantity:  
\begin{align*}
& \dfrac{y(1+o(1))}{x^b \log(x)}\int_{0.375}^{0.427} \int_{\max\{0.08,0.474-\alpha_1\}}^{\min\{0.526-\alpha_1,(0.655-\alpha_1)/2\}}\!\!\! \int_{0.08}^{\alpha_2} \int_{0.08}^{\min\{\alpha_3,\frac{(1-\sum_{i=1}^3\alpha_i)}{2}\}} \dfrac{\omega\left(\frac{1-\sum_{i=1}^4 \alpha_i}{\alpha_4}\right)}{\alpha_1 \alpha_2 \alpha_3 \alpha_4^2}\mbox{d}\underline{\alpha}  <\dfrac{0.013 y}{x^b \log(x)}.
\end{align*}
Next, $\Upsilon_3(n)$ is split up as follows: 
\begin{align}
\Upsilon_3(n)&=  \sum_{\substack{n = p_1p_2 m \\ x^{0.08} \leq p_2 < p_1 < X \\  \ell_1^* \in [0.285,0.315)\cup ( 0.345, 0.375] \\ \ell_2^* > 0.595 - \ell_1^*}} \psi(m,p_2)+  \sum_{\substack{n = p_1p_2 m \\ x^{0.08} \leq p_2 < p_1 < X \\  \ell_1^* \in [0.285,0.315)\cup ( 0.345, 0.375] \\ \ell_2^* \in [0.515-\ell_1^*, 0.595 - \ell_1^*]}} \psi(m,p_2)
 \label{line9}\\
 &+  \sum_{\substack{n = p_1p_2 m \\ x^{0.08} \leq p_2 < p_1 < X \\  \ell_1^* \in [0.285,0.315)\cup ( 0.345, 0.375] \\ \ell_2^* \in (0.485-\ell_1^*, 0.515 - \ell_1^*) \\ \ell_2^* > (0.655-\ell_1^*)/2}} \psi(m,p_2)+  \sum_{\substack{n = p_1p_2 m \\ x^{0.08} \leq p_2 < p_1 < X \\  \ell_1^* \in [0.285,0.315)\cup ( 0.345, 0.375] \\\ell_2^* \in (0.485-\ell_1^*, 0.515 - \ell_1^*) \\ \ell_2^* \leq (0.655-\ell_1^*)/2}} \psi(m,p_2) 
 \label{line10}\\
  &+  \sum_{\substack{n = p_1p_2 m \\ x^{0.08} \leq p_2 < p_1 < X \\  \ell_1^* \in [0.285,0.315)\cup ( 0.345, 0.375] \\\ell_2^* \in [0.405 - \ell_1^*, 0.485 - \ell_1^*]}} \psi(m,p_2)+  \sum_{\substack{n = p_1p_2 m \\ x^{0.08} \leq p_2 < p_1 < X \\  0.285 < \ell_1^* < 0.315 \\ \ell_2^* < 0.405-\ell_1^*}} \psi(m,p_2).
\label{line11}  
\end{align} 
Denote the first sum in  (\ref{line9}) by $\Theta_6(n)$. By condition (B)(c) of Corollary~\ref{cor:computations}, the second sum in  (\ref{line9}) is a sum of two good$^{\star\star}$ functions.  Denote the first sum in  (\ref{line10}) by $\Theta_7(n)$ and the second sum by $\Upsilon_{31}(n)$.   By condition (B)(c) of Corollary~\ref{cor:computations}, the first sum in  (\ref{line11}) is a sum of two good$^{\star\star}$ functions. Denote the second sum in  (\ref{line11}) by $\Upsilon_{32}(n)$. 

\vspace{3mm} Using  a generalised prime number theorem, we bound   $\sum_{n \in \mathcal{B}(y)} \Theta_i(n)$, $i \in \{6,7\}$, from above as follows: 
\begingroup
\allowdisplaybreaks
\begin{align*}
&\sum_{n \in \mathcal{B}(y)} \Theta_6(n) \leq  \dfrac{y(1+o(1))}{x^b \log(x)}\int_{[0.285,0.315]\cup [0.345,0.375]} \int_{0.595-\alpha_1}^{\min\{\alpha_1, (1-\alpha_1)/2\}}\dfrac{\omega\left(\frac{1-\alpha_1-\alpha_2}{\alpha_2}\right)}{\alpha_1 \alpha_2^2}\mbox{d}\alpha_2\mbox{d}\alpha_1 < \dfrac{0.08 y}{x^b  \log(x)},\\
&\sum_{n \in \mathcal{B}(y)} \Theta_7(n) \leq  \dfrac{y(1+o(1))}{x^b \log(x)}\int_{\substack{[0.285,0.315]\\ \,\cup [0.345,0.375]}} \int_{\max\{0.485-\alpha_1,(0.655-\alpha_1)/2\}}^{0.515-\alpha_1}\dfrac{\omega\left(\frac{1-\alpha_1-\alpha_2}{\alpha_2}\right)}{\alpha_1 \alpha_2^2}\mbox{d}\alpha_2\mbox{d}\alpha_1 < \dfrac{0.062 y}{x^b \log(x)}.
\end{align*}
\endgroup

Next we apply two more Buchstab iterations to $\Upsilon_{31}(n)$ and two more to $\Upsilon_{32}(n)$. 
\begin{align}
\Upsilon_{31}(n)+\Upsilon_{32}(n)&=\sum_{\substack{n = p_1p_2 m \\ x^{0.08} \leq p_2 < p_1 < X \\  \ell_1^* \in [0.285,0.315)\cup ( 0.345, 0.375] \\\ell_2^* \in (0.485-\ell_1^*, 0.515 - \ell_1^*) \\ \ell_2^* \leq (0.655-\ell_1^*)/2}} \psi(m,x^{0.08})  +  \sum_{\substack{n = p_1p_2 m \\ x^{0.08} \leq p_2 < p_1 < X \\  0.285 < \ell_1^* < 0.315  \\ \ell_2^* < 0.405-\ell_1^*}} \psi(m,x^{0.08}) \label{line18} \\ 
&-\sum_{\substack{n = p_1p_2p_3 m \\ x^{0.08} \leq p_3 < p_2 < p_1 < X \\  \ell_1^* \in [0.285,0.315)\cup ( 0.345, 0.375] \\\ell_2^* \in (0.485-\ell_1^*, 0.515 - \ell_1^*) \\ \ell_2^* \leq (0.655-\ell_1^*)/2}} \psi(m,x^{0.08})  -  \sum_{\substack{n = p_1p_2p_3 m \\ x^{0.08} \leq p_3 < p_2 < p_1 < X \\  0.285 < \ell_1^* < 0.315 \\ \ell_2^* < 0.405-\ell_1^*}} \psi(m,x^{0.08}) \label{line19}\\
&+\sum_{\substack{n = p_1p_2p_3p_4 m \\ x^{0.08} \leq p_4 <  p_3 < p_2 < p_1 < X \\  \ell_1^* \in [0.285,0.315)\cup ( 0.345, 0.375] \\\ell_2^* \in (0.485-\ell_1^*, 0.515 - \ell_1^*) \\ \ell_2^* \leq (0.655-\ell_1^*)/2}} \psi(m,p_4)  +  \sum_{\substack{n = p_1p_2p_3p_4 m \\ x^{0.08} \leq p_4 < p_3 < p_2 < p_1 < X \\  0.285 < \ell_1^* < 0.315 \\ \ell_2^* < 0.405-\ell_1^*}} \psi(m,p_4). \label{line20}
\end{align}
Both sums in (\ref{line18}) and  both sums in  (\ref{line19}) are good$^\star$ by condition (II)(iii) of  Corollary~\ref{cor:computations}. Denote the first sum  in (\ref{line20}) by $\Theta_8(n)$ and the second  by  $\Theta_{9}(n)$. Note that $\ell_1^* < 0.315$ implies $0.485-\ell_1^*>(0.655-\ell_1^*)/2$. Hence $\sum_{n \in \mathcal{B}(y)} \Theta_8(n)$ and $\sum_{n \in \mathcal{B}(y)} \Theta_{9}(n)$  are bounded above by the following quantities, respectively: 
\begin{align*}
& \dfrac{y(1+o(1))}{x^b \log(x)}\int_{0.345}^{0.375} \int_{0.485-\alpha_1}^{(0.655-\alpha_1)/2} \int_{0.08}^{\alpha_2}  \int_{0.08}^{\min\{\alpha_3,\frac{(1-\sum_{i=1}^3\alpha_i)}{2}\}} \dfrac{\omega\left(\frac{1-\alpha_1-\alpha_2-\alpha_3-\alpha_4}{\alpha_4}\right)}{\alpha_1 \alpha_2 \alpha_3 \alpha_4^2}\mbox{d}\underline{\alpha} < \dfrac{0.012 y}{x^b \log(x)},\\
& \dfrac{y(1+o(1))}{x^b \log(x)}\int_{0.285}^{0.315} \int_{0.08}^{0.405-\alpha_1} \int_{0.08}^{\alpha_2}  \int_{0.08}^{\min\{\alpha_3,\frac{(1-\sum_{i=1}^3\alpha_i)}{2}\}} \dfrac{\omega\left(\frac{1-\alpha_1-\alpha_2-\alpha_3-\alpha_4}{\alpha_4}\right)}{\alpha_1 \alpha_2 \alpha_3 \alpha_4^2}\mbox{d}\underline{\alpha} <\dfrac{0.003 y}{x^b \log(x)}.\\
\end{align*} 
To treat $\Upsilon_4(n)$, we note that 
 \begin{align}
\Upsilon_4(n)=  \sum_{\substack{n = p_1p_2 m \\ x^{0.08} \leq p_2 < p_1 < X \\  \ell_1^* <0.285 \\ \ell_2^* > 0.526-\ell_1^*}} \psi(m,p_2)+  \sum_{\substack{n = p_1p_2 m \\ x^{0.08} \leq p_2 < p_1 < X \\  \ell_1^* < 0.285 \\ \ell_2^* < 0.526-\ell_1^* \\ \ell_2^* \geq (0.655-\ell_1^*)/2}} \psi(m,p_2) 
  +  \!\!\!\!\! \sum_{\substack{n = p_1p_2 m \\ x^{0.08} \leq p_2 < p_1 < X \\  \ell_1^* < 0.285 \\ \ell_2^* < \min\{ (0.655-\ell_1^*)/2, 0.526-\ell_1^*\} }} \!\!\!\!\! \psi(m,p_2). \label{line12}
\end{align}
\endgroup
By condition (B)(b) of Corollary~\ref{cor:computations}, the first sum in  (\ref{line12}) is a good$^{\star\star}$ function. Denote the second sum in (\ref{line12}) by $\Theta_{10}(n)$ and the third sum by $\Upsilon_{41}(n)$. We bound   $\sum_{n \in \mathcal{B}(y)} \Theta_{10}(n)$ from above as follows: 
\begingroup
\allowdisplaybreaks
\begin{align*}
&\sum_{n \in \mathcal{B}(y)} \Theta_{10}(n) \leq  \dfrac{y(1+o(1))}{x^b \log(x)}\int_{0.655/3}^{0.285} \int_{(0.655-\alpha_1)/2}^{\min\{\alpha_1, 0.526-\alpha_1\}}\dfrac{\omega\left(\frac{1-\alpha_1-\alpha_2}{\alpha_2}\right)}{\alpha_1 \alpha_2^2}\mbox{d}\alpha_2\mbox{d}\alpha_1 < {  \dfrac{0.1221 y}{x^b \log(x)}}. 
\end{align*}
\endgroup

 Finally, we apply up to four Buchstab iterations to $\Upsilon_{41}(n)$. 
 \begingroup
\allowdisplaybreaks
 \begin{align}
\Upsilon_{41}(n)&= \sum_{\substack{n = p_1p_2 m \\ x^{0.08} \leq p_2 < p_1 < X \\  \ell_1^* < 0.285 \\ \ell_2^* < \min\{ (0.655-\ell_1^*)/2, 0.526-\ell_1^*\} }} \!\!\!\!\! \psi(m,x^{0.08}) - 
\sum_{\substack{n = p_1p_2p_3 m \\ x^{0.08} \leq p_3 <  p_2 < p_1 < X \\  \ell_1^* < 0.285 \\ \ell_2^* < \min\{ (0.655-\ell_1^*)/2, 0.526-\ell_1^*\} }} \!\!\!\!\! \psi(m,x^{0.08}) \label{line23} \\
&+
\sum_{\substack{n = p_1p_2p_3p_4 m \\ x^{0.08} \leq p_4 < p_3 <  p_2 < p_1 < X \\  \ell_1^* < 0.285 \\ \ell_2^* < \min\{ (0.655-\ell_1^*)/2, 0.526-\ell_1^*\} \\  \ell_4^* > (0.655-\ell_1^*-\ell_2^*-\ell_3^*)/2 }} \!\!\!\!\! \psi(m,p_4)  
+ \sum_{\substack{n = p_1p_2p_3p_4 m \\ x^{0.08} \leq p_4 < p_3 <  p_2 < p_1 < X \\  \ell_1^* < 0.285 \\ \ell_2^* < \min\{ (0.655-\ell_1^*)/2, 0.526-\ell_1^*\} \\  \ell_4^* \leq (0.655-\ell_1^*-\ell_2^*-\ell_3^*)/2 }} \!\!\!\!\! \psi(m,x^{0.08})  \label{line24} \\
&- \sum_{\substack{n = p_1p_2p_3p_4p_5 m \\ x^{0.08}  \leq p_5< p_4 < p_3 <  p_2 < p_1 < X \\  \ell_1^* < 0.285 \\ \ell_2^* < \min\{ (0.655-\ell_1^*)/2, 0.526-\ell_1^*\} \\  \ell_4^* \leq (0.655-\ell_1^*-\ell_2^*-\ell_3^*)/2 }} \!\!\!\!\! \psi(m,x^{0.08})  + \sum_{\substack{n = p_1p_2p_3p_4p_5p_6 m \\ x^{0.08} \leq p_6 < p_5< p_4 < p_3 < p_2 < p_1 < X \\  \ell_1^* < 0.285 \\ \ell_2^* < \min\{ (0.655-\ell_1^*)/2, 0.526-\ell_1^*\} \\  \ell_4^* \leq (0.655-\ell_1^*-\ell_2^*-\ell_3^*)/2 }} \!\!\!\!\! \psi(m,p_6)  \label{line25}
 \end{align}
 \endgroup
 Both sums in (\ref{line23}) are good$^\star$ by condition (II)(iii) of Corollary~\ref{cor:computations}. The second sum in (\ref{line24}) and the first sum in (\ref{line25}) are also good$^\star$,  by condition (II)(iv) of Corollary~\ref{cor:computations}. Denote the first sum in (\ref{line24})  by $\Theta_{11}(n)$ and the second sum in (\ref{line25}) by $\Theta_{12}(n)$. Quantity $\sum_{n \in \mathcal{B}(y)} \Theta_{11}(n)$ is bounded above by 
 \begin{align*}
 & \dfrac{y(1+o(1))}{x^b \log(x)}\int_{0.08}^{0.285} \int_{0.08}^{\min\{\alpha_1,\frac{ (0.655-\alpha_1)}{2}\}}\!\! \int_{0.08}^{\alpha_2}  \int_{\max\{0.08, \frac{(0.655-\sum_{i=1}^3\alpha_i)}{2}\}}^{\min\{\alpha_3,\frac{(1-\sum_{i=1}^3\alpha_i)}{2}\}} \dfrac{\omega\left(\frac{1-\sum_{i=1}^4 \alpha_i}{\alpha_4}\right)}{\alpha_1 \alpha_2 \alpha_3 \alpha_4^2}\mbox{d}\underline{\alpha} < \dfrac{0.296 y}{x^b \log(x)}.
 \end{align*}
 Finally, to bound $\sum_{n \in \mathcal{B}(y)} \Theta_{12}(n)$, note that $\ell_1^* + \dots + \ell_5^* \leq 0.655$ implies $\ell_3^* < 0.17$. Hence it is sufficient to count the number of integers $n=p_1 \dots p_6 m$ in $\mathcal{B}(y)$ which have $(m,P(x^{0.08}))=1$ and ($p_1, p_2 \in [x^{0.17},x^{0.285}]$ or $p_1 \in [x^{0.17},x^{0.285}]$, $p_2 \in [x^{0.08},x^{0.17}]$ or $p_1,p_2 \in [x^{0.08},x^{0.17}]$) and $p_i \in [x^{0.08},x^{0.17}]$ for $i \geq 3$. The double counting introduced by ignoring conditions $p_i> p_{i+1}$ can be removed by dividing by $6!$ or $5!$ or $4! 2!$, respectively.
 \begin{align*}
\sum_{n \in \mathcal{B}(y)} \!\Theta_{12}(n) \leq \! \dfrac{y(1+o(1))}{x^b\log(x)}\left(  \dfrac{\log\left( \frac{0.17}{0.08} \right)^6 }{0.08 \cdot 6!}  + \dfrac{\log\left( \frac{0.17}{0.08} \right)^5\log\left( \frac{0.285}{0.17}\right)  }{0.08 \cdot 5!}  +  \dfrac{\log\left( \frac{0.17}{0.08} \right)^4\log\left( \frac{0.285}{0.17}\right)^2  }{0.08 \cdot 4! 2!} \right)\! <\! \dfrac{0.04 y}{x^b\log(x)}.
\end{align*}
 
  In summary, we have shown that $1_{\mathbb{P}}(n)$ can be written as $1_{\mathbb{P}}(n) = \Psi(n) + \sum_{i=1}^{12} \Theta_i(n)$, where $\Psi(n)$ is a sum of no more than $100$ good$^\star$ and good$^{\star\star}$ functions and where  each $\Theta_i$ is non-negative and 
  \begin{align*}
 \sum_{i=1}^{12} \sum_{n \in \mathcal{B}(y)} \Theta_i(n) <  { \dfrac{ 0.9991y}{x^b\log(x)}}.
  \end{align*}
This concludes the proof of Proposition~\ref{proposition3} for the case $a \in (0.53,0.545]$.

\subsubsection{$a \in (0.57, 0.59]$} Next we treat the hardest case, $a \in (0.57,0.59]$. After two applications of the Buchstab identity with $z_1 = x^{0.075}$, and using $\ell_i^*$   as  a shorthand for $\log_x(p_i)$, we have 
\begingroup
\allowdisplaybreaks
\begin{align}
\psi(n,X) &= \psi(n,x^{0.075}) - \sum_{\substack{n = p_1 m \\ x^{0.075} \leq p_1 < X }} \psi(m,x^{0.075}) +  \sum_{\substack{n = p_1p_2 m \\ x^{0.075} \leq p_2 < p_1 < X }} \psi(m,p_2) \nonumber
 \\ \label{line224}
 &= \psi(n,x^{0.075}) - \sum_{\substack{n = p_1 m \\ x^{0.075} \leq p_1 < X }} \psi(m,x^{0.075}) \\
  &+  \sum_{\substack{n = p_1p_2 m \\ x^{0.075} \leq p_2 < p_1 < X  \\ \ell_1^* > 0.455}} \psi(m,p_2)+  \sum_{\substack{n = p_1p_2 m \\ x^{0.075} \leq p_2 < p_1 < X  \\ 0.42 < \ell_1^* \leq 0.455  }} \psi(m,p_2) 
  \label{line225} 
+  \sum_{\substack{n = p_1p_2 m \\ x^{0.075} \leq p_2 < p_1 < X  \\ 0.38 < \ell_1^* \leq 0.42  }} \psi(m,p_2) \\
  &+  \sum_{\substack{n = p_1p_2 m \\ x^{0.075} \leq p_2 < p_1 < X  \\ 0.315 < \ell_1^* \leq 0.38}} \psi(m,p_2)+  \sum_{\substack{n = p_1p_2 m \\ x^{0.075} \leq p_2 < p_1 < X  \\ 0.29 < \ell_1^* \leq 0.315  }} \psi(m,p_2) 
  \label{line226} 
+  \sum_{\substack{n = p_1p_2 m \\ x^{0.075} \leq p_2 < p_1 < X  \\ \ell_1^* \leq 0.29  }} \psi(m,p_2).
  \end{align}
  \endgroup
 By condition (IV)(i) of Corollary~\ref{cor:computations}, both $\psi(n, x^{0.075})$ and the sum $- \sum_{n = p_1 m, \, 0.075 \leq \ell_1^* \leq 0.5+\varepsilon_1 } \psi(m,x^{0.075}) $ in (\ref{line224}) are good$^\star$ functions.  Denote the first sum in  (\ref{line225}) by $\Upsilon_1(n)$ and the second sum by $\Upsilon_2(n)$.   The third sum in (\ref{line225})  is a  good$^{\star\star}$ function  by condition (D)(a) of Corollary~\ref{cor:computations}.  Denote the first sum in  (\ref{line226}) by $\Upsilon_3(n)$, the second  by $\Upsilon_4(n)$ and the third by $\Upsilon_5(n)$. We now treat $\Upsilon_1(n)$, $\Upsilon_2(n)$, $\Upsilon_3(n)$, $\Upsilon_4(n)$ and $\Upsilon_5(n)$ further, beginning with $\Upsilon_1(n)$.

\vspace{3mm} Applying two more Buchstab iterations to parts of $\Upsilon_1(n)$,  condition (IV)(ii) of Corollary~\ref{cor:computations} tells us that 
\begingroup
\allowdisplaybreaks
\begin{align}
\Upsilon_1(n) &= \sum_{\substack{n = p_1p_2 m \\ x^{0.075} \leq p_2 < p_1 < X  \\ \ell_1^* > 0.455\\ \ell_2^* > 0.62-\ell_1^*}} \psi(m,p_2) + \sum_{\substack{n = p_1p_2 m \\ x^{0.075} \leq p_2 < p_1 < X  \\ \ell_1^* > 0.455\\ \ell_2^* \in [0.58-\ell_1^*, 0.62-\ell_1^*]}} \psi(m,p_2) +\sum_{\substack{n = p_1p_2 m \\ x^{0.075} \leq p_2 < p_1 < X  \\ \ell_1^* > 0.455\\ \ell_2^* < 0.58-\ell_1^* \\ \ell_2^* \geq (0.685-\ell_1^*)/2}} \psi(m,p_2) \label{line28}\\ 
&+\Psi_1(n) + \sum_{\substack{n = p_1p_2p_3p_4 m \\ x^{0.075} \leq p_4<p_3< p_2 < p_1 < X  \\ \ell_1^* > 0.455\\ \ell_2^* < \min\{0.58-\ell_1^*, (0.685-\ell_1^*)/2\}}} \psi(m,p_4),\label{line29}
\end{align}
\endgroup
where $\Psi_1(n)$ is a sum of two good$^\star$ functions. $\Psi_1(n)$ consists of the $\sum \psi(m,x^{0.075})$ -- components of the Buchstab decomposition applied to  region $(\ell_1^* > 0.455, \ell_2^* < \min\{0.58-\ell_1^*, (0.685-\ell_1^*)/2\})$. Denote the first sum in (\ref{line28}) by $\Theta_1(n)$ and the third sum by $\Theta_2(n)$. The second sum in (\ref{line28}) is good$^{\star\star}$ by condition (D)(d) of Corollary~\ref{cor:computations}. Denote the big sum in (\ref{line29}) by $\Theta_3(n)$. 
Upper bounds for $\sum_{n \in \mathcal{B}(y)} \Theta_1(n)$, $\sum_{n \in \mathcal{B}(y)} \Theta_2(n)$ and $\sum_{n \in \mathcal{B}(y)} \Theta_3(n)$
 are given below:
\begingroup
\allowdisplaybreaks
\begin{align*}
& \dfrac{y(1+o(1))}{x^b \log(x)}\int_{0.455}^{0.5+\varepsilon_1} \int_{0.62-\alpha_1}^{(1-\alpha_1)/2}\dfrac{\omega\left(\frac{1-\alpha_1-\alpha_2}{\alpha_2}\right)}{\alpha_1 \alpha_2^2}\mbox{d}\alpha_2\mbox{d}\alpha_1 <\dfrac{0.2029 y}{x^b \log(x)},\\
& \dfrac{y(1+o(1))}{x^b \log(x)}\int_{0.455}^{0.475} \int_{(0.685-\alpha_1)/2}^{0.58-\alpha_1}\dfrac{\omega\left(\frac{1-\alpha_1-\alpha_2}{\alpha_2}\right)}{\alpha_1 \alpha_2^2}\mbox{d}\alpha_2\mbox{d}\alpha_1<\dfrac{0.0099 y}{x^b \log(x)},\\
& \dfrac{y(1+o(1))}{x^b \log(x)}\int_{0.455}^{0.5+\varepsilon_1} \int_{0.075}^{\min\{0.58-\alpha_1, (0.685-\alpha_1)/2\}} \!\!\int_{0.075}^{\alpha_2} \int_{0.075}^{\min\{\alpha_3,\frac{(1-\sum_{i=1}^3\alpha_i)}{2}\}} \dfrac{\omega\left(\frac{1-\sum_{i=1}^4 \alpha_i}{\alpha_4}\right)}{\alpha_1 \alpha_2 \alpha_3 \alpha_4^2}\mbox{d}\underline{\alpha} < \dfrac{0.0038 y}{x^b \log(x)}.
\end{align*}
\endgroup

\vspace{3mm}
 Next we apply two more Buchstab iterations to parts of $\Upsilon_2(n)$, one with $z_1=x^{0.075}$ and one with $z_1 = x^{0.105}$. Noting that condition (IV)(iii) of Corollary~\ref{cor:computations} allows $\beta \leq 0.105$, we find that 
\begingroup
\allowdisplaybreaks
\begin{align}
\Upsilon_2(n) &= \sum_{\substack{n = p_1p_2 m \\ x^{0.075} \leq p_2 < p_1 < X  \\ 0.42 < \ell_1^* \leq 0.455 \\ \ell_2^* > 0.685-\ell_1^* }} \psi(m,p_2)+\sum_{\substack{n = p_1p_2 m \\ x^{0.075} \leq p_2 < p_1 < X  \\ 0.42 < \ell_1^* \leq 0.455 \\ \ell_2^* \in  [0.58-\ell_1^*, 0.685-\ell_1^*]}} \psi(m,p_2) + \!\!\!\!\!\sum_{\substack{n = p_1p_2 m \\ x^{0.075} \leq p_2 < p_1 < X  \\ 0.42 < \ell_1^* \leq 0.455 \\ \ell_2^*\in((0.685-\ell_1^*)/2, 0.58-\ell_1^*) }} \psi(m,p_2) \label{line30}
\\
&+ \Psi_2(n) +  \sum_{\substack{n = p_1p_2p_3p_4 m \\ x^{0.105} \leq p_4 < p_3< p_2 < p_1 < X  \\ 0.42 < \ell_1^* \leq 0.455 \\ \ell_2^* \leq (0.685-\ell_1^*)/2 }} \psi(m,p_4), \label{line31}
\end{align}
\endgroup
where $\Psi_2(n)$ is a sum of three good$^\star$ functions. $\Psi_2(n)$ consists of the two $\sum \psi(m,z_1)$ -- components of the Buchstab decomposition applied to  region $( 0.42 < \ell_1^* \leq 0.455, \ell_2^* \leq (0.685-\ell_1^*)/2)$ and of one special good$^\star$ function, of the type described at the end of Section~\ref{ssec:preparation}, resulting from the jump from $z_1=x^{0.075}$ to $z_1=x^{0.105}$, which does not allow for $p_3< x^{0.105}$.  (This is just like in the treatment of $\Upsilon_{11}(n)$ in the case $a \in (0.53,0.545]$.)

\vspace{3mm} Denote the first sum in (\ref{line30}) by $\Theta_4(n)$ and the third sum by $\Theta_5(n)$. 
By condition (D)(b)  of Corollary~\ref{cor:computations}, the second sum in (\ref{line30}) is good$^{\star\star}$. Denote the sum in (\ref{line31}) by $\Theta_6(n)$. We have the following bounds: 
\begingroup
\allowdisplaybreaks
\begin{align*}
& \sum_{n \in \mathcal{B}(y)} \Theta_4(n) \leq \dfrac{y(1+o(1))}{x^b \log(x)}\int_{0.42}^{0.455} \int_{0.685-\alpha_1}^{(1-\alpha_1)/2}\dfrac{\omega\left(\frac{1-\alpha_1-\alpha_2}{\alpha_2}\right)}{\alpha_1 \alpha_2^2}\mbox{d}\alpha_2\mbox{d}\alpha_1 <\dfrac{0.0345 y}{x^b \log(x)},\\
& \sum_{n \in \mathcal{B}(y)} \Theta_5(n) \leq \dfrac{y(1+o(1))}{x^b \log(x)}\int_{0.42}^{0.455} \int_{(0.685-\alpha_1)/2}^{0.58-\alpha_1}\dfrac{\omega\left(\frac{1-\alpha_1-\alpha_2}{\alpha_2}\right)}{\alpha_1 \alpha_2^2}\mbox{d}\alpha_2\mbox{d}\alpha_1 <\dfrac{0.0502 y}{x^b \log(x)},\\
&\sum_{n \in \mathcal{B}(y)} \Theta_6(n) \leq  \dfrac{y(1+o(1))}{x^b\log(x)}\int_{0.42}^{0.455} \int_{0.105}^{ (0.685-\alpha_1)/2} \!\!\! \int_{0.105}^{\alpha_2} \int_{0.105}^{\min\{\alpha_3,\frac{(1-\sum_{i=1}^3\alpha_i)}{2}\}} \dfrac{\omega\left(\frac{1-\sum_{i=1}^4 \alpha_i}{\alpha_4}\right)}{\alpha_1 \alpha_2 \alpha_3 \alpha_4^2}\mbox{d}\underline{\alpha} <\dfrac{0.0004 y}{x^b \log(x)}. 
\end{align*}
\endgroup 

 \vspace{3mm}
 By condition (IV)(iv) of Corollary~\ref{cor:computations}, $\Upsilon_3(n)$ can be written as 
\begingroup
\allowdisplaybreaks
\begin{align}
\Upsilon_3(n) &= \Psi_3(n)+ \sum_{\substack{n = p_1p_2 m \\ x^{0.075} \leq p_2 < p_1 < X  \\ 0.315 < \ell_1^* \leq 0.38\\ \ell_2^* > 0.62-\ell_1^*}} \psi(m,p_2)  +    \sum_{\substack{n = p_1p_2 m \\ x^{0.075} \leq p_2 < p_1 < X  \\ 0.315 < \ell_1^* \leq 0.38\\ \ell_2^* \in [0.38-\ell_1^*, 0.455-\ell_1^*) \cup (0.545-\ell_1^*, 0.62-\ell_1^*]}} \!\!\!\!\!\!\!\!\!\! \psi(m,p_2) \label{line32}  \\
&+ \sum_{\substack{n = p_1p_2 m \\ x^{0.075} \leq p_2 < p_1 < X  \\ 0.315 < \ell_1^* \leq 0.38\\ \ell_2^* \in [0.455-\ell_1^*, 0.545-\ell_1^*] \\ \ell_2^* > \max\{0.145,(0.62-\ell_1^*)/2\}}} \!\!\!\!\!\!\!\!\!\!\!\!\!\!{ \psi(m,p_2)  }+ \sum_{\substack{n = p_1p_2p_3p_4 m \\ x^{0.075} \leq p_4<p_3<p_2 < p_1 < X  \\ 0.315 < \ell_1^* \leq 0.38\\ \ell_2^* \in [0.455-\ell_1^*, 0.545-\ell_1^*] \\ \ell_2^* \leq  \max\{0.145,(0.62-\ell_1^*)/2\} \\ \ell_4^* > 0.455-\ell_1^*}}\!\!\!\!\!\!\!\!\!\!\!\!\!\! \psi(m,p_4)+ \sum_{\substack{n = p_1p_2p_3p_4 m \\ x^{0.075} \leq p_4<p_3<p_2 < p_1 < X  \\ 0.315 < \ell_1^* \leq 0.38\\ \ell_2^* \in [0.455-\ell_1^*, 0.545-\ell_1^*] \\ \ell_2^* \leq \max\{0.145,(0.62-\ell_1^*)/2\} \\ \ell_4^* \leq 0.455-\ell_1^*}} \!\!\!\!\!\!\!\!\!\!\!\!\!\!\psi(m,p_4), \label{line33}
\end{align}
\endgroup
where $\Psi_3(n)$ is a sum of four good$^\star$ functions, resulting from Buchstab iterations applied to region $(0.315 < \ell_1^* \leq 0.38,  \ell_2^* \in [0.455-\ell_1^*, 0.545-\ell_1^*], \ell_2^* \leq \max\{0.145,(0.62-\ell_1^*)/2\})$. 
Denote the first big sum in (\ref{line32}) by $\Theta_7(n)$. By condition (D)(c) of Corollary~\ref{cor:computations}, the second big sum in (\ref{line32}) is good$^{\star\star}$. Denote the first  sum in (\ref{line33}) by $\Theta_8(n)$ and the second by $\Theta_9(n)$. By condition (D)(c) of Corollary~\ref{cor:computations}, the third sum in (\ref{line33}) is also good$^{\star\star}$.  Upper bounds for $\sum_{n \in \mathcal{B}(y)} \Theta_i(n)$, $i \in \{7,8,9\}$, 
 are given below:
\begin{align*}
&\dfrac{y(1+o(1))}{x^b \log(x)}\int_{0.315}^{0.38} \int_{0.62-\alpha_1}^{\min\{\alpha_1,(1-\alpha_1)/2\}}\dfrac{\omega\left(\frac{1-\alpha_1-\alpha_2}{\alpha_2}\right)}{\alpha_1 \alpha_2^2}\mbox{d}\alpha_2\mbox{d}\alpha_1 < \dfrac{0.0889 y}{x^b \log(x)},\\
& \dfrac{y(1+o(1))}{x^b \log(x)}\int_{0.315}^{0.38} \int_{\max\{0.145,(0.62-\alpha_1)/2\}}^{0.545-\alpha_1}\dfrac{\omega\left(\frac{1-\alpha_1-\alpha_2}{\alpha_2}\right)}{\alpha_1 \alpha_2^2}\mbox{d}\alpha_2\mbox{d}\alpha_1 < \dfrac{0.1993 y}{x^b \log(x)},\\ 
& \dfrac{y(1+o(1))}{x^b \log(x)}\int_{0.315}^{0.38} \int_{0.455-\alpha_1}^{\max\{0.145,(0.62-\alpha_1)/2\}}\!\! \int_{0.455-\alpha_1}^{\alpha_2} \int_{0.455-\alpha_1}^{\min\{\alpha_3,\frac{(1-\sum_{i=1}^3\alpha_i)}{2}\}} \dfrac{\omega\left(\frac{1-\sum_{i=1}^4 \alpha_i}{\alpha_4}\right)}{\alpha_1 \alpha_2 \alpha_3 \alpha_4^2}\mbox{d}\underline{\alpha} <\dfrac{0.0114 y}{x^b \log(x)}.
\end{align*}

\vspace{3mm}
By condition (IV)(v) of Corollary~\ref{cor:computations}, $\Upsilon_4(n)$  equals
\begingroup
\allowdisplaybreaks
\begin{align}
&\Psi_4(n)+ \sum_{\substack{n = p_1p_2 m \\ x^{0.075} \leq p_2 < p_1 < X  \\ 0.29 < \ell_1^* \leq 0.315 \\ \ell_2^* > 0.62-\ell_1^* }} \psi(m,p_2) + \sum_{\substack{n = p_1p_2 m \\ x^{0.075} \leq p_2 < p_1 < X  \\ 0.29 < \ell_1^* \leq 0.315  \\ \ell_2^* \in [0.58-\ell_1^*, 0.62-\ell_1^*] }} \psi(m,p_2) +  \sum_{\substack{n = p_1p_2 m \\ x^{0.075} \leq p_2 < p_1 < X  \\ 0.29 < \ell_1^* \leq 0.315 \\ \ell_2^* \in [(0.62-\ell_1^*)/2, 0.58-\ell_1^*) }} \psi(m,p_2)  \label{line34}\\
&+  \sum_{\substack{n = p_1p_2p_3p_4 m \\ x^{0.075} \leq p_4<p_3p_2 < p_1 < X  \\ 0.29 < \ell_1^* \leq 0.315 \\ \ell_2^* < (0.62-\ell_1^*)/2 \\ \ell_4^* > 0.42-\ell_1^* }} { \psi(m,p_4)} +  \sum_{\substack{n = p_1p_2p_3p_4 m \\ x^{0.075} \leq p_4<p_3p_2 < p_1 < X  \\ 0.29 < \ell_1^* \leq 0.315 \\ \ell_2^* < (0.62-\ell_1^*)/2 \\ \ell_4^* \in [0.38-\ell_1^*, 0.42-\ell_1^*] }} { \psi(m,p_4)} +  \sum_{\substack{n = p_1p_2p_3p_4 m \\ x^{0.075} \leq p_4<p_3p_2 < p_1 < X  \\ 0.29 < \ell_1^* \leq 0.315 \\ \ell_2^* < (0.62-\ell_1^*)/2 \\\ell_4^* < 0.38-\ell_1^*  }} { \psi(m,p_4)}. \label{line35}
\end{align}
\endgroup
for some $\Psi_4(n)$ which is a sum of two good$^\star$ functions. 
Denote the first big sum in (\ref{line34}) by $\Theta_{10}(n)$ and the third by $\Theta_{11}(n)$. By condition (D)(d) of Corollary~\ref{cor:computations}, the second big sum in (\ref{line34}) is good$^{\star\star}$. Denote the first  sum in (\ref{line35}) by $\Theta_{12}(n)$ and the third by $\Theta_{13}(n)$. By condition (D)(d) of Corollary~\ref{cor:computations}, the second sum in (\ref{line35}) is also good$^{\star\star}$.  Upper bounds for $\sum_{n \in \mathcal{B}(y)} \Theta_i(n)$, $i \in \{10,11,12,13\}$, 
 are given below:
\begingroup
\allowdisplaybreaks
\begin{align*}
& \dfrac{y(1+o(1))}{x^b \log(x)}\int_{0.31}^{0.315} \int_{0.62-\alpha_1}^{\alpha_1}\dfrac{\omega\left(\frac{1-\alpha_1-\alpha_2}{\alpha_2}\right)}{\alpha_1 \alpha_2^2}\mbox{d}\alpha_2\mbox{d}\alpha_1 < \dfrac{0.0007 y}{x^b \log(x)},\\
& \dfrac{y(1+o(1))}{x^b \log(x)}\int_{0.29}^{0.315} \int_{(0.62-\alpha_1)/2}^{0.58-\alpha_1}\dfrac{\omega\left(\frac{1-\alpha_1-\alpha_2}{\alpha_2}\right)}{\alpha_1 \alpha_2^2}\mbox{d}\alpha_2\mbox{d}\alpha_1 <\dfrac{0.1343 y}{x^b \log(x)},\\
& \dfrac{y(1+o(1))}{x^b \log(x)}\int_{0.29}^{0.315} \int_{0.42-\alpha_1}^{(0.62-\alpha_1)/2} \int_{0.42-\alpha_1}^{\alpha_2} \int_{0.42-\alpha_1}^{\min\{\alpha_3,\frac{(1-\sum_{i=1}^3\alpha_i)}{2}\}} \dfrac{\omega\left(\frac{1-\sum_{i=1}^4 \alpha_i}{\alpha_4}\right)}{\alpha_1 \alpha_2 \alpha_3 \alpha_4^2}\mbox{d}\underline{\alpha} <\dfrac{0.0020 y}{x^b \log(x)},\\ 
& \dfrac{y(1+o(1))}{x^b \log(x)}\int_{0.29}^{0.305} \int_{0.075}^{(0.62-\alpha_1)/2} \int_{0.075}^{\alpha_2} \int_{0.075}^{\min\{0.38-\alpha_1,\alpha_3,\frac{(1-\sum_{i=1}^3\alpha_i)}{2}\}} \dfrac{\omega\left(\frac{1-\sum_{i=1}^4 \alpha_i}{\alpha_4}\right)}{\alpha_1 \alpha_2 \alpha_3 \alpha_4^2}\mbox{d}\underline{\alpha}<\dfrac{0.0092 y}{x^b \log(x)}.
\end{align*}
\endgroup

\vspace{3mm}
Finally, using conditions (IV)(vi), (IV)(vii) and (IV)(viii) of Corollary~\ref{cor:computations}, we find that $\Upsilon_5(n)$ equals
\begingroup 
\allowdisplaybreaks
\begin{align}
\Psi_5(n)&+  \sum_{\substack{n = p_1p_2 m \\ x^{0.075} \leq p_2 < p_1 < X  \\ \ell_1^* \leq 0.29  \\  \ell_2^* > \min\{0.2275,(0.685-\ell_1^*)/2\}}} \psi(m,p_2) 
+  \sum_{\substack{n = p_1p_2p_3p_4 m \\ x^{0.075} \leq p_4<p_3<p_2 < p_1 < X  \\ \ell_1^* \leq 0.29  \\  \ell_2^* \leq \min\{0.2275,(0.685-\ell_1^*)/2\} \\ \ell_4^* > (0.62-\ell_1^*-\ell_2^*-\ell_3^*)/2 \\
\ell_2^* > 0.455-\ell_1^* \\ \ell_3^*<0.38-\ell_1^* }} \psi(m,p_4)  \label{line36}\\
&+ \sum_{\substack{n = p_1p_2p_3p_4 m \\ x^{0.075} \leq p_4<p_3<p_2 < p_1 < X  \\ \ell_1^* \leq 0.29  \\  \ell_2^* \leq \min\{0.2275,(0.685-\ell_1^*)/2\} \\ \ell_4^* > (0.62-\ell_1^*-\ell_2^*-\ell_3^*)/2 \\
\ell_2^* > 0.455-\ell_1^* \\ \ell_3^* \in [0.38-\ell_1^*,0.42-\ell_1^*]  }} \psi(m,p_4) +
 \sum_{\substack{n = p_1p_2p_3p_4 m \\ x^{0.075} \leq p_4<p_3<p_2 < p_1 < X  \\ \ell_1^* \leq 0.29  \\  \ell_2^* \leq \min\{0.2275,(0.685-\ell_1^*)/2\} \\ \ell_4^* > (0.62-\ell_1^*-\ell_2^*-\ell_3^*)/2 \\
\ell_2^* > 0.455-\ell_1^* \\  \ell_3^*>0.42-\ell_1^* \\ \ell_4^*<0.38-\ell_1^*}} \psi(m,p_4)    \label{line37} \\
&+ \sum_{\substack{n = p_1p_2p_3p_4 m \\ x^{0.075} \leq p_4<p_3<p_2 < p_1 < X  \\ \ell_1^* \leq 0.29  \\  \ell_2^* \leq \min\{0.2275,(0.685-\ell_1^*)/2\} \\ \ell_4^* > (0.62-\ell_1^*-\ell_2^*-\ell_3^*)/2 \\
\ell_2^* > 0.455-\ell_1^* \\ \ell_3^*>0.42-\ell_1^* \\ \ell_4^*\in [0.38-\ell_1^*,0.42-\ell_1^*]}} \psi(m,p_4)  + \sum_{\substack{n = p_1p_2p_3p_4 m \\ x^{0.075} \leq p_4<p_3<p_2 < p_1 < X  \\ \ell_1^* \leq 0.29  \\  \ell_2^* \leq \min\{0.2275,(0.685-\ell_1^*)/2\} \\ \ell_4^* > (0.62-\ell_1^*-\ell_2^*-\ell_3^*)/2 \\
\ell_2^* > 0.455-\ell_1^* \\ \ell_3^*>0.42-\ell_1^* \\ \ell_4^*>0.42-\ell_1^* }} \psi(m,p_4) \label{line38}\\
&+ \sum_{\substack{n = p_1p_2p_3p_4 m \\ x^{0.075} \leq p_4<p_3<p_2 < p_1 < X  \\ \ell_1^* \leq 0.29  \\  \ell_2^* \leq \min\{0.2275,(0.685-\ell_1^*)/2\} \\ \ell_4^* > (0.62-\ell_1^*-\ell_2^*-\ell_3^*)/2 \\
\ell_2^* \in [0.42-\ell_1^*, 0.455-\ell_1^*] \\ \ell_4^*>0.42-\ell_1^* \\ \ell_4^* \geq 0.105 }} \psi(m,p_4)   + 
\sum_{\substack{n = p_1p_2p_3p_4 m \\ x^{0.075} \leq p_4<p_3<p_2 < p_1 < X  \\ \ell_1^* \leq 0.29  \\  \ell_2^* \leq \min\{0.2275,(0.685-\ell_1^*)/2\} \\ \ell_4^* > (0.62-\ell_1^*-\ell_2^*-\ell_3^*)/2 \\
\ell_2^* \in [0.42-\ell_1^*, 0.455-\ell_1^*] \\ \ell_4^* \in [0.315-\ell_1^*,0.42-\ell_1^*] \\ \ell_4^* \geq 0.105 }} \psi(m,p_4)   
\label{line39}\\
&+ 
\sum_{\substack{n = p_1p_2p_3p_4 m \\ x^{0.075} \leq p_4<p_3<p_2 < p_1 < X  \\ \ell_1^* \leq 0.29  \\  \ell_2^* \leq \min\{0.2275,(0.685-\ell_1^*)/2\} \\ \ell_4^* > (0.62-\ell_1^*-\ell_2^*-\ell_3^*)/2 \\
\ell_2^* \in (0.38-\ell_1^*, 0.42-\ell_1^*)}} \psi(m,p_4) + 
\sum_{\substack{n = p_1p_2p_3p_4 m \\ x^{0.075} \leq p_4<p_3<p_2 < p_1 < X  \\ \ell_1^* \leq 0.29  \\  \ell_2^* \leq \min\{0.2275,(0.685-\ell_1^*)/2\} \\ \ell_4^* > (0.62-\ell_1^*-\ell_2^*-\ell_3^*)/2 \\
\ell_2^* \in [0.315-\ell_1^*, 0.38-\ell_1^*] \\ \ell_4^*>0.455-\ell_2^*-\ell_3^*}} \psi(m,p_4) 
\label{line40}\\
&+ 
\sum_{\substack{n = p_1p_2p_3p_4 m \\ x^{0.075} \leq p_4<p_3<p_2 < p_1 < X  \\ \ell_1^* \leq 0.29  \\  \ell_2^* \leq \min\{0.2275,(0.685-\ell_1^*)/2\} \\ \ell_4^* > (0.62-\ell_1^*-\ell_2^*-\ell_3^*)/2 \\
\ell_2^* \in [0.315-\ell_1^*, 0.38-\ell_1^*] \\ \ell_4^*\in[0.38-\ell_2^*-\ell_3^*,0.455-\ell_2^*-\ell_3^*]}} \psi(m,p_4)
+\sum_{\substack{n = p_1p_2p_3p_4 m \\ x^{0.075} \leq p_4<p_3<p_2 < p_1 < X  \\ \ell_1^* \leq 0.29  \\  \ell_2^* \leq \min\{0.2275,(0.685-\ell_1^*)/2\} \\ \ell_4^* > (0.62-\ell_1^*-\ell_2^*-\ell_3^*)/2 \\
\ell_2^* \in [0.315-\ell_1^*, 0.38-\ell_1^*] \\ \ell_4^*\in (0.455-\ell_1^*-\ell_2^*, 0.38-\ell_2^*-\ell_3^*)}} \psi(m,p_4)  
\label{line41} \\
&+\sum_{\substack{n = p_1p_2p_3p_4 m \\ x^{0.075} \leq p_4<p_3<p_2 < p_1 < X  \\ \ell_1^* \leq 0.29  \\  \ell_2^* \leq \min\{0.2275,(0.685-\ell_1^*)/2\} \\ \ell_4^* > (0.62-\ell_1^*-\ell_2^*-\ell_3^*)/2 \\
\ell_2^* \in [0.315-\ell_1^*, 0.38-\ell_1^*] \\  \ell_4^* \in [0.38-\ell_1^*-\ell_2^*, \min\{0.455-\ell_1^*-\ell_2^*, 0.38-\ell_2^*-\ell_3^* \}]}} \!\!\!\!\!\!\!\!\! \psi(m,p_4)  
+\sum_{\substack{n = p_1p_2p_3p_4 m \\ x^{0.075} \leq p_4<p_3<p_2 < p_1 < X  \\ \ell_1^* \leq 0.29  \\  \ell_2^* \leq \min\{0.2275,(0.685-\ell_1^*)/2\} \\ \ell_4^* > (0.62-\ell_1^*-\ell_2^*-\ell_3^*)/2 \\  \ell_2^* < 0.315-\ell_1^*  }}  \psi(m,p_4)   
 \label{line42}\\
&+\sum_{\substack{n = p_1p_2p_3p_4p_5p_6 m \\ x^{0.075} \leq p_6<p_5<p_4<p_3<p_2 < p_1 < X  \\ \ell_1^* \leq 0.29  \\  \ell_2^* \leq \min\{0.2275,(0.685-\ell_1^*)/2\} \\ \ell_4^* \leq (0.62-\ell_1^*-\ell_2^*-\ell_3^*)/2  }}  \psi(m,p_6), \label{line43}
\end{align}
\endgroup
where $\Psi_5(n)$ is a sum of nine good$^\star$ functions.   The first sum in (\ref{line37}) and the first sum in (\ref{line38}) are good$^{\star\star}$ by condition (D)(d) of Corollary~\ref{cor:computations}. The second sum in (\ref{line39}) is good$^{\star\star}$ by condition (D)(e) of Corollary~\ref{cor:computations}. The first sum in (\ref{line40}) is also good$^{\star\star}$ by condition (D)(d). The first sum in (\ref{line41}) is  good$^{\star\star}$ by condition (D)(g) and the first sum in (\ref{line42}) is  good$^{\star\star}$ by condition (D)(f). Enumerate the remaining nine big sums as $\Theta_{14}(n), \dots, \Theta_{22}(n)$. Upper bounds for $\sum_{n \in \mathcal{B}(y)} \Theta_i(n)$, $i \in \{14, \dots, 22\}$, are given below:
\begingroup
\allowdisplaybreaks
\begin{align*}
& \dfrac{y(1+o(1))}{x^b \log(x)}\int_{0.075}^{0.29} \int_{\min\{0.2275,(0.685-\ell_1^*)/2\}}^{\alpha_1}\dfrac{\omega\left(\frac{1-\alpha_1-\alpha_2}{\alpha_2}\right)}{\alpha_1 \alpha_2^2}\mbox{d}\alpha_2\mbox{d}\alpha_1 < \dfrac{0.1145 y}{x^b \log(x)},\\
& \dfrac{y(1+o(1))}{x^b \log(x)}\int_{0.075}^{0.29} \int_{0.455-\alpha_1}^{\min\{\alpha_1, \frac{(0.685-\alpha_1)}{2}\}}\!\!\! \int_{0.075}^{ 0.38-\alpha_1} \!\!\!\int_{\max\{0.075,\frac{(0.62-\sum_{i=1}^3\alpha_i)}{2}\}}^{\min\{\alpha_3,\frac{(1-\sum_{i=1}^3\alpha_i)}{2}\}} \dfrac{\omega\left(\frac{1-\sum_{i=1}^4 \alpha_i}{\alpha_4}\right)}{\alpha_1 \alpha_2 \alpha_3 \alpha_4^2}\mbox{d}\underline{\alpha} <\dfrac{0.0116 y}{x^b \log(x)},\\
& \dfrac{y(1+o(1))}{x^b \log(x)}\!\!\int_{0.075}^{0.29} \!\int_{0.455-\alpha_1}^{\min\{\alpha_1, \frac{(0.685-\alpha_1)}{2}\}}\!\!\!\! \int_{0.42-\alpha_1}^{ \alpha_2} \!\int_{\max\{0.075,\frac{(0.62-\sum_{i=1}^3\alpha_i)}{2}\}}^{\min\{\alpha_3,0.38-\alpha_1,\frac{(1-\sum_{i=1}^3\alpha_i)}{2}\}} \dfrac{\omega\!\left(\frac{1-\sum_{i=1}^4 \alpha_i}{\alpha_4}\right)}{\alpha_1 \alpha_2 \alpha_3 \alpha_4^2}\mbox{d}\underline{\alpha} <\dfrac{0.0129 y}{x^b \log(x)},\\
& \dfrac{y(1+o(1))}{x^b\log(x)}\!\!\int_{0.075}^{0.29} \!\int_{0.455-\alpha_1}^{\min\{\alpha_1, \frac{(0.685-\alpha_1)}{2}\}}\!\!\!\! \int_{0.42-\alpha_1}^{ \alpha_2} \!\int_{\max\{0.42-\alpha_1,\frac{(0.62-\sum_{i=1}^3\alpha_i)}{2}\}}^{\min\{\alpha_3,\frac{(1-\sum_{i=1}^3\alpha_i)}{2}\}} \dfrac{\omega\!\left(\frac{1-\sum_{i=1}^4 \alpha_i}{\alpha_4}\right)}{\alpha_1 \alpha_2 \alpha_3 \alpha_4^2}\mbox{d}\underline{\alpha} <\dfrac{0.0027 y}{x^b\log(x)},\\
& \dfrac{y(1+o(1))}{x^b \log(x)}\int_{0.105}^{0.29} \!\int_{0.42-\alpha_1}^{\min\{\alpha_1,0.455-\alpha_1\}} \int_{0.105}^{ \alpha_2} \int_{\max\{0.42-\alpha_1,\frac{(0.62-\sum_{i=1}^3\alpha_i)}{2}\}}^{\min\{\alpha_3,\frac{(1-\sum_{i=1}^3\alpha_i)}{2}\}} \dfrac{\omega\!\left(\frac{1-\sum_{i=1}^4 \alpha_i}{\alpha_4}\right)}{\alpha_1 \alpha_2 \alpha_3 \alpha_4^2}\mbox{d}\underline{\alpha} <\dfrac{0.0012 y}{x^b \log(x)},\\
& \dfrac{y(1+o(1))}{x^b \log(x)}\int_{0.075}^{0.29} \int_{\max\{0.075,0.315-\alpha_1\}}^{\min\{\alpha_1,0.38-\alpha_1\}} \int_{0.075}^{ \alpha_2} \int_{0.455-\alpha_2-\alpha_3}^{\min\{\alpha_3,\frac{(1-\sum_{i=1}^3\alpha_i)}{2}\}} \dfrac{\omega\!\left(\frac{1-\sum_{i=1}^4 \alpha_i}{\alpha_4}\right)}{\alpha_1 \alpha_2 \alpha_3 \alpha_4^2}\mbox{d}\underline{\alpha} < \dfrac{0.0048 y}{x^b \log(x)},\\
& \dfrac{y(1+o(1))}{x^b \log(x)}\int_{0.075}^{0.29} \int_{\max\{0.075,0.315-\alpha_1\}}^{\min\{\alpha_1,0.38-\alpha_1\}} \int_{0.075}^{ \alpha_2} \int_{0.455-\alpha_1-\alpha_2}^{\min\{\alpha_3,0.38-\alpha_2-\alpha_3\}} \dfrac{\omega\!\left(\frac{1-\sum_{i=1}^4 \alpha_i}{\alpha_4}\right)}{\alpha_1 \alpha_2 \alpha_3 \alpha_4^2}\mbox{d}\underline{\alpha} <\dfrac{0.0252 y}{x^b \log(x)},\\
& \dfrac{y(1+o(1))}{x^b \log(x)}\int_{0.075}^{0.29} \!\int_{0.075}^{\min\{\alpha_1,0.315-\alpha_1\}} \int_{0.075}^{ \alpha_2} \int_{\max\{0.075,\frac{(0.62-\sum_{i=1}^3\alpha_i)}{2}\}}^{\min\{\alpha_3,\frac{(1-\sum_{i=1}^3\alpha_i)}{2}\}} \dfrac{\omega\!\left(\frac{1-\sum_{i=1}^4 \alpha_i}{\alpha_4}\right)}{\alpha_1 \alpha_2 \alpha_3 \alpha_4^2}\mbox{d}\underline{\alpha} <\dfrac{0.0127 y}{x^b \log(x)},\\
& \dfrac{y(1+o(1))}{x^b \log(x)}\left(  \dfrac{\log\left( \frac{0.16}{0.075} \right)^6 }{0.075 \cdot 6!}  + \dfrac{\log\left( \frac{0.16}{0.075} \right)^5\log\left( \frac{0.29}{0.16}\right)  }{0.075 \cdot 5!}  +  \dfrac{\log\left( \frac{0.16}{0.075} \right)^4\log\left( \frac{0.29}{0.16}\right)^2  }{0.075 \cdot 4! 2!} \right)<\dfrac{0.0524 y}{x^b\log(x)}.
\end{align*} 
\endgroup
 In summary, we have shown that $1_{\mathbb{P}}(n)$ can be written as $1_{\mathbb{P}}(n) = \Psi(n) + \sum_{i=1}^{22} \Theta_i(n)$, where $\Psi(n)$ is a sum of no more than $100$ good$^\star$ and good$^{\star\star}$ functions and where  each $\Theta_i$ is non-negative and 
  \begin{align*}
 \sum_{i=1}^{22} \sum_{n \in \mathcal{B}(y)} \Theta_i(n) < \dfrac{ 0.9855y}{x^b\log(x)}.
  \end{align*}
This concludes the proof of Proposition~\ref{proposition3} for the case $a \in (0.57,0.59]$.

\subsubsection{$a \in (0.59, 0.61]$}

The treatment of $a\in (0.59,0.61]$ follows almost the exact same steps as the  previous case.  After two applications of the Buchstab identity with $z_1 = x^{0.07}$,  we have
\begingroup
\allowdisplaybreaks
\begin{align}
\psi(n,X) &= \psi(n,x^{0.07}) - \sum_{\substack{n = p_1 m \\ x^{0.07} \leq p_1 < X }} \psi(m,x^{0.07}) \label{a44} \\ &+  \sum_{\substack{n = p_1p_2 m \\ x^{0.07} \leq p_2<p_1 < X \\  \ell_1^* > 0.435}} \psi(m,p_2) 
+    \sum_{\substack{n = p_1p_2 m \\ x^{0.07} \leq p_2<p_1 < X  \\ 0.42 < \ell_1^* \leq  0.435 }} \psi(m,p_2) +  \sum_{\substack{n = p_1p_2 m \\ x^{0.07} \leq p_2<p_1 < X \\  0.365 < \ell_1^* \leq  0.42 }} \psi(m,p_2) \label{a45} \\ &+ \sum_{\substack{n = p_1p_2 m \\ x^{0.07} \leq p_2<p_1 < X \\ 0.33 < \ell_1^* \leq  0.365 }} \psi(m,p_2)  
+ \sum_{\substack{n = p_1p_2 m \\ x^{0.07} \leq p_2<p_1 < X \\  0.305 < \ell_1^* \leq  0.33 }} \psi(m,p_2) +  \sum_{\substack{ n = p_1p_2 m \\ x^{0.07} \leq p_2<p_1 < X \\  \ell_1^* \leq  0.305 }} \psi(m,p_2). \label{a46}
\end{align}
\endgroup
 By condition (V)(i) of Corollary~\ref{cor:computations}, both $\psi(n, x^{0.07})$ and the sum $- \sum_{n = p_1 m, \, 0.07 \leq \ell_1^* \leq 0.5+\varepsilon_1 } \psi(m,x^{0.07}) $ in (\ref{a44}) are good$^\star$ functions.  Denote the first sum in  (\ref{a45}) by $\Upsilon_1(n)$ and the second sum by $\Upsilon_2(n)$.   The third sum in (\ref{a45})  is a  good$^{\star\star}$ function  by condition (E)(a) of Corollary~\ref{cor:computations}.  Denote the first sum in  (\ref{a46}) by $\Upsilon_3(n)$, the second  by $\Upsilon_4(n)$ and the third by $\Upsilon_5(n)$. We now treat $\Upsilon_1(n)$, $\Upsilon_2(n)$, $\Upsilon_3(n)$, $\Upsilon_4(n)$ and $\Upsilon_5(n)$ further, beginning with $\Upsilon_1(n)$.

\vspace{3mm} Applying two more Buchstab iterations to parts of $\Upsilon_1(n)$,  condition (V)(ii) of Corollary~\ref{cor:computations} tells us that 
\begingroup
\allowdisplaybreaks
\begin{align}
&\Upsilon_1(n) = \Psi_1(n)+  \sum_{\substack{n = p_1p_2 m \\ x^{0.07} \leq p_2<p_1 < X \\  \ell_1^* > 0.435 \\ \ell_2^* > 0.635-\ell_1^*}} \psi(m,p_2)   +  \sum_{\substack{n = p_1p_2 m \\ x^{0.07} \leq p_2<p_1 < X \\ \ell_1^* > 0.435 \\ \ell_2^* \in [0.58-\ell_1^*, 0.635-\ell_1^*] }} \psi(m,p_2) \label{a47} \\
&+ \sum_{\substack{n = p_1p_2 m \\ x^{0.07} \leq p_2<p_1 < X \\ \ell_1^* > 0.435  \\ \ell_2^* < 0.58-\ell_1^* \\ \ell_2^* \geq \min\{0.105,(0.67-\ell_1^*)/2\}}} \psi(m,p_2) +  \sum_{\substack{n = p_1p_2p_3p_4 m \\ x^{0.07} \leq p_4<p_3<p_2<p_1 < X \\ \ell_1^* > 0.435  \\ \ell_2^* < \min\{0.58-\ell_1^*,0.105, (0.67-\ell_1^*)/2\}}} \psi(m,p_4),  \label{a48} 
\end{align}
\endgroup
where $\Psi_1(n)$ is  a sum of two good$^\star$ functions. 
Denote the first big sum in (\ref{a47}) by $\Theta_1(n)$. The second big sum in (\ref{a47}) is good$^{\star\star}$ by  condition (E)(a) of Corollary~\ref{cor:computations}.  Denote the first sum in (\ref{a48}) by $\Theta_2(n)$ and the second by $\Theta_3(n)$.  Upper bounds for $\sum_{n \in \mathcal{B}(y)} \Theta_i(n)$, $i \in \{1,2,3\}$,
 are given below:
\begingroup
\allowdisplaybreaks
\begin{align*}
& \dfrac{y(1+o(1))}{x^b \log(x)}\int_{0.435}^{0.5+\varepsilon_1} \int_{0.635-\alpha_1}^{(1-\alpha_1)/2}\dfrac{\omega\left(\frac{1-\alpha_1-\alpha_2}{\alpha_2}\right)}{\alpha_1 \alpha_2^2}\mbox{d}\alpha_2\mbox{d}\alpha_1 <\dfrac{0.2182 y}{x^b \log(x)},\\
& \dfrac{y(1+o(1))}{x^b \log(x)}\int_{0.435}^{0.5+\varepsilon_1} \int_{\min\{0.105,(0.67-\alpha_1)/2\}}^{0.58-\alpha_1}\dfrac{\omega\left(\frac{1-\alpha_1-\alpha_2}{\alpha_2}\right)}{\alpha_1 \alpha_2^2}\mbox{d}\alpha_2\mbox{d}\alpha_1 <\dfrac{0.0921 y}{x^b \log(x)},\\
& \dfrac{y(1+o(1))}{x^b \log(x)}\!\!\int_{0.435}^{0.5+\varepsilon_1} \int_{0.07}^{\min\{0.58-\alpha_1, 0.105,(0.67-\alpha_1)/2\}} \!\!\!\!\int_{0.07}^{\alpha_2} \!\int_{0.07}^{\min\{\alpha_3,\frac{(1-\sum_{i=1}^3\alpha_i)}{2}\}}\! \dfrac{\omega\left(\frac{1-\sum_{i=1}^4 \alpha_i}{\alpha_4}\right)}{\alpha_1 \alpha_2 \alpha_3 \alpha_4^2}\mbox{d}\underline{\alpha} <\dfrac{0.0083 y}{x^b \log(x)}.
\end{align*}
\endgroup
\vspace{3mm}

 Next we apply two more Buchstab iterations to parts of $\Upsilon_2(n)$, one with $z_1=x^{0.07}$ and one with $z_1 = x^{0.09}$. Noting that condition (V)(iii) of Corollary~\ref{cor:computations} allows $\beta \leq 0.09$, we find that 
\begingroup
\allowdisplaybreaks
\begin{align}
\Upsilon_2(n) &= \Psi_2(n)+ \sum_{\substack{n = p_1p_2 m \\ x^{0.07} \leq p_2<p_1 < X \\  0.42 < \ell_1^* \leq  0.435\\ \ell_2^* \geq 0.67-\ell_1^* }} \psi(m,p_2) + \sum_{\substack{n = p_1p_2 m \\ x^{0.07} \leq p_2<p_1 < X \\  0.42 < \ell_1^* \leq  0.435\\ \ell_2^* \in (0.58-\ell_1^*, 0.67-\ell_1^*)}} \psi(m,p_2) \label{a49} \\
&+\sum_{\substack{n = p_1p_2 m \\ x^{0.07} \leq p_2<p_1 < X \\  0.42 < \ell_1^* \leq  0.435\\  \ell_2^* \in [(0.67-\ell_1^*)/2, 0.58-\ell_1^*] }} \psi(m,p_2) + \sum_{\substack{n = p_1p_2p_3p_4 m \\ x^{0.07} \leq p_4<p_3< p_2<p_1 < X \\  0.42 < \ell_1^* \leq  0.435\\ \ell_2^* <  (0.67-\ell_1^*)/2 \\ \ell_4^* \geq 0.09}} \psi(m,p_4), \label{a50}
\end{align} 
\endgroup
where $\Psi_2(n)$ is a sum of three good$^\star$ functions.  Denote the first big sum in (\ref{a49}) by $\Theta_4(n)$. The second big sum in (\ref{a49}) is good$^{\star\star}$ by  condition (E)(b) of Corollary~\ref{cor:computations}.  Denote the first sum in (\ref{a50}) by $\Theta_5(n)$ and the second by $\Theta_6(n)$.  Upper bounds for $\sum_{n \in \mathcal{B}(y)} \Theta_i(n)$, $i \in \{4,5,6\}$,
 are given below:
\begingroup
\allowdisplaybreaks
\begin{align*}
& \dfrac{y(1+o(1))}{x^b \log(x)}\int_{0.42}^{0.435} \int_{0.67-\alpha_1}^{(1-\alpha_1)/2}\dfrac{\omega\left(\frac{1-\alpha_1-\alpha_2}{\alpha_2}\right)}{\alpha_1 \alpha_2^2}\mbox{d}\alpha_2\mbox{d}\alpha_1 < \dfrac{0.0189 y}{x^b \log(x)},\\
& \dfrac{y(1+o(1))}{x^b \log(x)}\int_{0.42}^{0.435} \int_{(0.67-\alpha_1)/2}^{0.58-\alpha_1}\dfrac{\omega\left(\frac{1-\alpha_1-\alpha_2}{\alpha_2}\right)}{\alpha_1 \alpha_2^2}\mbox{d}\alpha_2\mbox{d}\alpha_1 < \dfrac{0.0356 y}{x^b \log(x)},\\
& \dfrac{y(1+o(1))}{x^b \log(x)}\int_{0.42}^{0.435} \int_{0.09}^{ (0.67-\alpha_1)/2} \int_{0.09}^{\alpha_2} \int_{0.09}^{\min\{\alpha_3,\frac{(1-\sum_{i=1}^3\alpha_i)}{2}\}} \dfrac{\omega\left(\frac{1-\sum_{i=1}^4 \alpha_i}{\alpha_4}\right)}{\alpha_1 \alpha_2 \alpha_3 \alpha_4^2}\mbox{d}\underline{\alpha} <\dfrac{0.001 y}{x^b \log(x)}. 
\end{align*}
\endgroup
 
\vspace{3mm}
 By condition (V)(iv) of Corollary~\ref{cor:computations}, $\Upsilon_3(n)$ can be written as  
\begingroup
\allowdisplaybreaks
\begin{align}
\Upsilon_3(n) &= \Psi_3(n)+ \sum_{\substack{n = p_1p_2 m \\ x^{0.07} \leq p_2<p_1 < X \\  0.33 < \ell_1^* \leq  0.365 \\ \ell_2^* > 0.635-\ell_1^*}} \psi(m,p_2)  + \sum_{\substack{n = p_1p_2 m \\ x^{0.07} \leq p_2<p_1 < X \\  0.33 < \ell_1^* \leq  0.365 \\ \ell_2^* \in [0.365-\ell_1^*, 0.435-\ell_1^*] \cup [0.565-\ell_1^*, 0.635-\ell_1^*]}} \psi(m,p_2) \label{a51}\\
&+\!\!\!\!\sum_{\substack{n = p_1p_2 m \\ x^{0.07} \leq p_2<p_1 < X \\  0.33 < \ell_1^* \leq  0.365 \\ \ell_2^* \in [0.1524, 0.565-\ell_1^*)}}\!\!\!\! \psi(m,p_2)   + \!\!\!\! \sum_{\substack{n = p_1p_2p_3p_4 m \\ x^{0.07} \leq p_4<p_3<p_2<p_1 < X \\  0.33 < \ell_1^* \leq  0.365 \\ \ell_2^* \in (0.435-\ell_1^*, 0.1524) \\ \ell_4^* > 0.435-\ell_1^*}} \!\!\!\! \psi(m,p_4)+ \!\!\!\! \sum_{\substack{n = p_1p_2p_3p_4 m \\ x^{0.07} \leq p_4<p_3<p_2<p_1 < X \\  0.33 < \ell_1^* \leq  0.365 \\ \ell_2^* \in (0.435-\ell_1^*, 0.1524) \\ \ell_4^* \leq 0.435-\ell_1^*}}  \!\!\!\!\psi(m,p_4), \label{a52}
\end{align}
\endgroup
where $\Psi_3(n)$ is a sum of two good$^\star$ functions.  Denote the first big sum in (\ref{a51}) by $\Theta_7(n)$. The second big sum in (\ref{a51}) is good$^{\star\star}$ by  condition (E)(c) of Corollary~\ref{cor:computations}.  Denote the first sum in (\ref{a52}) by $\Theta_8(n)$ and the second by $\Theta_9(n)$. The third sum is again good$^{\star\star}$ by  condition (E)(c). We then have: 
\begingroup
\allowdisplaybreaks
\begin{align*}
&\sum_{n \in \mathcal{B}(y)} \Theta_7(n) \leq  \dfrac{y(1+o(1))}{x^b \log(x)}\int_{0.33}^{0.365} \int_{0.635-\alpha_1}^{\min\{\alpha_1,(1-\alpha_1)/2\}}\dfrac{\omega\left(\frac{1-\alpha_1-\alpha_2}{\alpha_2}\right)}{\alpha_1 \alpha_2^2}\mbox{d}\alpha_2\mbox{d}\alpha_1 <\dfrac{0.0367 y}{x^b \log(x)},\\
&\sum_{n \in \mathcal{B}(y)} \Theta_8(n) \leq  \dfrac{y(1+o(1))}{x^b \log(x)}\int_{0.33}^{0.365} \int_{0.1524}^{0.565-\alpha_1}\dfrac{\omega\left(\frac{1-\alpha_1-\alpha_2}{\alpha_2}\right)}{\alpha_1 \alpha_2^2}\mbox{d}\alpha_2\mbox{d}\alpha_1 <\dfrac{0.1186 y}{x^b \log(x)},\\ 
&\sum_{n \in \mathcal{B}(y)} \Theta_9(n) \leq  \dfrac{y(1+o(1))}{x^b \log(x)}\int_{0.33}^{0.365} \int_{0.435-\alpha_1}^{0.1524}\int_{0.435-\alpha_1}^{\alpha_2} \int_{0.435-\alpha_1}^{\min\{\alpha_3,\frac{(1-\sum_{i=1}^3\alpha_i)}{2}\}} \dfrac{\omega\left(\frac{1-\sum_{i=1}^4 \alpha_i}{\alpha_4}\right)}{\alpha_1 \alpha_2 \alpha_3 \alpha_4^2}\mbox{d}\underline{\alpha}<\dfrac{0.0211 y}{x^b \log(x)}.
\end{align*}
\endgroup

\vspace{3mm}
 By condition (V)(v) of Corollary~\ref{cor:computations}, $\Upsilon_4(n)$ can be written as
\begingroup
\allowdisplaybreaks
\begin{align}
\Upsilon_4(n) &= \Psi_4(n)+ \sum_{\substack{n = p_1p_2 m \\ x^{0.07} \leq p_2<p_1 < X \\  0.305 < \ell_1^* \leq  0.33 \\ \ell_2^* > 0.635-\ell_1^*}} \psi(m,p_2)  + \sum_{\substack{n = p_1p_2 m \\ x^{0.07} \leq p_2<p_1 < X \\  0.305 < \ell_1^* \leq  0.33\\ \ell_2^* \in (0.58-\ell_1^*, 0.635-\ell_1^*]  }} \psi(m,p_2)  \label{a53} \\
&+ \!\!\!\!\sum_{\substack{n = p_1p_2 m \\ x^{0.07} \leq p_2<p_1 < X \\  0.305 < \ell_1^* \leq  0.33 \\ \ell_2^* \in [(0.635-\ell_1^*)/2, 0.58-\ell_1^*] }} \!\!\!\!\psi(m,p_2)  + \!\!\!\!\sum_{\substack{n = p_1p_2p_3p_4 m \\ x^{0.07} \leq  p_4< p_3< p_2<p_1 < X \\  0.305 < \ell_1^* \leq  0.33 \\ \ell_2^* < (0.635-\ell_1^*)/2 \\ \ell_4^* > 0.42-\ell_1^* }}\!\!\!\! \psi(m,p_4)    + \!\!\!\!\sum_{\substack{n = p_1p_2p_3p_4 m \\ x^{0.07} \leq  p_4< p_3< p_2<p_1 < X \\  0.305 < \ell_1^* \leq  0.33 \\ \ell_2^* < (0.635-\ell_1^*)/2 \\ \ell_4^* \leq 0.42-\ell_1^* }} \!\!\!\!\psi(m,p_4), \label{a54}
\end{align}
\endgroup
where $\Psi_4(n)$ is a sum of two good$^\star$ functions.  Denote the first big sum in (\ref{a53}) by $\Theta_{10}(n)$. The second big sum in (\ref{a53}) is good$^{\star\star}$ by  condition (E)(a) of Corollary~\ref{cor:computations}.  Denote the first sum in (\ref{a54}) by $\Theta_{11}(n)$ and the second by $\Theta_{12}(n)$. The third sum is again good$^{\star\star}$ by  condition (E)(a). We then have: 
\begingroup
\allowdisplaybreaks
\begin{align*}
&\sum_{n \in \mathcal{B}(y)} \Theta_{10}(n) \leq   \dfrac{y(1+o(1))}{x^b \log(x)}\int_{0.305}^{0.33} \int_{0.635-\alpha_1}^{\alpha_1}\dfrac{\omega\left(\frac{1-\alpha_1-\alpha_2}{\alpha_2}\right)}{\alpha_1 \alpha_2^2}\mbox{d}\alpha_2\mbox{d}\alpha_1 < \dfrac{0.0043 y}{x^b \log(x)},\\
&\sum_{n \in \mathcal{B}(y)} \Theta_{11}(n) \leq  \dfrac{y(1+o(1))}{x^b \log(x)}\int_{0.305}^{0.33} \int_{(0.635-\alpha_1)/2}^{0.58-\alpha_1}\dfrac{\omega\left(\frac{1-\alpha_1-\alpha_2}{\alpha_2}\right)}{\alpha_1 \alpha_2^2}\mbox{d}\alpha_2\mbox{d}\alpha_1 <\dfrac{0.1178 y}{x^b \log(x)},\\
&\sum_{n \in \mathcal{B}(y)} \Theta_{12}(n) \leq  \dfrac{y(1+o(1))}{x^b \log(x)}\!\!\int_{0.305}^{0.33} \int_{0.42-\alpha_1}^{\frac{(0.635-\alpha_1)}{2}} \!\!\int_{0.42-\alpha_1}^{\alpha_2} \int_{0.42-\alpha_1}^{\min\{\alpha_3,\frac{(1-\sum_{i=1}^3\alpha_i)}{2}\}} \dfrac{\omega\left(\frac{1-\sum_{i=1}^4 \alpha_i}{\alpha_4}\right)}{\alpha_1 \alpha_2 \alpha_3 \alpha_4^2}\mbox{d}\underline{\alpha} <\dfrac{0.0062 y}{x^b \log(x)}.
\end{align*}
\endgroup

\vspace{3mm}
Finally, using conditions (V)(vi) and (V)(vii) of Corollary~\ref{cor:computations}, we find that $\Upsilon_5(n)$ equals
\begingroup
\allowdisplaybreaks
\begin{align}
\Psi_5(n)&+ \sum_{\substack{n = p_1p_2 m \\ x^{0.07} \leq p_2<p_1 < X \\   \ell_1^* \leq  0.305 \\ \ell_2^* \in [0.365-\ell_1^*,0.42-\ell_1^*) \cup (0.58-\ell_1^*,0.635-\ell_1^*]}} \psi(m,p_2) + 
\sum_{\substack{n = p_1p_2 m \\ x^{0.07} \leq p_2<p_1 < X \\   \ell_1^* \leq  0.305 \\ \ell_2^* \in (0.2099, 0.58-\ell_1^*]}} \psi(m,p_2)  \label{a55}\\
&+\!\!\!\!\sum_{\substack{n = p_1p_2p_3p_4 m \\ x^{0.07} \leq p_4<p_3<p_2<p_1 < X \\   \ell_1^* \leq  0.305 \\ \ell_2^* \in [0.42-\ell_1^*,0.2099] \\ \ell_3^* > 0.42-\ell_1^*  \\ \ell_4^* > 0.42-\ell_1^*}} \!\!\!\!\psi(m,p_4)+\!\!\!\!\sum_{\substack{n = p_1p_2p_3p_4 m \\ x^{0.07} \leq p_4<p_3<p_2<p_1 < X \\   \ell_1^* \leq  0.305 \\ \ell_2^* \in [0.42-\ell_1^*,0.2099] \\ \ell_3^* > 0.42-\ell_1^*  \\ \ell_4^*\in [0.365-\ell_1^*, 0.42-\ell_1^*]}}\!\!\!\! \psi(m,p_4)  +\!\!\!\!\sum_{\substack{n = p_1p_2p_3p_4 m \\ x^{0.07} \leq p_4<p_3< p_2<p_1 < X \\   \ell_1^* \leq  0.305 \\ \ell_2^* \in [0.42-\ell_1^*,0.2099] \\ \ell_3^* > 0.42-\ell_1^*  \\ \ell_4^* <0.365-\ell_1^*}} \!\!\!\! \psi(m,p_4) \label{a56}\\
&+\sum_{\substack{n = p_1p_2p_3p_4 m \\ x^{0.07} \leq p_4<p_3< p_2<p_1 < X \\   \ell_1^* \leq  0.305 \\ \ell_2^* \in [0.42-\ell_1^*,0.2099] \\ \ell_3^*  \in [0.365-\ell_1^*, 0.42-\ell_1^*]}} \psi(m,p_4)+\sum_{\substack{n = p_1p_2p_3p_4 m \\ x^{0.07} \leq p_4<p_3< p_2<p_1 < X \\   \ell_1^* \leq  0.305 \\ \ell_2^* \in [0.42-\ell_1^*,0.2099] \\ \ell_3^*  < 0.365-\ell_1^* }} \psi(m,p_4) \label{a57} \\
&+\sum_{\substack{n = p_1p_2p_3p_4 m \\ x^{0.07} \leq p_4<p_3< p_2<p_1 < X \\   \ell_1^* \leq  0.305 \\  \ell_2^* < 0.365-\ell_1^* \\ \ell_4^* > 0.42-\ell_2^*-\ell_3^* \\ \ell_4^* > (0.635-\ell_1^*-\ell_2^*-\ell_3^*)/2 }} \psi(m,p_4)+\sum_{\substack{n = p_1p_2p_3p_4 m \\ x^{0.07} \leq p_4<p_3< p_2<p_1 < X \\   \ell_1^* \leq  0.305 \\  \ell_2^* < 0.365-\ell_1^* \\ \ell_4^* \in [0.365-\ell_2^*-\ell_3^*, 0.42-\ell_2^*-\ell_3^*]  \\ \ell_4^* > (0.635-\ell_1^*-\ell_2^*-\ell_3^*)/2 }} \psi(m,p_4)
\label{a58}\\
&+\sum_{\substack{n = p_1p_2p_3p_4 m \\ x^{0.07} \leq p_4<p_3< p_2<p_1 < X \\   \ell_1^* \leq  0.305 \\  \ell_2^* < 0.365-\ell_1^* \\ \ell_4^* < 0.365-\ell_2^*-\ell_3^* \\ \ell_4^* > (0.635-\ell_1^*-\ell_2^*-\ell_3^*)/2 }} \psi(m,p_4)+\sum_{\substack{n = p_1p_2p_3p_4p_5p_6 m \\ x^{0.07} \leq p_6<p_5< p_4<p_3< p_2<p_1 < X \\   \ell_1^* \leq  0.305 \\  \ell_2^* < 0.365-\ell_1^* \\ \ell_4^* \leq(0.635-\ell_1^*-\ell_2^*-\ell_3^*)/2 }} \psi(m,p_6), \label{a59}
\end{align}
\endgroup
where $\Psi_5(n)$ is a sum of six good$^\star$ functions.   The first sum in (\ref{a55}), the second sum in (\ref{a56}), the first sum in (\ref{a57}) and the second sum in (\ref{a58})  are  good$^{\star\star}$ by condition (E)(a) of Corollary~\ref{cor:computations}. Enumerate the remaining seven big sums as $\Theta_{13}(n), \dots, \Theta_{19}(n)$. Upper bounds for $\sum_{n \in \mathcal{B}(y)} \Theta_i(n)$, $i \in \{13, \dots, 19\}$, are given below:
\begingroup
\allowdisplaybreaks
\begin{align*}
& \dfrac{y(1+o(1))}{x^b \log(x)}\int_{0.2099}^{0.305} \int_{0.2099}^{\min\{\alpha_1,0.58-\alpha_1\}}\dfrac{\omega\left(\frac{1-\alpha_1-\alpha_2}{\alpha_2}\right)}{\alpha_1 \alpha_2^2}\mbox{d}\alpha_2\mbox{d}\alpha_1 <\dfrac{0.1723 y}{x^b \log(x)},\\
& \dfrac{y(1+o(1))}{x^b \log(x)}\int_{0.21}^{0.305} \int_{0.42-\alpha_1}^{0.2099} \int_{0.42-\alpha_1}^{\alpha_2} \int_{0.42-\alpha_1}^{\min\{\alpha_3,\frac{(1-\sum_{i=1}^3\alpha_i)}{2}\}} \dfrac{\omega\left(\frac{1-\sum_{i=1}^4 \alpha_i}{\alpha_4}\right)}{\alpha_1 \alpha_2 \alpha_3 \alpha_4^2}\mbox{d}\underline{\alpha}<\dfrac{0.0104 y}{x^b \log(x)},\\
& \dfrac{y(1+o(1))}{x^b \log(x)}\int_{0.21}^{0.305} \int_{0.42-\alpha_1}^{0.2099} \int_{0.42-\alpha_1}^{\alpha_2} \int_{0.07}^{\min\{\alpha_3,\frac{(1-\sum_{i=1}^3\alpha_i)}{2},0.365-\alpha_1\}} \dfrac{\omega\left(\frac{1-\sum_{i=1}^4 \alpha_i}{\alpha_4}\right)}{\alpha_1 \alpha_2 \alpha_3 \alpha_4^2}\mbox{d}\underline{\alpha} <\dfrac{0.0212 y}{x^b \log(x)},\\
& \dfrac{y(1+o(1))}{x^b\log(x)}\int_{0.21}^{0.305} \int_{0.42-\alpha_1}^{0.2099} \int_{0.07}^{\min\{\alpha_2,0.365-\alpha_1\}} \!\!\int_{0.07}^{\min\{\alpha_3,0.365-\alpha_1\}} \dfrac{\omega\left(\frac{1-\sum_{i=1}^4 \alpha_i}{\alpha_4}\right)}{\alpha_1 \alpha_2 \alpha_3 \alpha_4^2}\mbox{d}\underline{\alpha} <\dfrac{0.0397 y}{x^b \log(x)},\\
& \dfrac{y(1+o(1))}{x^b \log(x)}\!\!\int_{0.07}^{0.295}\!\! \int_{0.07}^{\min\{\alpha_1,0.365-\alpha_1\}}\!\!\!\! \!\int_{0.07}^{\alpha_2} \int_{\max\{\frac{(0.635-\sum_{i=1}^3\alpha_i)}{2}, 0.42-\alpha_2-\alpha_3\}}^{\min\{\alpha_3,\frac{(1-\sum_{i=1}^3\alpha_i)}{2}\}} \dfrac{\omega\left(\frac{1-\sum_{i=1}^4 \alpha_i}{\alpha_4}\right)}{\alpha_1 \alpha_2 \alpha_3 \alpha_4^2}\mbox{d}\underline{\alpha} <\dfrac{0.0105 y}{x^b \log(x)},\\
& \dfrac{y(1+o(1))}{x^b\log(x)}\int_{0.07}^{0.295} \int_{0.07}^{\min\{\alpha_1,0.365-\alpha_1\}}\int_{0.07}^{\alpha_2} \int_{\max\{0.07,\frac{(0.635-\sum_{i=1}^3\alpha_i)}{2}\}}^{\min\{\alpha_3, 0.365-\alpha_2-\alpha_3\}} \dfrac{\omega\left(\frac{1-\sum_{i=1}^4 \alpha_i}{\alpha_4}\right)}{\alpha_1 \alpha_2 \alpha_3 \alpha_4^2}\mbox{d}\underline{\alpha}<\dfrac{0.023 y}{x^b \log(x)},\\
& \dfrac{y(1+o(1))}{x^b \log(x)}\left(  \dfrac{\log\left( \frac{0.2}{0.07} \right)^6 }{0.07 \cdot 6!}  + \dfrac{\log\left( \frac{0.165}{0.07} \right)^5\log\left( \frac{0.23}{0.2}\right)  }{0.07 \cdot 5!}  +  \dfrac{\log\left( \frac{0.135}{0.07} \right)^5\log\left( \frac{0.295}{0.23}\right)  }{0.07 \cdot 5!} \right)<\dfrac{0.0379 y}{x^b \log(x)}.
\end{align*} 
\endgroup
 In summary, we have shown that $1_{\mathbb{P}}(n)$ can be written as $1_{\mathbb{P}}(n) = \Psi(n) + \sum_{i=1}^{19} \Theta_i(n)$, where $\Psi(n)$ is a sum of no more than $100$ good$^\star$ and good$^{\star\star}$ functions and where  each $\Theta_i$ is non-negative and 
  \begin{align*}
 \sum_{i=1}^{19} \sum_{n \in \mathcal{B}(y)} \Theta_i(n) < \dfrac{ 0.9938y}{x^b\log(x)}.
  \end{align*}
This concludes the proof of Proposition~\ref{proposition3} for the case $a \in (0.59,0.61]$.

\subsubsection{$a \leq 0.53$} For the remaining  cases, $a \leq 0.53$, $a \in  (0.545,0.57]$ and $a > 0.61$, option (2) and (3) of Definition~\ref{def:combinatorialconditions} were not used in the search for elements of $\mathcal{R}^\star(a)$ and $\mathcal{R}^{\star\star}(a)$. The corresponding Buchstab decompositions are also considerably simpler. Thus we  now skip past the detailed step-by-step derivation given in previous cases and immediately write down the final decomposition. We begin with case $a \leq 0.53$.  By conditions (I)(i), (I)(ii) and (I)(iii) of Corollary~\ref{cor:computations}, Buchstab iterations give 
\begingroup
\allowdisplaybreaks
\begin{align}
\psi(n,X)  &= \Psi_1(n)+  \sum_{\substack{n = p_1p_2 m \\ x^{0.07} \leq p_2 < p_1 < X \\ \ell_1^* > 0.36 \\  \ell_2^* > (0.71-\ell_1^*)/2 \\ \ell_2^* \not\in [0.64-\ell_1^*, 0.71 -\ell_1^*]  }} \psi(m,p_2) + \sum_{\substack{n = p_1p_2 m \\ x^{0.07} \leq p_2 < p_1 < X \\ \ell_1^* > 0.36 \\  \ell_2^* > (0.71-\ell_1^*)/2 \\ \ell_2^* \in [0.64-\ell_1^*, 0.71 -\ell_1^*]  }} \psi(m,p_2) \label{4line81}
\\
&+ \sum_{\substack{n = p_1p_2p_3p_4 m \\ x^{0.07} \leq p_4 < p_3 < p_2 < p_1 < X \\ \ell_1^* > 0.36 \\  \ell_2^* \leq  (0.71-\ell_1^*)/2 \\ \ell_3^* < 0.64-\ell_1^*-\ell_2^* }} \psi(m,p_4) + \sum_{\substack{n = p_1p_2p_3p_4 m \\ x^{0.07} \leq p_4 < p_3 < p_2 < p_1 < X \\ \ell_1^* > 0.36 \\  \ell_2^* \leq  (0.71-\ell_1^*)/2 \\ \ell_3^* \geq  0.64-\ell_1^*-\ell_2^* }} \psi(m,p_4) \label{4line82} \\
&+  \sum_{\substack{n = p_1p_2 m \\ x^{0.07} \leq p_2 < p_1 < X \\ \ell_1^* \in [0.29,0.36]  }} \psi(m,p_2) +  \sum_{\substack{n = p_1p_2 m \\ x^{0.07} \leq p_2 < p_1 < X \\ \ell_1^* < 0.29 \\ \ell_2^* > (0.71-\ell_1^*)/2 }} \psi(m,p_2) +  \sum_{\substack{n = p_1p_2p_3p_4 m \\ x^{0.07} \leq p_4 < p_3 < p_2 < p_1 < X \\ 0.22 < \ell_1^* < 0.29 \\  \ell_2^* \leq  (0.71-\ell_1^*)/2 \\ \ell_4^* >  (0.71-\ell_1^*-\ell_2^* -\ell_3^*)/2 \\ \ell_4^* > 0.36-\ell_1^*}} \psi(m,p_4)  \label{4line83}
\\
& 
+  \sum_{\substack{n = p_1p_2p_3p_4 m \\ x^{0.07} \leq p_4 < p_3 < p_2 < p_1 < X \\ 0.22 < \ell_1^* < 0.29 \\  \ell_2^* \leq  (0.71-\ell_1^*)/2 \\ \ell_4^* >  (0.71-\ell_1^*-\ell_2^* -\ell_3^*)/2 \\ \ell_4^* \leq 0.36-\ell_1^*}} \psi(m,p_4) +  \sum_{\substack{n = p_1p_2p_3p_4 m \\ x^{0.07} \leq p_4 < p_3 < p_2 < p_1 < X \\ \ell_1^* \leq 0.22 \\  \ell_2^* \leq  (0.71-\ell_1^*)/2 \\ \ell_4^* >  (0.71-\ell_1^*-\ell_2^* -\ell_3^*)/2}} \psi(m,p_4) \label{4line84}
\\
 &+  \sum_{\substack{n = p_1p_2p_3p_4p_5p_6 m \\ x^{0.07} \leq p_6<p_5<p_4 < p_3 < p_2 < p_1 < X \\ \ell_1^*<0.29 \\  \ell_2^* \leq  (0.71-\ell_1^*)/2 \\ \ell_4^* \leq  (0.71-\ell_1^*-\ell_2^* -\ell_3^*)/2 \\ \mathcal{C}_1}} \psi(m,p_6)  +  \sum_{\substack{n = p_1p_2p_3p_4p_5p_6 m \\ x^{0.07} \leq p_6<p_5<p_4 < p_3 < p_2 < p_1 < X \\ \ell_1^*<0.29 \\  \ell_2^* \leq  (0.71-\ell_1^*)/2 \\ \ell_4^* \leq  (0.71-\ell_1^*-\ell_2^* -\ell_3^*)/2 \\ \mathcal{C}_2}} \psi(m,p_6) \label{4line85}\\
 &+  \sum_{\substack{n = p_1p_2p_3p_4p_5p_6 m \\ x^{0.07} \leq p_6<p_5<p_4 < p_3 < p_2 < p_1 < X \\ \ell_1^*<0.29 \\  \ell_2^* \leq  (0.71-\ell_1^*)/2 \\ \ell_4^* \leq  (0.71-\ell_1^*-\ell_2^* -\ell_3^*)/2 \\ \mathcal{C}_3}} \psi(m,p_6) +  \sum_{\substack{n = p_1p_2p_3p_4p_5p_6 m \\ x^{0.07} \leq p_6<p_5<p_4 < p_3 < p_2 < p_1 < X \\ \ell_1^*<0.29 \\  \ell_2^* \leq  (0.71-\ell_1^*)/2 \\ \ell_4^* \leq  (0.71-\ell_1^*-\ell_2^* -\ell_3^*)/2 \\ \mathcal{C}_4}} \psi(m,p_6), \label{4line86}
\end{align}
\endgroup 
where $\Psi_1(n)$ is a sum of eight good$^\star$ functions and  
where $\mathcal{C}_1$ denotes the condition ($\ell_i^*, \ell_j^* \in [0.145,0.18]$ for some $i \neq j$), $\mathcal{C}_2$ denotes the condition ($\ell_3^* \in [0.145,0.29]$ and $\ell_i^* \in [0.145,0.18]$ for at most one $i$), $\mathcal{C}_3$ denotes the condition ($\ell_1^* > 0.18$, $\ell_2^* \in [0.145,0.29]$ and $\ell_3^* <0.145$) and $\mathcal{C}_4$ denotes the condition ($\ell_2^* < 0.145$). Conditions $\mathcal{C}_1$, $\mathcal{C}_2$, $\mathcal{C}_3$ and $\mathcal{C}_4$ partition the set of $(\ell_1^*,\ell_2^*,\ell_3^*,\ell_4^*, \ell_5^*, \ell_6^*)$ which satisfy $0.07 \leq \ell_6^* < \ell_5^* < \ell_4^* < \ell_3^* < \ell_2^* < \ell_1^* \leq 0.29$, $\ell_2^* \leq  (0.71-{ \ell_1^*})/2$ and $\ell_4^* \leq (0.71-{\ell_1^*}-{\ell_2^*}-{\ell_3^*})/2$.

\vspace{3mm} 
Denote the first big sum in (\ref{4line81}) by $\Theta_1(n)$. The second sum in (\ref{4line81}) is good$^{\star\star}$ by condition (A)(b) of Corollary~\ref{cor:computations}.  Denote the first  sum in (\ref{4line82}) by $\Theta_2(n)$. The second sum in (\ref{4line82}) is good$^{\star\star}$ by condition (A)(c) of Corollary~\ref{cor:computations}. The first sum in (\ref{4line83}) is good$^{\star\star}$ by condition (A)(a) of Corollary~\ref{cor:computations}. Denote the second sum by $\Theta_3(n)$ and the third sum by $\Theta_4(n)$. The first sum in (\ref{4line84}) is  good$^{\star\star}$ by condition (A)(d) of Corollary~\ref{cor:computations}.  Denote the second sum by $\Theta_5(n)$.  Since $\ell_i^*, \ell_j^* \in [0.145,0.18]$ implies $\ell_i^*+\ell_j^* \in [0.29,0.36]$, the first sum in (\ref{4line85}) is  good$^{\star\star}$ by condition (A)(e) of Corollary~\ref{cor:computations}. By condition (A)(f), the second sum is also good. Denote the first sum in (\ref{4line86}) by $\Theta_6(n)$ and the second sum by $\Theta_7(n)$. 

\vspace{3mm} Upper bounds for $\sum_{n \in \mathcal{B}(y)} \Theta_i(n)$, $i \in \{1, \dots, 7\}$, are given below: 
\begingroup
\allowdisplaybreaks
\begin{align*}
&\dfrac{y(1+o(1))}{x^b \log(x)}\int_{0.36}^{0.5+\varepsilon_1} \int_{[(0.71-\alpha_1)/2, 0.64-\alpha_1] \cup [0.71-\alpha_1, (1-\alpha_1)/2 ]}\dfrac{\omega\left(\frac{1-\alpha_1-\alpha_2}{\alpha_2}\right)}{\alpha_1 \alpha_2^2}\mbox{d}\alpha_2\mbox{d}\alpha_1 <\dfrac{0.513 y}{x^b \log(x)},\\
& \dfrac{y(1+o(1))}{x^b \log(x)}\int_{0.36}^{0.5+\varepsilon_1} \int_{0.07}^{\min\{\alpha_1,(0.71-\alpha_1)/2\}} \int_{0.07}^{\min\{\alpha_2,0.64-\alpha_1-\alpha_2\}} \int_{0.07}^{\alpha_3} \dfrac{\omega\left(\frac{1-\sum_{i=1}^4 \alpha_i}{\alpha_4}\right)}{\alpha_1 \alpha_2 \alpha_3 \alpha_4^2}\mbox{d}\underline{\alpha} <\dfrac{0.079 y}{x^b \log(x)},\\
& \dfrac{y(1+o(1))}{x^b\log(x)}\int_{0.07}^{0.29} \int_{(0.71-\alpha_1)/2}^{\alpha_1}\dfrac{\omega\left(\frac{1-\alpha_1-\alpha_2}{\alpha_2}\right)}{\alpha_1 \alpha_2^2}\mbox{d}\alpha_2\mbox{d}\alpha_1 <\dfrac{0.08 y}{x^b \log(x)},\\
& \dfrac{y(1+o(1))}{x^b \log(x)}\int_{0.22}^{0.29} \int_{0.07}^{\min\{\alpha_1,\frac{(0.71-\alpha_1)}{2}\}} \int_{0.07}^{\alpha_2} \int_{\max\{0.36-\alpha_1, \frac{(0.71-\alpha_1-\alpha_2-\alpha_3)}{2}\}}^{\alpha_3} \dfrac{\omega\left(\frac{1-\sum_{i=1}^4 \alpha_i}{\alpha_4}\right)}{\alpha_1 \alpha_2 \alpha_3 \alpha_4^2}\mbox{d}\underline{\alpha}  <\dfrac{0.112 y}{x^b \log(x)},\\
& \dfrac{y(1+o(1))}{x^b \log(x)}\int_{0.07}^{0.22} \int_{0.07}^{\min\{\alpha_1,\frac{(0.71-\alpha_1)}{2}\}} \int_{0.07}^{\alpha_2} \int_{\max\{0.07, \frac{(0.71-\alpha_1-\alpha_2-\alpha_3)}{2}\}}^{\alpha_3} \dfrac{\omega\left(\frac{1-\sum_{i=1}^4 \alpha_i}{\alpha_4}\right)}{\alpha_1 \alpha_2 \alpha_3 \alpha_4^2}\mbox{d}\underline{\alpha} <\dfrac{0.063 y}{x^b \log(x)},\\
& \dfrac{y(1+o(1))}{x^b \log(x)} \dfrac{1}{4!}\int_{0.18}^{0.29} \int_{0.145}^{0.29} \int_{0.07}^{0.145}\int_{0.07}^{0.145}\int_{0.07}^{0.145}\int_{0.07}^{0.145}\dfrac{\omega\left(\frac{1-\alpha_1-\alpha_2-\alpha_3-\alpha_4-\alpha_5-\alpha_6}{\alpha_6}\right)}{0.07 \alpha_1 \alpha_2\alpha_3\alpha_4\alpha_5 \alpha_6}  \mbox{d}\underline{\alpha} <\dfrac{0.056 y}{x^b \log(x)} , \\
& \dfrac{y(1+o(1))}{x^b \log(x)} \dfrac{1}{5!}\int_{0.07}^{0.29} \int_{0.07}^{0.145} \int_{0.07}^{0.145}\int_{0.07}^{0.145}\int_{0.07}^{0.145}\int_{0.07}^{0.145}\dfrac{\omega\left(\frac{1-\alpha_1-\alpha_2-\alpha_3-\alpha_4-\alpha_5-\alpha_6}{\alpha_6}\right)}{0.07 \alpha_1 \alpha_2\alpha_3\alpha_4\alpha_5 \alpha_6}  \mbox{d}\underline{\alpha} < \dfrac{0.035 y}{x^b \log(x)}.
\end{align*}
\endgroup
  In summary, we have shown that $1_{\mathbb{P}}(n)$ can be written as $1_{\mathbb{P}}(n) = \Psi(n) + \sum_{i=1}^{7} \Theta_i(n)$, where $\Psi(n)$ is a sum of no more than $100$ good$^\star$ and good$^{\star\star}$ functions and where  each $\Theta_i$ is non-negative and 
  \begin{align*}
 \sum_{i=1}^{7} \sum_{n \in \mathcal{B}(y)} \Theta_i(n) < \dfrac{ 0.938y}{x^b\log(x)}.
  \end{align*}
This concludes the proof of Proposition~\ref{proposition3} for the case $a \leq 0.53$.

\subsubsection{$a \in (0.545, 0.57]$}

 Next we cover the case $a \in (0.545,0.57]$. By conditions (III)(i), (III)(ii), (III)(iii) and (III)(iv) of Corollary~\ref{cor:computations}, Buchstab iterations give  
\begingroup
\allowdisplaybreaks
\begin{align}
\psi(n,X)  &= \Psi_1(n)+  \sum_{\substack{n = p_1p_2 m \\ x^{0.075} \leq p_2 < p_1 < X \\ \ell_1^* > 0.475  \\ \ell_2^* > 0.6-\ell_1^*   }} \psi(m,p_2) +\sum_{\substack{n = p_1p_2 m \\ x^{0.075} \leq p_2 < p_1 < X \\ \ell_1^* > 0.475  \\ \ell_2^* \leq 0.6-\ell_1^*   }} \psi(m,p_2)  \label{5line87} +\sum_{\substack{n = p_1p_2 m \\ x^{0.075} \leq p_2 < p_1 < X \\ 0.4 \leq \ell_1^* \leq 0.475     }} \psi(m,p_2) \\
&+\sum_{\substack{n = p_1p_2 m \\ x^{0.075} \leq p_2 < p_1 < X \\ \ell_1^* < 0.4  \\ \ell_2^* > 0.6-\ell_1^*  }} \psi(m,p_2)+\sum_{\substack{n = p_1p_2 m \\ x^{0.075} \leq p_2 < p_1 < X \\ \ell_1^* < 0.4  \\ \ell_2^* \in [0.4-\ell_1^*, 0.475 -\ell_1^*] \cup [0.525-\ell_1^*, 0.6-\ell_1^*]  }} \psi(m,p_2)\label{5line88}\\
&+\sum_{\substack{n = p_1p_2 m \\ x^{0.075} \leq p_2 < p_1 < X \\ \ell_1^* < 0.4  \\ \ell_2^* \in (0.475 -\ell_1^*,0.525-\ell_1^*) \\ \ell_2^* > 0.14   }} \psi(m,p_2)+\sum_{\substack{n = p_1p_2p_3p_4 m \\ x^{0.075} \leq p_4<p_3<p_2 < p_1 < X \\ \ell_1^* < 0.4  \\ \ell_2^* \in (0.475 -\ell_1^*,0.525-\ell_1^*) \\ \ell_2^*\leq 0.14   }} \psi(m,p_4) \label{5line89}\\
&+\sum_{\substack{n = p_1p_2p_3p_4 m \\ x^{0.075} \leq p_4<p_3<p_2 < p_1 < X \\ \ell_2^* < 0.4-\ell_1^* \\ \ell_4^* > (0.615-\ell_1^*-\ell_2^*-\ell_3^*)/2 \\ \ell_4^*> 0.475-\ell_1^*-\ell_2^* }} \psi(m,p_4) +\sum_{\substack{n = p_1p_2p_3p_4 m \\ x^{0.075} \leq p_4<p_3<p_2 < p_1 < X \\ \ell_2^* < 0.4-\ell_1^* \\ \ell_4^* > (0.615-\ell_1^*-\ell_2^*-\ell_3^*)/2 \\\ell_4^*\in [0.4-\ell_1^*-\ell_2^* , 0.475-\ell_1^*-\ell_2^*]}} \psi(m,p_4)\label{5line90}\\
&+\sum_{\substack{n = p_1p_2p_3p_4 m \\ x^{0.075} \leq p_4<p_3<p_2 < p_1 < X \\ \ell_2^* < 0.4-\ell_1^* \\ \ell_4^* > (0.615-\ell_1^*-\ell_2^*-\ell_3^*)/2 \\\ell_4^*<0.4-\ell_1^*-\ell_2^* }} \psi(m,p_4)  
+\sum_{\substack{n = p_1p_2p_3p_4p_5p_6 m \\ x^{0.075} \leq p_6<p_5<p_4<p_3<p_2 < p_1 < X \\ \ell_2^* < 0.4-\ell_1^* \\ \ell_4^* \leq (0.615-\ell_1^*-\ell_2^*-\ell_3^*)/2 }} \psi(m,p_6),  \label{5line91} 
\end{align}
\endgroup
where $\Psi_1(n)$ is a sum of eight good$^\star$ functions. Denote the first big sum in (\ref{5line87}) by $\Theta_1(n)$. The second big sum is good$^{\star\star}$ by condition (C)(b) of Corollary~\ref{cor:computations} and the third is good$^{\star\star}$ by condition (C)(a). Denote the first  sum in (\ref{5line88}) by $\Theta_2(n)$.  The second sum is again good$^{\star\star}$ by condition (C)(b). Denote the first  sum in (\ref{5line89}) by $\Theta_3(n)$ and the second by $\Theta_4(n)$.   Denote the first  sum in (\ref{5line90}) by $\Theta_5(n)$.  The second sum is good$^{\star\star}$ by condition (C)(c). Denote the first  sum in (\ref{5line91}) by $\Theta_6(n)$ and the second by $\Theta_7(n)$.  

\vspace{3mm} Upper bounds for $\sum_{n \in \mathcal{B}(y)} \Theta_i(n)$, $i \in \{1, \dots, 7\}$, are given below:
\begingroup
\allowdisplaybreaks
\begin{align*}
& \dfrac{y(1+o(1))}{x^b  \log(x)}\int_{0.475}^{0.5+\varepsilon_1} \int_{0.6-\alpha_1}^{(1-\alpha_1)/2}\dfrac{\omega\left(\frac{1-\alpha_1-\alpha_2}{\alpha_2}\right)}{\alpha_1 \alpha_2^2}\mbox{d}\alpha_2\mbox{d}\alpha_1 <\dfrac{0.166 y}{x^b  \log(x)},\\
& \dfrac{y(1+o(1))}{x^b  \log(x)}\int_{0.3}^{0.4} \int_{0.6-\alpha_1}^{\min\{\alpha_1,(1-\alpha_1)/2\}}\dfrac{\omega\left(\frac{1-\alpha_1-\alpha_2}{\alpha_2}\right)}{\alpha_1 \alpha_2^2}\mbox{d}\alpha_2\mbox{d}\alpha_1<\dfrac{0.187 y}{x^b  \log(x)},\\
& \dfrac{y(1+o(1))}{x^b  \log(x)}\int_{0.475/2}^{0.385} \int_{\max\{0.14,0.475-\alpha_1\}}^{\min\{\alpha_1,0.525-\alpha_1\}}\dfrac{\omega\left(\frac{1-\alpha_1-\alpha_2}{\alpha_2}\right)}{\alpha_1 \alpha_2^2}\mbox{d}\alpha_2\mbox{d}\alpha_1 <\dfrac{0.302 y}{x^b  \log(x)},\\
& \dfrac{y(1+o(1))}{x^b  \log(x)}\int_{0.335}^{0.4} \int_{0.475-\alpha_1}^{\min\{0.14,0.525-\alpha_1\}} \int_{0.075}^{\alpha_2} \int_{0.075}^{\min\{\alpha_3,\frac{(1-\sum_{i=1}^3\alpha_i)}{2}\}} \dfrac{\omega\left(\frac{1-\sum_{i=1}^4 \alpha_i}{\alpha_4}\right)}{\alpha_1 \alpha_2 \alpha_3 \alpha_4^2}\mbox{d}\underline{\alpha} <\dfrac{0.032 y}{x^b  \log(x)},\\
& \dfrac{y(1+o(1))}{x^b  \log(x)}\int_{0.075}^{0.325} \!\!\int_{0.075}^{\min\{\alpha_1,0.4-\alpha_1\}}\!\!\! \int_{0.075}^{\alpha_2} \int_{\max\{\frac{(0.615-\sum_{i=1}^3\alpha_i)}{2},0.475-\alpha_1-\alpha_2\}}^{\min\{\alpha_3,\frac{(1-\sum_{i=1}^3\alpha_i)}{2}\}}\!\! \dfrac{\omega\left(\frac{1-\sum_{i=1}^4 \alpha_i}{\alpha_4}\right)}{\alpha_1 \alpha_2 \alpha_3 \alpha_4^2}\mbox{d}\underline{\alpha} <\dfrac{0.07 y}{x^b  \log(x)},\\
& \dfrac{y(1+o(1))}{x^b  \log(x)}\int_{0.075}^{0.325} \!\!\int_{0.075}^{\min\{\alpha_1,0.4-\alpha_1\}}\!\!\! \int_{0.075}^{\alpha_2} \int_{\max\{0.075,\frac{(0.615-\sum_{i=1}^3\alpha_i)}{2}\}}^{\min\{\alpha_3,\frac{(1-\sum_{i=1}^3\alpha_i)}{2},0.4-\alpha_1-\alpha_2\}} \dfrac{\omega\left(\frac{1-\sum_{i=1}^4 \alpha_i}{\alpha_4}\right)}{\alpha_1 \alpha_2 \alpha_3 \alpha_4^2}\mbox{d}\underline{\alpha} <\dfrac{0.01 y}{x^b  \log(x)},\\
& \dfrac{y(1+o(1))}{x^b  \log(x)}\left(  \dfrac{\log\left( \frac{0.155}{0.075} \right)^6 }{0.075 \cdot 6!}  + \dfrac{\log\left( \frac{0.155}{0.075} \right)^5\log\left( \frac{0.4}{0.155}\right)  }{0.075 \cdot 5!}  +  \dfrac{\log\left( \frac{0.155}{0.075} \right)^4\log\left( \frac{0.4}{0.155}\right)^2  }{0.075 \cdot 4! 2!} \right) <\dfrac{0.1 y}{x^b \log(x)}.
\end{align*}
\endgroup
In the final upper bound, which is the upper bound for $\sum_{n \in \mathcal{B}(y)} \Theta_7(n)$, we used that $\ell_1^* + \dots + \ell_5^* \leq 0.615$ implies $\ell_3^* \leq 0.155$.  In summary, $1_{\mathbb{P}}(n)$ can be written as $1_{\mathbb{P}}(n) = \Psi(n) + \sum_{i=1}^{7} \Theta_i(n)$, where $\Psi(n)$ is a sum of no more than $100$ good$^\star$ and good$^{\star\star}$ functions and where  each $\Theta_i$ is non-negative and 
  \begin{align*}
 \sum_{i=1}^{7} \sum_{n \in \mathcal{B}(y)} \Theta_i(n) < \dfrac{ 0.867y}{x^b\log(x)}.
  \end{align*}
This concludes the proof of Proposition~\ref{proposition3} for the case $a \in (0.545,0.57]$.

\subsubsection{$a > 0.61$} This is our final case.
  By conditions (VI)(i), (VI)(ii), (VI)(iii), (VI)(iv), (VI)(v) and (VI)(vi) of Corollary~\ref{cor:computations}, Buchstab iterations give  
\begingroup
\allowdisplaybreaks 
\begin{align}
\psi(n,X)  &= \Psi_1(n)+  \sum_{\substack{n = p_1p_2 m \\ x^{0.065} \leq p_2 < p_1 < X \\ \ell_1^* > 0.42  \\ \ell_2^* > 0.645-\ell_1^* }} \psi(m,p_2) +\!\!\!\!\!\!  \sum_{\substack{n = p_1p_2 m \\ x^{0.065} \leq p_2 < p_1 < X \\  \ell_1^* > 0.42  \\ \ell_2^* \in [0.58-\ell_1^*, 0.645-\ell_1^*] }} \!\!\!\!\!\!  \psi(m,p_2)  +  \sum_{\substack{n = p_1p_2 m \\ x^{0.065} \leq p_2 < p_1 < X \\ \ell_1^* > 0.42  \\ \ell_2^* < 0.58-\ell_1^* \\ \ell_2^* \geq 0.1 }} \psi(m,p_2) \label{6line92}\\
&+\sum_{\substack{n = p_1p_2p_3p_4 m \\ x^{0.065} \leq p_4<p_3<p_2 < p_1 < X \\ \ell_1^* > 0.42  \\ \ell_2^* < \min\{0.58-\ell_1^*, 0.1\} }} \!\!\!\! \psi(m,p_4)+ \sum_{\substack{n = p_1p_2 m \\ x^{0.065} \leq p_2 < p_1 < X \\ 0.355 \leq \ell_1^* \leq 0.42 }} \psi(m,p_2) +\!\!\!\! \sum_{\substack{n = p_1p_2 m \\ x^{0.065} \leq  p_2 < p_1 < X \\  \ell_1^* < 0.355  \\ \ell_2^* > 0.645-\ell_1^* }} \!\!\!\!\psi(m,p_2) \label{6line93}\\
&+ \sum_{\substack{n = p_1p_2 m \\ x^{0.065} \leq p_2 < p_1 < X \\ \ell_1^* < 0.355  \\ \ell_2^* \in [0.355-\ell_1^*, 0.42 -\ell_1^*) \cup (0.58-\ell_1^*, 0.645-\ell_1^*] }} \!\!\!\!\!\!\!\psi(m,p_2) + \sum_{\substack{n = p_1p_2 m \\ x^{0.065} \leq  p_2 < p_1 < X \\ 0.325 < \ell_1^* < 0.355  \\ \ell_2^* \in [0.42 -\ell_1^*,0.58-\ell_1^*] \\ \ell_2^* > (0.645-\ell_1^*)/2 }} \psi(m,p_2) \label{6line94} \\
&+ \sum_{\substack{n = p_1p_2p_3p_4 m \\ x^{0.065} \leq p_4 < p_3< p_2 < p_1 < X \\ 0.325 < \ell_1^* < 0.355  \\ \ell_2^* \in [0.42 -\ell_1^*,0.58-\ell_1^*] \\ \ell_2^* \leq (0.645-\ell_1^*)/2 \\ \ell_4^*>0.42-\ell_1^* }} \psi(m,p_4) +  \sum_{\substack{n = p_1p_2p_3p_4 m \\ x^{0.065} \leq p_4< p_3< p_2 < p_1 < X \\ 0.325 < \ell_1^* < 0.355  \\ \ell_2^* \in [0.42 -\ell_1^*,0.58-\ell_1^*] \\ \ell_2^* \leq (0.645-\ell_1^*)/2 \\\ell_4^*\in [0.355-\ell_1^*,0.42-\ell_1^*]   }}  \psi(m,p_4) \label{6line95}\\ &+ \sum_{\substack{n = p_1p_2 m \\ x^{0.065} \leq  p_2 < p_1 < X \\  \ell_1^* \leq 0.325  \\ \ell_2^* \in (\max\{0.42-\ell_1^*,0.2099\},0.58-\ell_1^*] }} \psi(m,p_2)  + \sum_{\substack{n = p_1p_2p_3p_4 m \\ x^{0.065} \leq  p_4<p_3< p_2 < p_1 < X \\  \ell_1^* \leq 0.325  \\ \ell_2^* \in [0.42-\ell_1^*,0.2099) \\ \ell_4^* > (0.42-\ell_2^*-\ell_3^*)/2 \\ \ell_3^* < 0.355-\ell_1^*}} \psi(m,p_4)  
\label{6line96} \\
&+ \sum_{\substack{n = p_1p_2p_3p_4 m \\ x^{0.065} \leq  p_4<p_3< p_2 < p_1 < X \\  \ell_1^* \leq 0.325  \\ \ell_2^* \in [0.42-\ell_1^*,0.2099) \\ \ell_4^* > (0.42-\ell_2^*-\ell_3^*)/2 \\  \ell_3^* \in [ 0.355-\ell_1^*, 0.42-\ell_1^*]}} \psi(m,p_4)   + \sum_{\substack{n = p_1p_2p_3p_4 m \\ x^{0.065} \leq  p_4<p_3< p_2 < p_1 < X \\  \ell_1^* \leq 0.325  \\ \ell_2^* \in [0.42-\ell_1^*,0.2099) \\ \ell_4^* > (0.42-\ell_2^*-\ell_3^*)/2 \\  \ell_3^* > 0.42-\ell_1^* \\ \ell_4^* < 0.355-\ell_1^*}} \psi(m,p_4)  \label{6line97}\\
&+\sum_{\substack{n = p_1p_2p_3p_4 m \\ x^{0.065} \leq  p_4<p_3< p_2 < p_1 < X \\  \ell_1^* \leq 0.325  \\ \ell_2^* \in [0.42-\ell_1^*,0.2099) \\ \ell_4^* > (0.42-\ell_2^*-\ell_3^*)/2 \\  \ell_3^* > 0.42-\ell_1^* \\ \ell_4^* \in [ 0.355-\ell_1^*, 0.42-\ell_1^*]}} \psi(m,p_4) + \sum_{\substack{n = p_1p_2p_3p_4 m \\ x^{0.065} \leq  p_4<p_3< p_2 < p_1 < X \\  \ell_1^* \leq 0.325  \\ \ell_2^* \in [0.42-\ell_1^*,0.2099) \\ \ell_4^* > (0.42-\ell_2^*-\ell_3^*)/2 \\  \ell_3^* > 0.42-\ell_1^* \\ \ell_4^* > 0.42-\ell_1^*}} \psi(m,p_4) \label{6line98}\\
&+\sum_{\substack{n = p_1p_2p_3p_4p_5p_6 m \\ x^{0.065} \leq  p_5<p_6< p_4<p_3< p_2 < p_1 < X \\  \ell_1^* \leq 0.325  \\ \ell_2^* \in [0.42-\ell_1^*,0.2099) \\ \ell_4^* \leq (0.42-\ell_2^*-\ell_3^*)/2 }} \psi(m,p_6) + \sum_{\substack{n = p_1p_2p_3p_4 m \\ x^{0.065} \leq  p_4<p_3< p_2 < p_1 < X \\  \ell_1^* \leq 0.325  \\ \ell_2^* < 0.355 - \ell_1^* \\ \ell_4^* > (0.645-\ell_1^*-\ell_2^*-\ell_3^*)/2}} \psi(m,p_4)  \label{6line99} \\
&+ \sum_{\substack{n = p_1p_2p_3p_4p_5p_6 m \\ x^{0.065} \leq  p_6<p_5<p_4<p_3< p_2 < p_1 < X \\  \ell_1^* \leq 0.325  \\ \ell_2^* < 0.355 - \ell_1^* \\ \ell_4^* \leq (0.645-\ell_1^*-\ell_2^*-\ell_3^*)/2}} \psi(m,p_6),  \label{6line100}
\end{align}
\endgroup
where $\Psi_1(n)$ is a sum of twelve good$^\star$ functions.

\vspace{3mm} 
Denote the first big sum in (\ref{6line92}) by $\Theta_1(n)$ and the third by $\Theta_2(n)$. The second sum in (\ref{6line92}) is good$^{\star\star}$ by condition (F)(b) of Corollary~\ref{cor:computations}. Denote the first  sum in (\ref{6line93}) by $\Theta_3(n)$ and the third by $\Theta_4(n)$. The second sum in (\ref{6line93}) is good$^{\star\star}$ by condition (F)(a) of Corollary~\ref{cor:computations}. The first sum in  (\ref{6line94}) is  good$^{\star\star}$ by condition (F)(b). Denote the second sum in  (\ref{6line94}) by $\Theta_5(n)$. Denote the first sum in  (\ref{6line95}) by $\Theta_6(n)$. The second sum in (\ref{6line95}) is good$^{\star\star}$ by condition (F)(b) of Corollary~\ref{cor:computations}.  Denote the first sum in (\ref{6line96}) by $\Theta_7(n)$ and the second by $\Theta_8(n)$.   The first sum in  (\ref{6line97}) is  good$^{\star\star}$ by condition (F)(b). Denote the second sum in  (\ref{6line97}) by $\Theta_9(n)$. The first sum in  (\ref{6line98}) is also good$^{\star\star}$ by condition (F)(b). Denote the second sum in  (\ref{6line98}) by $\Theta_{10}(n)$. Denote the first sum in (\ref{6line99}) by $\Theta_{11}(n)$ and the second by $\Theta_{12}(n)$. Denote the sum in (\ref{6line100}) by $\Theta_{13}(n)$. 

\vspace{3mm}Upper bounds for $\sum_{n \in \mathcal{B}(y)} \Theta_i(n)$, $i \in \{1, \dots, 13\}$, are given below:
\begingroup
\allowdisplaybreaks
\begin{align*}
& \dfrac{y(1+o(1))}{x^b \log(x)}\int_{0.42}^{0.5+\varepsilon_1} \int_{0.645-\alpha_1}^{(1-\alpha_1)/2}\dfrac{\omega\left(\frac{1-\alpha_1-\alpha_2}{\alpha_2}\right)}{\alpha_1 \alpha_2^2}\mbox{d}\alpha_2\mbox{d}\alpha_1 <\dfrac{0.2194 y}{x^b \log(x)},\\
& \dfrac{y(1+o(1))}{x^b \log(x)}\int_{0.42}^{0.48} \int_{0.1}^{0.58-\alpha_1}\dfrac{\omega\left(\frac{1-\alpha_1-\alpha_2}{\alpha_2}\right)}{\alpha_1 \alpha_2^2}\mbox{d}\alpha_2\mbox{d}\alpha_1 <  {  \dfrac{0.1664 y}{x^b \log(x)}},\\ 
&\dfrac{y(1+o(1))}{x^b \log(x)}\int_{0.42}^{0.5} \int_{0.065}^{\min\{0.1,0.58-\alpha_1\}} \int_{0.065}^{\alpha_2} \int_{0.065}^{\min\{\alpha_3,\frac{(1-\sum_{i=1}^3\alpha_i)}{2}\}} \dfrac{\omega\left(\frac{1-\sum_{i=1}^4 \alpha_i}{\alpha_4}\right)}{\alpha_1 \alpha_2 \alpha_3 \alpha_4^2}\mbox{d}\underline{\alpha} <\dfrac{0.0170 y}{x^b \log(x)},\\
& \dfrac{y(1+o(1))}{x^b \log(x)}\int_{0.3225}^{0.355} \int_{0.645-\alpha_1}^{\min\{\alpha_1,(1-\alpha_1)/2\}}\dfrac{\omega\left(\frac{1-\alpha_1-\alpha_2}{\alpha_2}\right)}{\alpha_1 \alpha_2^2}\mbox{d}\alpha_2\mbox{d}\alpha_1 <\dfrac{0.0191 y}{x^b \log(x)},\\
& \dfrac{y(1+o(1))}{x^b \log(x)}\int_{0.325}^{0.355} \int_{(0.645-\alpha_1)/2}^{0.58-\alpha_1}\dfrac{\omega\left(\frac{1-\alpha_1-\alpha_2}{\alpha_2}\right)}{\alpha_1 \alpha_2^2}\mbox{d}\alpha_2\mbox{d}\alpha_1 <{ \dfrac{0.1166 y}{x^b \log(x)}},\\
& \dfrac{y(1+o(1))}{x^b \log(x)}\int_{0.325}^{0.355} \int_{0.42-\alpha_1}^{(0.645-\alpha_1)/2} \int_{0.42-\alpha_1}^{\alpha_2} \int_{0.42-\alpha_1}^{\min\{\alpha_3,\frac{(1-\sum_{i=1}^3\alpha_i)}{2}\}} \dfrac{\omega\left(\frac{1-\sum_{i=1}^4 \alpha_i}{\alpha_4}\right)}{\alpha_1 \alpha_2 \alpha_3 \alpha_4^2}\mbox{d}\underline{\alpha} <\dfrac{0.0282 y}{x^b\log(x)},\\
& \dfrac{y(1+o(1))}{x^b \log(x)}\int_{0.2099}^{0.325} \int_{0.2099}^{\min\{\alpha_1, 0.58-\alpha_1\}}\dfrac{\omega\left(\frac{1-\alpha_1-\alpha_2}{\alpha_2}\right)}{\alpha_1 \alpha_2^2}\mbox{d}\alpha_2\mbox{d}\alpha_1 <\dfrac{0.2102 y}{x^b \log(x)},\\
& \dfrac{y(1+o(1))}{x^b \log(x)}\int_{0.21}^{0.325} \int_{0.42-\alpha_1}^{0.21}\int_{0.065}^{\min\{\alpha_2,0.355-\alpha_1\}}\!\! \int_{\max\{0.065,\frac{(0.42-\alpha_2-\alpha_3)}{2}\}}^{\min\{\alpha_3,\frac{(1-\sum_{i=1}^3\alpha_i)}{2}\}} \dfrac{\omega\left(\frac{1-\sum_{i=1}^4 \alpha_i}{\alpha_4}\right)}{\alpha_1 \alpha_2 \alpha_3 \alpha_4^2}\mbox{d}\underline{\alpha} <\dfrac{0.0249 y}{x^b \log(x)},\\
& \dfrac{y(1+o(1))}{x^b \log(x)}\int_{0.21}^{0.325} \int_{0.42-\alpha_1}^{0.21}\int_{0.42-\alpha_1}^{\alpha_2} \int_{\max\{0.065,\frac{(0.42-\alpha_2-\alpha_3)}{2}\}}^{\min\{\alpha_3,0.355-\alpha_1,\frac{(1-\sum_{i=1}^3\alpha_i)}{2}\}} \dfrac{\omega\left(\frac{1-\sum_{i=1}^4 \alpha_i}{\alpha_4}\right)}{\alpha_1 \alpha_2 \alpha_3 \alpha_4^2}\mbox{d}\underline{\alpha} <\dfrac{0.0191 y}{x^b \log(x)},\\
& \dfrac{y(1+o(1))}{x^b \log(x)}\int_{0.21}^{0.325} \int_{0.42-\alpha_1}^{0.21}\int_{0.42-\alpha_1}^{\alpha_2} \int_{\max\{0.42-\alpha_1,\frac{(0.42-\alpha_2-\alpha_3)}{2}\}}^{\min\{\alpha_3,\frac{(1-\sum_{i=1}^3\alpha_i)}{2}\}} \dfrac{\omega\left(\frac{1-\sum_{i=1}^4 \alpha_i}{\alpha_4}\right)}{\alpha_1 \alpha_2 \alpha_3 \alpha_4^2}\mbox{d}\underline{\alpha} <\dfrac{0.0280 y}{x^b \log(x)},\\
& \dfrac{y(1+o(1))}{x^b \log(x)} \dfrac{\log\left( \frac{0.325}{0.21} \right)\!\left(\frac{4!}{5!}\log\left( \frac{0.145}{0.065} \right)^5+ \log\left( \frac{0.18}{0.145} \right)\log\left( \frac{0.145}{0.065} \right)^4 + \log\left(\frac{0.21}{0.18} \right)\log\left( \frac{0.11}{0.065} \right)^4 \right)}{0.065 \cdot 4! } <\dfrac{0.0471 y}{x^b \log(x)}, \\
& \dfrac{y(1+o(1))}{x^b \log(x)}\int_{0.065}^{0.325} \int_{0.065}^{\min\{\alpha_1,0.355-\alpha_1\}}\!\!\!\int_{0.065}^{\alpha_2} \int_{\max\{0.065,\frac{(0.645-\alpha_1-\alpha_2-\alpha_3)}{2}\}}^{\min\{\alpha_3,\frac{(1-\sum_{i=1}^3\alpha_i)}{2}\}} \dfrac{\omega\left(\frac{1-\sum_{i=1}^4 \alpha_i}{\alpha_4}\right)}{\alpha_1 \alpha_2 \alpha_3 \alpha_4^2}\mbox{d}\underline{\alpha} <{ \dfrac{0.0449 y}{x^b \log(x)}},\\
& \dfrac{y(1+o(1))}{x^b\log(x)}\left(  \dfrac{\log\left( \frac{0.1775}{0.065} \right)^6 }{0.065 \cdot 6!}  + \dfrac{\log\left( \frac{0.22}{0.1775}\right)\log\left( \frac{0.1775}{0.065} \right)^5  }{0.065 \cdot 5!}  +  \dfrac{\log\left( \frac{0.29}{0.22}\right)\log\left( \frac{0.135}{0.065} \right)^5 }{0.065 \cdot 5! } \right)<\dfrac{0.0576 y}{x^b \log(x)}.
\end{align*}
\endgroup
In summary, we have shown that $1_{\mathbb{P}}(n)$ can be written as $1_{\mathbb{P}}(n) = \Psi(n) + \sum_{i=1}^{13} \Theta_i(n)$, where $\Psi(n)$ is a sum of no more than $100$ good$^\star$ and good$^{\star\star}$ functions and where  each $\Theta_i$ is non-negative and 
  \begin{align*}
 \sum_{i=1}^{13} \sum_{n \in \mathcal{B}(y)} \Theta_i(n) < {  \dfrac{ 0.9985y}{x^b\log(x)}}.
  \end{align*}
This concludes the proof of Proposition~\ref{proposition3} for the case $a > 0.61$.   \hfill $\square$

\section*{Acknowledgements}

I would like to thank my supervisor, James Maynard, for suggesting this problem and for many helpful discussions and feedback. This work was supported by an EPSRC studentship.

\medskip Regarding the revised version of my paper, I would like to thank Adam Harper, Ben Green, Sary Drappeau, Adrien Mounier, Ingo Althöfer and the anonymous referee for their feedback.

\bibliography{bibstad}
\bibliographystyle{acm}

\end{document}